\documentclass[11pt]{article}
\usepackage[a4paper]{geometry}
\usepackage{amsthm}
\usepackage{amsmath}
\usepackage{amssymb}
\usepackage{amsfonts}
\usepackage{graphicx}
\usepackage{color}
\usepackage[bf,SL,BF]{subfigure}
\usepackage{url}
\usepackage{epstopdf}
\usepackage{pdfsync}
\usepackage[colorlinks]{hyperref}
\usepackage[all]{xypic}
\usepackage{comment}

\hypersetup{
    linkcolor=blue,
}

\newlength{\fixboxwidth}
\usepackage{epsfig,amsbsy,graphicx,multirow}

\usepackage{ algorithm, algorithmic}
\renewcommand{\algorithmiccomment}[1]{\bgroup\hfill//~#1\egroup}

\usepackage[all]{xy}
\usepackage{accents}
\newcommand{\ubar}[1]{\underaccent{\bar}{#1}}

 \numberwithin{equation}{section}

\def\V{\mathfrak{V}}

\def\d{\mathfrak{d}}
\def\db{{\bf d}}

\def\R{\mathbb{R}}
\def\Rc{\mathcal{R}}

\def\S{\mathcal{S}}
\def\cP{\mathcal{P}}

\def\E{\mathbb{E}}

\def\Inv{\operatorname{Inv}}

\def\M{\mathcal{M}}

\def\C{{\bf C}}
\def\Pr{\mathcal{P}}

\def\B{\mathcal{B}}

\def\er{\mathcal{E}}

\def\one{\mathrm{I}}
\def\two{\mathrm{II}}

\def\G{\mathcal{G}}
\def\D{\mathcal{D}}

\def\I{\mathcal{I}}
\def\J{\mathcal{J}}
\def\L{\mathcal{L}}

\def\W{\mathfrak{W}}

\def\H{\mathcal{H}}

\def\N{\mathcal{N}}

\def\<{\big\langle}
\def\>{\big\rangle}
\def\Img{\operatorname{Im}}
\def\Ker{\operatorname{Ker}}
\def\Cond{\operatorname{Cond}}

\def\diiv{\operatorname{div}}

\def\Tr{\operatorname{Trace}}
\def\Trun{\operatorname{Trun}}

\def\Card{\operatorname{Card}}

\def\lsob{{\textrm{l}}}

\def\supp{{\operatorname{support}}}
\def\support{{\operatorname{supp}}}
\def\dim{{\operatorname{dim}}}
\def\f{{\operatorname{flux}}}
\def\Span{\operatorname{span}}
\def\Pr{\operatorname{Pr}}
\def\s{\sigma}
\def\pot{\mathrm{pot}}

\def\loc{{\rm loc}}

\def\app{{\mathrm{app}}}

\definecolor{red}{rgb}{0.9, 0, 0}

\newtheorem{Theorem}{Theorem}[section]
\newtheorem{Proposition}[Theorem]{Proposition}
\newtheorem{Lemma}[Theorem]{Lemma}
\newtheorem{Corollary}[Theorem]{Corollary}
\newtheorem{Remark}[Theorem]{Remark}
\newtheorem{Example}[Theorem]{Example}
\newtheorem{NE}[Theorem]{Illustration}
\newtheorem{Definition}[Theorem]{Definition}
\newtheorem{Construction}[Theorem]{Construction}

\newtheorem{Condition}[Theorem]{Condition}

\begin{document}
\title{ Universal Scalable Robust Solvers\\
from Computational Information  Games\\ and fast eigenspace adapted Multiresolution Analysis}

\date{\today}

\author{Houman Owhadi\footnote{California Institute of Technology, Computing \& Mathematical Sciences , MC 9-94 Pasadena, CA 91125, owhadi@caltech.edu} $\,$ and Clint Scovel\footnote{California Institute of Technology, Computing \& Mathematical Sciences , MC 9-94 Pasadena, CA 91125, clintscovel@gmail.com}}

\maketitle

\begin{abstract}
We show how the discovery/design of robust scalable numerical  solvers for arbitrary bounded linear operators can,  to some degree, be addressed/automated as a Game/Decision Theory problem  by reformulating the process of computing with partial information and limited  resources as that of playing underlying hierarchies of adversarial information  games. When the solution space is a Banach space $\B$ endowed with a quadratic norm $\|\cdot\|$, the optimal measure (mixed strategy) for such games (e.g. the adversarial recovery of $u\in \B$, given partial measurements  $([\phi_i, u])_{i\in \{1,\ldots,m\}}$ with $\phi_i\in \B^*$, using relative error in $\|\cdot\|$-norm as a loss) is a centered Gaussian field $\xi$ solely determined by the norm $\|\cdot\|$, whose conditioning (on measurements) produces optimal bets. When measurements are hierarchical, the process of conditioning this Gaussian field produces a hierarchy of   elementary gambles/bets (gamblets). These gamblets generalize the notion of  Wavelets and Wannier  functions in the sense that they are adapted to the norm $\|\cdot\|$ and induce a multi-resolution decomposition of $\B$  that is adapted to the eigensubspaces  of the operator defining the norm $\|\cdot\|$. When the operator is localized, we show that the resulting gamblets are localized both in space and frequency and introduce the Fast Gamblet Transform (FGT) with rigorous accuracy and (near-linear) complexity estimates. As the FFT can be used to solve and diagonalize arbitrary PDEs with constant coefficients, the FGT can be used to decompose a wide range of continuous linear operators (including arbitrary continuous linear bijections from $H^s_0(\Omega)$ to $H^{-s}(\Omega)$ or to $L^2(\Omega)$) into a sequence of independent linear systems with uniformly bounded condition numbers and leads to  $\mathcal{O}(N \operatorname{polylog} N)$  solvers and eigenspace adapted  Multiresolution Analysis (resulting in near linear complexity approximation of all eigensubspaces).
\end{abstract}

\tableofcontents
\section{Introduction}
\subsection{Motivations and historical perspectives}
\subsubsection{On universal solvers}
Is it possible to identify/design a scalable  solver that could be applied to nearly all linear operators?
One incentive to  ask this question is the vast and increasing  literature on the numerical approximation of linear operators
where the number of linear solvers seems to trail the number of possible linear systems.
  Paraphrasing Sard's assertion, one reason not to ask this question is the historical presupposition  that  \cite[pg.~223]{sard1967optimal} ``of course no one method of approximation of a linear operator can be universal.''
 Indeed,  this assertion is reasonable and is now rigorously supported by No Free Lunch theorems in Learning Theory (see
  \cite[Thm.~7.2]{devroye2013probabilistic} and
  \cite{wolpert1996lack})  and in  Optimization  \cite{wolpert1997no}.  However, such profound results do not preclude
 the existence of weak assumptions under which  universal algorithms may  exist. For example,  the recent success of Support Vector Machines
 \cite{steinwart2008support} is  an
 astonishing example which has transformed Learning Theory.
In this paper we investigate the possibility of achieving some degree of universality in answering this question in the setting  of linear operators on Banach spaces with quadratic norms (which include matrix equations). We show that, to some degree, a positive answer can be obtained when the operator is a continuous bijection
  under the following conditions on the image space: existence of a compact embedding and a multi-resolution decomposition (thereby generalizing the results of \cite{OwhadiMultigrid:2015}). The only condition on the actual operator is that
 it is a continuous bijection.

\subsubsection{On the game theoretic approach to numerical analysis}
Another purpose of this paper is to show that the discovery of these solvers can, to some degree, be automated through  a
 game/decision theoretic approach to numerical approximation and algorithm design  based on the observations that (1) to compute fast one must compute with partial information over hierarchies of increasing levels of complexity
(2) computing efficiently with partial information requires solving minimax problems against the missing information (3) these minimax problems are repeated and
mixed (game theoretic) optimal strategies emerge as natural solutions (and lead to natural Bayesian interpretations of the resulting methods and approximation errors).

  The connection between Information Theory and Numerical Analysis emerges naturally from the Information Based Complexity \cite{Woniakowski1986, Packel1987, Traub1988, Nemirovsky1992, Woniakowski2009} notion that computation on a continuous space (infinite dimensional space) can only be done with partial information.
 Here this notion will  be expanded to the principle that fast computation requires computing with partial information over a hierarchy of levels of complexity. Consider, for instance, the problem of inverting a $10^6\times 10^6$ matrix. A method that would require computing with all the entries of that matrix at once would lead to a slow method. To obtain a fast method one must compute with a few number of features of that matrix and these features typically do not represent all the entries of the matrix. Therefore one must bridge the information gap between these few features and the whole matrix. To be made efficient, this
  principle must be repeated over a hierarchy (e.g.  information gaps must be bridged between $4$,  $16$,  $64, \dots$ degrees of freedom). This principle  is evidently present in classical fast solvers such as
 multigrid methods \cite{Fedorenko:1961, Brandt:1973, Hackbusch:1978, Hackbusch:1985, RugeStuben1987, Stuben:2001}, multilevel finite element splitting \cite{Yserentant1986},  multilevel preconditioning \cite{Vassilevski89}, stabilized hierarchical basis methods \cite{Vassilevski1997sirev, VassilevskiWang1997a, VassilevskiWang1998},  multiresolution methods
 \cite{Beylkin:1995,Beylkin:1998,Beylkin:1998b, DorobantuEngquist1988, EngquistRunborg02},
the Fast Multipole Method \cite{GreengardRokhlin:1987}, Hierarchical matrices \cite{HackbuschGrasedyck:2002, Bebendorf:2008},
 Cholesky and multigrid solvers for Graph Laplacians \cite{kyng2016approximate, kyng2016sparsified} and fast solvers for Symmetric Diagonally Dominant Matrices \cite{cohen2014solving, spielman2004nearly, spielman2014nearly, koutis2010approaching, kelner2013simple}.
 While, for classical solvers,  these information gaps have been bridged by essentially guessing the form of interpolation operators, here we will, as in \cite{OwhadiMultigrid:2015} (see also \cite{OwZh:2016, owhadi2016, SchaeferSullivanOwhadi17}), reformulate the process of bridging these gaps as that of playing repeated adversarial games against the missing information and identify optimal mixed strategies for playing such games over hierarchies of
 increasing levels of complexity.

\subsubsection{On universal optimal recovery measures}\label{subsecdidhi3}
The game theoretic approach is relevant to numerical analysis/approximation for two main reasons: (1) Inaccurate approximations, in repeated intermediary calculations, lead to loss in CPU time and the total CPU time required to invert a given linear operator is the sum of these losses. Therefore finding optimal strategies for the repeated games describing intermediate numerical approximation steps translates into the minimization of the overall required CPU time. (2) As exposed in the reemerging field of probabilistic numerics/computing \cite{ChkrebtiiCampbell2015, schober2014nips, Owhadi:2014,Hennig2015, Hennig2015b, Briol2015, Conrad2015, OwhadiMultigrid:2015,  OwZh:2016, cockayne2016probabilistic, perdikaris2016multifidelity, raissi2017inferring, Cockayne2017,  SchaeferSullivanOwhadi17} (we refer to Section \ref{sec_correspondence} for an overview) by using a probabilistic description of numerical errors it is possible to seamlessly combine model and numerical errors in an encompassing Bayesian framework.
However, while  confidence intervals obtained from arbitrary priors may  be hard to justify to a numerical analyst, worst case measures (identified as optimal mixed strategies) are robust in  adversarial environments.
In this paper we will show (Section \ref{seccigsec}) that given a Banach space $\B$ endowed with a quadratic norm $\|\cdot\|$, one can identify a  Gaussian cylinder measure, solely determined by the norm $\|\cdot\|$, whose conditioning produces optimal mixed and pure strategies for playing the adversarial games inherent to numerical approximation. This measure is universal in the sense that it does not depend on the partial information entering in the numerical approximation problem. Furthermore its conditioning produces not only
a saddle point in the game theoretic formulation of numerical approximation, but also optimal recovery solutions  \cite{bounds1959michael, micchelli1977survey} that are optimal in the deterministic minimax formulation. In that sense these universal optimal
recovery measures form a natural bridge between probabilistic numerics and classical numerical analysis/approximation.

\paragraph{On the relation with Decision Theory.}
One of the landmark discoveries in Wald's theory of Statistical Decision Functions \cite{Wald:1945,Wald:essentially}, was the result that, under mild conditions, the optimal statistical decision function was obtained by extending the corresponding two person game to its mixed extension, obtaining a worst case measure as one component of a saddle point of the mixed extension, and then for the primary player to play as if the second player (nature) used this worst case  measure as their strategy. In  Section \ref{seccigsec} we  show that the minmax optimal solution to an optimal recovery problem can be  recovered in a similar way. However, in this case, the scaling properties of the loss function
of Section
\ref{subsecidyidudg2} lead  to a mixed extension
  which incorporates those properties. Although no true measure can
 be a component in a saddle point (i.e.~be a worst case measure)
 for this extension, we  obtain (Section \ref{sec_worstcase}) approximate saddle points to any degree of approximation using
Gaussian measures whose covariance operators are easily constructed (computable). These Gaussian measures converge to a Gaussian cylinder measure whose covariance  $Q:\B^{*}\rightarrow \B$
 is the same as that determining the inner product $\langle \cdot, \cdot\rangle :=[Q^{-1}\cdot,\cdot]$
of $\B$, establishing that such a Gaussian cylinder measure is a {\em universal} worst case (weak) probability measure, since, as discussed above, it is independent of the choice of the measurement functions (this {\em universal} Gaussian cylinder measure is an  isometry  from $\B^*$ to Gaussian space characterized by the fact that the image of $\phi \in \B^*$ is a real valued centered Gaussian random variable with variance $\|\phi\|_*^2$ where $\|\cdot\|_*$ is the dual norm of $\|\cdot\|$).

 \subsubsection{On operator adapted wavelets}

 Wavelets \cite{mallat1989theory, daubechies1990wavelet, coifman1992wavelet} have transformed signal and image processing.
Could they have a similar impact on numerical analysis?
This question has stimulated the development of adapted/adaptive wavelets aimed at  solving PDEs (or boundary integral equations) \cite{beylkin1991fast, Bacrywav92, Bacrywav92a, alpert1993wavelet, jawerth1993wavelet, Dahlkewav93, Dahlkewav94, Bertoluzzawav94, engquist1994fast, Vasilyevwav96, carnicer1996local, Monassewav98, Holwav98, Cogenwav00, Liandrat01, Cohendevrewav01, Barinkaetal01, Cohenalwav02, cohen2004adaptive, Sudarshanwav05, Dahmenwav05, Schwabwav08, Beylkinwav08, EngquistRunborg2009, yin2016coupling} or performing MRA on the solutions of PDEs \cite{Frohlwav94, Fargewav98, Sendovwav02}.
While first generation adaptive wavelets (such as   bi-orthogonal wavelets  \cite{cohen1992biorthogonal}, see \cite{stevenson2009adaptive} for an overview) can be constructed with
arbitrarily high preassigned  regularity (for adaptation to the regularity of the elements of the solution space of the operator)
and can replace  mesh refinement \cite{Cohenwav02} in numerical approximation (as an adaptation to the local regularity of a particular solution) their shift (and possible scale) invariance prevents their adaptation to irregular domains or non-homogeneous coefficients.

Second generation wavelets \cite{sweldens1998lifting, VassilevskiWang1997a, VassilevskiWang1998, Sudarshanwav05} (see \cite[Sec.~1.2]{Sudarshanwav05} for an overview) offer stronger
adaptability at the cost of a possible loss in shift and scale invariance.
The main idea of second generation wavelets is to start with a  \cite{sweldens1998lifting} ``lazy'' multiresolution decomposition of the solution space (such as hierarchical basis methods \cite{Yserentant1986, BankYserentant88}) that may not possess desirable properties (such as scale orthogonality with respect to the scalar product defined by the operator and vanishing polynomial moments) and then modify the hierarchy of basis functions to  achieve desirable properties, using  construction techniques
such as the  \emph{lifting scheme} of Sweldens, the \emph{stable construction} technique of Carnicer, Dahmen and Pe\~{n}a
\cite{carnicer1996local}, the orthogonalization procedure of Lounsbery et al. \cite{lounsbery1997multiresolution}, the \emph{wavelet-modified hierarchical basis} of Vassilevski and Wang
\cite{VassilevskiWang1997a, VassilevskiWang1998}, and the \emph{stable completion, Gram-Schmidt orthogonalization, and approximate Gram-Schmidt orthogonalization} of Sudarshan \cite{Sudarshanwav05}.

As emphasized in \cite[p.~83]{Sudarshanwav05} ideal adapted wavelets should be characterized by 3 properties:  (a) scale-orthogonality
(with respect to the scalar product associated with the operator norm to ensure block-diagonal stiffness matrices) (b) local support (or rapid decay) of the basis functions (to
ensures that the individual blocks are sparse) and (c) Riesz stability in the energy norm
(to ensure that the blocks are well-conditioned).
However, as discussed in \cite[p.~83]{Sudarshanwav05}, although adapted wavelets achieving 2 of these properties have been constructed,
``it is not known if there is a practical technique for ensuring all the three properties simultaneously
in general''.

In this paper we will introduce operator adapted wavelets (gamblets) exhibiting all 3 properties for local continuous linear bijections on Banach spaces.
 Gamblets are identified by conditioning the universal measure discussed Subsection \ref{subsecdidhi3} with respect to
 a (non operator adapted) multiresolution decomposition of the dual (or image) space and have (as a consequence) optimal adversarial  approximation properties (in both frameworks of optimal
recovery and game-theory).
  They are not only adapted to the regularity of the elements of the solution space but also to
 the eigen-subspaces of the operator itself (Theorem \ref{corunbcnOR}).
Through this adaptation, gamblets induce a near optimal sparse compression of the operator \eqref{corunbcnORfe}
and provide  near-linear complexity (Section \ref{secfgtas}) solutions to the problem of finding localized (Section \ref{secexpdecloc}) Wannier functions
 \cite{marzari1997maximally, kohn1959analytic, wannier1962dynamics, weinan2010localized, ozolicnvs2013compressed, OwhadiMultigrid:2015, OwZh:2016, hou2016sparse,  HouZhang2017II} (linear combinations of eigenfunctions concentrated around a given eigenvalue, that are localized in space).

\subsection{Outline of the paper}

We will present the main results and algorithms (with numerical Illustrations) in Sections \ref{secgamtrinrn} to \ref{sec8} and  proofs (along with further results) in Sections \ref{sec6} to  \ref{sec8proof}.
Section \ref{secgamtrinrn} presents the Gamblet Transform for the linear system $A x=b$ in $\R^N$ and for
 arbitrary continuous linear bijections mapping $H^s_0(\Omega)$ to $H^{-s}(\Omega)$ or to $L^2(\Omega)$. The main purpose of  Section \ref{secgamtrinrn} is to, at the cost of some redundancy,  facilitate the accessibility of the paper.
Section \ref{sec1or} introduces the Gamblet Transform (and its discrete version) on a Banach space $\B$ gifted with a quadratic norm $\|\cdot\|$.  The Gamblet Transform, which could be seen as a generalized  Wavelet Transform  \cite{daubechies1992ten, meyer1989orthonormal}  that is adapted to $(\B,\|\cdot\|)$, turns a multiresolution decomposition of $\B^*$ into a multiresolution decomposition of $\B$ with basis functions, called gamblets, that span orthogonal subspaces of $\B$,  akin to an eigenspaces,  enabling the multi-resolution decomposition of any element $u\in \B$ into components that are localized in \emph{space} and \emph{frequency}. As the Fourier Transform can be used to solve linear PDEs with constant coefficients,  Section  \ref{secsolvinvpb} shows that the Gamblet Transform can be used to transform an arbitrary continuous linear operator $\L$ mapping $\B$ to another Banach space $\B_2$ into a sequence of independent linear systems with uniformly bounded condition numbers.
Section \ref{seccigsec} introduces the Computational Information Games framework and shows how gamblets can be discovered and interpreted as elementary gambles/bets enabling computation with partial information of
 hierarchies of increasing levels of complexity. Sequences of approximations form a martingale under the mixed strategy emerging from the underlying games and underlying approximation errors are decomposed as sums of independent Gaussian fields acting at different levels of resolution. In particular Section \ref{seccigsec} provides a probabilistic description of numerical errors (in terms of posterior distributions) that can be used, as in Probabilistic Numerics or Scientific Computing, to seamlessly combine numerical errors with model errors in an encompassing Bayesian framework.
Section \ref{secexpdecloc} proves that gamblets are localized (exponentially decaying) based
 on properties of the dual space $\B^*$ or the image space $\B_2$. Therefore gamblets could  also be seen, as a generalization of Wannier basis functions \cite{marzari1997maximally, kohn1959analytic, wannier1962dynamics}. This exponential decay also provides a rigorous justification of the \emph{screening effect} seen in Kriging \cite{stein2002screening} where conditioning a spatial random field on a (homogeneously distributed) cloud of points leads to exponential decay in correlations.
Based on the exponential decay,
 Section \ref{secfgtas} introduces the Fast Gamblet Transform (FGT) whose complexity is $\mathcal{O}(N \ln^{3d} N)$ to compute the hierarchy of gamblets, and $\mathcal{O}(N \ln^{d+1} N)$ to decompose $u\in \B$ over the gamblet basis or invert a linear system $\L u =g$ with $g\in \B_2$. Since gamblets induce a multiresolution decomposition of $\B$  that is adapted to the eigensubspaces of the operator $Q:\B^*\rightarrow \B$ defining the norm $\|\cdot\|$, the
 FGT can  not only be used as a fast solver, but also as fast projection on approximations of these eigensubspaces, as a fast operator compression algorithm, as
 a near-linear complexity PCA algorithm \cite{jolliffe2002principal}, or  as a near-linear complexity active subspace decomposition method \cite{constantine2014active, liu2012active}.
Finally, Section \ref{sec_correspondence}  provides a short review of the reemerging and fascinating interplay between the fields of
  Numerical Analysis, Approximation Theory, and Statistical Inference that provides both the historical
 background of the paper and indicates future developments that are currently available.

\subsection{On the degree of universality of the method}
Given a Banach space $(\B,\|\cdot\|)$ and a nested hierarchy of measurement functions
 (a hierarchy of elements $\phi_i^{(k)} \in \B^*$
such that each level $k$ measurement function $\phi_i^{(k)}$ is a linear combination of level $k+1$ measurement functions $\phi_j^{(k+1)}$), the method (the gamblet transform), produces under \emph{stability conditions}, a hierarchy of localized elements of $\B$ that are scale-orthogonal with respect to the scalar product induced by $\|\cdot\|$ with well conditioned stiffness matrices.
These  \emph{stability conditions} (Conditions \ref{cond1OR} and \eqref{eqljdhelkjdhkh3}) are conditions involving the interplay
between the norm $\|\cdot\|$ (or equivalently its dual form) and the measurement functions $\phi_i^{(k)}$.
The universality of the method is derived from the fact that, given the measurement functions $\phi_i^{(k)}$,   these \emph{stability conditions} are an invariant (modulo proportional change of constants) of the equivalent class of the norm $\|\cdot\|$, i.e. if these \emph{stability conditions} are satisfied by another norm $\|\cdot\|_1$ on $\B$ such that $C_1 \|\cdot\|_1\leq \|\cdot\|\leq C_2 \|\cdot\|_1$, then they must also be satisfied by $\|\cdot\|$ (with constants scaled by $C_1$ and $C_2$).
As a consequence, if $\|\cdot\|$ is the operator norm of a continuous linear bijection $\L$ between $(\B,\|\cdot\|)$  and another Banach space
$(\B_2,\|\cdot\|_2)$ (with  quadratic norm $\|\cdot\|_2$) then, as shown in Theorem \ref{thmsjdhdhgd}, the \emph{stability conditions} can be expressed as conditions on the norm  $\|\cdot\|_2$ placed on the image space and the values of the continuity constants of $\L$ and $\L^{-1}$.  Consequently, modulo
this dependence on the continuity constants,
  these conditions are independent of the operator $\L$ itself. Said another way, the stability conditions do not depend on the structure of the operator $\L$ but only on its continuity constants.
In the Sobolev space setting of Section \ref{secgamtrinrn}, this transfer of  \emph{stability conditions} allows us to show that the method is efficient when $\|\cdot\|$ is the operator norm of an arbitrary continuous (symmetric, positive) linear bijection between $H^s_0(\Omega)$ and $H^{-s}(\Omega)$ by selecting measurement functions satisfying the required \emph{stability conditions} for the $\|\cdot\|_{H^{-s}(\Omega)}$ norm (Conditions \ref{condjehkgdedgu} and \ref{condkjehdu}).
Similarly, in the linear algebra setting of Section \ref{secgamtrinrn}, involving the inversion of the matrix system $A x=b$ (where $A$ is symmetric and positive), the
\emph{stability conditions} (Conditions \ref{conddiscrip3ordismatdisQQ} and \ref{condilwhiuhd}) are invariant
(modulo a proportional change of constants) with respect to the quadratic form defined by $A$. Therefore those \emph{stability conditions} are satisfied if
$A$ is obtained  by discretizing (using a stable numerical method) a continuous linear bijection between $H^s_0(\Omega)$ and $H^{-s}(\Omega)$.

\section{The Gamblet Transform on $\R^N$ and on Sobolev spaces}\label{secgamtrinrn}
\subsection{The exact gamblet transform on $\R^N$}
Here we write $|\cdot|$ for the Euclidean norm on $\R^N$.
Consider an $N\times N$ symmetric positive definite matrix  $A$, and
let  $|\cdot|_A$ denote the $A$-norm defined by $|u|_A:=\sqrt{u^T A u}$ and let
$\langle \cdot,\cdot \rangle_{A}$ defined by $\langle u, v\rangle_{A}:=u^T A v$
denote the corresponding inner product.  We say that two vectors $u, v$ of $\R^N$ are $A$-orthogonal if they are orthogonal with respect to the
$A$-inner product, i.e.~if $u^T A v=0$.
 For $V$ a linear subspace of $\R^N$, note that $v\in V$ is the $A$-orthogonal projection
of $u$ on $V$ if $v=\operatorname{argmin}_{w}|u-w|_A$. Hereafter we refer to such $A$-orthogonal projections
simply as $A$-projections.

Relabel $\{1,\ldots,N\}$ using an index tree $\I^{(q)}$ of depth $q\in \mathbb{N}^*$  defined below.
 \begin{Definition}\label{defindextree}
We say that $\I^{(q)}$ is an index tree of depth $q$   if it is the finite set of $q$-tuples of the form $i=(i_1,\ldots,i_q)$. For $1\leq k \leq q$ and $i=(i_1,\ldots,i_q)\in \I^{(q)}$,  write
$i^{(k)}:=(i_1,\ldots,i_k)$  and $\I^{(k)}:=\{i^{(k)}\,:\, i\in \I^{(q)}\}$.
\end{Definition}
Write $I^{(k)}$ for the $\I^{(k)}\times \I^{(k)}$ identity matrix.
\begin{Construction}\label{constpiQQ}
For $k\in \{1,\ldots,q-1\}$ let $\pi^{(k,k+1)}$ be a $\I^{(k)}\times \I^{(k+1)}$ matrix such that $\pi^{(k,k+1)}(\pi^{(k,k+1)})^T=I^{(k)}$.
\end{Construction}
 Algorithm \ref{alggambletcomutationnesQQ} below (the gamblet transform) computes a change of basis on $\R^N$ through an $A$-orthogonal decomposition
\begin{equation}\label{eqkjhdkejdiuhQQ}
\R^N=\V^{(1)}\oplus_A \W^{(2)}\oplus_A \cdots \oplus_A \W^{(q)}\,,
\end{equation}
where  $\oplus_A$ is the $A$-orthogonal direct sum,
the terms of which will be defined shortly. To analyze the performance of Algorithm \ref{alggambletcomutationnesQQ}, we will develop some conditions on $A$ regarding its relationship with the matrices $\pi^{(k,k+1)}$ of Construction \ref{constpiQQ}. To that end,
for $k\in \{1,\ldots,q\}$ write
 \begin{equation}\label{eqpikq}
 \pi^{(k,q)}:=\pi^{(k,k+1)}\pi^{(k+1,k+2)}\cdots \pi^{(q-1,q)}
 \end{equation}
  and let $\pi^{(q,k)}$ be the transpose of $\pi^{(k,q)}$. Write $\pi^{(q,q)}=I^{(q)}$.
\begin{Condition}\label{conddiscrip3ordismatdisQQ}
There exists constants
$C_\d\geq 1$ and $H\in (0,1)$  such that the following conditions are satisfied.
\begin{enumerate}
\item\label{linconmat5disQQ}  $\frac{1}{C_\d \sqrt{\lambda_{\min}(A)}}H^k  \leq \inf_{x\in \Img(\pi^{(q,k)}) } \frac{\sqrt{x^T A^{-1}x}}{|x|}$ for $k\in \{1,\ldots,q\}$.
\item\label{linconmat6disQQ}
$\sup_{x\in \Ker(\pi^{(k,q)}) }\frac{\sqrt{x^T  A^{-1}x}}{|x|}\leq \frac{C_\d}{ \sqrt{\lambda_{\min}(A)}} H^{k}$ for $k\in \{1,\ldots,q-1\}$.
\end{enumerate}
\end{Condition}
\begin{Remark}\label{rmklxkclkjdlf}
Item \ref{linconmat6disQQ} of Condition \ref{conddiscrip3ordismatdisQQ} is equivalent to the Poincar\'{e} inequality
\begin{equation}
\inf_{y\in \R^{\I^{(k)}}} |z-\pi^{(q,k)}y|\leq C_\d  H^k \frac{\sqrt{z^T  A z}}{\lambda_{\min}(A)}\text{ for $z\in \R^N$  and $k\in \{1,\ldots,q\}$.}
\end{equation}
Item \ref{linconmat5disQQ} of Condition \ref{conddiscrip3ordismatdisQQ} is implied by the inverse Poincar\'{e} inequality
\begin{equation}
C_\d^{-1}  H^k \frac{\sqrt{x^T  A x}}{\lambda_{\min}(A)}\leq |x| \text{ for $x\in \Img(\pi^{(q,k)})$ and  $k\in \{1,\ldots,q\}$.}
\end{equation}
\end{Remark}

We will show in Subsection \ref{subsecnatcondga} and the following sections that if $A$ is obtained as the discretization of a continuous linear bijection $\L$ between Sobolev spaces, then  matrices $\pi^{(k,k+1)}$ satisfying Condition \ref{conddiscrip3ordismatdisQQ} can naturally be identified from a multi-resolution decomposition of the image space, independently from the operator itself (i.e.~this identification is easy for linear systems obtained by discretizing continuous linear bijections between Sobolev spaces). Graph Laplacians are other
prototypical examples of practical importance \cite{belkin2001laplacian} and Remark \ref{rmklxkclkjdlf} implies that  Condition \ref{conddiscrip3ordismatdisQQ} represent natural analytical (Poincar\'{e} and inverse Poincar\'{e}) inequalities involving the interplay between the matrices  $\pi^{(k,k+1)}$ and the structure of the Graph Laplacian.

 Although \eqref{eqkjhdkejdiuhQQ} is not an exact eigenspace decomposition, the following  Theorem shows,
  under Condition \ref{conddiscrip3ordismatdisQQ},
 that  it shares many of its important characteristics.
\begin{Theorem}\label{thmdpedjoejdo}
Under Condition \ref{conddiscrip3ordismatdisQQ} it holds true that, for some constant $C$ depending only on $C_\d$,
\begin{equation}
C H^{2} \leq \lambda_{\min}(A) \frac{v^T A v}{|A v|^2}\leq C \text{ for } v\in \V^{(1)}
\end{equation}
and for $k\in \{2,\ldots,q\}$,
\begin{equation}
C H^{2k} \leq \lambda_{\min}(A) \frac{v^T A v}{|A v|^2}\leq C H^{2k-2}\text{ for } v\in \W^{(k)}
\end{equation}
\end{Theorem}

Given $b\in \R^N$, Algorithm \ref{alggambletcomutationnesQQ} also computes the solution  $u\in \R^N$  of the linear system
\begin{equation}\label{eqkjkdjdhjQQ}
A u=b
\end{equation}
and performs the $A$-orthogonal decomposition of $u$ over the right hand side of
\eqref{eqkjhdkejdiuhQQ}, i.e.
\begin{equation}\label{esdddkejdiuhQQ}
u=v^{(1)}+v^{(2)}+\cdots+v^{(q)}\,.
\end{equation}
 Since the decomposition in \eqref{esdddkejdiuhQQ} is $A$-orthogonal, $v^{(1)}$ is the
$A$-projection of $u$ on $\V^{(1)}$ and, for $k\in \{2,\ldots,q\}$, $v^{(k)}$  is the
$A$-projection of $u$ on $\W^{(k)}$. Write $u^{(1)}:=v^{(1)}$ and, for $k\in \{2,\ldots,q\}$,
 $u^{(k)}:=v^{(1)}+v^{(2)}+\cdots+v^{(k)}$ for
the corresponding sequence of successive approximations of $u$ with $A$-orthogonal increments.
Writing $\V^{(k)}=\V^{(1)}\oplus_A \W^{(2)}\oplus_A \cdots \oplus_A \W^{(k)}$,
 it follows that $\V^{(k)}\subset \V^{(k+1)}$ and $u^{(k)}$ is the $A$-projection of $u$ on $\V^{(k)}$.

The following Theorem demonstrates how  Condition \ref{conddiscrip3ordismatdisQQ} implies performance guarantees when using  Algorithm \ref{alggambletcomutationnesQQ} to solve the linear system \eqref{eqkjkdjdhjQQ}.
\begin{Theorem}\label{thmljshjhxj6QQ}
We have $u^{(q)}=u$ and, under Condition \ref{conddiscrip3ordismatdisQQ}, there exists a constant $C$ depending only on $C_\d$ such that, for $k\in \{1,\ldots,q-1\}$, we have
\begin{equation}\label{corunbcnORfeQQ}
|u-u^{(k)}|_A \leq \frac{C}{\sqrt{\lambda_{\min}(A)}}  H^k |b|\, .
\end{equation}
\end{Theorem}

Let us now describe Algorithm \ref{alggambletcomutationnesQQ}.
For $1<r<k$ and a $k$-tuple of the form $i=(i_1,\ldots,i_k)$ we write $i^{(r)}:=(i_1,\ldots,i_r)$.
\begin{Construction}\label{consjk}
For $k\in \{2,\ldots,q\}$ let $\J^{(k)}$ be a finite set of $k$-tuples of the form $j=(j_1,\ldots,j_k)$
such that $\{j^{(k-1)}\mid j\in \J^{(k)}\}=\I^{(k-1)}$ and for $i\in \I^{(k-1)}$, $\Card\{ j\in \J^{(k)}\mid j^{(k-1)}=i\}=\Card\{ s\in \I^{(k)}\mid s^{(k-1)}=i\}-1$.
\end{Construction}

Write $J^{(k)}$ for the $\J^{(k)}\times \J^{(k)}$ identity matrix.
\begin{Construction}\label{conswkQQ}
For $k=2,\ldots,q$ let $W^{(k)}$ be a $\J^{(k)}\times \I^{(k)}$ matrix such that  $\Img(W^{(k),T})=\Ker(\pi^{(k-1,k)})$ and $W^{(k)}(W^{(k)})^T=J^{(k)}$.
\end{Construction}

\begin{algorithm}[!ht]
\caption{Exact Gamblet Transform/Solve on $\R^N$.}\label{alggambletcomutationnesQQ}
\begin{algorithmic}[1]
\STATE\label{step5giORQQ}  $A^{(q)}_{i,j}=A $
\STATE\label{step5giORQQ2b} $\psi_i^{(q)}=e_i$
\STATE\label{stepjhdghdQQ} $b^{(q)}_i=b_i$
\FOR{$k=q$ to $2$}
\STATE\label{step7gORQQ} $B^{(k)}= W^{(k)}A^{(k)}W^{(k),T}$
\STATE\label{step8gORQQ} $w^{(k)}=B^{(k),-1} W^{(k)} b^{(k)}$
\STATE\label{step9gORQQ}  For $i\in \J^{(k)}$, $\chi^{(k)}_i=\sum_{j \in \I^{(k)}} W_{i,j}^{(k)} \psi_j^{(k)}$
\STATE\label{step10gORQQ} $v^{(k)}=\sum_{i\in \J^{(k)}}w^{(k)}_i \chi^{(k)}_i$
\STATE\label{step11gORQQ}  $ N^{(k)}= A^{(k)} W^{(k),T} B^{(k),-1}$
\STATE\label{step12gORQQ} $R^{(k-1,k)}=\pi^{(k-1,k)}(I^{(k)}- N^{(k)} W^{(k)})$
\STATE\label{step15gORQQ} $b^{(k-1)}=R^{(k-1,k)} b^{(k)}$
\STATE\label{step13gORQQ} $A^{(k-1)}= R^{(k-1,k)}A^{(k)}R^{(k,k-1)}$
\STATE\label{step14gORQQ} For $i\in \I^{(k-1)}$, $\psi^{(k-1)}_i=\sum_{j \in  \I^{(k)}} R_{i,j}^{(k-1,k)} \psi_j^{(k)}$
\ENDFOR
\STATE\label{step16gORQQ} $ w^{(1)}=A^{(1),-1}b^{(1)}$
\STATE\label{step17gORQQ} $v^{(1)}=\sum_{i \in \I^{(1)}} w^{(1)}_i \psi^{(1)}_i$
\STATE\label{step18gORQQ} $u=v^{(1)}+v^{(2)}+\cdots+v^{(q)}$
\end{algorithmic}
\end{algorithm}

Algorithm \ref{alggambletcomutationnesQQ} takes $A$, $b$, the matrices $\pi^{(k,k+1)}$ and $W^{(k)}$ as inputs and produces the following outputs: (1) $u$ the solution of \eqref{eqkjkdjdhjQQ} and its decomposition \eqref{esdddkejdiuhQQ} (2) families of nested vectors of $\R^N$, $(\psi_i^{(k)})$ indexed by $k\in \{1,\ldots,q\}$ and $i\in \I^{(k)}$ spanning the nested linear subspaces $\V^{(k)}:=\operatorname{span}\{\psi^{(k)}_i \mid i\in \I^{(k)}\}$ of $\R^N$
(3) families of  vectors of $\R^N$, $(\chi_i^{(k)})$ indexed by $k\in \{2,\ldots,q\}$ and $i\in \J^{(k)}$ spanning $A$-orthogonal linear subspaces $\W^{(k)}:=\operatorname{span}\{\chi^{(k)}_i \mid i\in \J^{(k)}\}$ of $\R^N$ such that the decomposition \eqref{eqkjhdkejdiuhQQ} holds and $\V^{(k)}=\V^{(k-1)}\oplus_A \W^{(k)}$
(4) positive definite $\I^{(k)}\times \I^{(k)}$ matrices $A^{(k)}$ indexed by $k\in \{1,\ldots,q\}$
(5) positive definite $\J^{(k)}\times \J^{(k)}$ matrices $B^{(k)}$.

Since the subspaces entering in the decomposition \eqref{eqkjhdkejdiuhQQ} are $A$-orthogonal and adapted to the eigenspaces of $A$, \eqref{eqkjkdjdhjQQ} can be solved independently on each one of them by solving $q$ well conditioned linear systems. More precisely
Algorithm \ref{alggambletcomutationnesQQ}, transforms the $N\times N$ linear system \eqref{eqkjkdjdhjQQ} into $q$ independent linear systems
\begin{equation}\label{eqlinsysQQ}
A^{(1)} w^{(1)}=  b^{(1)}\text{ and }B^{(k)} w^{(k)}= W^{(k)} b^{(k)}, \, k \in \{2,\ldots,q\}\, ,
\end{equation}
and the following theorem guarantees that these systems are well conditioned.
\begin{Theorem}\label{thmconddisbndisbismatdisQQ}
Under Condition \ref{conddiscrip3ordismatdisQQ}, there exists a constant $C$ depending only on $C_\d$ such that
$C^{-1}  I^{(1)} \leq \frac{A^{(1)}}{\lambda_{\min}(A)}\leq C H^{-2} I^{(1)}$, $\operatorname{Cond}(A^{(1)})\leq  C H^{-2}$.
and for $k\in \{2,\ldots,q\}$
$ C^{-1}H^{-2(k-1)} J^{(k)}   \leq \frac{B^{(k)}}{\lambda_{\min}(A)} \leq C H^{-2k} J^{(k)}$
   and $\operatorname{Cond}(B^{(k)})\leq  C H^{-2}$.
\end{Theorem}

Let $R^{(k,k+1)}$ be the $\I^{(k)}\times \I^{(k+1)}$ interpolation matrices computed in Algorithm \ref{alggambletcomutationnesQQ}, write $R^{(k,q)}=R^{(k,k+1)}\cdots R^{(q-1,q)}$ and let $R^{(q,k)}$ be its $\I^{(q)}\times \I^{(k)}$ matrix transpose.
The matrices $A^{(k)}$ correspond to a $\I^{(k)}\times \I^{(k)}$ matrix compression of $A$ in the sense that, under Condition \ref{conddiscrip3ordismatdisQQ},
\begin{equation}\label{eqkjdkjdhdh}
\big|(A^{-1}  - R^{(q,k)} A^{(k),-1} R^{(k,q)}) b \big|_A \leq \frac{C}{\sqrt{\lambda_{\min}(A)}}  H^k |b|
\end{equation}
for $b \in \R^N$, where the constant $C$ depends only on $C_\d$. The compression \eqref{eqkjdkjdhdh} can be interpreted as  numerical homogenization and \cite{BeOw:2010} and \cite[Sec.~2.5]{OwhadiMultigrid:2015} establish its optimality (up to a multiplicative constant when compared to an exact eigenspace decomposition).

The vectors $\psi_i^{(k)}$ and $\chi_i^{(k)}$ can be interpreted as algebraic wavelets
 (in Section \ref{seccigsec}  we will  justify the name {\em gamblets}) adapted to the matrix $A$.  The vectors  $\psi_i^{(k)}, i \in \I^{(k)}$, spanning $\V^{(k)}$, satisfy the nesting relation $\psi_i^{(k)}=\sum_{j\in \I^{(k+1)}} R^{(k,k+1)}_{i,j}\psi_j^{(k+1)}$. The vectors
 $\chi_i^{(k)}=\sum_{j\in \I^{(k)}} W^{(k)}_{i,j}\psi_j^{(k)}, i\in \J^{(k)}$, spanning
$ \W^{(k)}$, are obtained by $A$-orthogonalization with respect to
the decomposition $\V^{(k)}=\V^{(k-1)}\oplus_A \W^{(k)}$.
Furthermore for $k\in \{1,\ldots,q\}$ and $w\in \R^{(k)}$, the minimizer of $|v|_A$ over $v\in \R^N$ subject to $w=\pi^{(k,q)}  v$ is
$v=\sum_{i\in \I^{(k)}} w_i \psi_i^{(k)}$.

\subsection{The gamblet transform for arbitrary symmetric linear operators on Sobolev spaces}\label{subsecnatcondga}

\begin{Definition}\label{defkdjdh99}
Given $\Omega$ a bounded open subset of $\R^d$ with uniformly lipschitz  boundary and $s\in \mathbb{N}$, let
$H^s(\Omega)$ be the Sobolev space \cite[Sec.~2.2.1]{gazzola2010polyharmonic} gifted with the norm
\begin{equation}
\|u\|^2_{H^s(\Omega)}:=\sum_{t=0}^s \|D^t u\|^2_{L^2(\Omega)}\text{ for } u\in H^s(\Omega)
\end{equation}
where $D^t u$ is the total derivative of $u$ of order $t$,  $D^0 u=u$, $|D^t u|=(D^t u \cdot D^t u)^\frac{1}{2}$ and $D^t u \cdot D^t v=\sum_{i_1,\ldots,i_t=1}^d \frac{\partial^t u}{\partial_{i_1}\cdots \partial_{i_t} }\frac{\partial^t v}{\partial_{i_1}\cdots \partial_{i_t} }$.
Write $H^s_0(\Omega)$ for the closure of the set of smooth functions with compact support in $\Omega$ with respect to the norm $\|\cdot\|_{H^s(\Omega)}$. Write $\Delta^k$ the $k$-th iterate of the Laplacian.
For $s=2k$, write $\|u\|_{H^s_0(\Omega)}=\|\Delta^k u\|_{L^2(\Omega)}$. For $s=2k+1$ write  $\|u\|_{H^s_0(\Omega)}=\|\nabla \Delta^k u\|_{L^2(\Omega)}$. Recall  \cite[Thm.~2.2]{gazzola2010polyharmonic}  that $\|\cdot\|_{H^s_0(\Omega)}$ defines a norm that is equivalent to $\|\cdot\|_{H^s(\Omega)}$ on $H^s_0(\Omega)$. Let $(H^{-s}(\Omega),\|\cdot\|_{H^{-s}(\Omega)})$ be the dual of
 $(H^s_0(\Omega),\|\cdot\|_{H^s_0(\Omega)})$ using the usual dual pairing obtained from the  Gelfand triple
$H^s_0(\Omega) \subset L^2(\Omega) \subset H^{-s}(\Omega)$  of Sobolev spaces;
  for $g\in H^{-s}(\Omega)$, $\|g\|_{H^{-s}(\Omega)}=\sup_{v\in H^s_0(\Omega)}\frac{\int_{\Omega} gv}{\|v\|_{H^s_0(\Omega)}}$.
\end{Definition}

Let
\begin{equation}\label{eqkjdlhejdjhii}
\L\,:\, H^s_0(\Omega)\rightarrow H^{-s}(\Omega)
\end{equation}
be an  symmetric continuous linear bijection  between $H^s_0(\Omega)$ and $H^{-s}(\Omega)$,
 and write $C_{\L}:=\sup_{u\in \H^s_0(\Omega)} \|\L u\|_{H^{-s}(\Omega)}/ \|u\|_{H^s_0(\Omega)}$ and $C_{\L^{-1}}:=\sup_{u\in \H^s_0(\Omega)}  \|u\|_{H^s_0(\Omega)}/\|\L u\|_{H^{-s}(\Omega)}$ for its continuity constants.
Write   $[\cdot,\cdot]$ for the  duality pairing between $H^{-s}(\Omega)$ and
 $H^s_0(\Omega)$ defined by the $L^2(\Omega)$  integral $[g,v]:=\int_{\Omega} gv,\, g \in H^{-s}(\Omega), v\in H^s_0(\Omega)$.
 Let $\|\cdot\|$ be the norm on $H^s_0(\Omega)$ defined by
\[ \|u\|:=[\L u, u]^\frac{1}{2}\]
 and write $\<\cdot,\cdot\>$  for the corresponding scalar product.

Let $\I^{(k)}$  and $\pi^{(k,k+1)}$ be defined as in Definition \ref{defindextree} and Construction \ref{constpiQQ}. Let $(\phi_i^{(q)})_{i\in \I^{(q)}}$ be orthonormal elements of $L^2(\Omega)$ and, for $k\in \{1,\ldots,q-1\}$ and $i\in \I^{(k)}$ define $\phi_i^{(k)}$ via induction by
\begin{equation}\label{eqndkjndnfj}
\phi_i^{(k)}=\sum_{j\in \I^{(k+1)}} \pi^{(k,k+1)}_j \phi_j^{(k+1)}
\end{equation}

For $k\in \{1,\ldots,q\}$ and $i\in \I^{(k)}$, let $\psi_i^{(k)}$ be the minimizer of the following variational problem.
\begin{equation}\label{eqminuhsob}
\begin{cases}
\text{Minimize }\|\psi\|\\
\text{Subject to }\psi\in H^s_0(\Omega)\text{ and }[\phi_j^{(k)},\psi]=\delta_{i,j}\text{ for }j\in \I^{(k)}
\end{cases}
\end{equation}

We will call the elements $\psi_i^{(k)}$ gamblets.
For $k\in\{1,\ldots,q\}$ let $\Theta^{(k)}$  be the $\I^{(k)}\times \I^{(k)}$ symmetric matrix defined by
$\Theta^{(k)}_{i,j}:=[\phi_i^{(k)},\L^{-1} \phi_j^{(k)}]$.
\begin{Theorem}\label{thmfix1}
For $k\in\{1,\ldots,q\}$, $\Theta^{(k)}$ is positive definite, and writing $\Theta^{(k),-1}$ its inverse, we have for $i\in \I^{(k)}$,
\begin{equation}
\psi_i^{(k)}=\sum_{j\in \I^{(k)}} \Theta^{(k),-1}_{i,j} \L^{-1} \phi_j^{(k)}
\end{equation}
\end{Theorem}

For $k\in \{2,\ldots,q\}$, let $\J^{(k)}$ and  $W^{(k)}$ be as in Constructions \ref{consjk} and \ref{conswkQQ}.
For $k\in \{2,\ldots,q\}$ and $i\in \J^{(k)}$ write
\begin{equation}
\chi^{(k)}_i:=\sum_{j\in \I^{(k)}}W^{(k)}_{i,j}\psi^{(k)}_j
\end{equation}
For $k\in \{1,\ldots,q\}$ write $\V^{(k)}:=\Span\{\psi_i^{(k)}\mid i\in \I^{(k)}\}$ and for $k\in \{2,\ldots,q\}$ write
$\W^{(k)}:=\Span\{\chi_i^{(k)}\mid i\in \J^{(k)}\}$. Let $\W^{(q+1)}$ be the $\<\cdot,\cdot\>$-orthogonal complement of
$\V^{(q)}$ in $H^s_0(\Omega)$.  Write $\oplus$  for the $\<\cdot,\cdot\>$-orthogonal direct sum.

\begin{Theorem}\label{thmfix2}
For $k\in \{2,\ldots,q\}$, $\V^{(k-1)}\subset \V^{(k)}$ and $\V^{(k)}=\V^{(k-1)}\oplus \W^{(k)}$. In particular,
\begin{equation}
H^s_0(\Omega)=\V^{(1)}\oplus\W^{(2)}\oplus \cdots \oplus \W^{(q)}\oplus \W^{(q+1)}
\end{equation}
Furthermore, $k\in \{1,\ldots,q\}$, the $\<\cdot,\cdot\>$ orthogonal projection of $u\in H^s_0(\Omega)$ onto $\V^{(k)}$ is
\begin{equation}
u^{(k)}=\sum_{i\in \I^{(k)}} [\phi_i^{(k)},u]\psi_i^{(k)}
\end{equation}
and, for $k\in \{2,\ldots,q\}$, the $\<\cdot,\cdot\>$ orthogonal projection of $u\in H^s_0(\Omega)$ onto $\W^{(k)}$ is
$u^{(k)}-u^{(k-1)}$.
\end{Theorem}

For $k\in \{1,\ldots,q\}$, write
\begin{equation}
\Phi^{(k)}:=\Span\{\phi_i^{(k)}\mid i \in \I^{(k)}\}\,.
\end{equation}

The analogue here to Condition \ref{conddiscrip3ordismatdisQQ} for the linear algebra case is as follows.
\begin{Condition}\label{condjehkgdedgu}
There exists constants
$C_{s}\geq 1$ and $H\in (0,1)$  such that the following conditions are satisfied for $k\in \{1,\ldots,q\}$.

\begin{enumerate}
\item\label{linconmat5disQQl}  $\|\phi \|_{L^2(\Omega)} \leq C_{s} H^{-k}\|\phi \|_{H^{-s}(\Omega)}$ for $\phi\in \Phi^{(k)}$.
\item\label{linconmat6disQQl} $\|\phi \|_{H^{-s}(\Omega)} \leq C_{s}  H^{k} \|\phi \|_{L^2(\Omega)}$ for $\phi\perp_{L^2(\Omega)} \Phi^{(k)}$.
\end{enumerate}
\end{Condition}

For $k\in \{1,\ldots,q\}$ let $A^{(k)}$ be the stiffness matrix of the operator $\L$ in $\V^{(k)}$, i.e.
$A^{(k)}_{i,j}=\<\psi_i^{(k)},\psi_j^{(k)}\>$ for $i,j\in \I^{(k)}$.
For $k\in \{2,\ldots,q\}$ let $B^{(k)}$ be the stiffness matrix of the operator $\L$ in $\W^{(k)}$, i.e.
$B^{(k)}_{i,j}=\<\chi_i^{(k)},\chi_j^{(k)}\>$ for $i,j\in \J^{(k)}$.

\begin{Theorem}\label{thmfix3}
Assume Condition \ref{condjehkgdedgu} to be satisfied. Then there exists a constant $C$ depending only on $C_s$, $C_\L$ and $C_{\L^{-1}}$ such that for $u\in \H^s_0(\Omega)$
\begin{equation}
\|u-u^{(k)}\|_{H^s_0(\Omega)}\leq C H^k \|\L u\|_{L^2(\Omega)}.
\end{equation}
Furthermore, $C^{-1}  I^{(1)} \leq \frac{A^{(1)}}{\lambda_{\min}(A)}\leq C H^{-2} I^{(1)}$, $\operatorname{Cond}(A^{(1)})\leq  C H^{-2}$, and for $k\in \{2,\ldots,q\}$,
$ C^{-1}H^{-2(k-1)} J^{(k)}   \leq \frac{B^{(k)}}{\lambda_{\min}(A)} \leq C H^{-2k} J^{(k)}$
   and $\operatorname{Cond}(B^{(k)})\leq  C H^{-2}$.
\end{Theorem}

\begin{Condition}\label{cond7fyf}
Assume that there exists an index tree $\bar{\I}^{(q)}$, as in Definition \ref{defindextree}, and a finite set $\aleph=\{1,\ldots,|\aleph|\}$  such that  $\I^{(k)}=\bar{\I}^{(k)}\times \aleph$ for all $k \in \{1,\ldots q\}$, i.e. independent of $k$,  each label
   $(t,\alpha)\in \bar{\I}^{(k)}\times \aleph$ is in one to one correspondence with a label $i\in \I^{(k)}$ and we write $i^{\aleph}:=t$ for $i=(t,\alpha)$. From the Construction \ref{consjk} of $\J^{k+1}$,
  we write $j^{(k),\aleph}:=t$ for $j\in \J^{(k+1)}$ and $j^{(k)}=(t,\alpha)\in \I^{(k)}$.
  Assume that  $\pi^{(k-,k)}$ and $W^{(k)}$ are cellular, i.e. (1)
 $\pi^{(k-1,k)}_{i,j}=0$ for $j^{(k-1),\aleph}\not=i^{\aleph}$  and (2) $W^{(k)}_{i,j}=0$ for $j^{(k-1),\aleph}\not=i^{(k-1),\aleph}$ for $k \in \{2,\ldots q\}$.
\end{Condition}
\begin{Remark}
\label{rem_cond7fyf}
When $|\aleph|=1$,  Condition \ref{cond7fyf} simplifies to  (1)
$\pi^{(k-1,k)}_{i,j}=0$
for $(i,j)\in \I^{(k-1)}\times \I^{(k)}$ with  $i\not=j^{(k-1)}$ and (2) $W^{(k)}_{j,i}=0$ for $(j,i)\in \J^{(k)}\times \I^{(k)}$ with  $j^{(k-1)}\not=i^{(k-1)}$. We refer to \cite[Construction~4.13]{OwhadiMultigrid:2015} for a possible cellular choice for $W^{(k)}$ so that $W^{(k)}W^{(k),T}=J^{(k)}$ when $\pi^{(k,k+1)}$ is cellular with constant non-zero entries in each cell.
\end{Remark}

 \begin{Example}\label{egkajhdlkjdini}
 Let $h, \delta \in (0,1)$.
  For an open non-void convex subset $\tau$ of  $\Omega \subset \R^{d}$ write $\cP_{s-1}(\tau)$ the space of $d$-variate polynomials on $\tau$ of degree at most $s-1$. Let $n={s+d-1 \choose d}$ be the dimension of $\cP_{s-1}(\tau)$. Let $\aleph=\{1,\ldots,n\}$.
  Let $\bar{\I}^{(q)}$ and $\I^{(q)}$ be index trees defined as in Condition \ref{cond7fyf}.
   Let  $(\tau_t^{(k)})_{t\in \bar{\I}^{(k)}}$ be convex uniformly Lipschitz convex sets forming a nested partition of $\Omega$, i.e. $\Omega=\cup_{i\in \bar{\I}^{(k)}}\tau_i^{(k)}$ is a disjoint union except for the boundaries, for $k\in \{1,\ldots,q\}$ and $\tau_i^{(k)}=\cup_{j\in \bar{\I}^{(k+1)}: j^{(k)}=i}\tau_j^{(k+1)}$ for $k\in \{1,\ldots,q-1\}$. Assume that each $\tau_i^{(k)}$,  contains a ball of radius $h^k$, and is contained in a ball of radius $\delta h^k$. For $i\in \bar{\I}^{(k)}$, let $(\phi^{(k)}_{i,\alpha})_{\alpha\in \aleph}$ be an $L^2(\tau_i^{(k)})$-orthonormal basis of $\cP_{s-1}(\tau_i^{(k)})$. Note that the set of indices $(i,\alpha)\in \bar{\I}^{(q)}\times \aleph$ form the index tree  $\I^{(q)}$ and $\I^{(k)}=\bar{I}^{(k)}\times \aleph$.
\end{Example}

 \begin{Proposition}\label{propkajhdlkjdsob}
 Condition \ref{condjehkgdedgu} is satisfied  with the measurement functions
 presented in Example \ref{egkajhdlkjdini}  with  $H=h^s$ and a constant $C_{s}$ depending only on $d, \delta$ and $s$.
\end{Proposition}

\begin{Remark}
The gamblet transform turns the resolution of  $\L u=g$ into that of a sequence of independent linear systems with uniformly bounded condition numbers. When $\L$ is a local operator, and the measurement functions are localized, Subsection \ref{subsecexpdein} shows that the resulting gamblets are also localized, the fast gamblet transform described in Subsection \ref{subsefgtin} enables the computation of gamblets in $\L u=g$ in $\mathcal{O}(N \log^{3d}(N))$ complexity and the resolution of  $\L u=g$ (up to grid-size accuracy in $H^s_0(\Omega)$-norm) in $\mathcal{O}(N \log^{d+1}(N))$ complexity.
\end{Remark}

\subsection{Exponential decay of gamblets}\label{subsecexpdein}

Consider the operator $\L$ introduced in \eqref{eqkjdlhejdjhii}.
\begin{Construction}[Locality of measurement functions]\label{consmeasfomi}
Let $h>0$ and $\delta \in (0,1)$. For $|\aleph|,|\beth| \in \mathbb{N}^{*}$, Let  $\aleph, \beth$ be the finite sets $\{1,\ldots,|\aleph|\}$  and $\{1,\ldots,|\beth|\}$. Let $(\tau_{i})_{i\in \beth}$ be a partition of $\Omega$ such that each $\tau_i$ is convex, uniformly Lipschitz, contains a ball of radius $\delta h$ and is contained in a ball of radius $h$.
Let $(\phi_{i,\alpha})_{(i,\alpha)\in \beth \times \aleph} \in H^{-s}(\Omega)$ be  such that  the support of $\phi_{i,\alpha}$ is contained in $\tau_i$. For $i\in \beth$, let $\Omega_i\subset \Omega$ such that
  $\Omega_i$  contains $\tau_i$,  $\Omega_i^c \cap \Omega$ is at distance at least $ h$ from $\tau_i$, and
 $\partial \Omega_i$  is Lipschitz.
\end{Construction}

\begin{Definition}\label{defgraphmatdistbis}
Let $\I$ be a finite index set.
 For an $\I\times \I$ matrix $X$, we defined the graph distance $\db^X$ on $\I$ as follows.
For $(i,j)\in \I \times \I$ we define  $\db^X_{i,j}$, the  graph distance of $X$ between $i$ and $j$, as the minimal length of paths connecting $i$ and $j$ within the matrix graph of $X$, i.e. $\db^X_{i,j}$ is the smallest number $m$ such that there exists indices $i_0,i_1,\ldots,i_m\in \I$ with $i_0=i$, $i_m=j$ and
$X_{i_{t-1},i_t}\not=0$ for $t\in \{1,\ldots,m\}$ with the convention that $d_{i,j}^X=+\infty$ if no such path exists. Moreover, $\db^X_{i,j}=0$ if and only if $i=j$.
\end{Definition}

Let $\C$ be the $\beth\times \beth$ (connectivity) matrix defined by $\C_{i,j}=1$ if there exists
 a $(\chi_i,\chi_j) \in H^s_0(\Omega_i) \times H^s_0(\Omega_j)$ such that $\<\chi_j,\chi_i\>\not=0$, and $\C_{i,j}=0$ otherwise.
Let $\db:=\db^{\C}$ be the graph distance on $\beth$ induced by the connectivity matrix $\C$.
For $n\in \mathbb{N}$, let $\Omega_{i,n}:=\cup_{j:\db(i,j)\leq n}\Omega_j$, and note that
$\Omega_{i,0}=\Omega_i, i \in \beth$.

For $(i,\alpha)\in \beth\times \aleph$, let $\psi_{i,\alpha}$ be the minimizer of $\|\psi\|$ over $\psi \in H^s_0(\Omega)$ subject to
$[\phi_{j,\beta},\psi]=\delta_{i,j}\delta_{\alpha,\beta}$ for $(j,\beta)\in \beth\times \aleph$.
In addition, for each subset $ \Omega_i, i \in \beth$ and
each $\alpha\in  \aleph$, let $\psi_{i,\alpha}^n$ be the minimizer of $\|\psi\|$ over
 $\psi \in H^s_0(\Omega_{i,n})$ subject to $[\phi_{j,\beta},\psi]=\delta_{i,j}\delta_{\alpha,\beta}$ for $(j,\beta)\in \beth\times \aleph$.

Write $\Phi:=\Span\{\phi_{i,\alpha}\mid (i,\alpha)\in \beth \times \aleph\}$.
\begin{Condition}\label{condkjehdu}
There exists $C_{\min}, C_{\max}>0$ such that for $\varphi \in H^{-s}(\Omega)$.
 \begin{equation}\label{eqkjddlkjdldjjiejsob}
C_{\min}\, \inf_{\phi\in \Phi}\|\varphi-\phi\|_{H^{-s}(\Omega)}^2 \leq \sum_{i\in \beth} \inf_{\phi\in \Phi}\| \varphi-\phi\|_{H^{-s}(\Omega_i)}^2\leq C_{\max}\,\inf_{\phi\in \Phi} \|\varphi-\phi\|_{H^{-s}(\Omega)}^2\,.
\end{equation}
\end{Condition}
\begin{Theorem}\label{thmfix4}
Given Condition \ref{condkjehdu},
there exists a constant $C$ depending only on $C_{\L}, C_{\L^{-1}}$ and $C_{\max}/C_{\min}$ such that for $(i,\alpha)\in \beth \times \aleph$,
$
\|\psi_{i,\alpha}\|_{H^s_0(\Omega\setminus\Omega_{i,n})}\leq \|\psi_{i,\alpha}^0\| e^{-n/ C}
$
and
\begin{equation}
\|\psi_{i,\alpha}-\psi_{i,\alpha}^n\|_{H^s_0(\Omega)}\leq \|\psi_{i,\alpha}^0\| e^{-n/ C},\quad n \in \mathbb{N}\, .
\end{equation}
\end{Theorem}

\begin{Condition}\label{condmeasfuncloc}
Given the
 Construction \ref{consmeasfomi} with  values $\delta$ and $h$, and Condition \ref{condkjehdu},
define
$\V^\perp:=\{f\in H^s_0(\Omega)\mid [\phi_{i,\alpha},f]=0\text{ for }(i,\alpha)\in \beth \times \aleph \}$.
There exists a constant $C_{\lsob}\geq |\aleph|$  such that
 \begin{equation}\label{eqconlocmeas}
 \|D^t f\|_{L^2(\Omega)}\leq C_{\lsob} h^{s-t} \|f\|_{H^s_0(\Omega)}  \text{ for }t\in \{0,1,\ldots,s\},\quad f
\in \V^\perp\, ,
 \end{equation}
 \begin{equation}\label{eqconlocmeasnext}
 \sum_{i\in \beth, \alpha \in \aleph}  [\phi_{i,\alpha},f]^2 \leq C_{\lsob} \big( \|f\|_{L^2(\Omega)}^2+h^{2s} \|f\|_{H^s_0(\Omega)}^2\big), \quad f\in H^s_0(\Omega)\,,
 \end{equation}
 and
\begin{equation}\label{eqkjdhddkuieheu}
|x|^2\leq C_{\lsob} h^{-2s} \|\sum_{\alpha \in \aleph} x_\alpha \phi_{i,\alpha}\|_{H^{-s}(\tau_i)}^2,\quad
i\in \beth, x\in \R^\aleph\, .
\end{equation}
\end{Condition}

\begin{Theorem}\label{thmegdkj3hrrsuitebis}
Assume that the operator $\L$ is local in the sense that $\<\psi,\psi'\>=0$ if $\psi$ and $\psi'$ have disjoint supports. Let  $(\phi_{i,\alpha})_{(i,\alpha)\in \beth \times \aleph}$ satisfy
 Condition \ref{condmeasfuncloc} with values $\delta$ and $h$ from Construction \ref{consmeasfomi}.   Then there exists a constant $C_{d,\delta,s}$ depending only on  $d, \delta$ and $s$ such that  Condition \ref{condkjehdu} holds true with $C_{\min}^{-1}\leq C_{d,\delta,s}$ and $ C_{\max}\leq C_{d,\delta,s}$. In particular, there exists a constant $C$ depending only on
 $d, \delta, s, C_\L$ and $C_{\L^{-1}}$ such that
for $(i,\alpha)\in \beth \times \aleph$,
$
\|\psi_{i,\alpha}\|_{H^s_0(\Omega\setminus\Omega_{i,n})}\leq C e^{-n/ C}
$
and
\begin{equation}
\|\psi_{i,\alpha}-\psi_{i,\alpha}^n\|_{H^s_0(\Omega)}\leq C e^{-n/ C}
\end{equation}
\end{Theorem}

For a local operator,
Theorem \ref{thmegdkj3hrrsuitebis} demonstrates that
  Condition \ref{condmeasfuncloc} is sufficient to guarantee the localization of the gamblets.
The following result demonstrates that the examples of measurement functions that follow it  satisfy the locality Condition \ref{condmeasfuncloc}.
\begin{Theorem}\label{thmpropegkdejkdhdjkbis}
The measurement functions of Examples \ref{egkdejkdhdjk}, \ref{egkdejkdhkjhdjk} and \ref{egkdejkdhdjkbis},  satisfy Condition \ref{condmeasfuncloc}.
\end{Theorem}

\begin{Example}\label{egkdejkdhdjk}
Let $|\aleph|=1$ and $|\beth|=m$.
Let $\delta, h \in (0,1)$. Let $\tau_1,\ldots,\tau_m$ be a partition of $\Omega$
such that each $\tau_i$ is Lipschitz, convex, contained in a ball of radius $h$ and contains a ball of radius $\delta h$.
Let $\phi_i=1_{\tau_i}/\sqrt{|\tau_i|}$ where $1_{\tau_i}$ is the indicator function of $\tau_i$.
\end{Example}

\begin{Example}\label{egkdejkdhkjhdjk}
Let $\delta, h \in (0,1)$. Let $\tau_1,\ldots,\tau_{|\beth|}$ be a partition of $\Omega$
such that each $\tau_i$ is Lipschitz, convex, contained in a ball of radius $h$ and contains a ball of radius $\delta h$. Let $|\aleph|={s+d-1 \choose d}$.
For $i\in \{1,\ldots,|\beth|\}$, let $(\phi_{i,\alpha})_{\alpha \in \aleph}$ be an $L^2(\tau_i)$
 orthonormal basis of $\cP_{s-1}(\tau_i)$, the space of $d$-variate polynomials on $\tau_i$ of degree at most
 $s-1$.
\end{Example}

\begin{Example}\label{egkdejkdhdjkbis}
Assume that $s>d/2$.
Let $\delta, h \in (0,1)$. Let $\tau_1,\ldots,\tau_m$ be a partition of $\Omega$ such that each $\tau_i$ is convex, uniformly Lipschitz, contains a ball of center $x_i$ radius $\delta h$ and is contained in a ball of center $x_i$ and radius $h$.
 Let $|\aleph|=1$ and $\phi_i(x):=h^{d/2}\delta(x-x_i)$.
\end{Example}

\subsection{Fast Gamblet Transform}\label{subsefgtin}

When $A$ is sparse (e.g. banded) then Algorithm \ref{alggambletcomutationnesQQ} can be accelerated by using sparse matrices $\pi^{(k-1,k)}$ and $W^{(k)}$. For instance, the matrices $\pi^{(k-1,k)}$ and $W^{(k)}$ can be chosen to be cellular, as defined in  Condition \ref{cond7fyf}.

Under these sparsity conditions, we introduce in Section \ref{secfgtas} the Fast Gamblet Transform (Algorithm \ref{fastgambletsolvecase1g}) obtained by truncating the stiffness matrices $A^{(k)}$ and localizing the computation of the
vectors $\psi_i^{(k)}$ in Algorithm \ref{alggambletcomutationnesQQ}.
More precisely, if $A$ has $\mathcal{O}(N)$ non-zero entries and under the graph distance
 (Definition \ref{defgraphmatdistbis}) induced by $A$ on $\I^{(q)}$, a ball of radius $\rho$
has $\mathcal{O}(\rho^d)$ elements, and if
 Conditions \ref{conddiscrip3ordismatdis}, \ref{cond7fyf},  and \ref{condilwhiuhd} are  satisfied,
then  Theorem \ref{tmdiscrete1}
 proves that the
 complexity of the Fast Gamblet Transform is at most
$\mathcal{O}(N \ln^{3d} N)$ to compute the gamblets $\psi_i^{(k)}$, $\chi_i^{(k)}$ and their stiffness matrices $A^{(k)}$ and $B^{(k)}$, and $\mathcal{O}(N \ln^{d+1} N)$ to invert the linear system \eqref{eqkjkdjdhjQQ} and perform the decomposition \eqref{esdddkejdiuhQQ}
up to $|\cdot|_A$-norm accuracy $\mathcal{O}(H^{q})$.

The mechanism behind this acceleration is based on  the localization (exponential decay) of each gamblet $\psi_i^{(k)}$ away from $i$.  Item \ref{itt2} of Condition \ref{condilwhiuhd}, which provides a sufficient condition for this exponential decay, is the matrix version of Condition \ref{condkjehdu}.

Returning to  Sobolev spaces, recall that
Theorem \ref{thmpropegkdejkdhdjkbis} asserts that
 Condition \ref{condmeasfuncloc} is satisfied for a large  common class of measurement functions. Moreover,
  when $\L$ is local   Theorem \ref{thmegdkj3hrrsuitebis} asserts that  Condition  \ref{condmeasfuncloc} implies that
Condition \ref{condkjehdu} is satisfied, which is a condition on the range space and does not depend on the operator itself. Consequently,
as  with Condition \ref{conddiscrip3ordismatdisQQ} for matrix systems,
 if $A$ is obtained as the discretization of a local continuous bijection $\L$, then the equivalence
between Item \ref{itt2} of Condition \ref{condilwhiuhd}
and Condition \ref{condkjehdu} mentioned above implies
 Item \ref{itt2} of Condition \ref{condilwhiuhd} is a condition on the image space of that operator
 (independent from the operator itself) and the matrices $\pi^{(k,k+1)}$, that is,
 Item \ref{itt2} of Condition \ref{condilwhiuhd} is  naturally satisfied for
 continuous local bijections on Sobolev spaces.

\subsection{Non divergence form and non symmetric operators}

 Let $\L: H^s_0(\Omega) \rightarrow L^2(\Omega)$ be a continuous linear bijection.
 A prototypical example of $\L$ is the non-divergence form operator
 \begin{equation}
\L u = \sum_{0\leq |\alpha|\leq s} a_{\alpha}(x)  D^{\alpha}  u \text{ for }u\in H^s_0(\Omega)
\end{equation}
where $a$ is a tensor with entries in  $L^\infty(\Omega)$ such that $\L^{-1}$ is well defined and continuous.
The inverse problem $\L u=g$ is equivalent to $\L^* \L=\L^* g$ where $\L^*$ is the adjoint of $\L$.
Let $Q^{-1}$ be the linear operator mapping $H^s_0(\Omega)$ to $H^{-s}(\Omega)$ defined as
$Q^{-1}:=\L^* \L$. Observe that $Q^{-1}$ and $Q$ are continuous and, by applying the results of Subsections \ref{subsecnatcondga}, \ref{subsecexpdein} and \ref{subsefgtin} to the operator $Q^{-1}$,
 the inverse problem $\L u=g$ can be solved
in $\mathcal{O}(N \operatorname{polylog} N)$-complexity by the solving the linear system
$Q^{-1} u=\L^* \L u=\L^* g$ using gamblets defined by the norm induced by $Q$, i.e. by using gamblets minimizing the norm $\|\cdot\|$ where $\|u\|=[Q^{-1} u, u]^\frac{1}{2}=\|\L u\|_{L^2(\Omega)}$ for $u\in H^s_0(\Omega)$ and measurement functions satisfying Condition \ref{condkjehdu} (as in Example \ref{egkdejkdhkjhdjk}).

\section{The gamblet transform on a Banach space}\label{sec1or}
\subsection{Setting}\label{subsecttt}
For
any topological vector space  $V$,  we write the dual pairing between $V$
  and it topological dual $V^{*}$
by $[\cdot,\cdot]$.
Let $(\B,\|\cdot\|)$ be a reflexive separable Banach space such that the $\|\cdot\|$ norm is quadratic, i.e.
$\|u\|^2=[Q^{-1}u,u]$ for $u\in \B$,  and $Q$ is a symmetric positive bijective linear operator mapping $\B^*$  to $\B$ ($[v^*, Q w^*]=[w^*, Q v^*]$ and $[v^*,Q v^*]\geq 0$ for $v^*,w^*\in \B^*$). Write $\<\cdot,\cdot\>$  the corresponding inner product on $\B$ defined by
\begin{equation}
\<u,v\>:=[Q^{-1} u,v] \text{ for }u,v \in \B\,.
\end{equation}
Although $\B$ is also a Hilbert space under the quadratic norm $\|\cdot\|$, we will keep using the Banach terminology to emphasize the fact that our dual pairings will not be based on the $\<\cdot,\cdot\>$ scalar product but on  nonstandard (via the Riesz representation) realization of the dual space.

Let $(\B_0,\|\cdot\|_{0})$ be a Banach subspace of $\B^*$ such that the natural embedding $i:\B_0 \rightarrow \B^*$ is compact and dense. We summarize this embedding in the following diagram
\begin{equation}\label{eqQaQm1}
\text{\xymatrix{ \B  \ar@/^/[rr]|{Q^{-1}}  && \B^* (\supset \B_0 ) \ar@/^/[ll]|{Q} }}
\end{equation}
Write $\<\cdot,\cdot\>_*$  the scalar product on $\B^*$ defined by $\<\phi',\phi\>_*=[\phi',Q\phi]$ and let $\|\cdot\|_*$ be the corresponding norm. Observe that $\|\cdot\|_*$ is the natural norm induced by duality on $\B^*$, i.e.  $\|\phi\|_*=\sup_{v\in \B,\,v\not=0}[\phi,v]/\|v\|$ for $\phi \in \B^*$, and $Q$ and $Q^{-1}$ are isometries in the sense that $\|Q^{-1}u\|_*=\|u\|$ for $u\in \B$.

\begin{Example}\label{egproto00}
As a simple prototypical running example we will consider $\B=H^1_0(\Omega)$ and $\B_0=L^2(\Omega)$ with the following specifications.
 $\Omega$ is a bounded open subset of $\R^d$ (of arbitrary dimension $d\in \mathbb{N}^*$) with piecewise Lipschitz boundary.
 $\|\cdot\|_0=\|\cdot\|_{L^2(\Omega)}$ and
$\|u\|^2=\|u\|_a^2=\int_{\Omega} (\nabla u)^T a \nabla u$ is the energy norm defined by a uniformly elliptic $d\times d$ symmetric matrix $a$ with entries in $L^\infty(\Omega)$ (i.e. satisfying $\lambda_{\min}(a) |l|^2 \leq l^T a(x) l\leq \lambda_{\max}(a) |l|^2$ for $l\in \R^d$ with some strictly positive finite constants $\lambda_{\max}(a)$ and $\lambda_{\min}(a)$). Note that for this example the induced dual space
$\B^*$ can be identified with $H^{-1}(\Omega)$  with the norm
$\|\varphi\|_*=\sup_{v\in H^1_0(\Omega)}\int_{\Omega}\varphi v/\|v\|_a$. Under this identification the induced operator $Q^{-1}$ is the differential operator $-\diiv(a\nabla)$.
\end{Example}

\begin{Example}\label{egprotoalsobolevlori}
As another running example we will consider  $(\B_0, \|\cdot\|_0)=(L^2(\Omega), \|\cdot\|_{L^2(\Omega)})$ and
$(\B,\|\cdot\|)=(H^s_0(\Omega),\|\cdot\|)$, where $\|u\|^2=[Q^{-1} u,u]$ and  $Q$ is a symmetric continuous linear bijection mapping $H^{-s}(\Omega)$ to $H^{s}_0(\Omega)$ such that, for some constant $C_{e}\geq 1$,
\begin{equation}\label{eqjkhkkjhuiiu}
C_{e}^{-1}\|u\|_{H^s_0(\Omega)} \leq \|u\| \leq C_{e} \|u\|_{H^s_0(\Omega)},\,\text{ for } u\in H^s_0(\Omega)\,.
\end{equation}
Note that using  \cite[Thm.~2.2]{gazzola2010polyharmonic} the left and/or right hand sides of
\eqref{eqjkhkkjhuiiu} could be replaced by $\|u\|_{H^s(\Omega)}$ modulo changes in norm-equivalence constants.
Note that for this example the induced dual space
$\B^*$ can be identified with $H^{-s}(\Omega)$ with the induced norm $\|\phi\|_*=\sup_{v\in \H^s_0(\Omega)}\frac{[\phi,v]}{\|v\|}$, and \eqref{eqjkhkkjhuiiu} implies that
\begin{equation}\label{eqjjhykhkkjdedduiiu}
C_{e}^{-1}\|\phi\|_{H^{-s}(\Omega)} \leq \|\phi\|_* \leq C_{e} \|\phi\|_{H^{-s}(\Omega)},\,\text{ for } \phi\in H^{-s}(\Omega)
\end{equation}
\end{Example}

 \begin{figure}[h!]
	\begin{center}
			\includegraphics[width=0.7\textwidth]{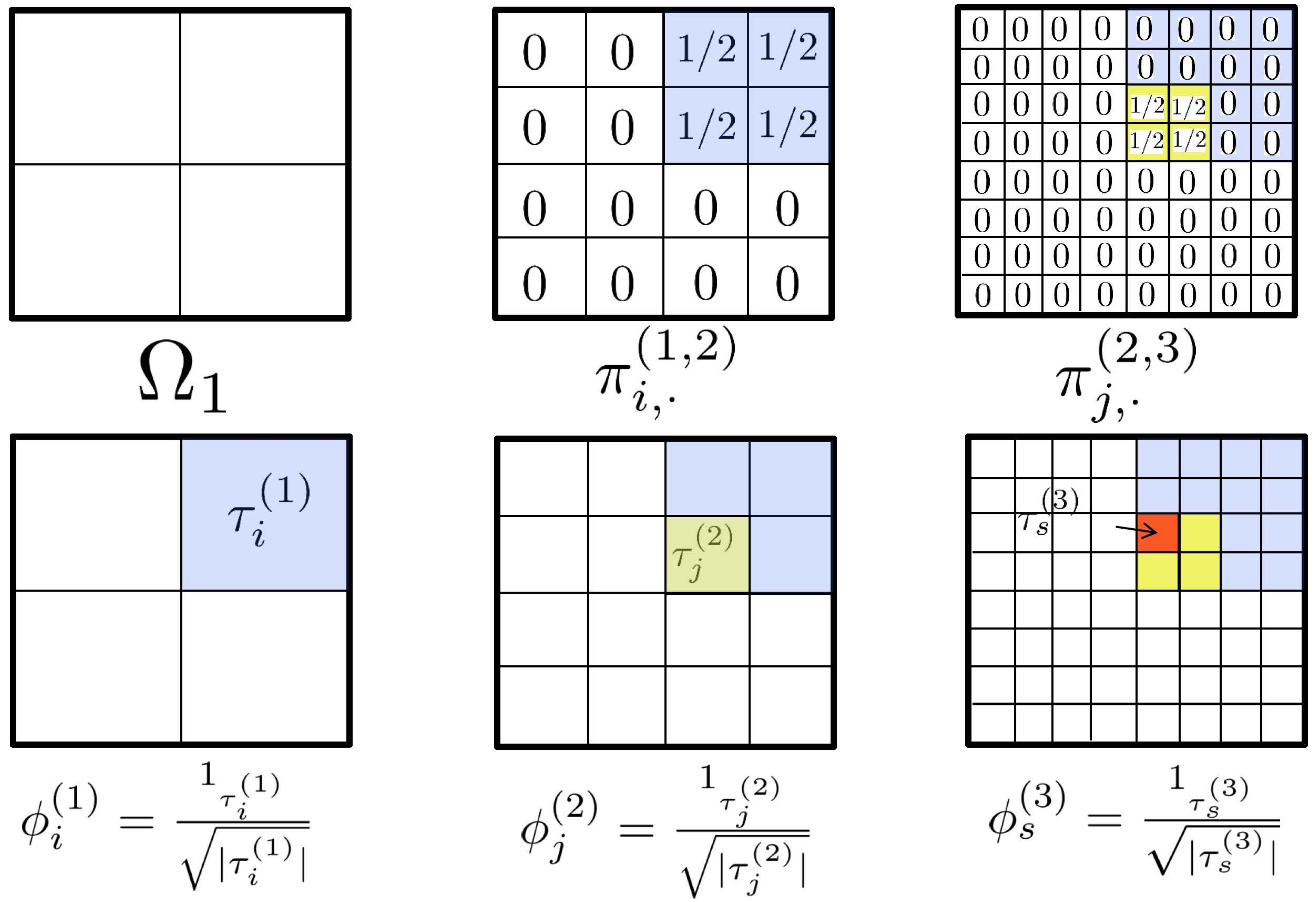}
		\caption{See Illustration \ref{nekdjdhhu}. $\Omega=(0,1)^2$. $\Omega_k$ corresponds to a uniform partition of $\Omega$ into $2^{-k}\times 2^{-k}$ squares. The bottom row shows the support of $\phi_i^{(1)}, \phi_j^{(2)}$ and $\phi_s^{(3)}$. Note that  $j^{(1)}=s^{(1)}=i$ and $s^{(2)}=j$. The top row shows the entries of $\pi^{(1,2)}_{i,\cdot}$ and $\pi^{(2,3)}_{j,\cdot}$.
}\label{figphipi}
	\end{center}
\end{figure}

\subsection{Hierarchy of nested measurements}\label{subsechajhe}

Let $q\in \mathbb{N}^* \cup \{\infty\}$. Let $\I^{(q)}$ be an index tree of level $q$ defined as in \ref{defindextree}.
For a finite set $\J$ write $|\J|:=\Card(\J)$ the number of elements of $\J$.
\begin{Construction}\label{constpi}
For $k\in \{1,\ldots,q-1\}$ let $\pi^{(k,k+1)}$ be a $\I^{(k)}\times \I^{(k+1)}$ matrix of rank $|\I^{(k)}|$.
\end{Construction}
Note that \ref{constpi} does not require the orthonormality condition $\pi^{(k,k+1)}(\pi^{(k,k+1)})^T=I^{(k)}$ of \ref{constpiQQ}.

\begin{Construction}\label{constphik}
Let $(\phi^{(q)}_i)_{i\in \I^{(q)}}$ be linearly independent elements of $\B^*$. For $k\in \{1,\ldots,q-1\}$ and $i\in \I^{(k)}$ let $\phi^{(k)}_i \in \B^*$ be defined by induction via
\begin{equation}\label{eq:eigdeiud3dd}
\phi^{(k)}_i=\sum_{j\in \I^{(k+1)}}\pi^{(k,k+1)}_{i,j}  \phi^{(k+1)}_j\, .
\end{equation}
\end{Construction}
Observe that $\text{rank}(\pi^{(k,k+1)})=|\I^{(k)}|$ implies the linear independence of the elements $(\phi^{(k)}_i)_{i\in \I^{(k)}}$.
For $k\in \{1,\ldots,q\}$, write
 \begin{equation}\label{eqdefPhik}
 \Phi^{(k)}:=\Span\{\phi_i^{(k)}\mid i\in \I^{(k)}\}\,.
 \end{equation}
Observe that \eqref{eq:eigdeiud3dd} implies that the spaces $\Phi^{(k)}$ are nested, i.e., $\Phi^{(k)}\subset \Phi^{(k+1)}$.

\begin{NE}\label{nekdjdhhu}
As a running illustration consider Example \ref{egproto00} with $\Omega=(0,1)^2$  illustrated in Figure \ref{figphipi}. For $k\in \N^*$, let $\Omega_{k}$ be a regular grid partition of $\Omega$ into $2^{-k}\times 2^{-k}$ squares $\tau^{(k)}_i$  and let $\phi_i^{(k)}=\frac{1_{\tau^{(k)}_i}}{\sqrt{|\tau^{(k)}_i|}}$ where  $1_{\tau^{(k)}_i}$ is the indicator function of $\tau^{(k)}_i$ and $|\tau^{(k)}_i|$ is the volume of $\tau^{(k)}_i$. The nesting of the indicator functions implies that of the measurement functions, i.e. \eqref{eq:eigdeiud3dd}. In this particular example, the nesting matrices $\pi^{(k,k+1)}$ are also cellular (in the sense that $\pi^{(k,k+1)}_{i,j}=0$ for $j^{(k)}\not=i$) and orthonormal (in the sense that $\pi^{(k,k+1)}(\pi^{(k,k+1)})^T=I^{(k)}$ where $I^{(k)}$ is the $\I^{(k)}\times \I^{(k)}$ identity matrix). Note that the measurement functions $\phi_i^{(k)}$ form a multiresolution decomposition of $L^2(\Omega)$.
\end{NE}

 \begin{Example}\label{egkajhdlkjdinibis}
 Consider, for Example \ref{egprotoalsobolevlori}, the measurement functions introduced in Example \ref{egkajhdlkjdini}.
   Observe that the measurement functions $(\phi_i^{(k)})_{i\in \I^{(k)}}$ are nested as in \eqref{eq:eigdeiud3dd} and satisfy $\int_{\Omega}\phi_i^{(k)}\phi_j^{(k)}=\delta_{i,j}$ for $i\not=j$, which implies $\pi^{(k,k+1)}(\pi^{(k,k+1)})^T=I^{(k)}$.
   Observe also that the matrices $\pi^{(k,k+1)}$ are cellular in the sense of Condition \ref{cond7fyf}, i.e.  $\pi^{(k,k+1)}_{i,j}=0$ for $j^{(k),\aleph}\not=i^{\aleph}$.
   Note also that the matrices $W^{(k)}$ can naturally be chosen so that $W^{(k)}W^{(k),T}=J^{(k)}$ and $W^{(k)}_{i,j}=0$ for $j^{(k),\aleph}\not=i^{(k),\aleph}$.
 Illustration \ref{nekdjdhhu} presents a particular instance of the proposed measurement functions for $s=1$.
\end{Example}

\subsection{The Gamblet Transform}\label{subseckjshgdhgdhOR}

Due to their game theoretic origin and interpretation (presented in Section \ref{seccigsec}),
we will refer to the following  defined  hierarchy of elements $\psi^{(k)}_i$ as gamblets.
\begin{Definition}(Gamblets)\label{defpsi}
For  $k\in \{1,\ldots,q\}$ and $i\in  \I^{(k)}$, let $\psi^{(k)}_i$ be the minimizer of
\begin{equation}\label{eq:dfddeytfewdaisq}
\begin{cases}
\text{Minimize }  &\|\psi\|\\
\text{Subject to } &\psi \in \B\text{ and }[\phi_j^{(k)}, \psi]=\delta_{i,j}\text{ for } j\in \I^{(k)}\,.
\end{cases}
\end{equation}
\end{Definition}

For $k\in\{1,\ldots,q\}$ let $\Theta^{(k)}$ be the symmetric positive definite $\I^{(k)}\times \I^{(k)}$ matrix defined by
\begin{equation}\label{eqtheta1}
\Theta^{(k)}_{i,j}:=[\phi_i^{(k)},Q \phi_j^{(k)}]\,,
\end{equation}
and write $\Theta^{(k),-1}:=(\Theta^{(k)})^{-1}$. The gamblets have the following  explicit form.
\begin{Theorem}\label{thmwhdguyd}
For $k\in\{1,\ldots,q\}$ and $i\in \I^{(k)}$, we have
\begin{equation}
\psi_i^{(k)}=\sum_{j\in \I^{(k)}} \Theta^{(k),-1}_{i,j} Q \phi_j^{(k)}\,.
\end{equation}
\end{Theorem}

 \begin{figure}[h!]
	\begin{center}
			\includegraphics[width=\textwidth]{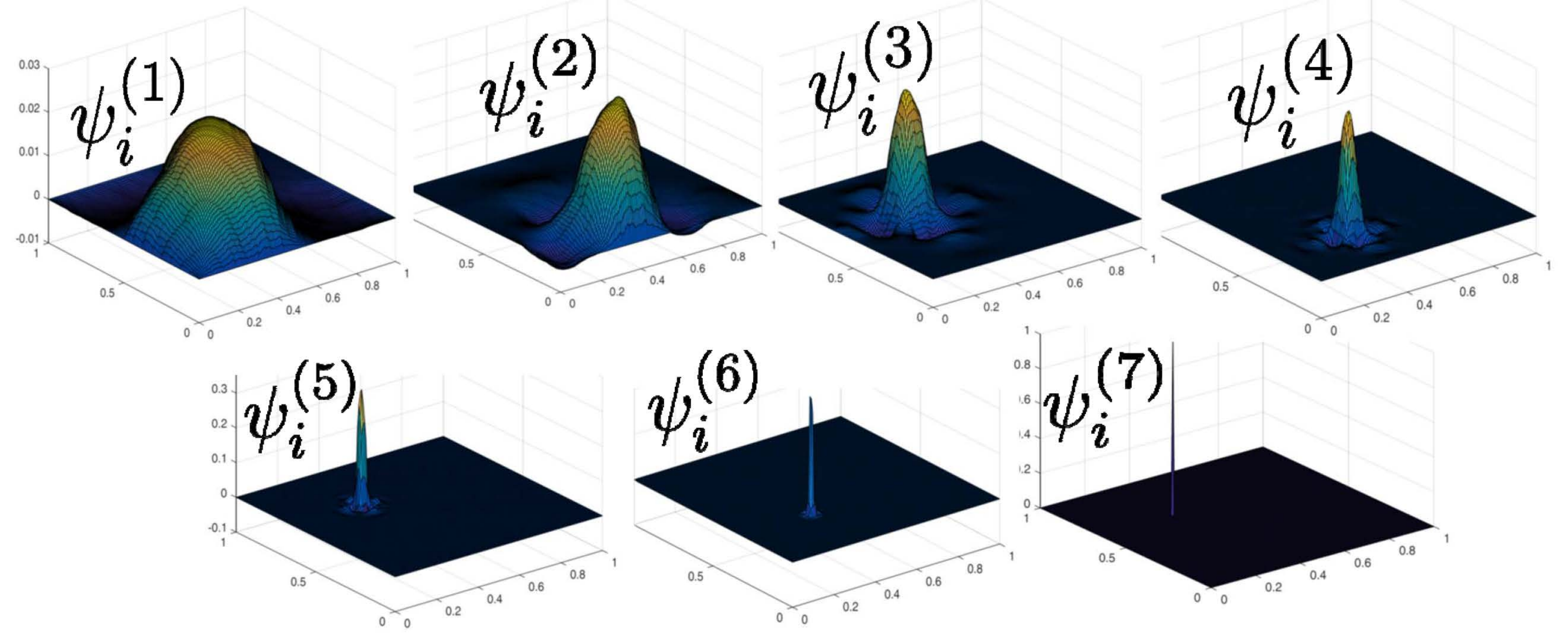}
		\caption{Gamblets $\psi_i^{(k)}$ for $1\leq k \leq 7$. See Illustration \ref{nekdjdhhdeu2}.
}\label{figpsi7}
	\end{center}
\end{figure}

\begin{NE}\label{nekdjdhhdeu2}
In the context of Example \ref{egproto00} and Illustration \ref{nekdjdhhu},
Figure \ref{figpsi7} provides a numerical illustration of the gamblets $\psi_i^{(k)}$ $(1\leq k \leq 7)$ using the Discrete Gamblet Transform described in Subsection \ref{subsecdisgambtra}.
\end{NE}

The following
 proposition provides the link between
gamblets and orthogonal projection in the immediately following Theorem \ref{thmgugyug0OR}.
We say that  a linear operator $Q:\B^* \rightarrow \B$ is positive  symmetric if
 $[\varphi_{1},Q \varphi_{2}]=[\varphi_{2},Q \varphi_{1}] $ and
 $[\varphi_{1},Q \varphi_{1}] \geq 0$ for $\varphi_{1},\varphi_{2} \in \B^* $.
 When such a  $Q$ is a continuous bijection it determines a Hilbert space inner product by
$\langle \varphi_{1},\varphi_{2}\rangle_{Q} :=[\varphi_{1},Q \varphi_{2}]$.
\begin{Proposition}
\label{prop_Gambletprojection}
Let $\B$ be a  separable Banach space, and $\B^*$ be a realization its dual with
$[\cdot,\cdot]$ the corresponding dual pairing.
 For a positive symmetric bijection $Q:
\B^* \rightarrow \B$, consider the Hilbert space $\B$ equipped with the inner product
$\langle u_{1},u_{2}\rangle:=[Q^{-1}u_{1}, u_{2}]$ and
the Hilbert space $\B^* $ equipped with the inner product
$\langle \varphi_{1},\varphi_{2}\rangle_{*}:=[ \varphi_{1},Q\varphi_{2}].$
Consider a
collection  $\phi_1,\ldots,\phi_m$ of $m$  linearly independent elements of $\B^*$, and let
$\Phi\subset \B^{*}$ denote its span.
Define the Gram  matrix $\Theta$ by
\begin{equation}\label{eqdefthet}
 \Theta_{ij}:=[ \phi_{i},Q\phi_{j}], \quad i,j=1,\ldots m\,,
 \end{equation}
and the elements $\psi_{i} \in \B, i=1,\ldots m$ by
\begin{equation}\label{eqhgudgud}
\psi_{i}:= \sum_{j=1}^{m}{\Theta^{-1}_{ij}Q\phi_{j}}, \quad i=1,\ldots, m\, .
\end{equation}
The  collection $\{\phi_{i},\psi_{j},\,i,j=1,\ldots m\}$ is a biorthogonal system, in that
$[\phi_{i},\psi_{j}]=\delta_{ij},\,i,j=1,\ldots m$.
Moreover, the operator
 $P:\B\rightarrow \B$ defined by
$P:= \sum_{i=1}^{m}{\psi_{i}\otimes \phi_{i}}\, $
is the  orthogonal projection onto $Q\Phi$,
and  the  operator $P^{*}:\B^{*}\rightarrow B^{*}$ defined by
$P^{*}:= \sum_{i=1}^{m}{Q^{-1}\psi_{i}\otimes Q\phi_{i}}\, $
 is the orthogonal projection onto $\Phi$.
In addition, $P^{*}$ is the adjoint of $P$ in the sense that
$[\varphi, P\psi]=[P^{*}\varphi, \psi], \, \varphi \in \B^{*}, \psi \in \B$, and we have
$P^{*}=Q^{-1}PQ\,.$
\end{Proposition}

Consequently, we can characterize the gamblets as components of an orthogonal projection.
For  $k\in \{1,\ldots,q\}$, write
\begin{equation}\label{eqdefvk}
\V^{(k)}:=\operatorname{span}\{\psi^{(k)}_i \mid i\in \I^{(k)}\} .
\end{equation}
\begin{Theorem}\label{thmgugyug0OR}
We have $\V^{(k)}=Q \Phi^{(k)}$ and $\V^{(k)} \subset \V^{(k+1)}$.
Furthermore, the   mapping $u^{(k)}:\B \rightarrow \B$
defined by
  \begin{equation}\label{eqbdudbuysbdneq1}
u^{(k)}(u):=\sum_{i\in \I^{(k)}} [\phi_i^{(k)},u] \psi_i^{(k)}\,
\end{equation}
is the orthogonal projection of $\B$ onto $\V^{(k)}$ and therefore has the variational formulation
\begin{equation}\label{eqdkjkduednOR}
\|u-u^{(k)}(u)\|=\inf_{v\in \V^{(k)}}   \|u - v\|,\quad u\in \B\, .
\end{equation}
\end{Theorem}

 \begin{figure}[h!]
	\begin{center}
			\includegraphics[width=0.7\textwidth]{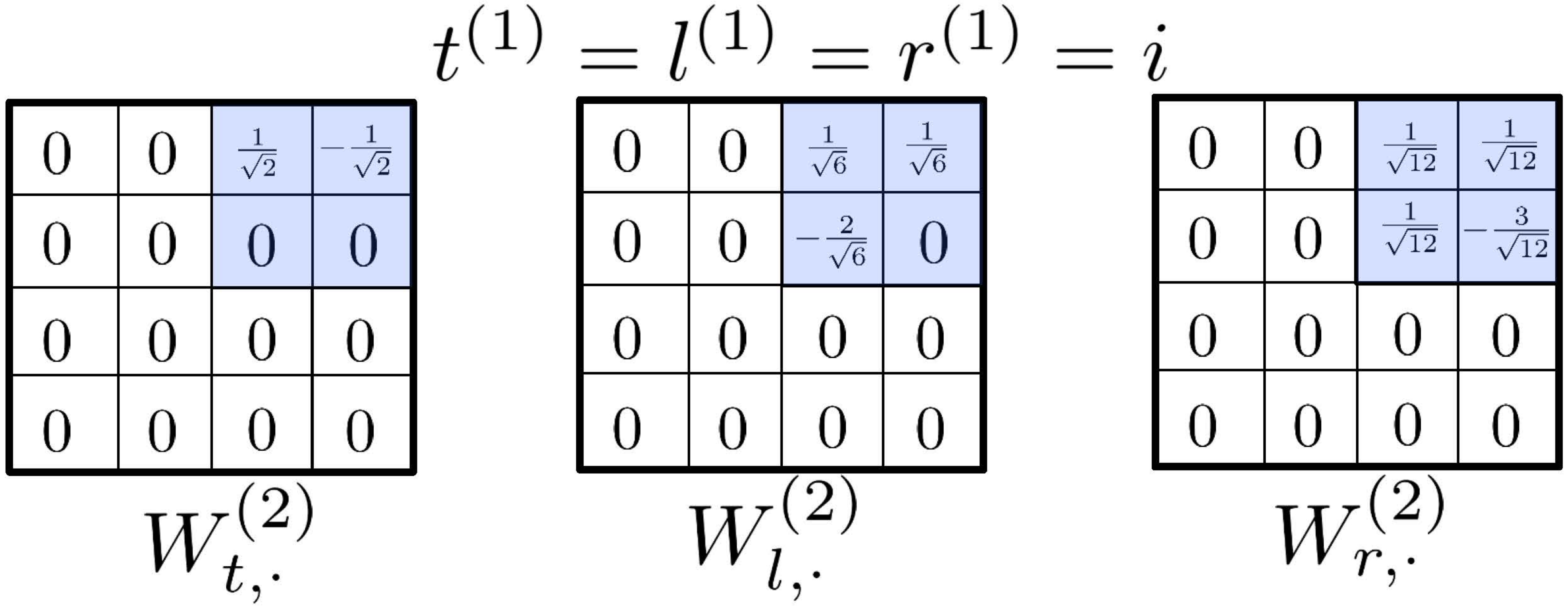}
		\caption{See Illustrations \ref{nekdjdhhu} and \ref{nekdjdhhu2}. Entries of $W^{(2)}_{t,\cdot}$, $W^{(2)}_{l,\cdot}$ and $W^{(2)}_{r,\cdot}$ with $t^{(1)}=l^{(1)}=r^{(1)}=i$.
}\label{figw1}
	\end{center}
\end{figure}

For $1<r<k$ and a $k$-tuple of the form $i=(i_1,\ldots,i_k)$ we write $i^{(r)}:=(i_1,\ldots,i_r)$.
For $k=2,\ldots,q$ let $\J^{(k)}$ be defined as in \ref{consjk}.
\begin{Construction}\label{conswk}
For $k=2,\ldots,q$ let $W^{(k)}$ be a $\J^{(k)}\times \I^{(k)}$ matrix such that  $\Img(W^{(k),T})=\Ker(\pi^{(k-1,k)})$.
\end{Construction}
Note that Construction \ref{conswk} does not require the orthonormality condition $W^{(k)}W^{(k),T}=J^{(k)}$ of Construction
\ref{conswkQQ}.

\begin{NE}\label{nekdjdhhu2}
We refer to Figure \ref{figw1} for an illustration of the entries of $W^{(k)}$ in the context discussed in Illustration \ref{nekdjdhhu}
and Figure \ref{figphipi}. Note that In this particular example, the  matrices $W^{(k)}$ are also cellular (in the sense that $W^{(k)}_{i,j}=0$ for $j^{(k)}\not=i$) and orthonormal (in the sense that $W^{(k)}(W^{(k)})^T=J^{(k)}$ where $J^{(k)}$ is the $\J^{(k)}\times \J^{(k)}$ identity matrix).
\end{NE}

\begin{Definition}
\label{def_chi}
For $k\in \{2,\ldots,q\}$ and $i\in \J^{(k)}$, write
  \begin{equation}\label{eqjkhdkdh}
\chi^{(k)}_i:=\sum_{j \in \I^{(k)}} W_{i,j}^{(k)} \psi_j^{(k)}
 \end{equation}
and
\begin{equation}\label{eqdefwk}
\W^{(k)}:=\operatorname{span}\{\chi^{(k)}_i \mid i\in \J^{(k)}\} .
\end{equation}
\end{Definition}
 \begin{figure}[h!]
	\begin{center}
			\includegraphics[width=\textwidth]{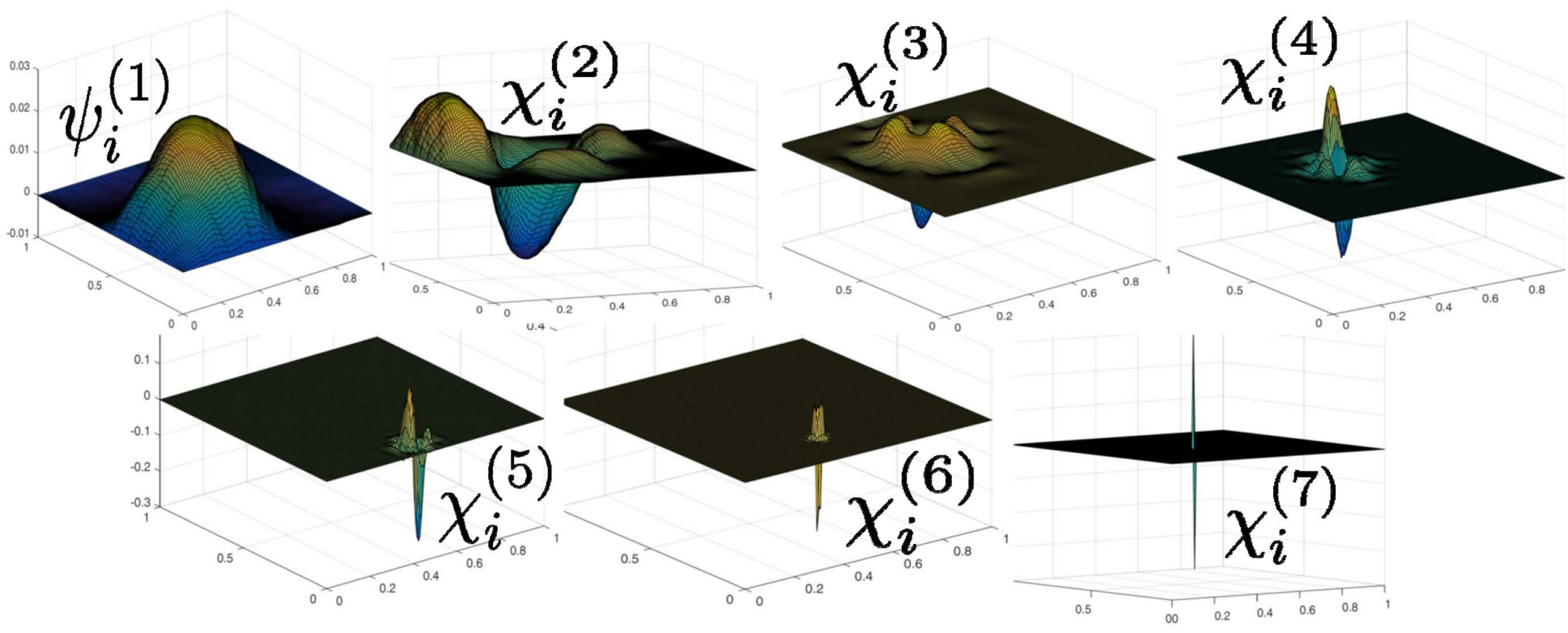}
		\caption{Gamblets $\psi_i^{(1)}$ and orthogonalized gamblets $\chi_i^{(k)}$ for $2\leq k \leq 7$. See Illustration \ref{nekdjdhhdseddeu2}.
}\label{figchi7}
	\end{center}
\end{figure}
\begin{NE}\label{nekdjdhhdseddeu2}
In the context of Example \ref{egproto00} and Illustrations \ref{nekdjdhhu} and \ref{nekdjdhhu2},
 Figure \ref{figpsi7} provides a numerical illustration of the Gamblets $\psi_i^{(1)}$ and orthogonalized gamblets $\chi_i^{(k)}$ ($2\leq k \leq 7$) using the Discrete Gamblet Transform described in Subsection \ref{subsecdisgambtra}.
\end{NE}

Write $\W^{(q+1)}$ for the $\<\cdot,\cdot\>$-orthogonal complement of $\V^{(q)}$ in $\B$.
\begin{Theorem}\label{thmgugyug2OR}
For $k\in \{2,\ldots,q\}$, $\W^{(k)}$ is the orthogonal complement of $\V^{(k-1)}$ in $\V^{(k)}$ with respect to the scalar product $\<\cdot,\cdot\>$, i.e. writing $\oplus$ the  $\<\cdot,\cdot\>$-orthogonal direct sum, we have
\begin{equation}\label{eqdedhhiuhe3OR}
\B=\V^{(1)}\oplus \W^{(2)} \oplus  \cdots \oplus \W^{(q)}\oplus \W^{(q+1)}\,.
\end{equation}
Furthermore,
$u=u^{(1)}(u)+(u^{(2)}(u)-u^{(1)})(u)+\cdots+(u^{(q)}(u)-u^{(q-1)}(u))+(u-u^{(q)}(u))$
is the  $\<\cdot,\cdot\>$-orthogonal decomposition of $u\in \B$ in corresponding to \eqref{eqdedhhiuhe3OR}.
\end{Theorem}

 \begin{figure}[h!]
	\begin{center}
			\includegraphics[width=\textwidth]{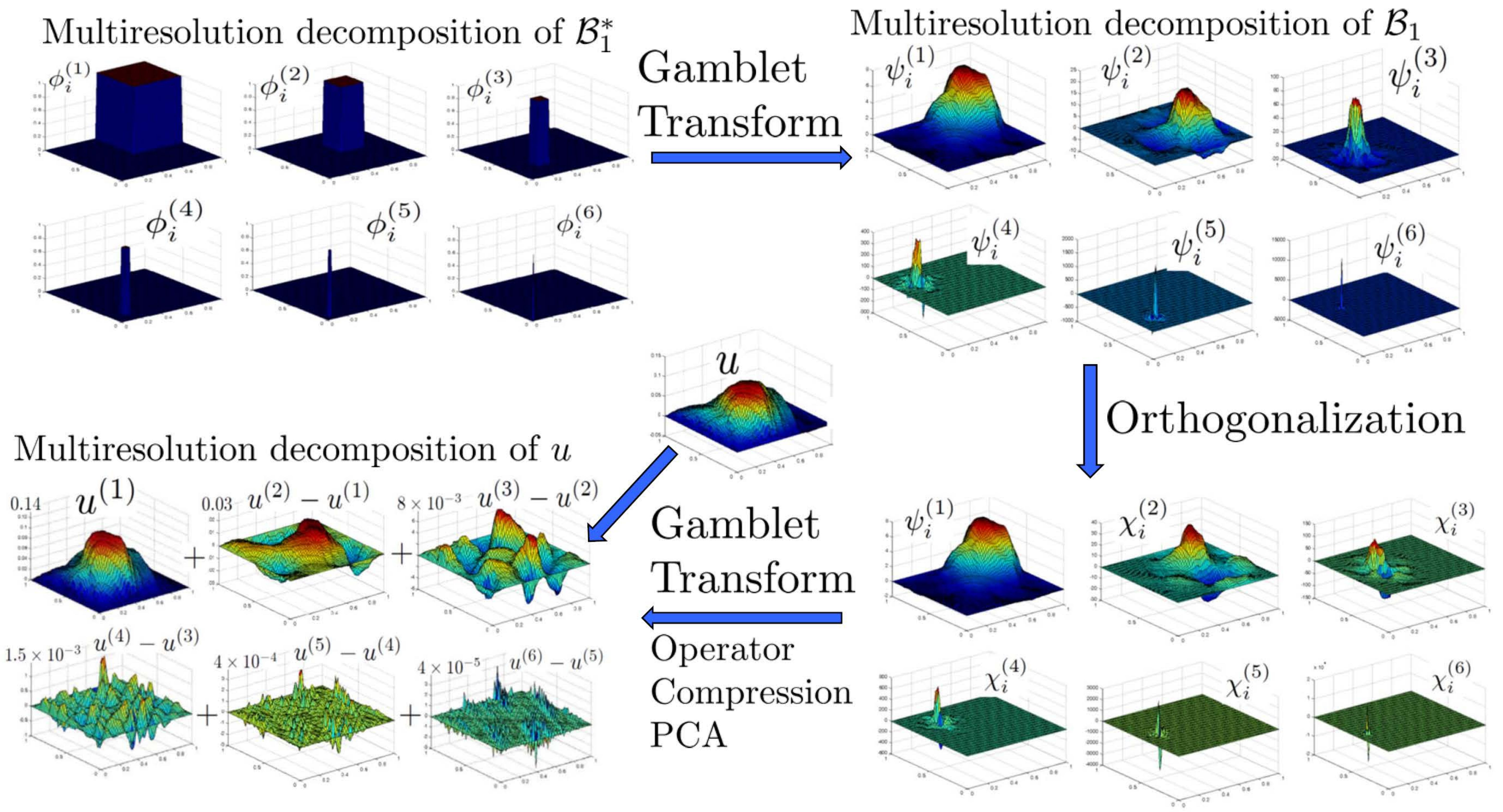}
		\caption{The Gamblet Transform.
}\label{figgt}
	\end{center}
\end{figure}

\begin{NE}\label{nsdeddrf}
Figure \ref{figgt} provides, in the context of Example \ref{egproto00} and Illustrations \ref{nekdjdhhu} and \ref{nekdjdhhu2},  an illustration of the Gamblet Transform described in Subsection \ref{subseckjshgdhgdhOR}.  Observe that the Gamblet Transform turns the hierarchy of nested measurements functions $\phi_i^{(k)}$ into the hierarchy of nested gamblets $\psi_i^{(k)}$. As it is commonly done with nested wavelets \cite{daubechies1992ten, meyer1989orthonormal}, the gamblets are then  orthogonalized into the elements $\chi_i^{(k)}$. These new gamblets enable the multiresolution decomposition of any element $u\in \B$ that could not only be employed for designing fast solvers, but also for operator compression, PCA or active subspace analysis. Observe that the transformation of the measurement functions $\phi_i^{(k)})\in \B^*$ into elements $\psi_i^{(k)}\in \B$
is also a generic method for automatically producing large and new classes of wavelets that are adapted to the space $(\B,\|\cdot\|)$.
More precisely we will show in Subsection \ref{subseghdhhgdjh} that the Gamblet transform turns a multiresolution decomposition of the compact embedding $(\B_0,\|\cdot\|_0)\rightarrow (\B,\|\cdot\|)$ into a multiresolution decomposition of the operator $Q^{-1}: (\B,\|\cdot\|) \rightarrow (\B^*,\|\cdot\|_*)$, which in the context  of Example \ref{egproto00} and Illustration \ref{nekdjdhhu}, corresponds to turning the  multiresolution decomposition of the compact embedding $(L_2(\Omega),\|\cdot\|_{L^2(\Omega)})\rightarrow (H^{-1}(\Omega),\|\cdot\|_{*})$ induced by Haar basis functions, into the multiresolution decomposition of the operator $-\diiv(a\nabla): (H^1_0(\Omega),\|\cdot\|_{a})\rightarrow (H^{-1}(\Omega),\|\cdot\|_{*})$.
\end{NE}

\subsection{Bounded condition numbers}\label{subseghdhhgdjh}
Let $A^{(k)}$ be the $\I^{(k)}\times \I^{(k)}$ stiffness matrix defined by
\begin{equation}\label{eqtheta2}
A^{(k)}_{i,j}=\< \psi_i^{(k)}, \psi_j^{(k)}\>
\end{equation}
Let $B^{(k)}$ be the $\J^{(k)}\times \J^{(k)}$ (stiffness) matrix
$B^{(k)}_{i,j}:=\<\chi_i^{(k)},\chi_j^{(k)}\>$.
Observe that
\begin{equation}\label{eqjgfytfjhyyyg}
B^{(k)}= W^{(k)}A^{(k)}W^{(k),T}
\end{equation}

For $k\in \{2,\ldots,q\}$ let
\begin{equation}\label{eqphikchi}
\Phi^{(k),\chi}:=\{\phi\in \Phi^{(k)}\mid \phi=\sum_{i\in \I^{(k)}} x_i \phi_i^{(k)}\text{ with } x\in \Ker(\pi^{(k-1,k)})\}
\end{equation}

Let $I^{(k)}$ be the identity $\I^{(k)}\times \I^{(k)}$ matrix and let $J^{(k)}$ be the identity $\J^{(k)}\times \J^{(k)}$ matrix.
\begin{Condition}\label{cond1OR}
There exists some constants $C_{\Phi}\geq 1$ and $H\in (0,1)$ such that
\begin{enumerate}
\item\label{itcor1} $\sup_{\phi\in \B_0} \frac{\|\phi \|_*}{\|\phi \|_0}\leq C_\Phi$.
\item\label{lincondORhdg} $\frac{1}{C_\Phi} H^k \leq \inf_{\phi\in \Phi^{(k)}} \frac{\|\phi \|_*}{\|\phi \|_0}$ for $k\in \{1,\ldots,q\}$.
\item\label{itcor2}  $\sup_{\varphi\in \B_0}\inf_{\phi \in \Phi^{(k)}} \frac{\|\varphi-\phi\|_*}{\|\varphi\|_0}\leq C_{\Phi} H^k$ for $k\in \{1,\ldots,q\}$.
\item\label{itcor3} $\sup_{\phi\in \Phi^{(k),\chi} }\frac{\| \phi\|_*}{\|\phi\|_0}\leq C_{\Phi} H^{k-1}$ for $k\in \{2,\ldots,q\}$.
\item\label{itcor4}  $\frac{1}{C_{\Phi}}|x|^2 \leq \|\sum_{i\in \I^{(k)}} x_i \phi_i^{(k)}\|_0^2 \leq C_{\Phi} |x|^2$ for $k\in \{1,\ldots,q\}$ and $x\in \R^{\I^{(k)}}$.
\item $C_{\Phi}^{-1} J^{(k)}\leq W^{(k)}W^{(k),T}\leq C_{\Phi} J^{(k)}$ for $k\in \{2,\ldots,q\}$.
\end{enumerate}
\end{Condition}
We write $\Cond(X):=\frac{\lambda_{\max}(X)}{\lambda_{\min}(X)}$ for the condition number of a symmetric positive matrix $X$.
\begin{Theorem}\label{corunbcnOR}
Under Condition \ref{cond1OR} it holds true that there exists a constant $C$ depending only on $C_{\Phi}$ such that
\begin{equation}\label{corunbcnORfe}
\|u-u^{(k)}(u)\|\leq C H^k \|Q^{-1} u\|_0,
\end{equation}
 for $u\in Q \B_0$ and $k\in \{1,\ldots,q\}$.
Furthermore, $C^{-1} I^{(1)} \leq A^{(1)}\leq C H^{-2} I^{(1)}$,  $\operatorname{Cond}(A^{(1)})\leq  C H^{-2}$,
and
\begin{equation}
  C^{-1} H^{-2(k-1)} J^{(k)} \leq B^{(k)}\leq C H^{-2k} J^{(k)}
\end{equation}
 for $k\in \{2,\ldots,q\}$.  In particular,
   \begin{equation}
   \operatorname{Cond}(B^{(k)})\leq  C H^{-2}\,.
   \end{equation}
\end{Theorem}

 \begin{figure}[h!]
	\begin{center}
			\includegraphics[width=0.8 \textwidth]{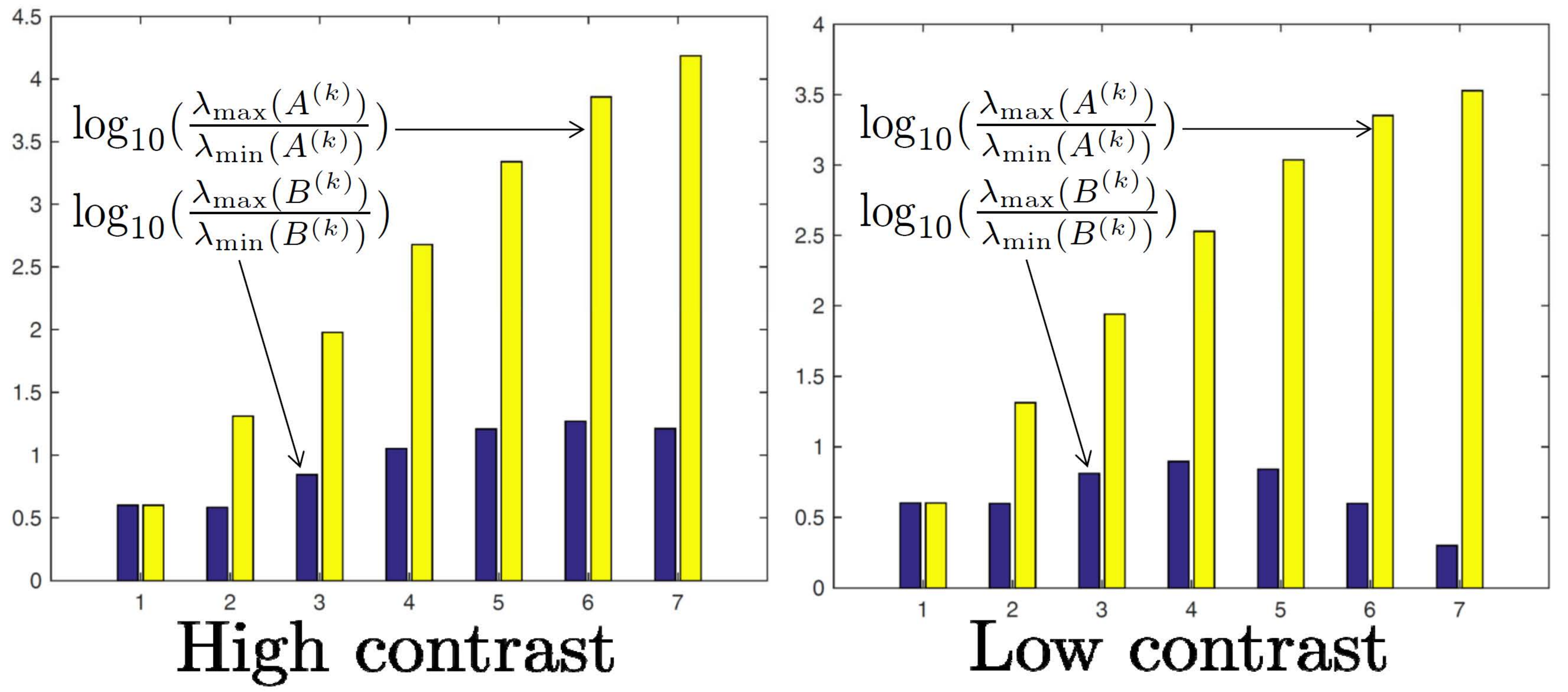}
		\caption{Condition numbers of $A^{(k)}$  and $B^{(k)}$, in $\log_{10}$ scale,  for $1\leq k \leq 7$. See Illustration \ref{ilfigbcn}.
}\label{figbcn}
	\end{center}
\end{figure}

\begin{NE}\label{ilfigbcn}
Figure \ref{figbcn} provides, in the context of Example \ref{egproto00} and Illustrations \ref{nekdjdhhu} and \ref{nekdjdhhu2},  an illustration of the condition numbers of the matrices $A^{(k)}$ and $B^{(k)}$. We define contrast as the ratio $\lambda_{\max}(a)/\lambda_{\min}(a)$. Figure \ref{figevranges} provides an illustration of the ranges of   the eigenvalues  of $A^{(7)}$, $A^{(1)}$, $B^{(2)},\ldots, B^{(7)}$. While the subspaces $\V^{(1)}, \W^{(2)}, \W^{(3)},\ldots$ are not exact eigenspaces
(e.g. in the context of Example \ref{egproto00} they are not orthogonal in $L^2(\Omega)$ and the angle between two successive subspace is of the order of a power of $H$) they retain several important characteristics of eigensubspaces: (1) Theorem \ref{thmgugyug2OR} shows that they are orthogonal in the $\<\cdot,\cdot\>$ scalar product (e.g., for  Example \ref{egproto00},  in the scalar product associated with the energy norm $\|\cdot\|_a$) (2) Theorem \ref{corunbcnOR} (and Figure \ref{figevranges}) shows that the ranges of eigenvalues of the operator $Q$ within each subspace define intervals of uniformly bounded lengths in log scale (3)
\cite{OwZh:2016} shows that the projections of the solutions of the
hyperbolic and parabolic versions of \ref{figevranges} on subspaces $\W^{(k)}$ (obtained using, possibly complex valued gamblets, that are non only adapted to the coefficients of the PDE but also to the implicit numerical scheme used for its resolution) produces space-time multiresolution decompositions of those solutions
(the evolution of their projected solution on $\W^{(k)}$ is slow for $k$ small and fast for $k$ large). In that sense,
 gamblets induce a multiresolution decomposition of $\B$ that is, to some degree, adapted to the eigen-subspaces of the operator $Q$.
\end{NE}

 \begin{figure}[h!]
	\begin{center}
			\includegraphics[width=0.8 \textwidth]{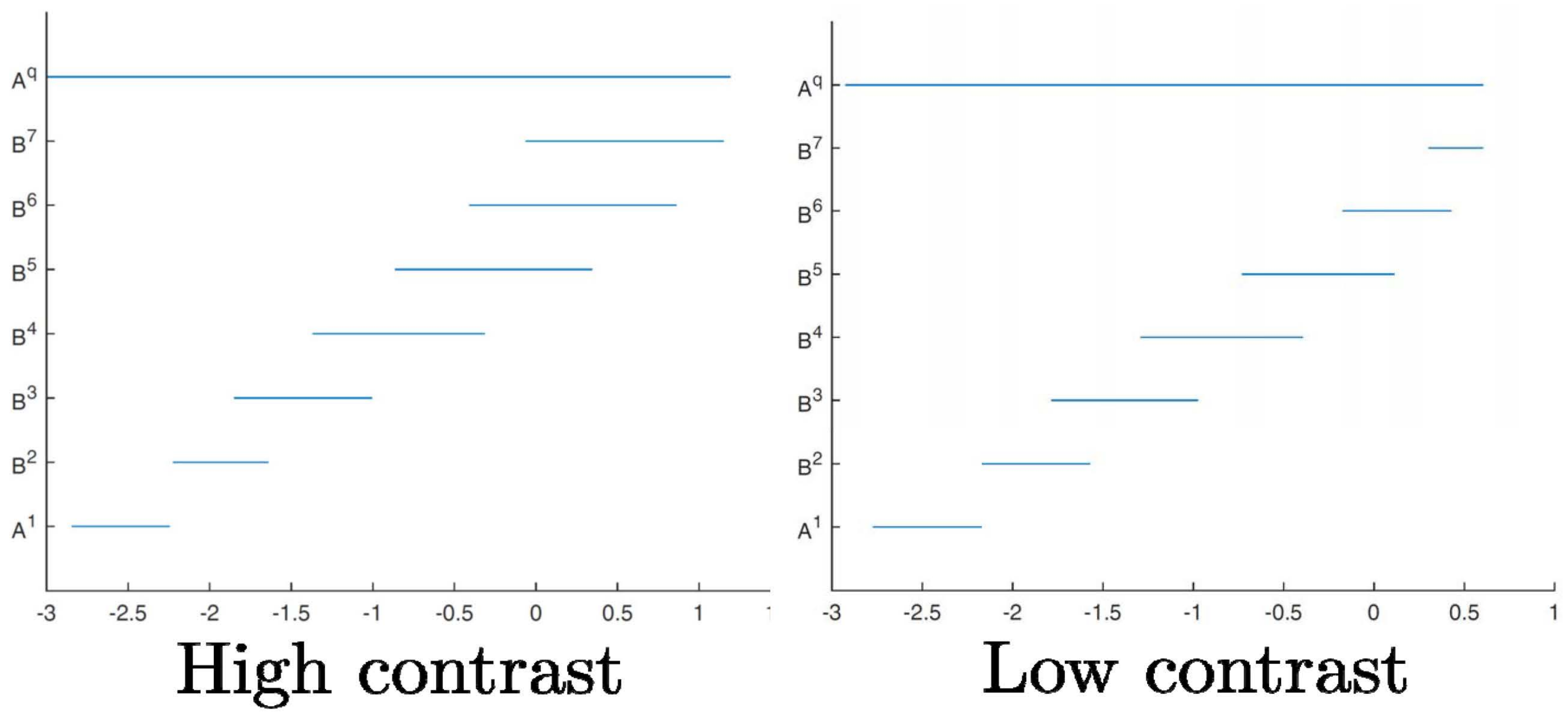}
		\caption{Ranges of the eigenvalues  of $A^{(7)}$, $A^{(1)}$, $B^{(2)},\ldots, B^{(7)}$, in $\log_{10}$ scale.  See Illustration \ref{ilfigbcn}.
}\label{figevranges}
	\end{center}
\end{figure}

\subsection{Hierarchical computation of gamblets}
Write $B^{(k),-1}$ the inverse of $B^{(k)}$.
Let $N^{(k)}$ be the $\I^{(k)}\times \J^{(k)}$ matrix defined by
\begin{equation}\label{eqjdhiudhiue}
N^{(k)}:=A^{(k)} W^{(k),T} B^{(k),-1},
\end{equation}
and write $N^{(k),T}$ its transpose.
Write $\pi^{(k+1,k)}$ the transpose of $\pi^{(k,k+1)}$.  $\text{rank}(\pi^{(k,k+1)})=|\I^{(k)}|$ implies that $\pi^{(k,k+1)}\pi^{(k+1,k)}$ is invertible. Let $\bar{\pi}^{(k,k+1)}$
 be the $\I^{(k)}\times \I^{(k+1)}$ matrix defined as the pseudo-inverse of $\pi^{(k+1,k)}$, i.e.
\begin{equation}\label{eqpibardef}
\bar{\pi}^{(k,k+1)}=(\pi^{(k,k+1)}\pi^{(k+1,k)})^{-1} \pi^{(k,k+1)}\,.
\end{equation}

Since the  spaces $\V^{(k)}$ are nested there exists a $\I^{(k)}\times  \I^{(k+1)}$   matrix $R^{(k,k+1)}$ such that
for $1\leq k \leq q-1$ and $i\in \I^{(k)}$
\begin{equation}\label{eq:ftfytftfx}
\psi^{(k)}_i=\sum_{j \in  \I^{(k+1)}} R_{i,j}^{(k,k+1)} \psi_j^{(k+1)}
\end{equation}
It is natural, using an analogy with multigrid theory, to refer to  $R^{(k,k+1)}$ as the restriction matrix and to its transpose
$R^{(k+1,k)}:=(R^{(k,k+1)})^T$ as the interpolation/prolongation matrix.
The following theorem is the basis of Algorithm \ref{alggambletcomutationnes}, which describes how gamblets are computed in a nested manner (from level $q$ to level $1$) by solving well conditioned linear systems.

\begin{Theorem}\label{thmpsdk}
It holds true that for $k\in \{1,\ldots,q-1\}$ and $i\in \I^{(k)}$,
\begin{equation}\label{eqdidhduhh}
\psi_i^{(k)}=\sum_{l\in \I^{(k+1)}}\bar{\pi}^{(k,k+1)}_{i,l}\psi^{(k+1)}_l- \sum_{j\in \J^{(k+1)}} (\bar{\pi}^{(k,k+1)} N^{(k+1)})_{i,j} \chi^{(k+1)}_j\,.
\end{equation}
In particular,
\begin{equation}\label{eqhuhiddeuv}
R^{(k,k+1)}= \bar{\pi}^{(k,k+1)}(I^{(k+1)}- N^{(k+1)} W^{(k+1)})\,.
\end{equation}
\end{Theorem}

\begin{algorithm}[!ht]
\caption{Hierarchical computation of gamblets.}\label{alggambletcomutationnes}
\begin{algorithmic}[1]
\STATE\label{step5giOR}  $A^{(q)}_{i,j}=\<\psi^{(q)}_i,\psi^{(q)}_j\> $   \COMMENT{Level $q$, $\I^{(q)}\times \I^{(q)}$ stiffness matrix}
\FOR{$k=q$ to $2$}
\STATE\label{step7gOR} $B^{(k)}= W^{(k)}A^{(k)}W^{(k),T}$ \COMMENT{Eq.~\eqref{eqjgfytfjhyyyg}}
\STATE\label{step9gOR}  For $i\in \J^{(k)}$, $\chi^{(k)}_i=\sum_{j \in \I^{(k)}} W_{i,j}^{(k)} \psi_j^{(k)}$  \COMMENT{Eq.~\eqref{eqjkhdkdh}}
\STATE\label{step11gpiOR} $\bar{\pi}^{(k-1,k)}=(\pi^{(k-1,k)}\pi^{(k,k-1)})^{-1} \pi^{(k-1,k)}$ \COMMENT{Eq.~\ref{eqpibardef}}
\STATE\label{step11gOR}  $ N^{(k)}= A^{(k)} W^{(k),T} B^{(k),-1}$ \COMMENT{Eq.~\eqref{eqjdhiudhiue}}
\STATE\label{step12gOR} $R^{(k-1,k)}=\bar{\pi}^{(k-1,k)}(I^{(k)}- N^{(k)} W^{(k)})$ \COMMENT{Eq.~\eqref{eqhuhiddeuv}}
\STATE\label{step13gOR} $A^{(k-1)}= R^{(k-1,k)}A^{(k)}R^{(k,k-1)}$ \COMMENT{Eq.~\eqref{eqhuhiuv}}
\STATE\label{step14gOR} For $i\in \I^{(k-1)}$, $\psi^{(k-1)}_i=\sum_{j \in  \I^{(k)}} R_{i,j}^{(k-1,k)} \psi_j^{(k)}$ \COMMENT{Eq.~\eqref{eq:ftfytftfx}}
\ENDFOR
\end{algorithmic}
\end{algorithm}

\subsection{Discrete Gamblet transform}\label{subsecdisgambtra}

Let $u\in \B$. For $k\in \{1,\ldots,q\}$ let $b^{(k)}\in \R^{\I^{(k)}}$ be defined by $b^{(k)}_i:=\<\psi_i^{(k)},u\>$ (observe that \eqref{eq:ftfytftfx} implies $b^{(k-1)}=R^{(k-1,k)} b^{(k)}$).
Let
\begin{equation}\label{eqljekdhldkjd}
v^{(1)}:=\sum_{i\in \I^{(1)}} w^{(1)}_i \psi_i^{(1)}
\end{equation}
 where $w^{(1)}\in \R^{\I^{(1)}}$ is the solution
of
\begin{equation}\label{eqawb}
A^{(1)} w^{(1)}=b^{(1)}\,.
\end{equation}
Let
\begin{equation}\label{eqldjhdkjhed}
v^{(k)}:=\sum_{i\in \J^{(k)}} w^{(k)}_i \chi_i^{(k)}
\end{equation}
 where $w^{(k)}\in \R^{\J^{(k)}}$ is the solution of
\begin{equation}\label{eqjhgiggy6}
B^{(k)} w^{(k)}=W^{(k)}b^{(k)}\,.
\end{equation}

The following theorem is the basis of Algorithm \ref{alggamblettransf}, which describes how the gamblet transform of $u\in \B$, i.e. the decomposition of $u\in \B$ over $\V^{(1)}\oplus \W^{(2)}\oplus \cdots \oplus \W^{(q)}$, can be obtained by solving $q$ independent and well conditioned (under Condition \ref{cond1OR}, by Theorem \ref{corunbcnOR}) linear systems.

\begin{Theorem}\label{thmreffytfuyf}
Let $u\in \B$ and, for $k\in \{1,\ldots,q\}$.
It holds true that $u^{(q)}(u)=\sum_{k=1}^q v^{(k)}$ where the $v^{(k)}$ are defined in \eqref{eqljekdhldkjd} and \eqref{eqldjhdkjhed}
via the solutions of the (independent) linear systems \eqref{eqawb} and
and \eqref{eqjhgiggy6}. Furthermore $v^{(1)}=u^{(1)}(u)$ and for $k\in \{2,\ldots,q\}$, $v^{(k)}=u^{(k)}(u)-u^{(k-1)}(u)$.
\end{Theorem}

\begin{algorithm}[!ht]
\caption{Gamblet transform of $u\in \B$.}\label{alggamblettransf}
\begin{algorithmic}[1]
\STATE\label{step4gOR} For $i\in \I^{(q)}$, $b^{(q)}_i=\<\psi_i^{(q)},u\>$
\FOR{$k=q$ to $2$}
\STATE\label{step8gOR} $w^{(k)}=B^{(k),-1} W^{(k)} b^{(k)}$ \COMMENT{Eq.~\eqref{eqjhgiggy6}}
\STATE\label{step10gOR} $v^{(k)}=\sum_{i\in \J^{(k)}}w^{(k)}_i \chi^{(k)}_i$ \COMMENT{Thm.~\ref{thmreffytfuyf}, $v^{(k)}:=u^{(k)}-u^{(k-1)}\in \W^{(k)}$}
\STATE\label{step15gOR} $b^{(k-1)}=R^{(k-1,k)} b^{(k)}$ \COMMENT{Eq.~\eqref{eq:ftfytftfx}}
\ENDFOR
\STATE\label{step16gOR} $ w^{(1)}=A^{(1),-1}b^{(1)}$ \COMMENT{Eq.~\eqref{eqawb}}
\STATE\label{step17gOR} $v^{(1)}=\sum_{i \in \I^{(1)}} w^{(1)}_i \psi^{(1)}_i$ \COMMENT{Thm.~\ref{thmreffytfuyf}}
\STATE\label{step18gOR} $u^{(q)}(u)=v^{(1)}+v^{(2)}+\cdots+v^{(q)}$    \COMMENT{Thm.~\ref{thmreffytfuyf}}
\end{algorithmic}
\end{algorithm}

\subsection{Discrete gamblet transform/computation}\label{subsecgamdis}
When $\B$ is infinite dimensional  the practical application of the gamblet transform may require its discretization.
Algorithm \ref{alggambletcomutationnes} and \ref{alggamblettransf} only  require
the specification of these level $q$ gamblets and their stiffness matrices for their applications.
Therefore, when explicit/analytical formulas are available for level $q$ gamblets $\psi_i^{(q)}$ this discretization can naturally be done on the corresponding (meshless) gamblet basis.
However, when such formulas are not available the gamblet transform must be applied to a prior discretization $\B^{\d}$ of the space $\B$. This discretization can be done using linearly independent basis elements $(\varPsi_i)_{i\in \N}$ spanning
\begin{equation}\label{eqkdklrflkff}
\B^\d:=\Span\{\varPsi_i \mid i \in \N\},
\end{equation}
 a finite-dimensional subspace of $\B$.  $\N$ is a finite-set  and we write $N:=|\N|$. Let $\I^{(q)}$ be an index tree of depth $q$ (as in Def.~\ref{defindextree}) relabeling $\N$.

For $k\in \{2,\ldots,q\}$ let $\pi^{(k-1,k)}$ and $W^{(k)}$ be as in Construction \ref{constpi} and \ref{conswk}.
 Algorithm \ref{algdiscgambletsolvecase1g} describes the nested calculation of discrete gamblets from the basis functions $(\varPsi_i)_{i\in \N}$ and the corresponding gamblet transform of $u\in \B^\d$.

\begin{algorithm}[!ht]
\caption{Discrete Gamblet computation and transform of $u\in \B^\d$.}\label{algdiscgambletsolvecase1g}
\begin{algorithmic}[1]
\STATE\label{step3gdgpbdis} For $i\in \I^{(q)}$, $\psi^{(q)}_i= \varPsi_i$  \COMMENT{Level $q$ gamblets}
\STATE\label{step4gdis} Use Algorithm \ref{alggambletcomutationnes} to compute the  gamblets $\psi_i^{(k)}$, $\chi_i^{(k)}$
and matrices $A^{(k)}$, $B^{(k)}$ and $R^{(k-1,k)}$
\STATE\label{step7gdis} Use Algorithm \ref{alggamblettransf} to compute the gamblet transform of $u$  \COMMENT{$u=u^{(q)}(u)=v^{(1)}+v^{(2)}+\cdots+v^{(q)}$ }
\end{algorithmic}
\end{algorithm}

Let $A$  be the $\I^{(q)}\times \I^{(q)}$ stiffness matrix defined by
\begin{equation}\label{eqstiffnmata}
A_{i,j}=\<\varPsi_i,\varPsi_j\>\,.
\end{equation}
Define $\pi^{(k,q)}$ as in \eqref{eqpikq}.

\begin{Condition}\label{conddiscrip3ordismatdis}
There exists constants
$C_\d\geq 1$ and $H\in (0,1)$  such that the following conditions are satisfied.
\begin{enumerate}
\item\label{linconmat1dis} $C_\d^{-1} J^{(k)}\leq W^{(k)}W^{(k),T}\leq C_\d J^{(k)}$ for $k\in \{2,\ldots,q\}$.
\item\label{linconmat3dis}  $C_\d^{-1} I^{(k)} \leq \pi^{(k,q)}\pi^{(q,k)} \leq C_\d I^{(k)}$ for $k\in \{1,\ldots,q-1\}$.
\item\label{linconmat4dis}  $\sup_{x\in \R^{\I^{(q)}}} \inf_{y\in \R^{\I^{(k)}}} \frac{\sqrt{ (x-\pi^{(q,k)}y)^T A^{-1} (x-\pi^{(q,k)}y)}}{|x|}\leq \frac{C_\d}{ \sqrt{\lambda_{\min}(A)}} H^k$ for $k\in \{1,\ldots,q-1\}$.
\item\label{linconmat5dis}  $\frac{1}{C_\d \sqrt{\lambda_{\min}(A)}}H^k  \leq \inf_{y\in \R^{\I^{(k)}}} \frac{\sqrt{y^T \pi^{(k,q)} A^{-1}\pi^{(q,k)} y}}{|y|}$ for $k\in \{1,\ldots,q\}$.
\item\label{linconmat6dis}
$\sup_{y\in \Ker(\pi^{(k-1,k)}) }\frac{\sqrt{y^T \pi^{(k,q)} A^{-1}\pi^{(q,k)}y}}{|y|}\leq \frac{C_\d}{ \sqrt{\lambda_{\min}(A)}} H^{k-1}$ for $k\in \{2,\ldots,q\}$.
\end{enumerate}
\end{Condition}

\begin{Theorem}\label{thmconddisbndisbismatdis}
Let $A^{(1)}$ and $(B^{(k)})_{k\in \{2,\ldots,q\}}$ be the matrices computed in Algorithm \ref{algdiscgambletsolvecase1g}.
Under Conditions \ref{conddiscrip3ordismatdis} there exists a constant $C$ depending only on $C_\d$ such that
$C^{-1}  I^{(1)} \leq \frac{A^{(1)}}{\lambda_{\min}(A)}\leq C H^{-2} I^{(1)}$
  and $\operatorname{Cond}(A^{(1)})\leq  C H^{-2}$.
Furthermore, for $k\in \{2,\ldots,q\}$ it holds true that
$ C^{-1}H^{-2(k-1)} J^{(k)}   \leq \frac{B^{(k)}}{\lambda_{\min}(A)} \leq C H^{-2k} J^{(k)}$
   and $\operatorname{Cond}(B^{(k)})\leq  C H^{-2}$.
\end{Theorem}

\begin{Remark}\label{rmknt}
If $\pi^{(k,q)}\pi^{(q,k)}$ is a multiple of identity for $k\in \{1,\ldots,q-1\}$ then items \ref{linconmat4dis} and \ref{linconmat6dis} of Condition \ref{conddiscrip3ordismatdis} are equivalent to
$\sqrt{x^T A^{-1} x}\leq \frac{C_\d}{\sqrt{\lambda_{\min}(A)}} H^k |x|$ for $x\in \Ker(\pi^{(q,k)})$ and $k\in \{1,\ldots,q-1\}$.
\end{Remark}

 \begin{figure}[h!]
	\begin{center}
			\includegraphics[width=0.75 \textwidth]{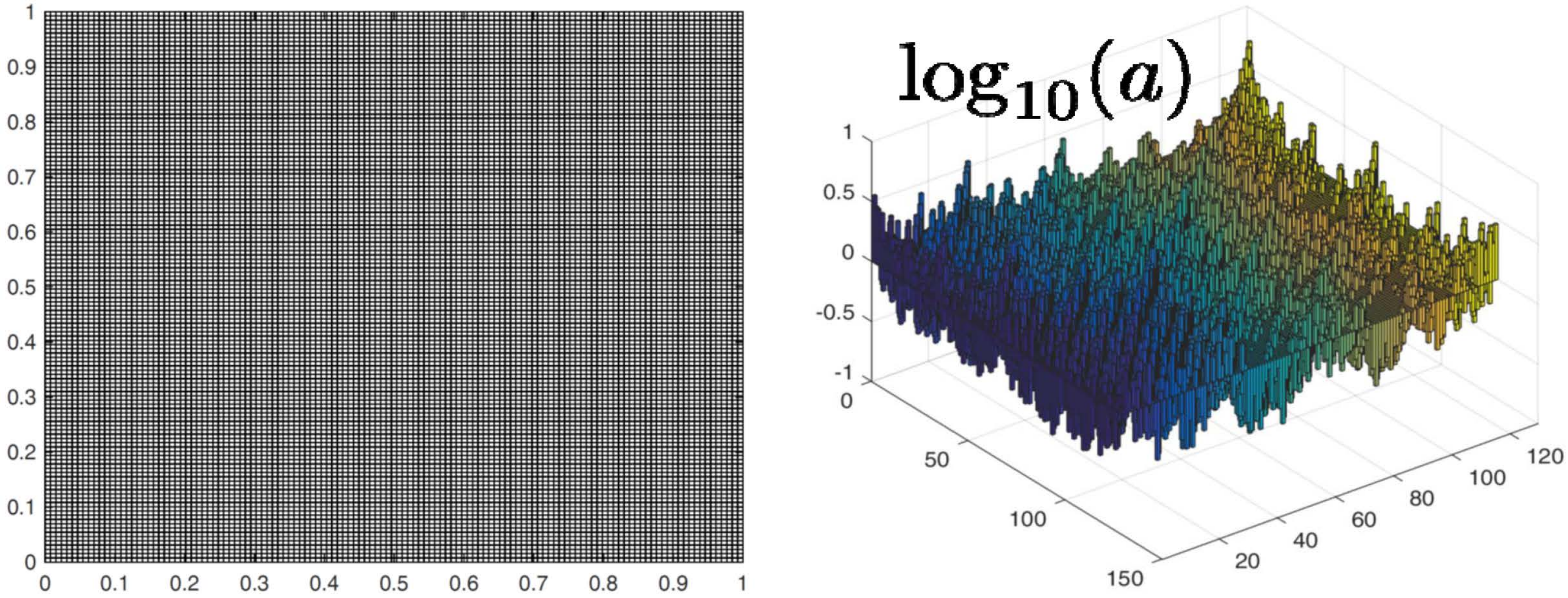}
		\caption{The fine grid and $a$ in $\log_{10}$ scale. See Illustration \ref{illoganumsim}.
}\label{figlog10a}
	\end{center}
\end{figure}
 \begin{figure}[h!]
	\begin{center}
			\includegraphics[width=\textwidth]{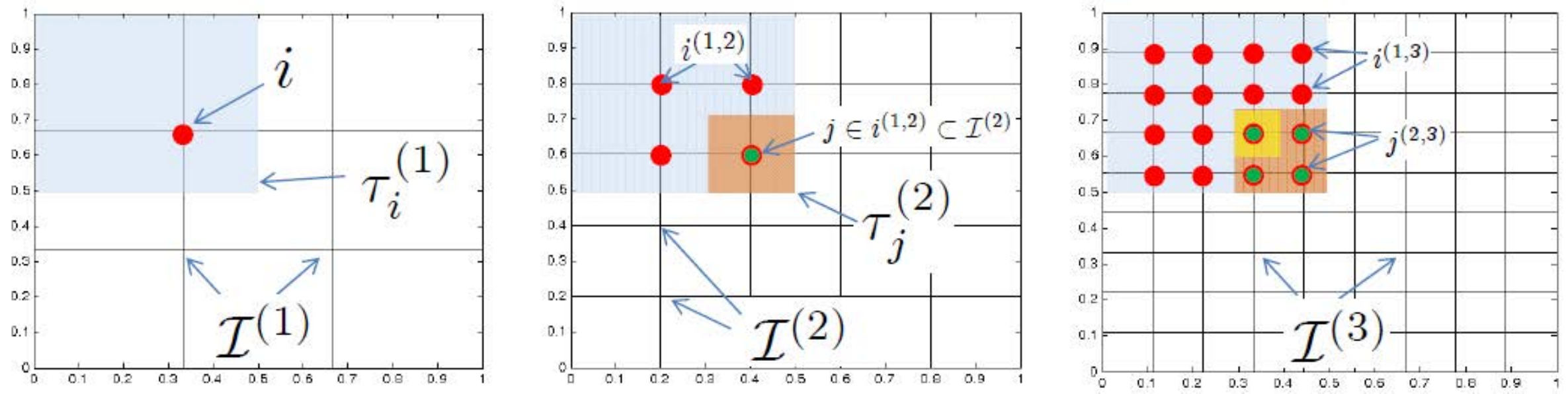}
		\caption{$\I^{(1)}$, $\I^{(2)}$ and $\I^{(3)}$. See Illustration \ref{illoganumsim}.}\label{fig:Pi}
	\end{center}
\end{figure}

\begin{NE}\label{illoganumsim}
Our running numerical example is obtained from the numerical discretization of  Example \ref{egproto00}.
More precisely we consider the uniform grid of $\Omega=(0,1)^2$ with $2^q\times 2^q$ interior points ($q=7$) illustrated in Figure
\ref{figlog10a}. $a$ is piecewise constant on each square of that grid, and given by
$a(x)=\prod_{k=1}^7 \Big(1+0.2 \cos\big(2^k \pi (\frac{i}{2^q+1}+\frac{j}{2^q+1})\big)\Big) \Big(1+0.2 \sin\big(2^k \pi (\frac{j}{2^q+1}-3\frac{i}{2^q+1})\big)\Big)$, as illustrated in $\log_{10}$ scale in Figure \ref{illoganumsim}.
To construct the hierarchy of indices $\I^{(k)}$ we partition the unit square into nested sub-squares $\tau_i^{(q)}$ of side $2^{-k}$ as illustrated in Figure \ref{fig:Pi} and we label each node $i\in \I^{(q)}$ of the fine mesh (of resolution $2^{-q}$) by the square $\tau_i^{(q)}$ of side $2^{-q}$ containing that node, and we determine the hierarchy of labels through set inclusion: i.e. for $i\in \I^{(q)}$, $i^{(k)}\in \I^{(k)}$ is the label of the square of side $2^{-k}$ containing $i$. The finite-element discretization  of $\B$ is  obtained using
continuous nodal bilinear basis elements $\varPsi_i$ spanned by $\{1,x_1,x_2,x_1 x_2\}$ in each square of the fine mesh. As described in Algorithm \ref{algdiscgambletsolvecase1g}, these fine mesh bilinear finite elements form our level $q$ gamblets (i.e. $\psi_i^{(q)}=\varPsi$). Although the continuous Gamblet Transform relies on the specification of the measurement functions $\phi_i^{(k)}$, the discrete Gamblet Transform only requires the specification of the matrices $\pi^{(k,k+1)}$ and $W^{(k)}$. For our numerical example these matrices are the same as those presented in a continuous setting in Illustrations \ref{nekdjdhhu} and \ref{nekdjdhhu2}. We refer to Figures \ref{figpsi7} and \ref{figchi7} for the corresponding illustrations of the gamblets $\psi_i^{(k)}$ and $\chi_i^{(k)}$. We refer to Figures \ref{figbcn} and \ref{figevranges} for the corresponding illustrations of the condition numbers of $A^{(k)}$, $B^{(k)}$ and the intervals containing their eigenvalues.
\end{NE}

\section{Operator inversion with the gamblet transform}\label{secsolvinvpb}

Here we discuss the gamblet transform in the context of inverse problems and its relationship with the
 choice of  error norms in the context of nonstandard dual pairings.
\subsection{The inverse problem and its variational formulation}
Let $(\B_{2},\|\cdot\|_{2})$  be a separable Banach space and let
 $\L:\B \rightarrow \B_{2}$ be a continuous linear bijection.
Since $\L$ is continuous, and therefore bounded,
 it follows from the open mapping theorem that $\L^{-1}$ is also bounded.
 Let us define its continuity constants to be
\begin{equation}\label{eqdjkehjdidudhus}
C_{\L^{-1}}:=\sup_{g \in \B_2} \frac{\|\L^{-1}g\|}{\|g\|_2} \text{ and } C_{\L}:=
\sup_{v \in \B} \frac{\|\L v \|_2}{\|v\|}.
\end{equation}
For a given $g \in  \B_2$, consider solving the inverse problem
\begin{equation}\label{eqn:scalar}
\L u =g\,
\end{equation}
for $u\in \B$.
Since $Q^{-1}u=Q^{-1}\L^{-1}g$ and $\<u,v\>=[Q^{-1}u,v]$, \eqref{eqn:scalar} is equivalent to the weak formulation (in the error norm $\|\cdot\|$)
\begin{equation}\label{eqn:scalarfem}
\<u,v\>=[Q^{-1}\L^{-1}g,v] \text{ for }v\in \B\,.
\end{equation}

\begin{Example}\label{egprotoa}
We will consider the following prototypical PDE as a running illustrative example,
\begin{equation}\label{eqn:scalarprotoa}
\begin{cases}
    -\diiv \big(a(x)  \nabla u(x)\big)=g(x) \quad  x \in \Omega; \\
    u=0 \quad \text{on}\quad \partial \Omega,
    \end{cases}
\end{equation}
where $\Omega$ is as in Example \ref{egproto00}, and  $a$ is a uniformly elliptic $d\times d$ matrix (which may or may not be symmetric) with entries in $L^\infty(\Omega)$. We define $\lambda_{\min}(a)$ as the largest constant and $\lambda_{\max}(a)$ as the smallest constant such that for all $x\in \Omega$ and $l\in \R^d$,
\begin{equation}
\lambda_{\min}(a) |l|^2 \leq l^T a(x) l \text{ and } l^T a^T(x)a(x) l \leq \big(\lambda_{\max}(a)\big)^2|l|^2.
\end{equation}
\end{Example}

\subsection{Identification of $Q$ and variational formulation of $\L u =g$}\label{subgeneralcaseNNNOR}
The practical application of the gamblet transform to the resolution of \eqref{eqn:scalarfem} requires the identification of the operator $Q$ identifying the error norm.
\subsubsection{General case}\label{subgeneralcaseNNN}
The identification of the operator $Q$ can be done by selecting $\G$, a self-adjoint positive continuous linear bijection from $\B_2$ onto $\B_2^*$ and $\D$,
a continuous linear bijection from $\B_2^*$ onto $\B^*$ such that $\D \G \L$ is self-adjoint and positive.
 $Q$ is then defined by $Q=\L^{-1} \G^{-1} \D^{-1}$, i.e.
\begin{equation}
\label{gc_constr}
Q^{-1}=\D \G \L
\end{equation}
making the following diagram
\begin{equation}\label{eqcase1NNN}
\text{\xymatrix{
\B \ar[d]^{Q^{-1}} \ar[r]^\L &\B_2\ar[d]^\G\\
\B^*           &\B_2^* \ar[l]^{\D}}}
\end{equation}
commutative. Throughout the rest of this paper all such diagrams will be commutative.
$\G$ can, a priori, be chosen independently from $\L$ and the definition $\|u\|^2=[Q^{-1}u, u]$ leads to
\begin{equation}
\|u\|^2=[\D \G \L u,  u]
\end{equation}
The variational formulation of the equation $\L u=g$ is then
\begin{equation}\label{eqlfkjNNN}
[Q^{-1} u, \varPsi]=[\D \G g,\varPsi]\text{ for } \varPsi\in \B
\end{equation}
which can be written (using the scalar product $\<u,\varPsi\>=[Q^{-1} u, \varPsi]$)
\begin{equation}\label{eqklajhdjfNNN}
\<u,\varPsi\>=[\D \G g,\varPsi] \text{ for } \varPsi\in \B
\end{equation}

\begin{Example}\label{egprotoh10normNNN}
Consider the prototypical Example \ref{egprotoa} and assume $a$ to be symmetric.  Let
$\B=H^1_0(\Omega)$ and define $\|\cdot\|$ as the energy norm $\|u\|^2=\|u\|_a^2:=\int_{\Omega} (\nabla u)^T a \nabla u$ (as in Example \ref{egproto00}).
Let $\B_2=H^{-1}(\Omega)$ and let $\|g\|_2=\|g\|_{H^{-1}(\Omega)}=\sup_{v\in H^1_0(\Omega)}\frac{\int_{\Omega} gv}{\|v\|_{H^1_0(\Omega)}}$ where $\|v\|_{H^1_0(\Omega)}^2:=\int_{\Omega}|\nabla v|^2$.
Observe that $\B_2^*=H^1_0(\Omega)$ and  the (dual) norm on $\B_2^*$ is $\|v\|_2^*=\|v\|_{H^1_0(\Omega)}$.
 The gamblet transform for the PDE \eqref{eqn:scalarprotoa} can be defined by (1)
considering  the operator $\L=-\diiv(a\nabla)$ mapping  $\B$ onto $\B_2$ (2) taking $\D=-\Delta$ and $\G=\D^{-1}=-\Delta^{-1}$ in Diagram \ref{eqcase1NNN}, where    $\D=-\Delta$ is the Laplace-Dirichlet operator mapping $H^1_0(\Omega)$ onto $H^{-1}(\Omega)$ (note that $\|v\|_2^2=[\G^{-1} v,v]$). Under these choices $Q^{-1}$ is (as in Example \ref{egproto00}) the operator $- \diiv(a\nabla)$ illustrated in the following diagram
\begin{equation}\label{eqcase1exnewbNNN}
\text{\xymatrixcolsep{4pc}\xymatrix{
H^1_0(\Omega) \ar[d]^{Q^{-1}} \ar[r]^{-\diiv(a\nabla)} &H^{-1}(\Omega)\ar[d]^{-\Delta^{-1}}\\
H^{-1}(\Omega)           &H^1_0(\Omega)\,\, . \ar[l]^{-\Delta}}\, }\,
\end{equation}
The variational formulation of the PDE \eqref{eqn:scalarprotoa} can then be written
\begin{equation}\label{eqklajhdjdef2bihjhsNNN}
\int_{\Omega}(\nabla \varPsi)^T a  \nabla u=\int_{\Omega} \varPsi g \text{ for } \varPsi\in H^1_0(\Omega)\, .
\end{equation}
Observe that by placing the norm $\|g\|_{H^{-1}(\Omega)}$ on $\B_2$ we have $C_\L\leq \sqrt{\lambda_{\max}(a)}$ and  $C_{\L^{-1}}\leq 1/\sqrt{\lambda_{\min}(a)}$.
\end{Example}

\begin{Example}\label{egprotoalsobolevl}
Let $\L$ be a continuous, symmetric, positive linear bijection from $H^s_0(\Omega)$ to $H^{-s}(\Omega)$ and
define $Q:=\L^{-1}$ to be its inverse.
 Write $\|u\|:=[ \L u,u]^\frac{1}{2}$ and let $\|\cdot\|_*$ be the dual norm of $\|\cdot\|$.
Define $(\B,\|\cdot\|)$ and $(\B^*,\|\cdot\|_*)$ as in Example \ref{egprotoalsobolevlori}.
 Consider the inverse problem \eqref{eqn:scalar}.
Let $\B_2=H^{-s}(\Omega)$ and let $\|g\|_2=\|g\|_{H^{-s}(\Omega)}$. Observe that
 $\B_2^*=H^s_0(\Omega)$ and  the (dual) norm on $\B_2^*$ is $\|v\|_2^*=\|v\|_{H^s_0(\Omega)}$.
 The gamblet transform for the inverse problem \eqref{eqn:scalar} can then be defined by
considering  the operator $\L$ mapping  $\B$ onto $\B_2$ (2) taking $\D=(-\Delta)^s$ and $\G=\D^{-1}$ in
 Diagram \ref{eqcase1NNN},
where $(- \Delta)^{s}$ is the operator defined by the $s$-th iterate of the Laplacian mapping $H^s_0(\Omega)$ to $H^{-s}(\Omega)$. Under these choices $Q^{-1}$ is  the operator $\L$ as illustrated in the following diagram

\begin{equation}\label{eqcase1exnedewbNdNN}
\text{\xymatrixcolsep{4pc}\xymatrix{
H^s_0(\Omega) \ar[d]^{Q^{-1}} \ar[r]^{\L} &H^{-s}(\Omega)\ar[d]^{\G}\\
H^{-s}(\Omega)           &H^s_0(\Omega)\,\, . \ar[l]^{(- \Delta)^s}}}
\end{equation}
\end{Example}

 \begin{Lemma}\label{lemdkljedhlkfjh}
The norm inequalities \eqref{eqjkhkkjhuiiu} in Example \ref{egprotoalsobolevlori} are equivalent to the continuity of the operator $\L$ (mapping $H^s_0(\Omega)$ to $H^{-s}(\Omega)$) and of its inverse in Example \ref{egprotoalsobolevl}.
Moreover,
let us introduce  two-parameter versions
 \begin{equation}\label{eqjkhkkjhuiiu2}
C_{e}^{-1}\|u\|_{H^s_0(\Omega)} \leq \|u\| \leq C_{e,2} \|u\|_{H^s_0(\Omega)},\,\text{ for } u\in H^s_0(\Omega)\,
\end{equation}
of \eqref{eqjkhkkjhuiiu} in Example \ref{egprotoalsobolevlori} and its consequence
\begin{equation}\label{eqjjhykhkkjdedduiiu2}
C_{e}^{-1}\|\phi\|_{H^{-s}(\Omega)} \leq \|\phi\|_* \leq C_{e,2} \|\phi\|_{H^{-s}(\Omega)},\,\text{ for } \phi\in H^{-s}(\Omega)\, .
\end{equation}
  Then, if  $C_{e}$  and $C_{e,2}$ are  the smallest constants such that \ref{eqjkhkkjhuiiu2} holds, then
$C_{e,2}^2=C_{\L}$, $C_{e}^2 \geq C_{\L^{-1}}$, and $C_{e}\leq  C_{\L} C_{\L^{-1}}^2$.
 \end{Lemma}

\begin{Example}\label{egkljkhdekjd}
A particular instance of Example \ref{egprotoalsobolevl} is the self-adjoint differential operator $\L$ mapping $H^s_0(\Omega)$ to $H^{-s}(\Omega)$ and defined by
\begin{equation}\label{eqjgkgjgkjhgjhg}
\L u = \sum_{0\leq |\alpha|, |\beta|\leq s} (-1)^{|\alpha|}D^{\alpha}(a_{\alpha,\beta}(x) D^{\beta} u),\, \text{ for }u\in H^s_0(\Omega)\, ,
\end{equation}
where $a$ is symmetric with entries in  $L^\infty(\Omega)$ and $\alpha, \beta$ are $d$ dimensional multi-indices $\alpha=(\alpha_1,\ldots,\alpha_d)$ (with $\alpha_i \in \mathbb{N}$ and $|\alpha|=\sum_{i=1}^d \alpha_i$).
Lemma \ref{lemdkljedhlkfjh} implies that \eqref{eqjkhkkjhuiiu} is equivalent to the continuity of $\L$ and $\L^{-1}$.
The continuity of $\L$ (right inequality in \eqref{eqjkhkkjhuiiu}) follows from the uniform bound on the entries of $a$. The left hand side of \eqref{eqjkhkkjhuiiu} is a classical coercivity condition (ensuring the well-posedness of \eqref{eqn:scalar}) and we refer to \cite{Agmon58coer} for its characterization.
\end{Example}

\subsubsection{General case with $\D=\L^*$}\label{subgeneralcase}
Write $\L^*$  the adjoint of $\L$ defined as the operator mapping $\B_2^*$ onto $\B^*$ such that for $(v,w^*)\in \B\times \B_2^*$, $[w^*,\L v]=[\L^* w, v]$.
Consider the general case discussed in Subsection \ref{subgeneralcaseNNN} and take $\D=\L^*$. In that case, we have $Q^{-1}=\L^* \G \L$, i.e.
\begin{equation}\label{eqkhhuhuhij}
Q= \L^{-1} \G^{-1} \L^{-1,*},
\end{equation}
and Diagram \ref{eqcase1NNN} reduces to
\begin{equation}\label{eqcase1}
\text{\xymatrix{
\B \ar[d]^{Q^{-1}} \ar[r]^\L &\B_2\ar[d]^\G\\
\B^*           &\B_2^* \,\,.\ar[l]^{\L^*}}}
\end{equation}
Furthermore, $\|u\|^2=[\G \L u, \L u]$ and the
 variational formulation of the equation $\L u=g$ is then
\begin{equation}\label{eqlfkj}
[Q^{-1} u, \varPsi]=[Q^{-1}\L^{-1}g,\varPsi],\,\text{ for } \varPsi\in \B\, ,
\end{equation}
which can be written (using the scalar product $\<u,\varPsi\>=[Q^{-1} u, \varPsi]=[\G \L u, \L \varPsi]$)
\begin{equation}\label{eqklajhdjf}
\<u,\varPsi\>=[\L^* \G g,\varPsi],\, \text{ for } \varPsi\in \B\, .
\end{equation}

\begin{Example}\label{egprotoflux}
Consider the prototypical Example \ref{egprotoa}.
 The gamblet transform for the PDE \eqref{eqn:scalarprotoa} can be defined by (1) considering the operator $\L=-\diiv(a\nabla)$ mapping  $\B=H^1_0(\Omega)$ onto $\B_2=H^{-1}(\Omega)$ (2) taking $\G=-\Delta^{-1}$ in
 Diagram \ref{eqcase1}, i.e. the inverse of  $-\Delta$, the Laplace-Dirichlet operator mapping $H^1_0(\Omega)$ onto $H^{-1}(\Omega)$. Under these choices $Q^{-1}$ is the operator $- \diiv(a^T\nabla)\Delta^{-1}\diiv(a\nabla \cdot)$ illustrated in the following diagram
\begin{equation}\label{eqcase1ex}
\text{\xymatrixcolsep{4pc}\xymatrix{
H^1_0(\Omega) \ar[d]^{Q^{-1}} \ar[r]^{-\diiv(a\nabla)} &H^{-1}(\Omega)\ar[d]^{-\Delta^{-1}}\\
H^{-1}(\Omega)           &H^1_0(\Omega)\,\,. \ar[l]^{-\diiv(a^T\nabla)}}}
\end{equation}
Furthermore, $\|\cdot\|$ is the flux norm  $\|u\|=\|u\|_\f$ introduced in \cite{BeOw:2010} (and generalized in \cite{Sym12, Wang:2012}, see also \cite{bebendorf2016low} for its application to low rank approximation with high contrast coefficients),   defined by $\|u\|^2_\f=\|\nabla \Delta^{-1}\L u\|_{L^2(\Omega)}^2$ (recall that \cite{BeOw:2010} $\|u\|_\f=\big\|(a\nabla u)_{\pot}\big\|_{(L^2(\Omega))^d}$ where $(a\nabla u)_{\pot}$ is the potential part of the vector field $a\nabla u$).
The variational formulation of the PDE \eqref{eqn:scalarprotoa}
 can then be written
\begin{equation}\label{eqklajhdjdef2}
\int_{\Omega}(a \nabla \varPsi)^T_{\pot}  (a\nabla u)_{\pot}=\int_{\Omega}(\nabla \varPsi)^T a^T \nabla  \Delta^{-1}g \text{ for } \varPsi\in H^1_0(\Omega)\, .
\end{equation}
Writing $\|u\|_{H^1_0(\Omega)}^2:=\int_{\Omega}|\nabla u|^2$, recall \cite{BeOw:2010} that for $u\in H^1_0(\Omega)$,
\begin{equation}
\lambda_{\min}(a) \|u\|_{H^1_0(\Omega)} \leq \|u\|_\f\leq \lambda_{\max}(a) \|u\|_{H^1_0(\Omega)}\, .
\end{equation}
Observe also that if $u$ is the solution of \eqref{eqn:scalarprotoa} then $\|u\|_\f=\|\Delta^{-1}g\|_{H^1_0(\Omega)}$. Since the flux norm of $u$ is independent from $a$,  $C_\L$ and $C_{\L^{-1}}$ are also independent from $a$. In particular placing the norm $\|g\|_2:=||g\|_{H^{-1}(\Omega)}=\|\nabla \Delta^{-1}g\|_{L^2(\Omega)}$ on $\B_2=H^{-1}(\Omega)$ we have $C_\L=1$ and $C_{\L^{-1}}=1$). Therefore, under these choices,
 the efficiency of the corresponding gamblet transform is robust to high contrast in $a$ (e.g. the conditions numbers of the stiffness matrix $B^{(k)}$ are uniformly bounded independently from $a$).
\end{Example}

\begin{Example}\label{egprotoalsobolevlnondivrep}
Let $\L$ be a continuous linear bijection from $H^{s}_0(\Omega)$ to $H^{-s}(\Omega)$. We do not assume $\L$ to be symmetric.
The inverse problem \eqref{eqn:scalar} is equivalent to
\begin{equation}\label{eqldjkehhgjcckjhjsk}
\L^* (-\Delta)^{-s} \L=\L^* (-\Delta)^{-s} g
\end{equation}
Let $Q=(\L^* (-\Delta)^{-s} \L)^{-1}$ be the symmetric operator mapping $H^{-s}(\Omega)$ to $H^s_0(\Omega)$.
 Write $\|u\|:=[Q^{-1}u,u]$ and let $\|\cdot\|_*$ be the dual norm of $\|\cdot\|$.
Define $(\B,\|\cdot\|)$ and $(\B^*,\|\cdot\|_*)$ as in Example \ref{egprotoalsobolevlori}.
 Consider the inverse problem \eqref{eqn:scalar}.
Let $\B_2=H^{-s}(\Omega)$ and let $\|g\|_2=\|g\|_{H^{-s}(\Omega)}$. Observe that
 $\B_2^*=\H^s_0(\Omega)$.
 The gamblet transform for the inverse problem \eqref{eqn:scalar} can then be defined by
 considering  the operator $\L$ mapping  $\B$ onto $\B_2$ (2) taking $\D=L^*$ and $\G=(-\Delta)^{-s}$ in
Diagram \ref{eqcase1NNN}. Under these choices $Q^{-1}$ is  the (continuous, symmetric, positive) operator $\L^* (-\Delta)^{-s} \L$  as illustrated in the following diagram
\begin{equation}\label{eqcase1exnedewbNdNNnondivrep}
\text{\xymatrixcolsep{4pc}\xymatrix{
H^s_0(\Omega) \ar[d]^{Q^{-1}} \ar[r]^{\L} &H^{-s}(\Omega)\ar[d]^{(-\Delta)^{-s}}\\
H^{-s}(\Omega)           &H^s_0(\Omega)\,\,. \ar[l]^{\L^*}}}
\end{equation}
Note that, under the norm $\|\cdot\|$, the solution $u$ of \eqref{eqn:scalar} satisfies $\|u\|=\|g\|_{H^{-s}(\Omega)}$ which leads to the robustness of the corresponding gamblets with respect to the  coefficients of $\L$.

\end{Example}

\subsubsection{Self-adjoint case with the energy norm}\label{subsecselfadjoint}
From the general case discussed in Subsection \ref{subgeneralcase}, assume that
  $\B_2=\B^*$, $\L$ is self-adjoint and positive definite (i.e.~for $v,w \in \B$, $[\L v,w]=[\L w,v]$, $[\L v, v]\geq 0$ and $[\L v v, v]=0$ is equivalent to $v=0$). Take $\G=\L^{-1}$, i.e.~$Q^{-1}=\L$. In that case $\|\cdot\|$ is the energy norm
  $\|u\|^2=[\L u, u]$ on $\B$ and Diagram \ref{eqcase1} reduces to
\begin{equation}\label{eqcase2}
\text{\xymatrix{ \B  \ar@/^/[rr]|\L  && \B^* . \ar@/^/[ll]|{Q} }}
\end{equation}
The variational formulation of the equation $\L u=g$ can then be written (using the scalar product $\<u,\varPsi\>=[\L u, \varPsi]$)
\begin{equation}\label{eqklajhdjf2}
\<u,\varPsi\>=[ g,\varPsi],\, \text{ for } \varPsi\in \B\, .
\end{equation}
Observe that, since  $\|g\|_2=\|g\|_*=\|Q^{-1} u\|_*=\|u\|$, we have  $C_\L=1$ and $C_{\L^{-1}}=1$.

\begin{Example}\label{egprotocase12}
Consider the prototypical Example \ref{egprotoa} when $a=a^T$.
The gamblet transform for the PDE \eqref{eqn:scalarprotoa} can be defined by (1) considering the (self-adjoint) operator $\L=-\diiv(a\nabla)$ mapping  $\B=H^1_0(\Omega)$ onto $\B_2=\B_1^*=H^{-1}(\Omega)$ (2) taking $\G=\L^{-1}$ in Diagram \ref{eqcase1}. Under these choices $Q^{-1}=\L$  and Diagram \ref{eqcase1} reduces to
\begin{equation}\label{eqcase2ex}
\text{\xymatrix{ H^1_0(\Omega)  \ar@/^/[rr]|{-\diiv(a\nabla \cdot)}  && H^{-1}(\Omega)\,. \ar@/^/[ll]|{Q} }}
\end{equation}
 $\|\cdot\|$ is the energy norm $\|u\|=\|u\|_a$, defined by $\|u\|_a^2:=\int_{\Omega}(\nabla u)^T a \nabla u$, for $u\in H^1_0(\Omega)$.
 $\|g\|_2=\|g\|_*$ is the dual norm $\|g\|_*=\sup_{v\in H^1_0(\Omega)}\frac{\int_{\Omega} g v}{\|v\|_a}=\|u\|_a$.
 Therefore,  $C_\L=1$ and $C_{\L^{-1}}=1$.
The variational formulation of the PDE \eqref{eqn:scalarprotoa} can then be written as in \eqref{eqklajhdjdef2bihjhsNNN}.
\end{Example}

\subsubsection{When $\B_2=\B_2^*$}\label{subsecselfadjointbis}
From the general case discussed in Subsection \ref{subgeneralcase}, assume that
  $\B_2^*=\B_2$ and choose $\G:=\textrm{id}$ to be the identity operator on $\B_2$. Then
$Q^{-1}=\L^*  \L$ (i.e. $Q=\L^{-1}  \L^{*,-1}$),  $\|u\|^2=[ \L u, \L u]$ for $u\in \B$, and Diagram \ref{eqcase1} reduces to
\begin{equation}\label{eqcase3}
\text{\xymatrix{
\B \ar[d]^{Q^{-1}} \ar[r]^\L &\B_2\ar[d]^{\textrm{id}}\\
\B^*           &\B_2^*\, .\ar[l]^{\L^*}}}
\end{equation}
The variational formulation of the equation $\L u=g$ can then be written (using the scalar product $\<u,\varPsi\>=[\L u, \L \varPsi]$)
\begin{equation}\label{eqklajhdjfds23}
\<u,\varPsi\>=[ g,\L \varPsi] \text{ for } \varPsi\in \B\, .
\end{equation}
Note that, by using the norm $\|g\|_2=[g,g]$ on $\B_2$, we have $C_{\L}=1$ and $C_{\L^{-1}}=1$.

\begin{Example}\label{egprotocase13}
Consider the prototypical Example \ref{egprotoa}.
The gamblet transform for the PDE \eqref{eqn:scalarprotoa} can be defined by (1) considering the operator $\L=-\diiv(a\nabla)$ mapping  $\B=\{u\in H^1_0(\Omega)\mid \diiv(a\nabla u)\in L^2(\Omega) \}$ onto $\B_2=L^2(\Omega)$ (2) taking $\G=\textrm{id}$ as in Diagram \ref{eqcase3}. Under these choices $Q^{-1}=\L^*  \L$  and $\|u\|=\|\L u\|_{L^2(\Omega)}$. The variational formulation of the PDE \eqref{eqn:scalarprotoa} can then be written
\begin{equation}\label{eqklajhdjdef23}
\int_{\Omega}\diiv (a\nabla \varPsi) \diiv(a \nabla u)=\int_{\Omega}(\nabla \varPsi)^T a \nabla g \text{ for } \varPsi\in \B\, .
\end{equation}
Note that, by taking $\|\cdot\|_2=\|\cdot\|_{L^2(\Omega)}$, we have $C_\L=C_{\L^{-1}}=1$.
\end{Example}

\begin{Example}\label{egprotoalsobolevlnondiv}
Let $\L$ be a continuous linear bijection from $H^s_0(\Omega)$ to $L^2(\Omega)$.
The inverse problem \eqref{eqn:scalar} is equivalent to
\begin{equation}\label{eqldjkejcckjhjsk}
\L^* \L=\L^* g
\end{equation}
Let $Q=(\L^* \L)^{-1}$ be the symmetric operator mapping $H^{-s}(\Omega)$ to $H^s_0(\Omega)$.
 Write $\|u\|:=\|\L u\|_{L^2(\Omega)}=[Q^{-1}u,u]$ and let $\|\cdot\|_*$ be the dual norm of $\|\cdot\|$.
Define $(\B,\|\cdot\|)$ and $(\B^*,\|\cdot\|_*)$ as in Example \ref{egprotoalsobolevlori}.
 Consider the inverse problem \eqref{eqn:scalar}.
Let $\B_2=L^2(\Omega)$ and let $\|g\|_2=\|g\|_{L^2(\Omega)}$. Observe that
 $\B_2^*=\B_2$.
 The gamblet transform for the inverse problem \eqref{eqn:scalar} can then be defined by
 considering  the operator $\L$ mapping  $\B$ onto $\B_2$ (2) taking $\D=L^*$ and $\G=I_d$ in Diagram \ref{eqcase1NNN}. Under these choices $Q^{-1}$ is  the operator $\L^* \L$ (with the associated norm $\|u\|=\|\L u\|_{L^2(\Omega)}$) as illustrated in the following diagram
\begin{equation}\label{eqcase1exnedewbNdNNnondiv}
\text{\xymatrixcolsep{4pc}\xymatrix{
H^s_0(\Omega) \ar[d]^{Q^{-1}} \ar[r]^{\L} &L^2(\Omega)\ar[d]^{I_d}\\
H^{-s}(\Omega)           &L^2(\Omega)\, . \ar[l]^{\L^*}}}
\end{equation}
\end{Example}

\begin{Example}\label{egprotoalsobolevluygunondiv}
A particular instance of Example \ref{egprotoalsobolevlnondiv} is the differential operator
 $\L:H^s_0(\Omega)\rightarrow L^2(\Omega)$  defined by
\begin{equation}
\L u = \sum_{0\leq |\alpha|\leq s} a_{\alpha}(x)  D^{\alpha}  u,\,\text{ for }u\in H^s_0(\Omega),
\end{equation}
where $a$ is a tensor with entries in  $L^\infty(\Omega)$ such that $\L^{-1}$ is well defined and continuous.
\end{Example}

\subsection{Identification of measurement functions}
The application of the gamblet transform to the resolution of \eqref{eqn:scalar} requires the prior identification of the measurement functions $\phi_i^{(k)} \in \B^*$ that satisfy Condition \ref{cond1OR}. We will now show that these measurement functions can simply be obtained as the image, by $\D \G$ of a multiresolution decomposition on $\B_2$.
Let $(\B_c,\|\cdot\|_c)$ be a Banach subspace of $\B_{2}$ such that the natural embedding $i:\B_c \rightarrow \B_{2}$ is compact and dense.

\begin{Construction}\label{constphikfk}
Let $(f^{(q)}_i)_{i\in \I^{(q)}}$ be linearly independent elements of $\B_c$. For $k\in \{1,\ldots,q-1\}$ and $i\in \I^{(k)}$ let $f^{(k)}_i \in \B_c$ be defined by induction via
\begin{equation}\label{eq:eigdeiud3ddfk}
f^{(k)}_i=\sum_{j\in \I^{(k+1)}}\pi^{(k,k+1)}_{i,j}  f^{(k+1)}_j
\end{equation}
where the $\pi^{(k,k+1)}$ and $\I^{(k)}$ are as in Construction \ref{defindextree} and \ref{constpi}.
\end{Construction}
For $k\in \{1,\ldots,q\}$ let

 \begin{equation}\label{eqdefPhikfk}
 F^{(k)}:=\Span\{f_i^{(k)}\mid i\in \I^{(k)}\}\,.
 \end{equation}
For $k\in \{2,\ldots,q\}$ let
\begin{equation}\label{eqphikchifk}
F^{(k),\chi}:=\{f\in F^{(k)}\mid f=\sum_{i\in \I^{(k)}} x_i f_i^{(k)}\text{ with } x\in \Ker(\pi^{(k-1,k)})\}
\end{equation}

The analogue of Condition \ref{cond1OR} to the multiresolution decomposition on $\B_{2}$ is as follows.
\begin{Condition}\label{cond1ORfk}
There exists some constants $C_{F}\geq 1$ and $H\in (0,1)$ such that
\begin{enumerate}
\item $\sup_{f\in \B_c} \frac{\|f \|_2}{\|f \|_c}\leq C_F$.
\item $\frac{1}{C_F} H^k \leq \inf_{f\in F^{(k)}} \frac{\|f \|_2}{\|f \|_0}$ for $k\in \{1,\ldots,q\}$.
\item $\sup_{g\in \B_c}\inf_{f \in F^{(k)}} \frac{\|g-f\|_2}{\|g\|_c}\leq C_{F} H^k$ for $k\in \{1,\ldots,q\}$.
\item $\sup_{f\in F^{(k),\chi} }\frac{\| f\|_2}{\|f\|_c}\leq C_{F} H^{k-1}$ for $k\in \{2,\ldots,q\}$.
\item $\frac{1}{C_{F}}|x|^2 \leq \|\sum_{i\in \I^{(k)}} x_i f_i^{(k)}\|_c^2 \leq C_{F} |x|^2$ for $k\in \{1,\ldots,q\}$ and $x\in \R^{\I^{(k)}}$.
\item $\frac{1}{C_{F}} J^{(k)}\leq W^{(k)}W^{(k),T}\leq C_{F} J^{(k)}$ for $k\in \{2,\ldots,q\}$.
\end{enumerate}
\end{Condition}
Let $\B_0=\D \G \B_c$ and $\|\phi\|_0=\|\G^{-1}\D^{-1} \phi\|_c$.
Observe that $\B_0$ is a Banach subspace of $\B^*$ and the natural embedding $i:\B_0 \rightarrow \B^*$ is compact and dense.
For $k\in \{1,\ldots,q\}$ and $i\in \I^{(k)}$ let
\begin{equation}
\phi_i^{(k)}=\D \G f^{(k)}
\end{equation}

\begin{Theorem}\label{thmsjdhdhgd}
If Condition \ref{cond1ORfk} is satisfied then Condition \ref{cond1OR} is satisfied with a constant $C_{\Phi}$ depending only on $C_{\L}$, $C_{\L^{-1}}$ and  $C_F$. Furthermore, the (finite-element) solution of \eqref{eqn:scalarfem} in $\V^{(k)}$ is $u^{(k)}(u)$, which satisfies $\|u-u^{(k)}(u)\|\leq C H^k \|g\|_c$.
\end{Theorem}

\begin{Example}\label{egleddoiejd}
Consider  Examples \ref{egproto00} and \ref{egprotoh10normNNN}  and take $(\B_0,\|\cdot\|_0)=(L^2(\Omega),\|\cdot\|_{L^2(\Omega)})$. Since $\D \G$ is the identity operator, Condition \ref{cond1ORfk} translates to (1) the compact embedding inequality
$\|\phi \|_{H^{-1}(\Omega)} \leq C \|\phi \|_{L^2(\Omega)} $ for $\phi\in L^2(\Omega)$ (2)  the inverse Sobolev inequality $\|\phi \|_{L^2(\Omega)} \leq C H^{-k}\|\phi \|_{H^{-1}(\Omega)}$ for $\phi\in \Phi^{(k)}$ (3) the approximation property
$\inf_{\phi \in \Phi^{(k)}} \|\varphi-\phi\|_{H^{-1}(\Omega)} \leq C H^k \|\varphi\|_{L^2(\Omega)}$ for $\varphi\in L^2(\Omega)$  (4)  the
 Poincar\'{e} inequality $\|\phi \|_{H^{-1}(\Omega)} \leq C  H^{k-1} \|\phi \|_{L^2(\Omega)}$ for $\phi\in \Phi^{(k),\chi}$  (5)  and the Riesz basis/frame inequality $\frac{1}{C}|x|^2 \leq \|\sum_{i\in \I^{(k)}} x_i \phi_i^{(k)}\|_{L^2(\Omega)}^2 \leq C |x|^2$ for $x\in \R^{\I^{(k)}}$. These conditions are therefore natural regularity conditions on the elements $\phi_i^{(k)}$ and it is easy to check that they are satisfied in the context of Illustrations \ref{nekdjdhhu} and \ref{nekdjdhhu2} with $H=1/2$.
\end{Example}

\begin{Example}\label{egleddoiddedejd}
Consider Examples \ref{egprotoalsobolevlori} and \ref{egprotoalsobolevl} (or a particular instance, Example \ref{egkljkhdekjd}) and take $(\B_0,\|\cdot\|_0)=(L^2(\Omega),\|\cdot\|_{L^2(\Omega)})$. Since
 $\D \G$ is the identity operator, Condition \ref{cond1ORfk} translates to (1) the compact embedding inequality
$\|\phi \|_{H^{-s}(\Omega)} \leq C \|\phi \|_{L^2(\Omega)} $ for $\phi\in L^2(\Omega)$ (2)  the inverse Sobolev inequality $\|\phi \|_{L^2(\Omega)} \leq C H^{-k}\|\phi \|_{H^{-s}(\Omega)}$ for $\phi\in \Phi^{(k)}$ (3) the approximation property
$\inf_{\phi \in \Phi^{(k)}} \|\varphi-\phi\|_{H^{-s}(\Omega)} \leq C H^k \|\varphi\|_{L^2(\Omega)}$ for $\varphi\in L^2(\Omega)$ (4) the
 Poincar\'{e} inequality $\|\phi \|_{H^{-s}(\Omega)} \leq C  H^{k-1} \|\phi \|_{L^2(\Omega)}$ for $\phi\in \Phi^{(k),\chi}$ (5) and the Riesz basis/frame inequality $\frac{1}{C}|x|^2 \leq \|\sum_{i\in \I^{(k)}} x_i \phi_i^{(k)}\|_{L^2(\Omega)}^2 \leq C |x|^2$ for $x\in \R^{\I^{(k)}}$. These conditions are, as in Example \ref{egleddoiejd}, natural regularity conditions on the elements $\phi_i^{(k)}$.
 \end{Example}

 \begin{Proposition}\label{propkajhdlkjd}
 The conditions equivalent to Conditions \ref{cond1ORfk}, derived in Example \ref{egleddoiddedejd}, are satisfied with the measurement functions
 presented in Examples \ref{egkajhdlkjdini} and \ref{egkajhdlkjdinibis} with  $H=h^s$ and a constant $C$ depending only on $d, \delta$ and $s$.
\end{Proposition}

\begin{algorithm}[!ht]
\caption{Discrete Gamblet solve of \eqref{eqklajhdjfdis}.}\label{algdiscgambletsolvecase1gso}
\begin{algorithmic}[1]
\STATE\label{step3gdgpbdisso} For $i\in \I^{(q)}$, $\psi^{(q)}_i= \varPsi_i$  \COMMENT{Level $q$ gamblets}
\STATE\label{step4gdisso} Use Algorithm \ref{alggambletcomutationnes} to compute the  gamblets $\psi_i^{(k)}$, $\chi_i^{(k)}$
and matrices $A^{(k)}$, $B^{(k)}$ and $R^{(k-1,k)}$
\STATE\label{step4gso} For $i\in \I^{(q)}$, $b^{(q)}_i=[\D \G g,\psi_i^{(q)}]$
\FOR{$k=q$ to $2$}
\STATE\label{step8gso} $w^{(k)}=B^{(k),-1} W^{(k)} b^{(k)}$ \COMMENT{Eq.~\eqref{eqjhgiggy6}}
\STATE\label{step10gso} $v^{(k)}=\sum_{i\in \J^{(k)}}w^{(k)}_i \chi^{(k)}_i$ \COMMENT{Thm.~\ref{thmreffytfuyf}, $v^{(k)}:=u^{(k)}-u^{(k-1)}\in \W^{(k)}$}
\STATE\label{step15gso} $b^{(k-1)}=R^{(k-1,k)} b^{(k)}$ \COMMENT{Eq.~\eqref{eq:ftfytftfx}}
\ENDFOR
\STATE\label{step16gso} $ w^{(1)}=A^{(1),-1}g^{(1)}$ \COMMENT{Eq.~\eqref{eqawb}}
\STATE\label{step17gso} $u^{(1)}=\sum_{i \in \I^{(1)}} w^{(1)}_i \psi^{(1)}_i$ \COMMENT{Thm.~\ref{thmreffytfuyf}}
\STATE\label{step18gso} $u^\d=u^{(1)}+v^{(2)}+\cdots+v^{(q)}$    \COMMENT{Thm.~\ref{thmreffytfuyf}}
\end{algorithmic}
\end{algorithm}

\subsection{Using the gamblet transform to solve inverse problems/linear systems}
Consider elements  $(\varPsi_i)_{i\in \N}\in \B$, as in  Subsection \ref{subsecgamdis}, to be used in
the discretization of the operator $\L$ to their  span $\B^\d:=\Span\{\varPsi_i, i\in \N\}\subset \B$.
For $g\in \B_2$ let $u^\d\in \B^\d$ be the Galerkin approximation of the solution of $\L u=g$ in $\B^\d$, i.e. the solution of the discrete linear system (obtained from the variational formulation \eqref{eqklajhdjfNNN})
\begin{equation}\label{eqklajhdjfdis}
\<u^\d,\varPsi\>=[\D \G g,\varPsi] \text{ for } \varPsi\in \B^\d\,.
\end{equation}
When $\B$ is finite-dimensional one can select $\B^\d=\B$ and obtain $u^\d=u$. When $\B$ is infinite-dimensional $u^\d$ is only an approximation of $u$ whose error corresponds to the distance between $u$ and $\B^\d$ ($u^\d=\arg\min_{v\in \B^\d} \|u-v\|$).
Algorithm \ref{algdiscgambletsolvecase1gso}, which is a direct variant of the discrete gamblet transform introduced in Algorithm \ref{alggamblettransf}, turns the resolution of \eqref{eqklajhdjfdis} into that of $q$ linearly independent linear systems (with uniformly bounded condition numbers under Condition \ref{conddiscrip3ordismatdis}).

\begin{Remark}
By  decomposing the inversion of \eqref{eqklajhdjfdis} into the inversion of $q$ linear systems of uniformly bounded condition numbers, Algorithm \ref{algdiscgambletsolvecase1gso}
provides an alternative regularization of ill-conditioned linear systems.
Recall that classical regularization methods include singular value truncation and the traditional (least square) Tikhonov regularization \cite{Neumaier98}.
Recall also that well conditioned linear systems  can be solved efficiently using iterative methods.
such as the Conjugate Gradient (CG) method \cite{HestenesStiefel1952}.
\end{Remark}

 \begin{figure}[h!]
	\begin{center}
			\includegraphics[width=\textwidth]{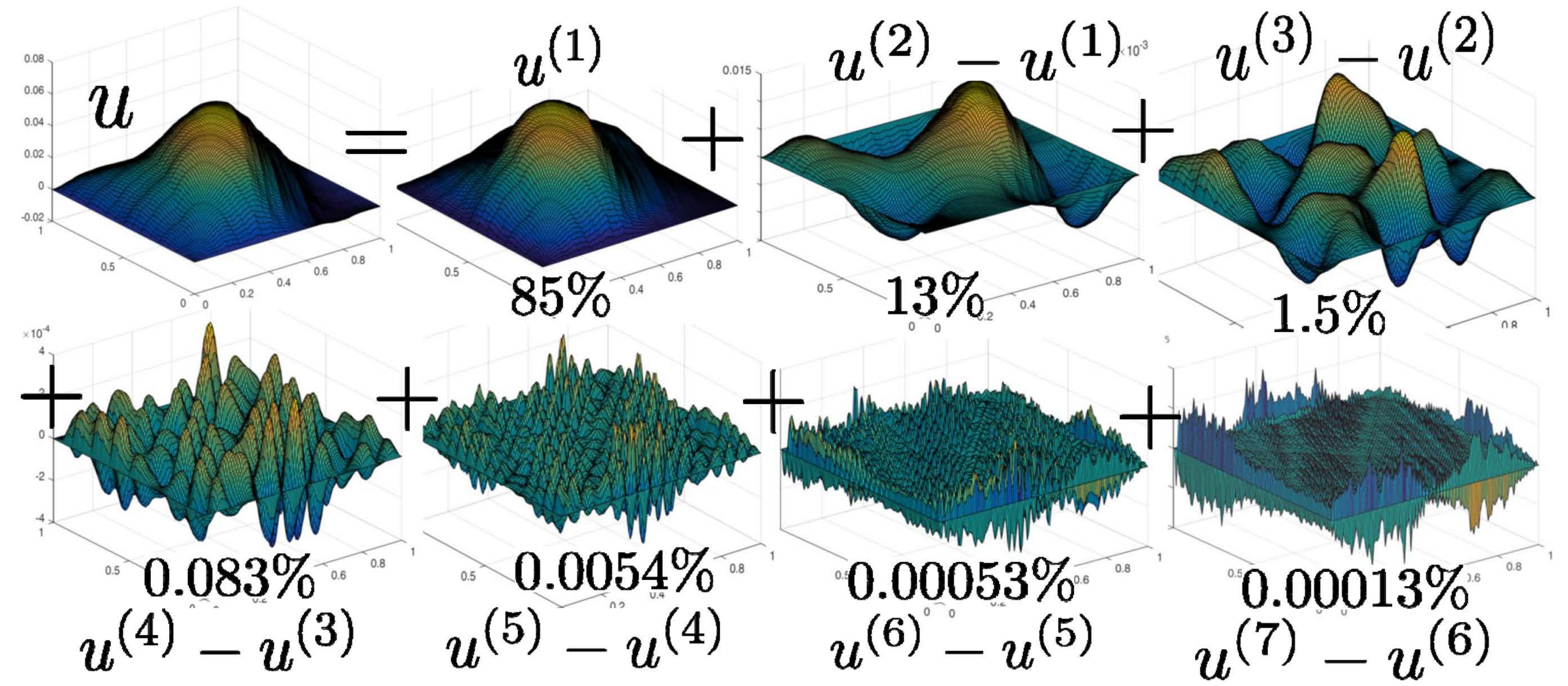}
		\caption{Decomposition of $u$ with smooth $g$. See Illustration \ref{nedkjdhijh3}.}\label{figu7}
	\end{center}
\end{figure}

\begin{NE}\label{nedkjdhijh3}
Consider the numerical example described in Illustration \ref{illoganumsim} of Example \ref{egproto00}. Figures \ref{figu7} and \ref{figu7dirac} show the decomposition the finite-element  solution $u$ of \eqref{eqn:scalarprotoa} into subband solutions $u^{(1)}$ and $u^{(k+1)}-u^{(k)}$ along with the relative energy content of each subband. We consider two right hand sides $g$, one smooth given by
$g(x)=\sum_{i\in \N} \big(\cos(3z_{i,1}+z_{i,2})+\sin(3z_{i,2})+\sin(7z_{i,1}-5z_{i,2})\big) \varPsi_i(x)$ (writing $z_i$ the positions of the interior nodes of the fine mesh) and one singular defined as the (approximate) discrete mass of Dirac $g(x)=4^{q}\psi_{i_0}$ where $i_0$ is the label of an interior node of the fine mesh in the center of the square. Observe that when $g$ is regular then the energy content in higher subbands quickly decreases towards $0$ (and therefore those subband solutions may not need to be computed depending on the desired accuracy which enables  computation in sublinear complexity) whereas when $g$ singular, the energy content in higher subbands remain significant and all subband solutions may need to be computed.
\end{NE}

 \begin{figure}[h!]
	\begin{center}
			\includegraphics[width=\textwidth]{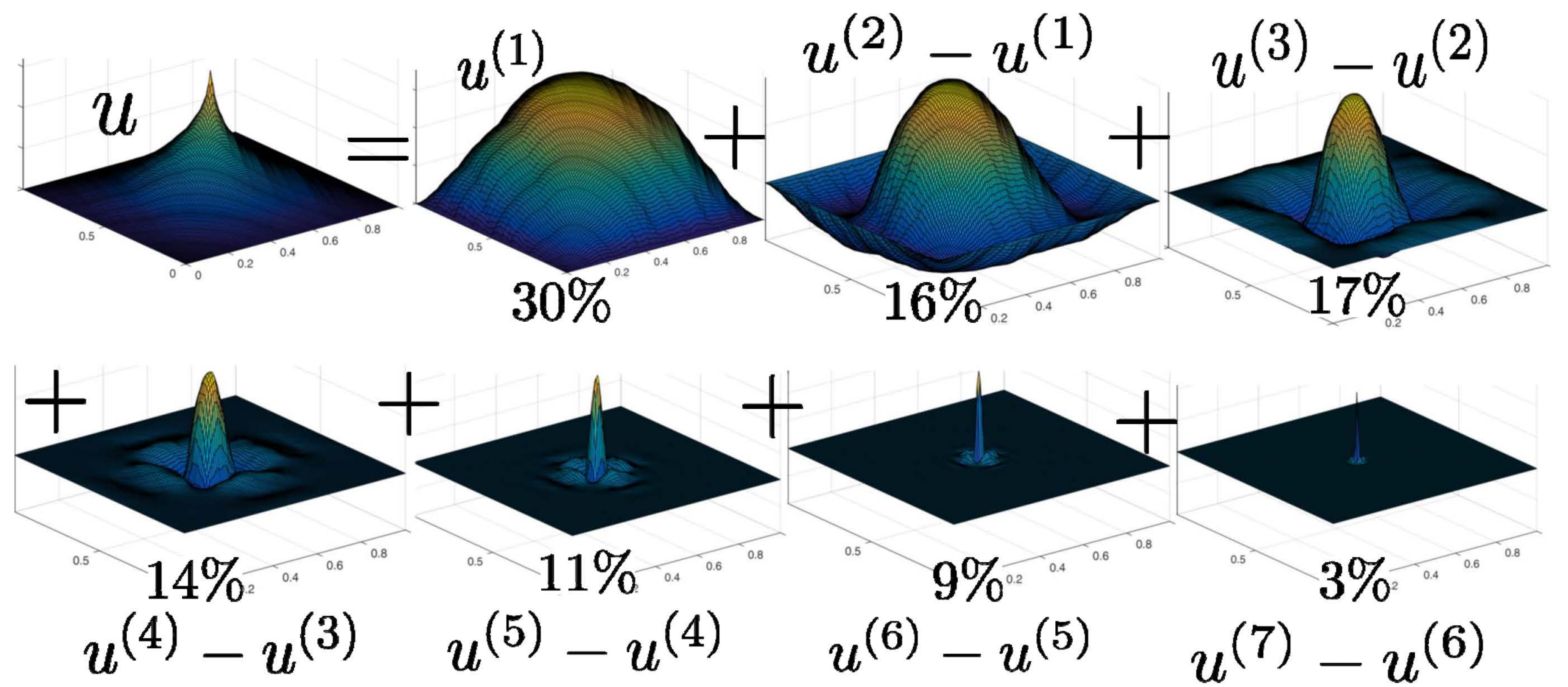}
		\caption{Decomposition of $u$ with a singular $g$. See Illustration \ref{nedkjdhijh3}.}\label{figu7dirac}
	\end{center}
\end{figure}

\section{Computational Information Games}\label{seccigsec}

Although the full development of our approach to Computational Information Games will involve
 Gelfand triples of Banach and Hilbert spaces, Gaussian cylinder measures and  their corresponding Wiener measures in the context of the interplay between Numerical Analysis, Approximation Theory and Statistical Decision
 Theory, as described at length in Section \ref{sec_correspondence},
to reduce the complexity of this paper,
here we
 consider Hilbert spaces paired with non-standard realizations of their dual spaces and consider Gaussian cylinder measures in the form of Gaussian fields, all to be defined. To begin, we now describe
  the connection between
 the Information-Based Complexity approach to the problem of Optimal Recovery and its connections with
 Game Theory.

\subsection{From Information Based Complexity to Game Theory}\label{subsecidyidudg}
It is well understood in Information Based Complexity \cite{Woniakowski2009} (IBC) that computation in infinite dimensional spaces can only be done with partial information. In the setting of Subsection \ref{subsecttt}
this means that, if $\B$ is infinite dimensional space, then one cannot directly compute with $u$ but only with a finite number of \emph{features} of $u$. These \emph{features} can represented as the vector $u_m:=\big([\phi_1,u],\ldots,[\phi_m,u]\big)\in \R^m$ where
$\phi_1,\ldots,\phi_m$ are $m$ linearly independent elements of $\B^*$. Similarly one can, for $g\in \B_2$, define $g_m\in \R^m$ (as a function of $g$) representing finite information about $g$. To solve
 the inverse the problem \eqref{eqn:scalar}, since one cannot directly compute with $u$ and $g$ but only with $u_m$ and $g_m$, one must identify a reduced operator mapping $u_m$ into $g_m$. If we know the mapping $\L$
  and the mapping $g \mapsto g_{m}$ then, as illustrated in \eqref{eqcase1NNNgydgyredop}, this
 identification requires  the determination of a mapping from $u_m$ to $u$, bridging the information gap between  $\R^m$ and $\B$.

\begin{equation}\label{eqcase1NNNgydgyredop}
\text{\xymatrixcolsep{10pc}\xymatrix{
u  \ar[r]^\L &g\ar[d]\\
u_m \ar[u] \ar@/^2.5pc/[r]_{\text{Reduced operator}}          &g_m \ar[l]^{\text{Discretized inverse problem}}}}
\end{equation}

We apply the Optimal Recovery approach, see e.g.~Micchelli and Rivlin \cite{micchelli1977survey}, to bridging the information gap  as follows:
Corresponding to the
collection  $\dot{\Phi}:=\phi_1,\ldots,\phi_m$ of $m$ linearly independent elements of $\B^*$, let
$\Phi:\B \rightarrow \R^{m}$ be defined by
$\Phi(u):=\big([\phi_1,u],\ldots,[\phi_m,u]\big)$ be the information operator. Moreover, it should cause no confusion to also denote by $\Phi\subset \B^{*}$ the span $\Phi:=\Span{\dot{\Phi}}$.  A solution operator
is a possibly nonlinear map
$\Psi:\R^{m} \rightarrow \B$ which uses only the values of the information operator $\Phi$.
 For any solution operator $\Psi$ and any state $u \in \B$, the relative error
 corresponding to the identity operator on $\B$ and the information operator
$\Phi:\B \rightarrow \R^{m}$  can be written
\[\mathcal{E}(u,\Psi):=\frac{\|u -\Psi(\Phi(u))\|}{\|u\|}\, ,\]
see e.g.~\cite{micchelli1984orthogonal},
from which the error associated with the solution operator $\Psi$ is
\[\mathcal{E}(\Psi):=\sup_{u \in \B}{\frac{\|u -\Psi(\Phi(u))\|}{\|u\|}}\, ,\]
and the optimal solution error is
\begin{equation}
\label{eobiiubuikn}
\mathcal{E^{*}}= \inf_{\Psi}{\mathcal{E}(\Psi)}=\inf_{\Psi}\sup_{u \in \B}{\frac{\|u -\Psi(\Phi(u))\|}{\|u\|}}\, .
\end{equation}
Micchelli \cite[Thm.~2]{micchelli1984orthogonal} provides the solution to this problem. In the setting of nonstandard dual pairings it appears as follows:
\begin{Theorem}[Micchelli]
\label{thm_micchelli}
Let $\B$ be a  separable Banach space, and $\B^*$ be a realization its dual with
$[\cdot,\cdot]$ the corresponding dual pairing.
 For a positive symmetric bijection $Q:
\B^* \rightarrow \B$, consider the Hilbert space $\B$ equipped with the inner product
$\langle u_{1},u_{2}\rangle:=[Q^{-1}u_{1}, u_{2}]$ .
Corresponding to the
collection  $\dot{\Phi}:=\phi_1,\ldots,\phi_m$ of $m$ linearly independent elements of $\B^*$, let
$\Phi:\B \rightarrow \R^{m}$ be defined by
$\Phi(u):=\big([\phi_1,u],\ldots,[\phi_m,u]\big)$ be the information operator.
Define the Gram  matrix $\Theta$ by
\[ \Theta_{ij}:=[ \phi_{i},Q\phi_{j}], \quad i,j=1,\ldots m\,, \]
and the elements $\psi_{i} \in \B, i=1,\ldots m$ by
\begin{equation}
\label{def_gam}
\psi_{i}:= \sum_{j=1}^{m}{\Theta^{-1}_{ij}Q\phi_{j}}, \quad i=1,\ldots, m\, .
\end{equation}
Then the mapping $\Psi:\R^{m} \rightarrow \B$ defined by
$\Psi(x):=\sum_{i=1}^{m}{x_{i}\psi_{i}}, x\in \R^{m}$ is the unique optimal  minmax solution to
\eqref{eobiiubuikn}.
\end{Theorem}

The minmax problem \eqref{eobiiubuikn} corresponds naturally to an adversarial zero sum game involving two players.
Let us denote by
\begin{equation}
\label{def_L}
L(\Phi,\B)
\end{equation}
 the set of $\bigl(\s(\Phi),\s(\B)\bigr)$-measurable  functions, where
$\s(\Phi)$ denotes the Borel  $\s$-algebra generated by $\dot{\Phi}$ and $\s(\B)$ the Borel
$\s$- algebra of $\B$.
  In this notation, using the fact, see e.g.~\cite[Thm.~2.12.3]{Bogachev1},
 that $v\in L(\Phi,\B)$ is equivalent to $v=\Psi \circ \Phi$ for some Borel measurable function $\Psi$, the game can be formulated as in the following diagram
  \begin{equation}\label{eqdkjdhkjhffORgame}
\text{\xymatrixcolsep{0pc}\xymatrix{
\text{(Player I)} & u^\one \ar[dr]_{\max}   &      &u^\two \ar[ld]^{\min}&\text{(Player II)}\\
&&\frac{\|u^\one -u^\two(u^\one)\|}{\|u^\one\|}& &
}}
\end{equation}
 and the objective of Player II (from a deterministic worst case numerical perspective) would be to minimize
\begin{equation}\label{eqdkjdhkjhffOR}
\inf_{u^{\two} \in L(\Phi,\B) }\sup_{u^\one\in \B}\frac{\|u^\one -u^\two(u^\one)\|}{\|u^\one\|}\, .
\end{equation}

Using  the notation $\Phi:=\Span{\dot{\Phi}}\subset \B^{*}$ we obtain the following corollary.
\begin{Corollary}[Micchelli]
\label{cor_micchelli}
Consider the situation of Theorem \ref{thm_micchelli}.
Then the  optimal minmax solution to \eqref{eqdkjdhkjhffOR}
is the orthogonal projection $P_{Q\Phi}$ onto $Q\Phi \subset \B$ and
$P_{Q\Phi}=\Psi\circ \Phi$, where $\Psi$ is the optimal minmax solution  to
\eqref{eobiiubuikn}.
\end{Corollary}

\subsection{The game theoretic solution to numerical approximation}\label{subsecidyidudg2}
Bivariate loss functions such as $\mathcal{E}(u^\one,u^\two):=\frac{\|u^\one -u^\two(u^\one)\|}{\|u^\one\|}$ with $
u^\one \in \B, \, u^\two \in  L(\Phi,\B) $, do not, in general have a saddle point. In fact it is easy to see
that
\[\sup_{u^\one\in \B} \inf_{u^{\two} \in L(\Phi,\B) }\frac{\|u^\one -u^\two(u^\one)\|}{\|u^\one\|}=0\, .\]
We know, from Von Neumann's remarkable minimax theorem \cite{VNeumann28}, that, at least finite games, although minimax problems do not, in general, have a saddle point of pure strategies, saddle points of mixed strategies do always exist. These mixed strategies are randomized strategies   obtained by lifting minimax problems to distributions over pure strategies \cite{VonNeumann:1944, Nash:1951}.
Although the information game described in \eqref{eqdkjdhkjhffORgame} is zero sum, it is not finite.
Nevertheless, we also know, from  Wald's Decision Theory  \cite{Wald:1945}, that under sufficient regularity conditions such games can be made compact and can, as a result, be approximated by a finite game (we also refer to Le Cam's substantial generalizations \cite{LeCam1955, LeCam} and also to \cite{strasser1985mathematical}).
Instead of looking for the deterministic  worst case (numerical analysis) solution we will  therefore lift \eqref{eqdkjdhkjhffOR} to a minimax problem over measures and look for a game theoretic, mixed strategy, saddle point/solution. Our motivation in doing so is twofold: (1) The game theoretic solution is in general easier to identify than the numerical analysis solution (2) When solving a large linear system, minimax problems such as \eqref{eqdkjdhkjhffORgame} occur in a repeated manner (over a range of levels of complexity) and mixed strategies are the optimal solutions of such repeated games.

Although, in the context of \eqref{eqdkjdhkjhffORgame}, mixed strategies for Player II  correspond to selecting
 $u^{\two}$ at random by placing a probability distribution over $L(\Phi,\B)$, we will show that the optimal mixed strategy for Player II is a pure (non random) strategy, obtained by (1)
 identifying the optimal mixed strategy of Player I as selecting $u^\one$ at random by placing a (weak)
 probability distribution $\pi^\one$ on $\B$ (and projecting that distribution onto the orthogonal complement of $Q \Phi$) (2) taking the conditional
 expectation of $\pi^\one$  on the observation of $\Phi$ (i.e. the vector $([\phi_1,u^\one],\ldots,[\phi_m,u^\one])$). Furthermore the optimal weak distribution for $\pi^\one$ corresponds to that of a Gaussian field on $\B$ with covariance operator $Q$.
To that end, we now introduce basic measure theoretic  terminology, including the notions of {\em weak} distributions in their  equivalent form
 of cylinder measures, and describe their relation with
 Gaussian fields.

\subsection{Measures, cylinder measures, weak distributions, and Gaussian fields}
Let us begin by establishing some notational conventions and  basic facts.
For a topological space $B$ we let $\s(B)$ denote the corresponding Borel $\s$-algebra.
For measurable spaces $(X_{1}, \Sigma_{1})$ and $(X_{2}, \Sigma_{2})$ the notation
$f:(X_{1},\Sigma_{1}) \rightarrow (X_{2},\Sigma_{2})$ will indicate that the function
$f:X_{1}\rightarrow X_{2}$ is measurable, that is $f^{-1}(A)\in \Sigma_{1}$ for $ A \in \Sigma_{2}$.
Let $\dot{\Phi}:=\{\phi_{i}, i=1,..,m\}$ denote the collection of functions
$\phi_{i}:\B \rightarrow \R$
defined by the elements $\phi_{i} \in \B^{*},i=1,\ldots ,m$ and let
$\s(\Phi)$ denote the induced $\s$-algebra of
 subsets of $\B$ generated by the collection $\dot{\Phi}$. Corresponding to
such a collection, let us  denote by
 $\Phi:\B\rightarrow \R^m$ the  mapping
defined by $\Phi(x):=\sum_{i=1}^{m}{x_{i}\phi_{i}}, x \in \R^{m}$.
Since each component $\phi_{i}$ is continuous it is Borel measurable. Therefore,
  a function
$f:\B\rightarrow \B$ is $\bigl(\s(\Phi),\s(\B)\bigr)$-measurable, that is
$f:(\B,\s(\Phi))\rightarrow \bigl(\B,\s(\B)\bigr)$, if and only if
 $f=\psi\circ \Phi$  where
$\psi:\R^{m}\rightarrow \B$ is Borel measurable, that is $\bigl(\s(R^{m}),\s(\B)\bigr)$
 measurable.
  Consequently, the added assumption that the solution map
$\Psi$ be measurable is equivalent to the function $u^{\two}:=\Psi\circ \Phi$ being
$\bigl(\s(\Phi),\s(\B)\bigr)$-measurable.
Recall the set
$L(\Phi,\B)$
  of $\bigl(\s(\Phi),\s(\B)\bigr)$-measurable  functions introduced in \eqref{def_L}.

According to Gross   \cite[Pg.~33]{gross1967abstract}, the notion of a weak distribution, introduced by Segal \cite{segal1956tensor}, is equivalent to  that of a cylinder measure.
To describe the latter notion,  for a Banach space $X$, the cylinder sets are sets  of the form
$F^{-1}(B)$ where $F:X \rightarrow \R^{n}$ for some $n$ is continuous,
and $B$ is a Borel subset of $\R^{n}$. The cylinder set algebra is the $\s$-algebra generated
by all choices of $F,n$, and $B$. According to Bogachev \cite[Thm.~A.3.7]{bogachev1998gaussian} when  $X$
is separable, this $\s$-algebra is the Borel $\s$-algebra.
According to Bogachev \cite{bogachev1998gaussian}, we say that $\nu$ is a {\em cylinder   measure}
if $\nu$ is {\em finitely} additive set function on the cylinder set $\s$-algebra such that
for every continuous linear map $F:X \rightarrow \R^{n}$, the pushforward
$F_{*}\nu$, defined by $F_{*}\nu(B):=\nu(F^{-1}(B))$  for Borel sets $B \subset \R^{n}$, is a true measure.
When these are centered Gaussian measures, we say that $\nu$ is a {\em Gaussian cylinder  measure}.

To define a Gaussian field,  we say that
 a linear subspace $W\subset L^{2}(\Omega,\Sigma,\mu)$, where $(\Omega,\Sigma,\mu)$ is probability space,
is a {\em Gaussian space} if each element
$w\in W$ is a centered Gaussian random variable. If $W$ is a closed subspace, we say
that it is a {\em Gaussian Hilbert space}.
Let us now summarize our definition of a weak Gaussian distribution defined in terms of a  Gaussian field
in the  setting of a Banach space and {\em realization of its dual}.
\begin{Definition}[Gaussian Field]
\label{def_Gaussianfield}
Let $\B$ be a  separable Banach space, and $\B^*$ be a realization its dual with
$[\cdot,\cdot]$ the corresponding dual pairing.
 For a positive symmetric bijection $Q:
\B^* \rightarrow \B$, consider the Hilbert space $\B^*$ equipped with the inner product
$\langle\varphi_{1},\varphi_{2}\rangle_{Q}:=[\varphi_{1}, Q\varphi_{2}]$. Then we say that
  $\xi$ is a  Gaussian field with
 covariance operator $Q$, which we write $\xi\sim \N(0,Q)$, if
\[\xi:(\B^*, \langle \cdot,\cdot\rangle_{Q})\rightarrow L^{2}(\Omega,\Sigma,\mu) \quad \text{is an
 isometry}\] to a Lebesgue probability space
such that image $\xi(\B^*)$ is a Gaussian space.
\end{Definition}
Let us now mention a particularly useful abuse of notation that we will use throughout the paper.
Consider a Gaussian field $\xi:(\B^*, \langle \cdot,\cdot\rangle_{Q})\rightarrow L^{2}(\Omega,\Sigma,\mu) $
and an element  $\varphi \in \B^*$. Then the random variable
$\xi(\varphi) \in  L^{2}(\Omega,\Sigma,\mu)$ is a  real-valued function
  $\xi(\varphi):\Omega \rightarrow \R$ on $\Omega$.
Since, for $\omega$ fixed,
the function $\xi(\varphi)(\omega)$ of $\varphi \in \B^{*}$ is linear,  for each $\omega \in \Omega$,
there exists  an element
$\hat{\xi}(\omega)$ in the algebraic dual
$\B^{*}$, so that
$ \xi(\varphi)(\omega)=[\varphi,\hat{\xi}(\omega)], \omega \in \Omega\, ,$
where the bracket $[\cdot,\cdot]$ is the bracket corresponding to the algebraic dual.
If we abuse notation by removing the hat from $\hat{\xi}$ and using  the same bracket notation for algebraic dual and topological dual,  then we obtain the notation
\[\xi(\varphi)=[\varphi,\xi]\, \]
where the function on the righthand side is defined by $[\varphi,\xi](\omega)=[\varphi,\xi(\omega)],
 \, \omega \in \Omega$ .
Consequently, using this notation, the isometric nature of the Gaussian field $\xi$ can be written as
\[\E\big[ [\varphi_{1},\xi][\varphi_{2},\xi]\big]=[\varphi_{1},Q \varphi_{2}]\,, \quad
 \varphi_{1},\varphi_{2} \in \B^*\, .\]
Moreover, since $\B^{*}$ is reflexive, using the close relationship between the algebraic dual of $\B^{*}$
and its topological dual $\B$,  $\xi$ has the interpretation as a $\B$-valued (weak) random variable. Consequently,
we say that a Gaussian field  $\xi:(\B^*, \langle \cdot,\cdot\rangle_{Q})\rightarrow L^{2}(\Omega,\Sigma,\mu) $ is a {\em Gaussian field  on $\B$.}

\begin{Remark}
Observe that if $\|\cdot\|$ is the norm on $\B$ defined by $\|u\|=[Q^{-1}u,u]^\frac{1}{2}$ and if $\xi \sim \N(0,Q)$ then for $\varphi, \phi\in \B^*$, $[\phi,\xi]\sim \N(0,\|\varphi\|_*^2)$ and $\E\big[[\varphi,\xi][\phi,\xi]\big]=\<\varphi,\phi\>_*$ where $\|\cdot\|_*$ is the dual norm of $\|\cdot\|$ and $\<\cdot,\cdot\>_*$ its associated scalar product.
\end{Remark}

The
conditional expectation of a Gaussian field is determined as the field of conditional expectations. That is,
for a Gaussian field $\xi$ and a sub $\s$-algebra  $\Sigma'\subset \Sigma$,
$\E[\xi|\Sigma']:\B^* \rightarrow L^{2}(\Omega,\Sigma,\mu)$ is defined by
\begin{equation}
\label{de_gfcond}
\E[\xi|\Sigma'](\varphi):=\E[\xi(\varphi)|\Sigma']\,,
\end{equation}
where here and throughout the rest of the paper  will refrain from constantly mentioning
 {\em almost everywhere}. Recall that we use the symbol $\Phi \subset \B^{*}$
  for the span of $\dot{\Phi}$.
Since  $\xi(\varphi) \in L^{2}(\Omega,\Sigma,\mu)$ is a centered Gaussian random variable for each
 element $ \varphi \in \Phi$ and $\Phi$ is finite dimensional,
 it follows that $\xi(\Phi) \subset L^{2}(\Omega,\Sigma,\mu)$ is a Gaussian Hilbert space.
Consequently, if we let $P_{\xi(\Phi)}:L^{2}(\Omega,\Sigma,\mu)\rightarrow L^{2}(\Omega,\Sigma,\mu)$
denote the orthogonal projection onto  $\xi(\Phi)$ and
let $\s(\Phi)$ denote the $\s$-algebra generated by $\Phi$, using the standard relation between
 conditioning on a $\s$-algebra and conditioning on the set of random variables generating it,
  according to
Janson \cite[Thm.~9.1]{janson1997gaussian},  for all $\varphi\in \B^* $ we have
\[\E[\xi(\varphi)|\s(\Phi)]=P_{\xi(\Phi)}\bigl(\xi(\varphi)\bigr)
 \]
 Consequently,  the
 conditional expectation is also a Gaussian field and has the particularly simple and useful form
\begin{equation}
\label{def_gausscond}
\E[\xi|\s(\Phi)]=P_{\xi(\Phi)}\xi\,,
\end{equation}
where the Gaussian field $P_{\xi(\Phi)}\xi:\B^* \rightarrow L^{2}(\Omega,\Sigma,\mu)$
is defined by$(P_{\xi(\Phi)}\xi)(\varphi):=P_{\xi(\Phi)}\bigl(\xi(\varphi)\bigr), \varphi
\in  \B^* $.

In this paper, we will use the { \em canonical} instantiation of the ambient space $L^{2}(\Omega,\Sigma,\mu)$  for the Hilbert space  $\B$
 throughout the rest of the paper.
Consider the countable product $\R^{\mathbb{N}}$ equipped with the product   $\gamma_{\mathbb{N}}$
of standard Gaussian measure on $\R$ and the resulting Lebesgue space
$L^{2}(\R^{\mathbb{N}}, \s(\R^{\mathbb{N}}),\gamma_{\mathbb{N}})$. It is well known that
$ \s(\R^{\mathbb{N}})=\s(\R)^{\mathbb{N}}$. We simplify notation by writing this space as
$L^{2}(\R^{\mathbb{N}},\gamma_{\mathbb{N}})$.

Let $f^{*}_{i}, i \in \mathbb{N}$ denote the set of functions on $\R^{\mathbb{N}}$
 which satisfy $f^{*}_{i}(x)=x_{i}, x \in \R^{\mathbb{N}}$.
Select an  orthonormal basis $\{e_{i} \in \B^{*} , i\in \mathbb{N}\}$  for $\B^{*}$, and
  consider the mapping $\xi:\B^{*} \rightarrow L^{2}(\R^{\mathbb{N}},\gamma_{\mathbb{N}})$
  determined by defining it on the basis elements as
$\xi(e_{i})= \acute{e}_{i}$ and extending it by linearity to the rest of $\B^{*}$.
Clearly, $\xi(e_{i})=\acute{e}_{i}$ is a centered Gaussian random variable on $\R^{\mathbb{N}}$
with variance $1$. Moreover,  Parseval's formula can be used to prove, that for $\varphi \in \B^{*}$,
$\xi(\varphi)$ is a centered Gaussian random variable of variance $\|\varphi\|^{2}$.
It follows that
$\xi:\B^{*} \rightarrow L^{2}(\R^{\mathbb{N}},\gamma_{\mathbb{N}})$  is an isometry and therefore
a Gaussian field by Definition \ref{def_Gaussianfield}.
  This field can be shown to be independent of the chosen
  orthonormal basis.
Moreover, as asserted by Strasser \cite[Ex.~68.7.3]{strasser1985mathematical}, it follows
 that such a Gaussian field corresponds to a Gaussian cylinder measure.
Therefore, henceforth we restrict the ambient space to be $L^{2}(\R^{\mathbb{N}},\gamma_{\mathbb{N}})$
chosen in this way.

\begin{Remark}
\label{rem_measures}
We have  mentioned the  equivalences between  cylinder measures and weak distributions,  and the
equivalence between Gaussian cylinder measures, Gaussian weak distributions and
Gaussian fields. Moreover, unless it is important to make the distinction
 we may refer to such weak objects simply as {\em measures}.
\end{Remark}

\subsection{Mixed extension of the game  and optimal minmax solutions}
In Section \ref{sec_worstcase} we describe how weak distributions arise naturally as worst case (weak) measures  for the
optimal recovery problem, and demonstrate that {\em Gaussian fields are universal worst case measures}
  in the sense that they are worst case measures independent of the measurement
functions $\dot{\Phi}$.
The following theorem shows that
$\xi \sim \N(0,Q)$ is such a universal worst case measure,
 producing the optimal minmax strategy.
\begin{Theorem}\label{thmdlkdjh3e}
The optimal   strategy  of \eqref{eqdkjdhkjhffORgame} for Player II is the pure strategy
  $u^\two=\E_{\xi \sim \N(0,Q)}\big[\xi \mid \s(\Phi)\big]$ corresponding to the mixed strategy
$\xi \sim \N(0,Q)$ of player $I$ which is a worst case (weak) measure in the sense described in
Section
\ref{sec_worstcase}. In particular, the function $u^\two \in L(\Phi,\B)$  defined by
\begin{equation}\label{eqbhbdhbdjhb3eor}
u^\two(u)=\E_{\xi \sim \N(0,Q)}\big[\xi \mid \text{$[\phi_i,\xi]=[\phi_i,u]$ for $i=1,\ldots,m$}\big]\,,\quad
u\in \B
\end{equation}
is  the  optimal  minmax solution of  \eqref{eqdkjdhkjhffOR}.
Moreover, the gamblets \eqref{def_gam} determined to be optimal by Theorem \ref{thm_micchelli},   have the representation
\begin{equation}
\label{def_gambletexp}
\psi_i=\E_{\xi \sim \N(0,Q)}\big[\xi\big| [\phi_j,\xi]=\delta_{i,j},\, j=1,\ldots m\,]\, .
\end{equation}
\end{Theorem}

\subsection{Repeated games across hierarchies of increasing levels of complexity}
It is not only computation with continuous operators that requires with partial information, to compute fast one must also compute with partial information. For example, the inversion of a $10^6\times 10^6$ matrix would be a slow process if one tries to compute with all the entries of that matrix at once, the only way to compute fast is to compute with a few \emph{features} of that matrix (that could be mapped to 64 degrees of freedom) and these \emph{features} typically do not represent all the  matrix entries.
Similarly, to obtain near optimal complexity solvers, one must compute with partial information over
 hierarchies of increasing levels of complexity and bridge hierarchies of information gaps.
In the proposed framework, we use the hierarchy of measurement functions $\phi_i^{(k)}$ introduced in Subsection \ref{subsechajhe} to generate a filtration on $\B$ representing a hierarchy of partial information about $u\in \B$.
As in Subsection \ref{subsecidyidudg}, the process of bridging information gaps across this hierarchy can then be formulation as an adversarial game, in which Player I
selects $u^\one \in \B$  and Player II is shown the level $k$ measurements $\big([\phi_i^{(k)},u^\one]\big)_{i\in \I^{(k)}}$ and must approximate $u^\one$ and level $k+1$ measurements (i.e., $\big([\phi_i^{(k)},u^\one]\big)_{i\in \I^{(k+1)}}$). This game is repeated across $k$ (the hierarchy of partial information/measurements about $u^\one$) and the choice of Player I does not change as $k$ progresses from $1$ to $q$.

We now extend Theorem \ref{thmdlkdjh3e} to the hierarchy.
\begin{Theorem}\label{thmdoiedhiu}
The optimal   strategy  of \eqref{eqdkjdhkjhffORgame} at level $\Phi^{(k)}$ for Player II
   is the pure strategy
  $u^{(k)}:=\E_{\xi \sim \N(0,Q)}\big[\xi \mid \s(\Phi^{(k)})\big]$ corresponding to the mixed strategy
$\xi \sim \N(0,Q)$ of player $I$. That is, the function $u^{(k)} \in L(\Phi^{(k)},\B)$  defined by
\begin{equation}\label{eqbhbdhbdjhb3e}
u^{(k)}(u):=\E_{\xi \sim \N(0,Q)}\big[\xi \mid \text{$[\phi^{(k)}_i,\xi]=[\phi^{(k)}_i,u]$ for $i=1,\ldots,m$}\big]\,,\quad
u\in \B
\end{equation}
is  the  optimal  minmax solution of  \eqref{eqdkjdhkjhffOR} at level $\Phi^{(k)}.$
Moreover, the gamblets at level $k$, defined in \ref{defpsi}
 with explicit representation in Theorem \ref{thmwhdguyd},  have the representation
\begin{equation}
\label{def_gambletexpk}
\psi^{(k)}_i=\E\big[\xi\big| [\phi_j^{(k)},\xi]=\delta_{i,j},\,  j\in \I^{(k)}\big]\,,\quad
i\in \I^{(k)}\,
\end{equation}
and the
interpolation matrix implicitly defined in  \eqref{eq:ftfytftfx} has  the representation
\begin{equation}\label{eqhjgjhgjgjgOR}
R^{(k,k+1)}_{i,j}= \E\big[ [\phi^{(k+1)}_j,\xi] \big| [\phi^{(k)}_l,\xi]=\delta_{i,l},\,l\in  \I^{(k)}\big],\quad i\in \I^{(k)}, j\in \I^{(k+1)}\, .
\end{equation}
Finally, the measure $\xi \sim \N(0,Q)$ is a worst case measure at all levels of the hierarchy.
\end{Theorem}

For a Gaussian field $\xi$, let
\begin{equation}\label{eqmdedeanvspde}
\xi^{(k)}:=\E\big[\xi\big| \sigma(\Phi^{(k)}) \big]
\end{equation}
denote the conditional expectation with respect to the $\s$-algebra $\sigma(\Phi^{(k)})$ generated by the observation functions at the $k$-th level. We now conclude the theoretical portion of this section with
 important Martingale properties of our constructions. We say that a sequence $\xi^{(k)}, k=1,\ldots$
 of Gaussian fields with common domain is a {\em martingale} if, for each $\varphi$ in its domain, its sequence of images
$[\varphi,\xi^{(k)}]$ is a martingale.
\begin{Theorem}\label{thmdgdjdgygugyd}
It holds true that (1) $\sigma(\Phi^{(1)}),\ldots,\sigma(\Phi^{(q)})$ forms a filtration, i.e. $\sigma(\Phi^{(k)})\subset \sigma(\Phi^{(k+1)})$
(2) $\xi^{(k)}$ is a martingale with respect to the filtration $\big(\sigma(\Phi^{(k)})\big)_{k\geq 1}$, i.e.
$\xi^{(k)}=\E\big[\xi^{(k+1)}\big| \sigma(\Phi^{(k)})\big]$ (3) $\xi^{(1)}$ and the increments $(\xi^{(k+1)}-\xi^{(k)})_{k\geq 1}$ are independent Gaussian fields. Furthermore,
\begin{equation}\label{eqmdedeanvspdeOR}
\xi^{(k)}=\sum_{i\in  \I^{(k)}} \psi^{(k)}_i [\phi^{(k)}_i,\xi]
\end{equation}
\end{Theorem}

Theorem  \ref{thmdoiedhiu} shows that, if  \eqref{eqdkjdhkjhffORgame} is used to measure loss, then   $u^{(k)}$ is not only optimal in a Galerkin sense (i.e. it is the best approximation of $u$ in $\V^{(k)}$ as shown in Theorem \ref{thmgugyug0OR}), it is also the optimal (pure) bet for Player II for playing the repeated game described in this section. Furthermore,  Theorem \ref{thmdoiedhiu} and  the following Theorem \ref{thmdgdjdgygugyd} show that
 the elements $\psi_i^{(k)}$ obtained in Subsection \ref{subseckjshgdhgdhOR} form a basis of  elementary  gambles/bets for playing the game,  providing the motivation for referring to them  as {\em gamblets}.
 Note that \eqref{eqhjgjhgjgjgOR} shows that $R^{(k,k+1)}_{i,j}$ can be identified as the best bet of Player II on the value of
 $[\phi^{(k+1)}_j,u]$ given the information that $[\phi^{(k)}_l,u]=\delta_{i,l}$ for $l\in  \I^{(k)}$.

Moreover, Theorem \ref{thmdgdjdgygugyd}   enables the application of classical  results concerning martingales to the numerical analysis of $\xi^{(k)}$ (and therefore $u^{(k)}$). In particular (1) Martingale (concentration) inequalities can  be used to control the fluctuations of $\xi^{(k)}$ (2) Optimal stopping times can be used to derive optimal strategies for stopping numerical simulations  based on  loss functions mixing computation costs with the cost of imperfect decisions (3) Taking $q=\infty$ in the construction of the basis elements $\psi^{(k)}_i$  and using the martingale convergence theorem imply that,  for all $\phi\in B^*$, $[\phi,\xi^{(k)}] \rightarrow [\phi,\xi^{(\infty)}]$
 as $k\rightarrow \infty$ (a.s. and in $L^1$). Furthermore, the independence of the increments $\xi^{(k+1)}-\xi^{(k)}$ is related to the orthogonal multiresolution decompositions \eqref{eqdedhhiuhe3OR}.

Let us now describe these results in the context of  Example \ref{egproto00}, Definition \ref{def_chi}, and
 Illustrations \ref{nekdjdhhu} and \ref{nekdjdhhu2}.

  \begin{figure}[h!]
	\begin{center}
			\includegraphics[width=\textwidth]{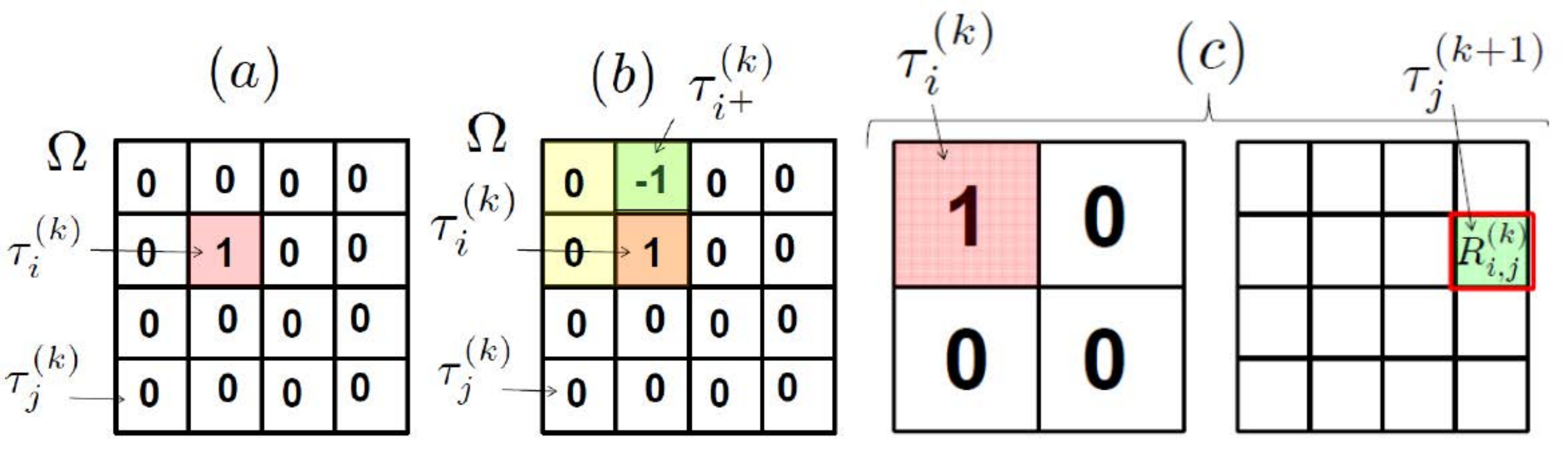}
		\caption{$\psi_i^{(k)}$ is the best bet of Player II on the value of $u\in \B$ given that $[\phi_j^{(k)},u]=\delta_{i,j}$ for $j\in \I^{(k)}$.  $\chi_i^{(k)}$ is the best bet of Player II on the value of $u\in \B$ given that $[\phi_j^{(k)},u]=W_{i,j}$ for $j\in \I^{(k)}$. $R_{i,j}^{(k,k+1)}$ is the best bet of Player II on the value of $[\phi_j^{(k+1)},u]$ given that $[\phi_j^{(k)},u]=\delta_{i,j}$ for $j\in \I^{(k)}$. See Illustration \ref{NEfigbets}.}\label{figbets}
	\end{center}
\end{figure}

\begin{NE}\label{NEfigbets}
Consider Figure \ref{figbets} and the context of Example \ref{egproto00} and Illustrations \ref{nekdjdhhu} and \ref{nekdjdhhu2} where $\phi_i^{(k)}=1_{\tau_i^{(k)}}/\sqrt{|\tau_i^{(k)}|}$ and $(\tau_s^{(k)}, s\in \I^{(k)})$ is a nested rectangular partition of $\Omega=(0,1)^2$. $\psi_i^{(k)}$ is Player II's best bet on the value of the solution $u\in (H^1_0(\Omega),\|\cdot\|_a)$  given  $\int_{\tau^{(k)}_j}u=\sqrt{|\tau^{(k)}_j|}\delta_{i,j}$ for $j\in \I^{(k)}$ (b) $\chi_i^{(k)}$ is Player II's best bet on $u$ given  $\int_{\tau^{(k)}_j}u=\sqrt{|\tau^{(k)}_j|} (\delta_{i,j}-\delta_{i^+,j})$ for $j\in \I^{(k)}$ (where $i^+$ is an adjacent square of $i$) (c) $R^{(k,k+1)}_{i,j}$ is Player II's best bet on  $\int_{\tau^{(k+1)}_j} u$  given  $\int_{\tau^{(k)}_s}u= \sqrt{|\tau^{(k)}_s|} \delta_{i,s}$ for $s\in \I^{(k)}$.
\end{NE}

\subsection{Probabilistic interpretation of numerical errors}\label{subseccoverror}
One popular objective of Probabilistic Numerics, see e.g.\cite{ChkrebtiiCampbell2015, schober2014nips, Owhadi:2014,Hennig2015, Hennig2015b, Briol2015, Conrad2015, OwhadiMultigrid:2015,cockayne2016probabilistic, Cockayne2017},
 is to, to some degree,  go beyond the classical deterministic bounds of numerical analysis and infer posterior  probability distributions on numerical approximation errors.
 In later sections we will demonstrate the existence of a game theoretic optimal Gaussian field $\xi$ in the estimation of the solution $u$ of a linear operator equation.  Then the martingale and multi-resolution decompositions of Theorems \ref{thmdgdjdgygugyd} and \ref{thmgugyug2OR} allow us
to represent the approximation  $u^{(k)}$ of $u$ as the conditional expectation of
  the conditional Gaussian field $\xi^{(k)}$, which  through the multiresolution analysis
 is  a sum of independent Gaussian fields.
 If we consider the Gaussian field $\xi^{(k)}$ as an approximation to the Gaussian field $\xi$, then the errors of this approximation are distributed according the Gaussian field $\xi-\xi^{(k)}$.

We will now determine the covariance operators
 of the   approximation error  $\xi-\xi^{(k)}$. To that end,
for $k\in\{1,\ldots,q\}$ let $\Theta^{(k)}$ be as in \eqref{eqtheta1} and let us consider
the Gaussian field $\xi\in \N(0,Q)$ as in Theorem \ref{thmdlkdjh3e}.
 Since, \eqref{def_gausscond}
implies that $\xi^{(k)}=P_{Q\Phi^{(k)}}\xi$ it follows that the Gaussian field of errors $\xi-\xi^{(k)}$
has covariance $ \Gamma^{(k)}:\B^{*}\rightarrow \B$ defined by
\[ \Gamma^{(k)}=(I-P_{Q\Phi^{(k)}})Q(I-P_{Q\Phi^{(k)}})^{*},\] and using the representation
$P_{Q\Phi^{(k)}}=\sum_{i \in \mathcal{I}^{(k)}}{\psi^{(k)}_{i}\otimes \phi^{(k)}_{i}}$ of
Proposition \ref{prop_Gambletprojection} of the orthogonal projection  and
$P_{Q\Phi^{(k)}}^{*}=Q^{-1}P_{Q\Phi^{(k)}}Q$ of its adjoint, we obtain
\[ \Gamma^{(k)}=(I-P_{Q\Phi^{(k)}})Q=Q-\sum_{i \in \mathcal{I}^{(k)}}{\psi^{(k)}_{i}\otimes Q\phi^{(k)}_{i}}\, .
\]
Moreover,  since for the initial estimate we have  $\xi^{(1)}=P_{Q\Phi^{(1)}}\xi$, it
 is a Gaussian field with covariance operator
$\Gamma^{(1),*}:\B_{*}\rightarrow \B$ defined
\[ \Gamma^{(1),*}=P_{Q\Phi^{(1)}}QP^{*}_{Q\Phi^{(1)}}=P_{Q\Phi^{(1)}}Q
\]
For $k\in \{1,\ldots,q-1\}$, since
 $\xi^{(k+1)}-\xi^{(k)}=P_{Q\Phi^{(k+1)}}\xi- P_{Q\Phi^{(k)}}\xi$,
$P_{Q\Phi^{(k+1)}} P_{Q\Phi^{(k)}}=P_{Q\Phi^{(k)}}$ and $ P_{Q\Phi^{(k)}}P_{Q\Phi^{(k+1)}}=P_{Q\Phi^{(k)}}$, it follows that $\xi^{(k+1)}-\xi^{(k)}$ is a
  Gaussian field with covariance operator $\Gamma^{(k+1),*}$ defined by
\[
\Gamma^{(k+1),*}=(P_{Q\Phi^{(k+1)}}- P_{Q\Phi^{(k)}})Q(P_{Q\Phi^{(k+1)}}- P_{Q\Phi^{(k)}})^{*}\\
=(P_{Q\Phi^{(k+1)}}- P_{Q\Phi^{(k)}})^{2}Q
=(P_{Q\Phi^{(k+1)}}- P_{Q\Phi^{(k)}})Q
\]
so that
\[ \Gamma^{(k+1), *}=\sum_{i \in \mathcal{I}^{(k+1)}}{\psi^{(k+1)}_{i}\otimes Q\phi^{(k+1)}_{i}}-
\sum_{i \in \mathcal{I}^{(k)}}{\psi^{(k)}_{i}\otimes Q\phi^{(k)}_{i}}\, .
\]

One notorious difficulty (complexity bottleneck) in the probabilistic numerics approaches to numerical analysis
 is the complexity of the inversion of dense
 covariance operators required by
 the computation of  posterior probabilities on numerical errors. However, \cite{SchaeferSullivanOwhadi17} shows that,
in the proposed  framework,  covariance operators can be inverted in near-linear complexity if $Q^{-1}$ is a local (e.g. differential) operator on a Sobolev space.

\subsection{Gaussian filtering}

The following proposition shows how Gaussian fields transform under transformation of their base space.
Its proof is straightforward.
\begin{Proposition}\label{prop_filtering}
Consider a continuous bijection $\L:\B\rightarrow \B_{2}$ between Banach spaces and
a Gaussian field
$\xi:(\B^*, \langle \cdot,\cdot\rangle_{Q})\rightarrow L^{2}(\Omega,\Sigma,\mu)$
with covariance $Q:\B^* \rightarrow \B$.
Define the pushforward field
$\L\xi:\B_{2}^*  \rightarrow L^{2}(\Omega,\Sigma,\mu)$
by
\[\L\xi(\varphi):=\xi(\L^{*}\varphi),\quad  \varphi \in \B_{2}^{*}\,.\]
Then $\L\xi$ is a Gaussian field on $\B_{2}^{*}$ with covariance operator
$Q_{2}=\L Q\L^{*}$. In particular,
\[\L \xi:(\B_{2}^*, \langle \cdot,\cdot\rangle_{Q_{2}})\rightarrow L^{2}(\Omega,\Sigma,\mu) \quad \text{is an
 isometry}\, \] with image space
$\L\xi(\B^*_{2})=\xi(\B^*)\subset L^{2}(\Omega,\Sigma,\mu)$,  the same Gaussian Hilbert space
as $\xi$.
\end{Proposition}

In applications, probability distributions can be more naturally placed on the range space of a linear operator, that is, in general, we have much better prior knowledge regarding the righthand side of an operator equation
 than on the set of solutions. Consequently,
to connect with our the analysis  of
the operator in  both \eqref{eqcase1} and \eqref{eqcase1NNN}  derived in Subsection \ref{subgeneralcase}, we invert the above proposition.
\begin{Proposition}\label{propdfflkj}
Let $\L:\B\rightarrow \B_{2}$ be a continuous bijection between Banach spaces.
Let  $\xi_2$ be a Gaussian field of $\B_2$ with covariance operator  $\G^{-1}$,
and let $\xi=\L^{-1}\xi_2$ be the pullback of $\xi_2$ under $\L$.
 It holds true  that $\xi$ is a Gaussian field on $\B$ with covariance operator  $Q= \L^{-1} \G^{-1} \L^{-1,*}$.
\end{Proposition}

From a Game Theoretic perspective this result can be understood as a transfer of optimal mixed strategy from a Game played on $\B_2$ to a game played on $\B$. From a Bayesian perspective this result allows us to construct accurate priors on the solution space of $\L$.
Consider for instance the prototypical PDE \eqref{eqn:scalarprotoa} and Example \ref{egprotocase13}. If $g$ lives in $L^2(\Omega)$ then $u$ lives in a subspace $V$ of $H^1_0(\Omega)$. Here, although it would be difficult to directly specify a good prior for $u$, it remains a simple task to specific a good prior on $g$ (e.g. white noise) and push that prior through the inverse of the operator to obtain a good prior on $u$. This is the strategy introduced in \cite{Owhadi:2014} where it is also shown that Rough Polyharmonic Splines \cite{OwhadiZhangBerlyand:2014} and Polyharmonic Splines  \cite{Harder:1972, Duchon:1976,Duchon:1977,Duchon:1978} can be re-discovered as  solutions of  Gaussian filtering problems, i.e. in the setting of Example \ref{egprotocase13}, if the $\phi_i^{(k)}$ are masses of Diracs (so that the value of $u$ is observed at a finite number of points of $\Omega$ in the game theoretic formulation) then
the gamblets \eqref{defpsi} and \eqref{prop_Gambletprojection} are Polyharmonic Splines  when $a$ is identity matrix and Rough Polyharmonic Splines in the general case.

\subsection{Emergence of probabilistic computation, quantum mechanics and the simulation hypothesis}
The proposed game theoretic interpretation of numerical approximation suggests that probabilistic computation emerges as a natural form of computation with limited resources and partial information.
 Could this approach be related to the form of computation described by Quantum Mechanics, which as shown in \cite{Caves2002-CAVQPA, Benavolietal2016}, could naturally be interpreted as a form of Bayesian computation with complex numbers, and if yes, what would that suggest about the nature of reality?
 Wheeler advocated \cite{Wheeler1990} that ``Quantum Physics requires a new view of reality'' integrating physics with  digital (quanta) information. As observed in  \cite{testingsimhyp2017}, two such views emerge from the presupposition that reality could be computed. The first one, which includes
Digital Physics \cite{Zuse1967} and the cellular automaton interpretation of Quantum Mechanics  \cite{HooftG2016}, proposes that the universe \emph{is} the computer. The second one, which includes the simulation hypothesis \cite{BOSTROM2003, testingsimhyp2017, WhitworthB2007}, suggests that the observable reality is entirely virtual and the system performing the simulation (the computer) is distinct from its simulation (the universe). \cite{testingsimhyp2017} argues that the second view could be analyzed (and tested) based on the assumption that the system performing the simulation has limited computational resources.
Therefore,  in the  simulation hypothesis, the system rendering reality would use computational complexity as a minimization/selection principle for algorithm design, and to achieve near optimal computational complexity by computing with partial information and limited resources, such a system would have to \emph{play dice}. Given these observations it is tempting to analyze/interpret/understand Quantum Mechanics as an optimal form of computation in presence of incomplete information (it is interesting to note that in \cite{Benavolietal2016} the Bayesian formulation of Quantum Mechanics is also derived in a game theoretic setting).

\subsection{Universal worst case measure for optimal recovery}
\label{sec_worstcase}
Here we will demonstrate  the assertion in  Theorem \ref{thmdlkdjh3e} that the appropriately chosen
 Gaussian cylinder measure, or equivalently Gaussian field, is a worst case measure with respect to the mixed extension of a game related to the optimal recovery problem. Moreover, it is universal in the sense that it is independent of the observation functions $\Phi$.
First let us demonstrate how the scaling properties of the optimal recovery problem
 lead naturally to a mixed extension which is also invariant to scalings.

As in Section \ref{subsecttt}, let
 $(\B,\|\cdot\|)$ be a reflexive separable Banach space such that the $\|\cdot\|$ norm is quadratic, i.e.
$\|u\|^2=[Q^{-1}u,u]$ for $u\in \B$,  and $Q$ is a symmetric positive bijective linear operator mapping $\B^*$  to $\B$,  and write $\<\cdot,\cdot\>$ for  the corresponding inner product on $\B$ defined by
\begin{equation}
\label{def_Q}
\<u,v\>:=[Q^{-1} u,v] \text{ for }u,v \in \B\,.
\end{equation}
Recall  the set $L(\Phi,\B) $  of $\bigl(\s(\Phi),\s(\B)\bigr)$-measurable  functions
introduced above  Corollary \ref{cor_micchelli}.
Define the {\em value} $\nu(v)$ of a putative solution $v \in L(\Phi,\B)$, by
\[    \nu(v):=
\sup_{x\in \B}\frac{\|x -v(x)\|^{2}}{\|x\|^{2}}
\]
so that the value
 \[\lambda_{*}:=   \inf_{v \in L(\Phi,\B) }\sup_{x\in \B}\frac{\|x -v(x)\|^{2}}{\|x\|^{2}}\]
of the minmax problem satisfies
$\lambda_{*}=   \inf_{v \in L(\Phi,\B) }{\nu(v)} $.
For $\epsilon \geq 0$, we say that $v\in L(\Phi,\B) $ is an $\epsilon$-optima of the minmax problem
if
 $ \nu(v)-\lambda_{*} \leq \epsilon\, . $
Evidently, the appropriate saddle function with which to define this game and its mixed extension appears to be
\begin{equation}
\Phi(v,x):=\frac{\|x -v(x)\|^{2}}{\|x\|^{2}}, \quad v \in L(\Phi,\B), x \in \B\setminus \{0\}\,
\end{equation}
so that
\[    \nu(v):=
\sup_{x\in \B}{\Phi(v,x)}
\]
and
\[\lambda_{\Phi}=   \inf_{v \in L(\Phi,\B) }{\nu(v)} =\inf_{v \in L(\Phi,\B) }\sup_{x\in \B}{\Phi(v,x)}\, .\]

The minmax problem can be shown to be reducible to linear solutions and so corresponds
to a  minmax problem with objective function
$\|x -v(x)\|^{2}$ subject to the constraint
$\|x\| \leq 1$. Since this problem is homogeneously related to the same problem
with constraint $\|x\| \leq t$ for every $ t >0$ this homogeneity appears to generate its connection
with a worst case distribution on the whole space $\B$ instead of its unit ball, as follows.
Let $\mathcal{M}_{2}(\B)$ denote the Borel probability measures with finite second moments on $\B$,
 and consider
the following saddle function
\begin{equation*}
\Psi(v,\mu):=\int{\|x -v(x)\|^{2}d\mu(x)}
\end{equation*}
and corresponding minmax problem
$\lambda_{\Psi}=   \inf_{v \in L(\Phi,\B) }\sup_{\mu\in \mathcal{M}_{2}(\B)}{\Psi(v,\mu)}\, .$
The saddle function $\Psi$ is clearly convex in $v$ and affine in $\mu$ and consequently it is
convex-concave.
The covariance operator $S_{\mu}:\B \rightarrow \B$ of any measure in $\mathcal{M}_{2}(\B)$ is
  $Q$-symmetric in that $QS_{\mu}^{*}=S_{\mu}Q$ where
 $S_{\mu}^{*}$ is the adjoint of $S_{\mu}$ defined through the dual pairing
$[S^{*}_{\mu}\varphi, u]=[\varphi, S_{\mu}u], \varphi \in \B^{*}, u \in \B$.
However, it is easy to show that scaling the covariance operator $S_{\mu}$ of a measure $\mu$
to $tS_{\mu}$ with $t >0$ produces a measure $\mu_{t}$  such that
\[\int{\|x\|^{2}d\mu_{t}}=t\int{\|x\|^{2}d\mu}\, .\]
Since
 Wasilkowski and Wozniakowsi \cite{wasilkowski1986average} show that  the optimal solution
is an orthogonal projection,  independent of such scaling, it follows that the optimal value is infinite.  Therefore, it appears appropriate to either constrain
the second moment or scale by the second moment. Since the above analysis led
to scaling for the worst case problem, this suggests we proceed with scaling. To that end,
consider the following scaled saddle function
\begin{equation}
\label{def_saddle}
\Psi(v,\mu):=\frac{\int{\|x -v(x)\|^{2}d\mu(x)}}{\int{\|x\|^{2}d\mu(x)}}
\end{equation}
This saddle function  is a fractional function and is easily seen to be quasi-concave in its second argument, see Mangasarian \cite[Sec.~9.6]{mangasarian1994nonlinear}.
Since it is convex in its first argument
it is a quasi-convex/quasi-concave saddle function. Therefore, given
 compactness of one of the domains the minmax theorem of Sion \cite{sion1958general}
may be used to demonstrate that it satisfies a minmax equality. The following theorem demonstrates
that $\Psi$ does indeed satisfy a minmax equality, without any need for compactness.

\begin{Theorem}
\label{thm_minmax}
Under the condition $dim(\Phi) < dim(\B)$, we have
\begin{equation}
\label{minmax1}
  \inf_{v \in L(\Phi,\B) }\sup_{\mu\in \mathcal{M}_{2}(\B)}{\Psi(v,\mu)}
=  \sup_{\mu\in \mathcal{M}_{2}(\B)}\inf_{v \in L(\Phi,\B) }{\Psi(v,\mu)}=1\,
\end{equation}

\end{Theorem}
By the classical relationship between saddle points and worst case components of minmax problems,
a measure $\mu^{*} \in \mathcal{M}_{2}(\B)$ is a worst case measure if it is a component
of a saddle point $(v^{*},\mu^{*})$ for $\Psi$, that is we have
\[\Psi(v^{*},\mu) \leq \Psi(v^{*},\mu^{*}) \leq  \Psi(v,\mu^{*}),\quad v\in  L(\Phi,\B), \mu \in  \mathcal{M}_{2}(\B).\,\]

When $\B$ is infinite dimensional, it is straightforward to show that such saddle points do not exist in the class of countably additive measures.
 On the other hand, in that case, we now show that if we extend the notion of saddle point slightly,
then a Gaussian cylinder measure is not only a component of a saddle point, it is {\em computable} in the sense that we can  determine countably additive Gaussian measures which are components of approximate saddle points, which approximate it.
To that end, let $CM$ denote the space of cylinder
 measures on $\B$ and  let $\mathcal{F}(\B)$ be the set of continuous linear finite-rank projections on $\B$, and
  define the {\em weak cylinder measure topology} $\omega_{CM}$ by
saying that
 \[\mu_{n} \xrightarrow{\omega_{CM}} \mu\]
 if
\begin{equation}
\label{def_omegaCM}
F_{*}\mu_{n}  \xrightarrow{\omega} F_{*}\mu,\quad  F  \in \mathcal{F}(\B)\, .
\end{equation}
We have the following.
\begin{Proposition}
\label{prop_cmcomplete}
The space $(CM,\omega_{CM})$ is sequentially complete.
\end{Proposition}
We say that a pair $(v^{*},\mu^{*}) \in  L(\Phi,\B) \times  \mathcal{M}_{2}(\B)$ is an $\epsilon$-saddle point
of $\Psi$ if
\[ \Psi(v^{*},\mu) -\epsilon \leq  \Psi(v^{*},\mu^{*}) \leq  \Psi(v,\mu^{*}) +\epsilon\, ,
\quad v \in   L(\Phi,\B),\mu
\in \mathcal{M}_{2}(\B)\, . \]

\begin{Definition}
\label{def_saddledef}
Consider a saddle function $\Psi: L(\Phi,\B) \times  \mathcal{M}_{2}(\B) \rightarrow \R$. We say that a pair
$(v^{*},\mu^{*}) \in  L(\Phi,\B) \times CM$ is a saddle point of $\Psi$ if there exists a sequence
$\mu^{*}_{n} \in \mathcal{M}_{2}(\B), n=1,\ldots $ such that
\[ \mu^{*}_{n} \xrightarrow{\omega_{CM}} \mu^{*}\, \] and
\[(v^{*},\mu_{n}^{*})\,\, \text{ is  a $\frac{1}{n}$-saddle point of $\Psi$}\]
for all $n$.
\end{Definition}

Our primary result in this section  asserts that the appropriately chosen Gaussian cylinder measure is a worst case measure for optimal recovery in infinite dimensions. To include the finite dimensional case recall
 the characterization of Anderson and Trapp  \cite[Thm.~6]{anderson1975shorted} of the short
 $ \Phi^{\perp}(Q)$ of the operator $Q$ to $\Phi^{\perp}$ defined by
\begin{equation}
\label{id_short}
  \bigl[ \Phi^{\perp}(Q)s,s\bigr]=\inf\Bigl\{\Bigl[ Q(s+t),(s+t)
  \Bigr], \, t \in \Phi  \Bigr\},\quad  s \in  \B^{*}\, .
\end{equation}

\begin{Theorem}
\label{thm_saddle}
Consider a separable  Hilbert space $\B$, with inner product defined by
$\langle u_{1},u_{2} \rangle:=[Q^{-1}u_{1},u_{2}]$ where $Q:\B^{*}\rightarrow \B$ is a continuous bijection.
Suppose that $dim(\Phi) < dim(\B)$.
 Let
$P_{Q\Phi}$ denote orthogonal projection onto $Q\Phi$.
 Consider the
short  $\Phi^{\perp}(Q)$ \eqref{id_short}  of the operator $Q$ to the subspace $(Q\Phi)^{\perp}$ and the corresponding Gaussian measure $\mu_{\Phi^{\perp}(Q)}$.
Then
the pair
\[(P_{Q\Phi},\mu_{\Phi^{\perp}(Q)})\]
 is a saddle point of
$\Psi$.
Moreover, in the infinite dimensional case,
 the approximating sequence can be chosen to be a sequence of classical Gaussian measures of probability on $\B$
(which define a sequence of Gaussian random vectors on $\B$).
\end{Theorem}
\begin{Remark}
\label{rmk_iubibu}
It is interesting to note that in the infinite dimensional case one can also prove that the Gaussian cylinder measure, without conditioning, is also a worst case measure. That is,
\[(P_{Q\Phi},\mu_{Q})\]
 is a saddle point of
$\Psi$.
\end{Remark}

\begin{Remark}[{\bf  Universal Worst Case Measure}]
\label{rm_universalcomputablewc}
Theorem \ref{thm_saddle} implies, for a set $\dot{\Phi}$
 of observation functions,
that the cylinder measure $\mu_{\Phi^{\perp}(Q)}$
 is a worst case measure.
 Consequently, we say that $\mu_{Q}$ is a universal measure, in that it generates
the worst case measure $\mu_{\Phi^{\perp}(Q)}$ through the shorting operation.
\end{Remark}
\begin{Remark}
Recently, \cite{OwhadiScovelSchur} have  established that, when $S$ positive trace class, that
 the short $\Phi^{\perp}(S)$ is the covariance operator associated with conditioning the
Gaussian measure $\mu_{S}$ on the subspace $S\Phi$. Extending this result to cylinder measures,
 one obtains that
 the worst case (cylinder) measure $\mu_{\Phi^{\perp}(Q)}$ then has the interpretation of being obtained by conditioning the universal measure $\mu_{Q}$.
   Moreover, in infinite dimensions,   Remark \ref{rmk_iubibu} asserts that
the Gaussian cylinder measure $\mu_{Q}$,  without conditioning, is also a worst case measure. Since
 Theorem \ref{thm_saddle} also asserts that $\mu_{Q}$ is the limit of approximate saddle points consisting of
countably additive (classical) Gaussian measures,
 this implies that  $\mu_{Q}$ is a {\bf computable} universal worst case measure.
Finally, the connection between the shorted operator and the conditional measures implies
 the following addendum to Theorem  \ref{thmdlkdjh3e}:
The optimal   strategy  of \eqref{eqdkjdhkjhffORgame} for Player II is the pure strategy
\begin{equation}\label{eqbhbdhbdjhb3eor2}
u^\two(u)=\E_{\xi \sim \N(0,Q)}\big[\xi \mid \text{$[\phi_i,\xi]=[\phi_i,u]$ for $i=1,\ldots,m$}\big]\,,\quad
u\in \B
\end{equation}
 corresponding to the worst case mixed strategy
\begin{equation}
\label{eq_yuvuyvuyv}
 u^{\one}\sim \xi -\E_{\xi \sim \N(0,Q)}{[\xi \mid [\phi_{i},\xi],  i \in \mathcal{I}]}
\end{equation}
 of player $I$.
\end{Remark}

\section{Exponential decay and localization of gamblets}\label{secexpdecloc}
Although the analysis of the exponential decay and localization of gamblets could be restricted to the discrete case (i.e. linear algebra with finite dimensional matrices), we will also perform this analysis in the continuous case and show that localization can be expressed as a property of the image (or dual) space that can be pulled back to a property of the solution space via the continuity of the operator.
The characterization of the exponential decay of gamblets is relative to a notion of \emph{physical} distance
  that is distinct from the norm $\|\cdot\|$ of the Banach space $\B$ under consideration, but more closely related to the metric structure of its domain when it is a function space over that domain.
Although, in general,
an arbitrary space $\B$  does not possess a natural \emph{physical} distance, in this section we demonstrate how
such a notion  emerges from a subspace decomposition of $\B$ in a way that generalizes
 the domain decomposition in the computation of PDEs.

\subsection{Subspace decomposition}\label{subsecejhdg9877eg8e}
In a first step we will, in this subsection, provide localization results based on a generalization of the subspace iteration method (and related conditions) introduced in \cite{KornhuberYserentant16, KornhuberYserentant16bis}. As in \cite{KornhuberYserentant16, KornhuberYserentant16bis}, this approach is analogous to a  Schwarz subspace decomposition and correction method \cite{xu1992iterative, griebel1995abstract}.

For $|\aleph|,|\beth| \in \mathbb{N}^{*}$, consider index sets
$\aleph:=\{1,\ldots,|\aleph|\}$  and $\beth:= \{1,\ldots,|\beth|\}$, and let
 $(\phi_{i,\alpha})_{(i,\alpha)\in \beth \times \aleph}$ be $|\beth|\times |\aleph|$ linearly independent elements of $\B^*$.  Throughout this section, all internal sums will be {\em non-direct} in that the components in the sum may have nontrivial intersection.

\begin{Construction}\label{constlocop}
For all $i\in \beth$
 let $\B_{i} \subset \B$ be a closed subspace
 such that (1) $\B=\B_{1}+\cdots+\B_{|\beth|}$ (2) For each $(i,\alpha) \in \beth \times \aleph$, there exists  $\tilde{\psi}_{i,\alpha}\in \B_{i}$ such that $[\phi_{j,\beta},\tilde{\psi}_{i,\alpha}]=\delta_{i,j}\delta_{\alpha,\beta}$ for $(j,\beta)\in \beth \times \aleph$. Equip each of these subspaces
  $\B_{i}$  with the norm $\|\cdot\|_{i}$ induced by $\|\cdot\|$.
\end{Construction}
Item (2) of Construction \ref{constlocop} ensures that there exists an element $\psi$ in the localized subspace $\B_{i}$ satisfying the constraints imposed by the measurement functions $(\phi_{i,\alpha})_{(i,\alpha)\in \beth \times \aleph}$ appearing in the variational formulation, derived from Definition \eqref{eq:dfddeytfewdaisq},
 of gamblets that we will use in this section. In particular,
 this property  implies that \eqref{eqhihdjkjhiudiduh} below has a solution for $n=0$.

Let $\V^\perp:=\{\psi \in \B\mid [\phi_{i,\alpha},\psi]=0\text{ for }(i,\alpha)\in \beth\times \aleph\}$ and
for $i\in \beth$ write
$\V^\perp_i:=\B_{i} \cap \V^\perp$.
For $i\in \beth$, let $P_{i}$ be the $\<\cdot,\cdot\>$-orthogonal projection mapping $\B$ onto $\V^{\perp}_i$, i.e. for $\psi \in \B$, $P_{i} \psi$ is the unique element of $\V^{\perp}_i$ such that
\begin{equation}
\<P_{i} \psi,\chi\>=\<\psi,\chi\>\text{ for }\chi\in \V^{\perp}_i\, .
\end{equation}
Write
\begin{equation}\label{eqP}
P:=P_{1}+\cdots+P_{|\beth|}
\end{equation}
and define $\lambda_{\min}(P)$  and $\lambda_{\max}(P)$ (respectively) as the largest and smallest constants such that for all $\chi\in \V^\perp$
\begin{equation}\label{eqkdjdjhdj}
\lambda_{\min}(P)\, \|\chi\|^2 \leq \<\chi, P \chi \> \leq \lambda_{\max}(P)\, \|\chi\|^2\,,
\end{equation}
and denote the condition number of $P$ by
\begin{equation}
\Cond(P):=\frac{\lambda_{\max}(P)}{\lambda_{\min}(P)}\,.
\end{equation}

\begin{Lemma}\label{lemequivpv}  $P$ restricted to $\V^\perp$ is  a symmetric linear operator
$P:\V^\perp \rightarrow \V^\perp $.  Furthermore $\V^\perp=\sum_{i\in \beth} \V_i^\perp$
 is equivalent to $\lambda_{\min}(P)>0$ and also equivalent to the bijectivity of $P$.
\end{Lemma}

In Lemma \ref{lemdeiudygddf} of  Subsection \ref{subsecdklejhd78} below, we will provide simple and natural conditions that are equivalent to the following condition in terms of the alignment of the measurement functions with the subspaces $\B_{i}$.
\begin{Condition}\label{confviperp}
Assume that  $\lambda_{\min}(P)>0$\,.
\end{Condition}

Let $\C$ be the $\beth\times \beth$ connectivity matrix defined by $\C_{i,j}=1$ if there exists $(\chi_i,\chi_j) \in \B_{i} \times \B_{j}$ such that $\<\chi_j,\chi_i\>\not=0$ and $\C_{i,j}=0$ otherwise.
Let $\db:=\db^{\C}$ be the graph distance on $\beth$ induced by the connectivity matrix $\C$ (see Definition \ref{defgraphmatdistbis}).
Let $(\psi_{i,\alpha})_{(i,\alpha)\in \beth \times \aleph}$ be the gamblets
 (per Definition \eqref{eq:dfddeytfewdaisq}) corresponding to $(\phi_{i,\alpha})_{(i,\alpha)\in \beth \times \aleph}$, i.e. for $(i,\alpha)\in \beth \times \aleph$,  $\psi_{i,\aleph}$ is the minimizer of $\|v\|$  over $v\in \B$ such that $[\phi_{j,\beta},v]=\delta_{i,j}\delta_{\alpha,\beta}$ for $(j,\alpha)\in \beth\times \aleph$.

Let us now  widen each $\B_{i}$ to a $\B_{i}^{n}$ for each $n\in \mathbb{N}$,
 by including its neighbors in a ball of radius $n$  in the graph distance $\db$, by  defining $\B_{i}^{n}:=\sum_{j:\db(i,j)\leq n}\B_{j}$. For each $n\in \mathbb{N}$, we now define some modified gamblets using these widened spaces  $\B_{i}^{n}$ as follows:
for $(i,\alpha)\in \beth \times \aleph$, let $\psi_{i,\alpha}^n$ be the unique minimizer of
\begin{equation}\label{eqhihdjkjhiudiduh}
\begin{cases}
\text{Minimize }\|\psi\|\\
\text{Subject to }\psi\in \B_{i}^{n} \text{ and }[\phi_{j,\beta},\psi]=\delta_{i,j}\delta_{\alpha,\beta}\text{ for }(j,\beta)\in \beth\times \aleph\,.
\end{cases}
\end{equation}
The following theorem shows that if $\Cond(P)<\infty$ then difference between $\psi_{i,\alpha}$  and
 $\psi_{i,\alpha}^n$ decays exponentially in $n$ so that the computation of $\psi_{i,\alpha}$  can be localized.
\begin{Theorem}\label{thmswkskjsh}
Under Condition \ref{confviperp}, it holds true that  $\|\psi_{i,\alpha}-\psi_{i,\alpha}^n\| \leq \big(\frac{\Cond(P)-1}{\Cond(P)+1}\big)^n \|\psi_{i,\alpha}^0\|$ for $n\geq 0$.
\end{Theorem}

Let $A$ be the $(\beth \times \aleph)\times (\beth \times \aleph)$ stiffness matrix defined by $A_{(i,\alpha),(j,\beta)}=\<\psi_{i,\alpha},\psi_{j,\beta}\>$. The following theorem shows that if $\Cond(P)<\infty$ then $A$ decays exponentially away from its diagonal (which will provide sufficient bounds on approximation errors introduced by  truncating $A$).
\begin{Theorem}\label{thmhgguyg65OR}
Under Condition \ref{confviperp}, it holds true that
\begin{equation}
|A_{(i,\alpha),(j,\beta)}|\leq \|\psi_{i,\alpha}^0\|\|\psi_{j,\beta}^0\| \big(\frac{\Cond(P)-1}{\Cond(P)+1}\big)^{\frac{\db(i,j)}{2}-1}\,
\end{equation}
for all  $(i,\alpha),(j,\beta)\in \beth \times \aleph$.
\end{Theorem}

Write
\begin{equation}
n_{\max}=\max_{i\in \beth} \operatorname{Card}\{j \in \beth\mid \db(i,j)\leq 1\}
\end{equation}
for the maximum number of elements of a  $\db$ ball of radius one.
Moreover, let
 $K_{\max}$ be the smallest constant such that
\begin{equation}
\|\chi\|^2 \leq K_{\max}\sum_{i\in \beth} \|\chi_i\|^2
\end{equation}
for $\chi=\sum_{i\in \beth} \chi_i$ with  $\chi_i\in \V_{i}^\perp, i \in \beth$.
Similarly,  define $K_{\min}$ as the largest constant such that, for all $\chi\in \V^\perp$, there exists a decomposition $\chi=\sum_{i\in \beth}\chi_i$ with $\chi_i\in \V_i^\perp, i \in \beth$ such that
\begin{equation}\label{eqlkjdhlkdhd}
K_{\min}\sum_{i\in \beth} \|\chi_i\|^2 \leq \|\chi\|^2\,.
\end{equation}

The strategy introduced in \cite{KornhuberYserentant16, KornhuberYserentant16bis} is to bound $\Cond(P)$ by $K_{\max}/K_{\min}$ as described by the following lemma, which is a simple generalization of Lemma 3.1 of \cite{KornhuberYserentant16}.
\begin{Lemma}\label{lemdkjdhjh3e}
It holds true that $ K_{\min}\leq \lambda_{\min}(P)$ and $\lambda_{\max}(P) \leq K_{\max}\leq n_{\max}$.
\end{Lemma}

The following Proposition shows that the inequality $ K_{\min}\leq \lambda_{\min}(P)$ obtained in  Lemma \ref{lemdkjdhjh3e} (and in \cite{KornhuberYserentant16} for divergence form elliptic PDEs) is in fact an equality.
\begin{Proposition}\label{propkjshkdjhdkjh}
It holds true that $\lambda_{\min}(P)=K_{\min}$.
\end{Proposition}

\subsection{Conditions on dual and image spaces}\label{subsecdklejhd78}
In a second step,  in this subsection, we will bound $\Cond(P)$ based on equivalent necessary and sufficient conditions expressed on  the dual space $\B^*$ or the image space $\B_2$.

For the simplicity of the notations we will continue using $[\cdot,\cdot]$ for the duality product between $\B_{i}^*$ and $\B_{i}$.
Write $\|\cdot\|_{*,i}$ for dual the norm induced by $\|\cdot\|_{i}$ on $\B_{i}^*$, defined by
 $\|\varphi\|_{*,i}:=\sup_{\psi \in \B_{i}}\frac{[\varphi,\psi]}{\|\psi\|_{i}}$ for $\varphi\in \B_{i}^*$.
Write $Q_i:\B_{i}^{*} \rightarrow \B_{i}$ for the positive symmetric linear bijection satisfying   $\|\varphi\|_{*,i}^2=[\varphi,Q_i \varphi]$ for $\varphi \in B_{i}^*$.
For $i\in \beth$, let  $\Rc_i$ be the adjoint of the subspace injection $\B_{i}\rightarrow \B$, so that for
 $\varphi\in \B^*$, $\Rc_i \varphi$
 is the unique element of $\B_{i}^*$ such that
$[\varphi,\psi]=[\Rc_i \varphi, \psi]$ for $\psi \in \B_{i}$. That is, $\Rc_i \varphi$ is obtained by restricting the action of $\varphi$ to $\B_{i}$.
\begin{Theorem}\label{propjguyug6}
It holds true that $\lambda_{\min}(P)$  and $\lambda_{\max}(P)$ are also (respectively) the largest and smallest constants such that any of the following conditions hold,
\begin{itemize}
\item For all $\varphi\in \B^*$,
\begin{equation}\label{eqkhohihi}
\lambda_{\min}(P)\, \sup_{\chi'\in \V^\perp} \frac{[\varphi,\chi']^2}{\|\chi'\|^2} \leq \sum_{i\in \beth} \sup_{\chi'\in \V_i^\perp} \frac{[\varphi,\chi']^2}{\|\chi'\|^2} \leq \lambda_{\max}(P)\, \sup_{\chi'\in \V^\perp} \frac{[\varphi,\chi']^2}{\|\chi'\|^2}\,.
\end{equation}
\item For all $\varphi\in \B^*$,
\begin{equation}\label{eqljdhelkjdhkh3}
\lambda_{\min}(P)\, \inf_{\phi\in \Phi}\|\varphi-\phi\|_*^2 \leq \sum_{i\in \beth} \inf_{\phi\in \Phi}\|\Rc_i (\varphi-\phi)\|_{*,i}^2\leq \lambda_{\max}(P)\, \inf_{\phi\in \Phi}\|\varphi-\phi\|_*^2 \,.
\end{equation}
\end{itemize}
\end{Theorem}
\begin{Remark}
We also refer to Lemma \ref{lempropjguyug6} for equivalent conditions expressed in terms of quadratic form inequalities satisfied
by the actions of the operators $Q$ and its localized versions $Q_i$ on equivalence classes induced by measurement functions. In the context of Example \ref{egprotoh10normNNN} these conditions can be viewed as inequalities on Green's functions acting on equivalence classes induced by measurement functions.
\end{Remark}

The following theorem allows us to express the localization property of gamblets as a property of, or condition on, the image space $\B_{2}$.

\begin{Theorem}\label{thmlkhkjhlkh}
For a continuous bijection $\L:\B\rightarrow \B_{2}$, let
$\B_{2,i}:=\L B_{i}$ be equipped with metric $\|\cdot\|_{2,i}$ induced as a subspace $\B_{2,i} \subset \B_{2}$.
For each $i\in \beth$, let $\L_i:B_{i}\rightarrow \B_{2,i}$ denote the corresponding induced bijection and
write $C_{\L_i}$ and $C_{\L_i^{-1}}$ for the continuity constants of  $\L_i$ and $\L_i^{-1}$. Write
 $\bar{C}_{\L}:=\max(C_{\L},\max_{i} C_{\L_i})$ and $\bar{C}_{\L^{-1}}:=\max(C_{\L^{-1}},\max_{i} C_{\L_i^{-1}})$.
Let $C_{\min}$ and $C_{\max}$ be  constants such that for all $\varphi\in \B^*$,
\begin{equation}
C_{\min}\, \inf_{\phi\in \Phi}\|\L Q (\varphi-\phi)\|_2^2 \leq \sum_{i\in \beth} \inf_{\phi\in \Phi}\|\L_i Q_i \Rc_i ( \varphi-\phi)\|_{2,i}^2\leq C_{\max}\, \inf_{\phi\in \Phi}\|\L Q (\varphi-\phi)\|_2^2 \,.
\end{equation}
It holds true that $\lambda_{\max}(P)\leq (\bar{C}_{\L} \bar{C}_{\L^{-1}})^2 C_{\max}$ and
$\lambda_{\min}(P)\geq  \frac{C_{\min}}{(\bar{C}_{\L} \bar{C}_{\L^{-1}})^2}$.
\end{Theorem}

\begin{Example}\label{egprotoh10normNNNlocal}
Consider Example \ref{egprotoh10normNNN}.
 Let $\Omega_1,\ldots,\Omega_{|\beth|}$ be open subsets of $\Omega$   such that $\Omega=\cup_{i\in \beth} \Omega_i$.
Let $\B_{i}=H^1_0(\Omega_i)$  with the norm $\|v\|_{i}^2=\int_{\Omega_i}(\nabla v)^T a \nabla v$.
Let $\B_{2,i}=H^{-1}(\Omega_i)$ with the norm $\|g\|_{2,i}=\sup_{v\in H^1_0(\Omega_i)}\frac{\int_{\Omega_i} gv}{\|v\|_{H^1_0(\Omega_i)}}=\|\nabla \Delta_i^{-1} g\|_{L^2(\Omega_i)}$ (writing $-\Delta_i$ the Laplace-Dirichlet operator on $\Omega_i$). Let $\L_i$ be the differential operator $-\diiv(a\nabla)$ mapping $\B_{i}$ to $\B_{2,i}$.
 Note that that the continuity constants provided in Example \ref{egprotoh10normNNN} imply $C_{\L_i}\leq \sqrt{\lambda_{\max}(a)}$  for all $i\in \beth$, and therefore $\bar{C}_{\L}\leq \sqrt{\lambda_{\max}(a)}$. Similarly we have
 $\bar{C}_{\L^{-1}}\leq 1/\sqrt{\lambda_{\min}(a)}$.
Observe that $\L Q $ and $\L_i Q_i$ are the identity operators (on the same spaces with different metrics, e.g. $\L Q$ is identity operator mapping $H^{-1}(\Omega)$ with the  $\|\cdot\|_{*}$-norm to $H^{-1}(\Omega)$ with the  usual/classical $\|\cdot\|_{H^{-1}(\Omega)}$-norm), and $\Rc_i \varphi=\varphi$ on $\Omega_i$.
Theorem \ref{thmlkhkjhlkh} implies that
\begin{equation}\label{eqlkjhkhiui}
\Cond(P)\leq \frac{C_{\max}}{C_{\min}} \big(\frac{\lambda_{\max}(a)}{\lambda_{\min}(a)}\big)^2,
\end{equation}
where $C_{\min}$ is the largest constant, and $C_{\max}$ is the smallest constant such that, for all $\varphi\in H^{-1}(\Omega)$,
\begin{equation}\label{eqkjddlkjdlj}
C_{\min}\, \inf_{\phi\in \Phi}\|\varphi-\phi\|_{H^{-1}(\Omega)}^2 \leq \sum_{i\in \beth} \inf_{\phi\in \Phi}\| \varphi-\phi\|_{H^{-1}(\Omega_i)}^2\leq C_{\max}\,\inf_{\phi\in \Phi} \|\varphi-\phi\|_{H^{-1}(\Omega)}^2,
\end{equation}
where we abuse notations by writing $\Phi \subset \B^*$ for $\L Q \Phi \subset \B_2$ since $\L Q$ is the identity operator on the same space with different metrics (and we have used a similar natural simplification for $\L_i Q_i$).
\end{Example}

The following theorem (whose proof is a straightforward consequence of Theorem \ref{propjguyug6}) allows us to
express the localization property of gamblets as a property of, or condition on, the dual space  $\B^*$.
\begin{Theorem}\label{thmlkhkdfual}
Let $\|\cdot\|_{e,*}$ and $\|\cdot\|_{e,*,i}$ be alternate norms on $\B^*$ and $\B_{i}^*$ and let $C_{e}\geq 1$
 be  a constant such that
$C_{e}^{-1}\|\cdot\|_{e,*} \leq \|\cdot\|_{*}\leq C_{e} \|\cdot\|_{e,*}$ and
$C_{e}^{-1}\|\cdot\|_{e,*,i} \leq \|\cdot\|_{*,i}\leq C_{e} \|\cdot\|_{e,*,i}$ for $i\in \beth$.
Let $C_{\min}$ and $C_{\max}$ be  constants such that for all $\varphi\in \B^*$,
for all $\varphi\in \B^*$,
\begin{equation}\label{eqljdhelfkjdhrtkh3}
C_{\min}\, \inf_{\phi\in \Phi}\|\varphi-\phi\|_{e,*}^2 \leq \sum_{i\in \beth} \inf_{\phi\in \Phi}\|\Rc_i (\varphi-\phi)\|_{e,*,i}^2\leq C_{\max}\, \inf_{\phi\in \Phi}\|\varphi-\phi\|_{e,*}^2 \,.
\end{equation}
It holds true that $\lambda_{\max}(P)\leq C_{e}^4 C_{\max}$ and
$\lambda_{\min}(P)\geq  C_{e}^{-4} C_{\min}$.
\end{Theorem}

The following corollary (whose proof is a direct consequence of Theorems \ref{propjguyug6} and \ref{thmlkhkdfual}) allows us to transfer the localization property via norm equivalence  onto the primary space $\B$.

\begin{Corollary}\label{corthmlkkjhkdfual}
Let $\|\cdot\|_{e}$ and $\|\cdot\|_{e,i}$ be alternate norms on $\B$ and $\B_{i}$ and let $C_{e}\geq 1$ be a
  constant such that
$C_{e}^{-1}\|\cdot\|_{e} \leq \|\cdot\|\leq C_{e} \|\cdot\|_{e}$ and
$C_{e}^{-1}\|\cdot\|_{e,i} \leq \|\cdot\|_{i}\leq C_{e} \|\cdot\|_{e,i}$ for $i\in \beth$.
Let $\|\cdot\|_{e,*}$ and $\|\cdot\|_{e,*,i}$ be the induced dual norms on $\B^*$ and $\B_{i}^*$.
If $C_{\min}$ and $C_{\max}$ are such that \eqref{eqljdhelfkjdhrtkh3} holds  for all $\varphi\in \B^*$, then
 $\lambda_{\max}(P)\leq C_{e}^4 C_{\max}$ and
$\lambda_{\min}(P)\geq  C_{e}^{-4}C_{\min}$.
\end{Corollary}

\begin{Example}\label{egdkj3hrr}
Consider Example \ref{egprotoalsobolevl}.
Let $(\phi_{i,\alpha})_{(i,\alpha)\in \beth\times \aleph} \in H^{-s}(\Omega)$ and write $\Phi=\Span\{\phi_{i,\alpha}\mid (i,\alpha)\in \beth \times \aleph\}$.
 Let $\Omega_1,\ldots,\Omega_{|\beth|}$ be open subsets of $\Omega$   such that $\Omega=\cup_{i\in \beth} \Omega_i$.
Let $\B_{i}=(H^s_0(\Omega_i),\|\cdot\|_{i})$  with the norm $\|\cdot\|_{i}$
induced by the restriction of $\|\cdot\|$ to $H^s_0(\Omega_i)$.
  Let $C_{e}\geq 1$ be a constant such that
$C_{e}^{-1}\|\cdot\|_{H^s_0(\Omega)} \leq \|\cdot\|\leq C_{e} \|\cdot\|_{H^s_0(\Omega)}$. Then we have
$C_{e}^{-1}\|\cdot\|_{H^s_0(\Omega_i)} \leq \|\cdot\|_{i}\leq C_{e} \|\cdot\|_{H^s_0(\Omega_i)}$ for $i\in \beth$.
Moreover, Corollary \ref{corthmlkkjhkdfual} implies that  $\lambda_{\max}(P)\leq  C_{e}^4 C_{\max}$ and
$\lambda_{\min}(P)\geq  C_{e}^{-4}C_{\min}$ where $C_{\min}$ is the largest constant, and $C_{\max}$ is the smallest constant, such that for all $\varphi\in H^{-s}(\Omega)$,
\begin{equation}\label{eqkjddlkjdldjjiej}
C_{\min}\, \inf_{\phi\in \Phi}\|\varphi-\phi\|_{H^{-s}(\Omega)}^2 \leq \sum_{i\in \beth} \inf_{\phi\in \Phi}\| \varphi-\phi\|_{H^{-s}(\Omega_i)}^2\leq C_{\max}\,\inf_{\phi\in \Phi} \|\varphi-\phi\|_{H^{-s}(\Omega)}^2\,.
\end{equation}
\end{Example}

\subsection{Conditions on localized measurement functions}\label{subsecdhgiudg6}

We now specify conditions on the relationship between the  subspaces $\B_{i}$
and the measurement functions $\phi_{i,\alpha}\in \B^{*}, (i,\alpha)\in \beth\times \aleph$.
\begin{Condition}\label{conduihiuh}
For $(i,\alpha)\in \beth\times \aleph$ there exists $\tilde{\psi}_{i,\alpha}\in \B_{i}$ such that (1) $[\phi_{j,\beta},\tilde{\psi}_{i,\alpha}]=\delta_{i,j}\delta_{\alpha,\beta}$ for $(j,\beta)\in \beth\times \aleph$, and (2) if $\tilde{\psi}_{i,\alpha} \not\in \B_{j}$ then $[\phi_{i,\alpha},v]=0$ for $v\in \B_{j}$.
\end{Condition}
For $\tilde{\psi}_{i,\alpha} \in \B_{i}, (i,\alpha)\in \beth\times \aleph$ satisfying
 Condition \ref{conduihiuh},
let $\tilde{P}:\B \rightarrow \B$ be the linear operator  defined by
\begin{equation}\label{eqkllhkjhju9}
\tilde{P} v:=\sum_{(i,\alpha)\in \beth\times \aleph} \tilde{\psi}_{i,\alpha} [\phi_{i,\alpha},v],\,\text{ for } v \in \B\, .
\end{equation}

\begin{Lemma}\label{lemdeiudygddf}
Condition \ref{conduihiuh} implies $\V^\perp=\V_1^\perp+\cdots+\V_{|\beth|}^\perp$ and $\lambda_{\min}(P)>0$.
\end{Lemma}

\begin{Condition}\label{conduihiuhno2}
Let $\tilde{P}$ be as in \eqref{eqkllhkjhju9}.
There exists $T_{\max}>0$ such that every $v\in \V^\perp$ can be decomposed as  $v=\sum_{i\in \beth} v_i$ with $v_i \in \B_{i}$ and (1) $\sum_{i\in \beth} \|v_i\|^2 \leq T_{\max} \|v\|^2$ and (2) $\sum_{i\in \beth} \|\tilde{P} v_i\|^2 \leq T_{\max} \|v\|^2$.
\end{Condition}

\begin{Theorem}\label{thmjhg8g87g}
Under Conditions \ref{conduihiuh} and \ref{conduihiuhno2} it holds true that $\frac{1}{4 T_{\max} }\leq K_{\min} = \lambda_{\min}(P)$.
\end{Theorem}

We will now show that \eqref{eqkjddlkjdlj} and \eqref{eqkjddlkjdldjjiej} are satisfied with
 measurement functions that are localized as in Construction \ref{consmeasfomi} and satisfy Condition \ref{condmeasfuncloc}.

\begin{Theorem}\label{thmegdkj3hrrsuite}
Given a bounded open subset of $\Omega \subset \R^d$ with uniformly lipschitz  boundary,
let  $s\in \mathbb{N}^*$ and consider  a continuous bijection
 $\L:H^s_0(\Omega)\rightarrow H^{-s}(\Omega)$ (as in Examples \ref{egprotoalsobolevlori}, \ref{egprotoalsobolevl} and \ref{egdkj3hrr}). Furthermore, assume that the operator $\L$ is local in the sense that $\<\psi,\psi'\>=0$ if $\psi$ and $\psi'$ have disjoint supports. Let  $(\phi_{i,\alpha})_{(i,\alpha)\in \beth \times \aleph}$ be localized at level $\delta$ as in Construction \ref{consmeasfomi} and satisfy Condition \ref{condmeasfuncloc}.   Let $P$ be defined as in \eqref{eqP}. Then there exists a constant $C_1$ depending only on  $d, \delta$ and $s$ such that inequalities \eqref{eqkjddlkjdldjjiej} hold with $C_{\min}^{-1}, C_{\max}\leq C_1$. In particular, there exists a constant $C_2$ depending only on $C_1$, $C_\L$ and $C_{\L^{-1}}$, such that
   $\Cond(P)\leq C_2$. Furthermore, for each $(i,\alpha)\in \beth\times \aleph$,
  \eqref{eqhihdjkjhiudiduh} has a minimizer for $n=0$ satisfying
    $\|\psi_{i,\alpha}^0\|\leq C_2$ .
\end{Theorem}

We summarize the results obtained in the examples presented in this Section in the following theorem.
\begin{Theorem}\label{thmjff76f57}
Consider Example \ref{egprotoh10normNNN} or \ref{egprotoalsobolevl}.
Let $(\phi_{i,\alpha})_{(i,\alpha)\in \beth\times \aleph}$ be as in Example \ref{egkdejkdhdjk}, \ref{egkdejkdhkjhdjk} or \ref{egkdejkdhdjkbis} or satisfy Condition \ref{condmeasfuncloc} and Construction \ref{consmeasfomi}(with $s=1$ for Example \ref{egprotoh10normNNN}) with values $\delta$ and $h$. Let $(\psi_{i,\alpha})_{(i,\alpha)\in \beth\times \aleph}$ be the gamblets corresponding to $(\phi_{i,\alpha})_{(i,\alpha)\in \beth\times \aleph}$ localized to sub-domains of diameter $n h$ as in \eqref{eqhihdjkjhiudiduh}. Then
$\|\psi_{i,\alpha}-\psi_{i,\alpha}^n\|\leq C e^{- n/C} h^{-s}$,
  with a constant $C$ depending only on $d, \delta, s$,
    $C_{e}, C_{e, 2}$ (for Example \ref{egprotoalsobolevl}) and $\lambda_{\min}(a),\lambda_{\max}(a)$ (for Example \ref{egprotoh10normNNN}).
\end{Theorem}

\subsection{Numerical homogenization}

\begin{Corollary}\label{corjff76f57bis}
Consider Example  \ref{egprotoalsobolevl} and let
 $(\phi_{i,\alpha})_{(i,\alpha)\in \beth\times \aleph}$ be as in Example \ref{egkdejkdhdjk}, \ref{egkdejkdhkjhdjk} or \ref{egkdejkdhdjkbis} under Construction \ref{consmeasfomi} with values $\delta$ and $h$.
Let $u^n$ be the finite-element solution of \eqref{eqn:scalar} in $\Span\{\psi_{i,\alpha}^n\mid (i,\alpha)\in \beth\times \aleph\}$.
 For $n\geq C(1+\ln \frac{1}{h})$ we have
 \begin{equation}
\frac{\|u-u^n\|_{H^s_0(\Omega)}}{\|\L u\|_{L^2(\Omega)}} \leq  C h^s
\end{equation}
where the constant $C$ depends only on $d, \delta, s, C_{e}$ and $C_{e, 2}$.
\end{Corollary}

\begin{Remark}
 The localization problem \cite{ BaLip10, OwZh:2011,  MaPe:2012, GrGrSa2012, OwhadiZhangBerlyand:2014, HouLiu2015, OwhadiMultigrid:2015, gallistl2016computation, KornhuberYserentant16, KornhuberYserentant16bis,  HouZhang2017II}, in the context of the prototypical Example \eqref{egprotoa},   has been one of the major challenges of numerical homogenization  \cite{WhHo87, BaOs:1983, BaCaOs:1994, Beylkin:1995, HoWu:1997, EEngquist:2003,   OwZh:2007a,  OwZh06c, OwZh:2007b, BraWu09, BeOw:2010, DesDonOw:2012}.
 Babu{\v{s}}ka and Lipton \cite{BaLip10} proved localization in the context of generalized finite-element methods
using local eigenfunctions and harmonic extensions of the solution. Owhadi and Zhang  \cite{OwZh:2011} proved localization using the resonance-error reduction techniques of Gloria \cite{Gloria:2011}.   M{\aa}lqvist and Peterseim \cite{MaPe:2012} proved localization in the context  Variational Multiscale Methods \cite{hughes1998variational}
based on the properties of the Clement interpolation operator \cite{clement1975approximation}. Owhadi, Zhang and Berlyand obtained localized basis functions with rough polyharmonic splines \cite{OwhadiZhangBerlyand:2014}.
\cite[Sec.~3.6]{OwhadiMultigrid:2015} provides a proof of localization of gamblets based on local Poincar\'{e} inequalities.
Kornhuber and Yserentant proved localization of the basis functions of \cite{MaPe:2012} based on the  Schwarz subspace decomposition and correction method \cite{xu1992iterative, griebel1995abstract}.
 Theorem \ref{thmlkhkdfual} has  been motivated by the necessity to provide simple and natural conditions for exponential decay for a wide range of operators by identifying stability conditions of the Schwarz projection operator as conditions on the interplay between measurement functions and the image (or dual) space (these conditions are independent of the operator itself if it is invertible and continuous).
Indeed, recently techniques based on mass chasing in the physical domain via  the introduction of mollifiers and local Poincar\'{e} inequalities
provide a simple proof of exponential decay for gamblets associated with second order elliptic operators  \cite{OwhadiMultigrid:2015}. Moreover, these results have been generalized by Hou and Zhang in \cite{HouZhang2017II}
 to higher order PDEs, presented in Example \ref{egkljkhdekjd}, by requiring the
  additional assumptions  that the operator is strongly elliptic, that $h$ is sufficiently small and that the measurement functions be higher order polynomials.
As corollaries of Theorem \ref{thmlkhkdfual}, Theorems \ref{thmegdkj3hrrsuite}, \ref{thmegdkj3hrrsuitebis} and \ref{thmjff76f57} and
  Condition \ref{condmeasfuncloc} provide simple and natural conditions
  (that can be expressed as Poincar\'{e}, frame and inverse Sobolev inequalities) for the localization  of arbitrary bounded invertible linear operators mapping $H^s_0(\Omega)$ to $H^{-s}(\Omega)$ (or $L^2(\Omega)$). In particular, for Example \ref{egkljkhdekjd}, Theorem \ref{thmjff76f57} does not require the additional  assumptions  used in \cite{HouZhang2017II} to obtain localization.
   Note also that  Theorem  \ref{thmlkhkdfual} provides a natural path to exponential decay
   based on equivalence between Green's functions.
\end{Remark}

\begin{Remark}
As in  ``all roads lead to Rome'', many routes lead to basis functions that can be represented as optimal recovery splines
 \cite{bounds1959michael, micchelli1977survey}. These routes include
Variational Multiscale Methods \cite{hughes1998variational},   Polyharmonic Splines  \cite{Harder:1972, Duchon:1976,Duchon:1977,Duchon:1978} and Rough Polyharmonic Splines \cite{OwhadiZhangBerlyand:2014} (which can  be recovered as gamblets using the operator $Q^{-1}=\L^* \L$ with $\L=-\diiv(a\nabla)$ and Dirac delta functions as measurement functions),
the local orthogonal decomposition method of M{\aa}lqvist and Peterseim \cite{MaPe:2012} (using $Q^{-1}=-\diiv(a\nabla)$  and
 linear combinations of piecewise linear elements as measurement functions),  the basis functions obtained from the Bayesian Inference interpretation of Numerical Homogenization \cite{Owhadi:2014} (for arbitrary operators $\L$, using arbitrary   measurement functions and $Q^{-1}=\L^* K^{-1} \L$ where $K$ is the covariance operator of the Gaussian prior placed on source terms),
  the gamblets introduced in \cite{OwhadiMultigrid:2015} (using $Q^{-1}=-\diiv(a\nabla)$  and arbitrary measurement functions, the hierarchical analysis is done with indicator functions in \cite{OwhadiMultigrid:2015},
    the operator compression rates of gamblets are also characterized by Hou and Zhang   \cite{HouZhang2017II},
in the context of the  numerical homogenization of Example \ref{egkljkhdekjd},
through a  direct comparison between numerical homogenization convergence rates and the spectrum of the operator \eqref{eqjgkgjgkjhgjhg} using the energy norm and higher order polynomials as measurement functions),
  the reduced bases of \cite{binev2017data} (for approximating solution spaces of parametric PDEs).
 In this paper, as in \cite{OwhadiMultigrid:2015}, the route leading to optimal recovery splines has been that of Game Theory  based on the conditioning of the universal measure associated with the norm $\|\cdot\|$.
\end{Remark}

\section{Fast Gamblet Transform and Solve}\label{secfgtas}

Consider the general setting of Section \ref{subgeneralcaseNNN},
where  $\L:\B \rightarrow \B_{2}$ is a continuous linear bijection from  a Banach space $\B$, equipped with
 inner product $\langle u, v \rangle:=[Q^{-1}u, v]$ determined by a  symmetric linear bijection
$Q:\B^{*}\rightarrow \B$, and  where
 $\G:\B_2 \rightarrow \B_2^*$ is a symmetric continuous linear bijection and $\D:B_2^*\rightarrow \B^*$ is a continuous linear bijection.
 The only constraint between these operators is that
$Q^{-1}=\D \G \L$ from \eqref{gc_constr} holds,
making  Diagram
\ref{eqcase1NNN} commute.
 Our goal is, for $ g \in \B_{2}$, to efficiently compute a
variational solution \eqref{eqklajhdjfNNN} $u$  of the equation $\L u=g$.
 To that end,
consider a discretization  of Subsection \ref{subsecgamdis} where, instead of solving on the
 full space $\B$, we look for a solution on  a  discrete subspace
 $\B^\d:=\Span\{\varPsi_i \in \B,  i \in \N\}$    consisting of the span  \eqref{eqkdklrflkff}  of
a finite set of  linearly independent basis elements.
 For fixed $g \in \B_{2}$, we develop the fast gamblet transform to solve the
discretized version \eqref{eqklajhdjfdis}
\begin{equation}
\label{def_discretetete}
\<u^\d,\varPsi\>=[\D \G g,\varPsi] \text{ for } \varPsi\in \B^\d\,
\end{equation}
 of the full variational formulation
 \eqref{eqklajhdjfNNN} for  the solution $u^\d$  of the equation $\L u=g$.

These basis elements $ \varPsi_i,  i \in  \N$ will be  used as level $q$ gamblets corresponding to
measurement functions  defined  in \eqref{eqjhddhghd} and \eqref{eqkjdhkdhstpsiphitil} in Section
\ref{subsectiondiimspa}.
Let $A$ be the stiffness matrix $A=A^{(q)}$, defined
by $A_{i,j}=\<\varPsi_i,\varPsi_j\>$,  in Algorithm \ref{algdiscgambletsolvecase1g} and satisfying Condition \ref{conddiscrip3ordismatdis}.
When the matrices $\pi^{(k,k+1)}$ and $W^{(k)}$ are cellular as defined in Condition \ref{cond7fyf}, see also Remark \ref{cond7fyf},
 (or when these matrices have exponentially decaying coefficients) and when $A$ is banded (or exponentially decaying), (grid size accuracy) approximations of the outputs of  Algorithms \ref{alggambletcomutationnes}, \ref{alggamblettransf} and \ref{algdiscgambletsolvecase1gso}
 can be computed in $N \operatorname{polylog} N$ complexity based on the exponential decay of the gamblets. Here we use this exponential decay to develop fast versions of these algorithms  by truncating the stiffness matrices $B^{(k)}$ and $A^{(k)}$
 away from the diagonal and localizing (over the hierarchy) the computation of the gamblets. The natural notion of distance for these truncations will be the graph distance induced by $A$ over the hierarchy of indices as described in the following definition.

\begin{Definition}\label{Defdlkejhdkjh3}
Let $k\in \{1,\ldots,q\}$. Let $C^{(k),\d}$ be the $\I^{(k)}\times \I^{(k)}$ matrix such that $C^{(k),\d}_{i,j}=1$ if there exists $s,t\in \I^{(q)}$ with $s^{(k)}=i$,
$t^{(k)}=j$ and $A_{i,j}\not=0$ and $C^{(k),\d}_{i,j}=0$ otherwise.
Let $\db^{(k)}$ be the graph distance on $\I^{(k)}$ induced by the connectivity matrix $C^{(k),\d}$ (i.e. $\db^{(k)}:=\db^{C^{(k),\d}}$ with Definition \ref{defgraphmatdistbis}).
\end{Definition}

\subsection{The algorithm}
The Fast Gamblet transform, Algorithm \ref{fastgambletsolvecase1g} below,  will be a localized version of Algorithm
\ref{algdiscgambletsolvecase1gso}, and as such will require a localized version of
Algorithm \ref{alggambletcomutationnes}. To that end, we
will use the truncation operator  of Definition \ref{deftrunc} along with the localized inverse operation in Definition \ref{deflocinv}, below, to define localized versions of the
 components of Algorithm \ref{algdiscgambletsolvecase1gso} along with
localized versions of the components in Algorithm  \ref{alggambletcomutationnes} which Algorithm
\ref{algdiscgambletsolvecase1gso} requires.
To express this relationship with these previous algorithms, in the
 righthand  column of Algorithm \ref{fastgambletsolvecase1g},  a symbol such as
 Alg.~\ref{algdiscgambletsolvecase1gso}.x
 means that that line of Algorithm \ref{fastgambletsolvecase1g} corresponds with  line x in Algorithm  \ref{algdiscgambletsolvecase1gso}. We also use the same convention for lines in Algorithm  \ref{alggambletcomutationnes}. Note that
the notation $g$ in  Algorithm \ref{fastgambletsolvecase1g}  corresponds with
$b$ in Alg.~\ref{algdiscgambletsolvecase1gso}. The algorithm also depends on a sequence
$\rho_{k}, k =1,\ldots q-1$ of localization radii.

\begin{algorithm}[!ht]
\caption{Fast Gamblet transform/solve of \eqref{def_discretetete}.}\label{fastgambletsolvecase1g}
\begin{algorithmic}[1]
\STATE\label{step2gif} For $i,j\in \I^{(q)}$, $A_{i,j}=\<\varPsi_i,\varPsi_j\>$ \COMMENT{Alg.~\ref{alggambletcomutationnes}.\ref{step5giOR}}
\STATE\label{step3gdgpbf} For $i\in \I^{(q)}$, $\psi^{(q),\loc}_i= \varPsi_i$  \COMMENT{Level $q$ gamblets}
\STATE\label{step5gif}  $A^{(q),\loc}= A$
\STATE\label{step4gf} For $i\in \I^{(q)}$, $g^{(q),\loc}_i=[\D \G g,\psi_i^{(q),loc}]$ \COMMENT{Alg.~\ref{algdiscgambletsolvecase1gso}.\ref{step4gso}}
\FOR{$k=q$ to $2$}
\STATE\label{step7gf} $B^{(k),\loc}= W^{(k)}A^{(k),\loc}W^{(k),T}$ \COMMENT{Alg.~\ref{alggambletcomutationnes}.\ref{step7gOR}}
\STATE\label{step8gf} $w^{(k),\loc}=(B^{(k),\loc})^{-1} W^{(k)} g^{(k),\loc}$
\COMMENT{Alg.~\ref{algdiscgambletsolvecase1gso}.\ref{step8gso}}
\STATE\label{step9gf}  For $i\in \J^{(k)}$, $\chi^{(k),\loc}_i=\sum_{j \in \I^{(k)}} W_{i,j}^{(k)} \psi_j^{(k),\loc}$
\COMMENT{Alg.~\ref{alggambletcomutationnes}.\ref{step9gOR}}
\STATE\label{step10gf} $v^{(k),\loc}=\sum_{i\in \J^{(k)}}w^{(k),\loc}_i \chi^{(k),\loc}_i$
\COMMENT{Alg.~\ref{algdiscgambletsolvecase1gso}.\ref{step10gso}}
\STATE\label{step11gpif} $\bar{\pi}^{(k-1,k)}=(\pi^{(k-1,k)}\pi^{(k,k-1)})^{-1} \pi^{(k-1,k)}$
\COMMENT{Alg.~\ref{alggambletcomutationnes}.\ref{step11gpiOR}}
\STATE\label{step11gf}  $ D^{(k,k-1),\loc}=   \Inv_{\rho_{k-1}}(B^{(k),\loc}, W^{(k)} A^{(k),\loc}\bar{\pi}^{(k,k-1)})$ \COMMENT{Alg.~\ref{alggambletcomutationnes}.\ref{step11gOR}}
\STATE\label{step12gf} $R^{(k-1,k),\loc}=\bar{\pi}^{(k-1,k)}- D^{(k-1,k),\loc} W^{(k)}$
\COMMENT{Alg.~\ref{alggambletcomutationnes}.\ref{step12gOR}}
\STATE\label{step13gf} $A^{(k-1),\loc}= \Trun(R^{(k-1,k),\loc}A^{(k),\loc}R^{(k,k-1),\loc},\rho_{k-2})$ \COMMENT{ Alg.~\ref{alggambletcomutationnes}.\ref{step13gOR}}
\STATE\label{step14gf} For $i\in \I^{(k-1)}$, $\psi^{(k-1),\loc}_i=\sum_{j \in  \I^{(k)}} R_{i,j}^{(k-1,k),\loc} \psi_j^{(k),\loc}$
\COMMENT{Alg.~\ref{alggambletcomutationnes}.\ref{step14gOR}}
\STATE\label{step15gf} $g^{(k-1),\loc}=R^{(k-1,k),\loc} g^{(k),\loc}$
\COMMENT{Alg.~\ref{algdiscgambletsolvecase1gso}.\ref{step15gso}}
\ENDFOR
\STATE\label{step16gf} $ w^{(1),\loc}=(A^{(1),\loc})^{-1}g^{(1),\loc}$
\COMMENT{Alg.~\ref{algdiscgambletsolvecase1gso}.\ref{step16gso}}
\STATE\label{step17gf} $u^{(1),\loc}=\sum_{i \in \I^{(1)}} w^{(1),\loc}_i \psi^{(1),\loc}_i$
\COMMENT{Alg.~\ref{algdiscgambletsolvecase1gso}.\ref{step17gso}}
\STATE\label{step18gf} $u^{\d,\loc}=u^{(1),\loc}+v^{(2),\loc}+\cdots+v^{(q),\loc}$
\COMMENT{Alg.~\ref{algdiscgambletsolvecase1gso}.\ref{step18gso}}
\end{algorithmic}
\end{algorithm}

Now we present  Definitions \ref{deflocinv} and \ref{deftrunc}
 describing the localized linear solves and truncations performed in lines
 \ref{step11gf} and \ref{step13gf} of
 Algorithm \ref{fastgambletsolvecase1g}. They depend on the specification, performed below, of a sequence
$\rho_{k}, k =1,\ldots q-1$ of localization radii.

\begin{Definition}\label{deflocinv}
For $k\in \{2,\ldots,q\}$ we write $ D^{(k-1,k),\loc}$ the transpose of the $\J^{(k)}\times \I^{(k-1)}$ matrix
$ D^{(k,k-1),\loc}:=   \Inv_{\rho_{k-1}}(B^{(k),\loc}, W^{(k)} A^{(k),\loc}\bar{\pi}^{(k,k-1)})$  defined as follows.
Let $i\in \I^{(k-1)}$. Let $\J_i^{(k)}:=\{j\in \J^{(k)}: \db^{(k-1)}(i,j^{(k-1)})\leq \rho_{k-1}\}$. Let $z$ be the $\J_i^{(k)}$ vector defined by $z_j=(W^{(k)} A^{(k),\loc}\bar{\pi}^{(k,k-1)})_{j,i}$. Let $X$ be the  $\J_i^{(k)}\times \J_i^{(k)}$ sub matrix of $B^{(k),\loc}$ defined by $X_{j,j'}=B^{(k),\loc}_{j,j'}$. Let $y$ be the  $\J_i^{(k)}$ vector defined as the solution of
$X y=z$. We then define $ D^{(k,k-1),\loc}_{j,i}:=y_j$ for $j\in \J_i^{(k)}$ and $ D^{(k,k-1),\loc}_{j,i}:=0$ for $j\not\in \J_i^{(k)}$.
\end{Definition}

\begin{Definition}\label{deftrunc}
For $k\in \{3,\ldots,q\}$ we write $A^{(k-1),\loc}= \Trun(R^{(k-1,k),\loc}A^{(k),\loc}R^{(k,k-1),\loc},\rho_{k-2})$ the $\I^{(k-1)}\times \I^{(k-1)}$ matrix defined by $A^{(k-1),\loc}_{i,j}= (R^{(k-1,k),\loc}A^{(k),\loc}R^{(k,k-1),\loc})_{i,j}$ for $\db^{(k-2)}(i^{(k-2)},j^{(k-2)})\leq 2 \rho_{k-2}$ and $A^{(k-1),\loc}_{i,j}= 0$ for $\db^{(k-2)}(i^{(k-2)},j^{(k-2)})> 2 \rho_{k-2}$.
\end{Definition}

The following condition describes the accuracies to which the linear solves of lines \ref{step8gf}, \ref{step11gf} and \ref{step16gf} of Algorithm \ref{fastgambletsolvecase1g} must be performed. The value of the constant $C_a$ will be determined in Theorem \ref{tmdiscrete1} below. Recall the norm notation  $|y|_M:=\sqrt{y^T M y}$ for a symmetric matrix $M$.
\begin{Condition}\label{Condonrhok}
We  assume that (1)
$\rho_k\geq C_a \big((1+\frac{1}{\ln(1/H)})\ln \frac{1}{H^k}+\ln \frac{1}{\epsilon}\big)$ for $k\in \{1,\ldots,q\}$ (2)
The equation $w^{(k),\loc}=(B^{(k),\loc})^{-1} W^{(k)} g^{(k),\loc}$ on Line \ref{step8gf} of Algorithm \ref{fastgambletsolvecase1g} is solved to accuracy $|w^{(k),\loc}-w^{(k),\app}|_{B^{(k),\loc}}\leq \frac{\epsilon}{2 k^2}$ (writing $w^{(k),\app}$ the approximation of $w^{(k),\loc}$) (3) The linear systems $X y=z$ in Definition \ref{deflocinv} describing the equation $ D^{(k,k-1),\loc}=   \Inv_{\rho_{k-1}}(B^{(k),\loc}, W^{(k)} A^{(k),\loc}\bar{\pi}^{(k,k-1)})$
on Line \ref{step11gf} of Algorithm \ref{fastgambletsolvecase1g}  are solved to accuracy $|y-y^{\app}|_{X}\leq C^{-1}_a H^{3-k+kd/2} \epsilon/k^2$
(4)The equation  $ w^{(1),\loc}=(A^{(1),\loc})^{-1}g^{(1),\loc}$ on Line \ref{step16gf} of Algorithm \ref{fastgambletsolvecase1g} is solved to accuracy $|w^{(1),\loc}-w^{(1),\app}|_{A^{(1),\loc}}\leq \frac{\epsilon}{2}$.
\end{Condition}

For $k\in \{1,\ldots,q\}$ and $i\in \I^{(k)}$, let
\begin{equation}
\I^{(k,q)}_i:=\{j\in \I^{(q)}\mid \db^{(k)}(i,j^{(k)})\leq r\}\,.
\end{equation}
and
\begin{equation}
\I^{(k)}_i:=\{j\in \I^{(k)}\mid \db^{(k)}(i,j)\leq r\}\,.
\end{equation}
where $r$ is a constant independent from $k$ and $q$ with $r=1$ as a prototypical example.
For $i\in \I^{(k)}$, let $A^{i}$ be $\I^{(k,q)}_i\times \I^{(k,q)}_i$ matrix defined by
$A^{i}_{j,l}=A^{(q)}_{j,l}$ for $j,l\in \I^{(k,q)}_i$. Let $A^{i,-1}:=(A^{i})^{-1}$.
For $i\in \I^{(k)}$ and $x\in \R^{\I^{(q)}}$ write $\Pr_i^{(k,q)} x$ the projection of $x$ onto the $\I^{(k,q)}_i$ coefficients, i.e.
$\Pr_i^{(k,q)} x \in \R^{\I^{(k,q)}_i}$ is defined by $(\Pr_i^{(k,q)} x)_j=x_j$ for $j\in \I^{(k,q)}_i$.
For $i\in \I^{(k)}$ and $y\in \R^{\I^{(k)}}$ write $\Pr_i^{(k)} y$ the projection of $y$ onto the $\I^{(k)}_i$ coefficients, i.e.
$\Pr_i^{(k)} y \in \R^{\I^{(k)}_i}$ is defined by $(\Pr_i^{(k)} y)_j=y_j$ for $j\in \I^{(k)}_i$.

The following conditions are translations (to the discrete setting) of the sufficient conditions for localization obtained in the continuous setting in Section \ref{secexpdecloc}. More precisely Item \ref{itt1} of Condition \ref{condilwhiuhd} is a condition on the
number of elements in a $\db^{(k)}$ ball of radius $\rho$. Item \ref{itt2} is the direct translation of \eqref{eqljdhelkjdhkh3}, which as shown in Theorem \ref{thmlkhkjhlkh} can be expressed as a condition on the image space (that is pulled back via the continuity of the operator). Item \ref{itt3} (which corresponds to \eqref{equbarhkloc}), is a regularity condition on  measurement functions corresponding to localized versions of the inverse Sobolev inequalities appearing  in Item \ref{linconmat5dis} of Condition \ref{conddiscrip3ordismatdis} or Line \ref{lincondORhdg} of Condition \ref{cond1OR}.

\begin{Condition}\label{condilwhiuhd}
There exists  constants $C_{\loc,1}, C_{\loc,2}, C_{\loc,3}, d, d_s>0$ such that for $k\in \{1,\ldots,q\}$,
\begin{enumerate}
\item\label{itt1} $\Card(\I^{(k)})\leq C_{\loc,3} H^{-kd_s}$ and for $i\in \I^{(k)}$,
 \begin{equation}\label{eqkdhkdhkdjh}
 \Card\{j: \db^{(k)}(i,j)\leq \rho\}\leq C_{\loc,3} \rho^{d}
 \end{equation}
\item\label{itt2} For $x\in \R^{\I^{(q)}}$ such that $x\not \in \Img(\pi^{(q,k)})$,
\begin{equation}\label{eqlskdjlkd}
\frac{1}{C_{\loc,1}} \leq \frac{\sum_{i\in \I^{(k)}} \inf_{y\in \R^{\I^{(k)}}}(x-\pi^{(q,k)} y)^T \Pr_i^{(k,q)} A^{i,-1}\Pr_i^{(k,q)} (x-\pi^{(q,k)} y)}{\inf_{y\in \R^{\I^{(k)}}}(x-\pi^{(q,k)} y)^T A^{-1}(x-\pi^{(q,k)} y)}\leq C_{\loc,2}
\end{equation}
\item\label{itt3} $\frac{1}{C_{\loc,3}}H^k \leq \inf_{y\in \R^{\I^{(k)}}} \frac{\sqrt{y^T \pi^{(k,q)} \Pr_i^{(k,q)} A^{i,-1} \Pr_i^{(k,q)} \pi^{(q,k)} y}}{|\Pr_i^{(k)} y|}$ for $k\in \{1,\ldots,q\}$.
\end{enumerate}
\end{Condition}

The following theorem provides rigorous estimates on the accuracy and performance of Algorithm \ref{fastgambletsolvecase1g}.
\begin{Theorem}\label{tmdiscrete1}
Let $u^\d$ be the solution of the discrete system \eqref{def_discretetete} and let $u^{\d, \loc}$
 be the output of Algorithm \ref{fastgambletsolvecase1g}. Under Conditions \ref{conddiscrip3ordismatdis}, \ref{cond7fyf} and \ref{condilwhiuhd},
if Condition \ref{Condonrhok} is satisfied with $C_a\geq C_0$ (where $C_0$ depends only on  $C_{\d},|\aleph|, d_s, C_{\loc,1}, C_{\loc,2}$ and $C_{\loc,3}$), then
(1) $\|u^\d - u^{\d,\loc}\|\leq  \epsilon \|g\|_{2}$.
and (2) the results of Theorem \ref{thmconddisbndisbismatdis} remain true with $B^{(k)}$ replaced by $B^{(k),\loc}$ and $A^{(1)}$ replaced by $A^{(1),\loc}$. Furthermore, writing $u^{(1),\loc}$, $v^{(k),\loc}$, $u^{(1)}$ and $v^{(k)}$ the outputs of Algorithm \ref{fastgambletsolvecase1g}
and \ref{algdiscgambletsolvecase1gso} respectively, and writing $u^{(k),\loc}:=u^{(1),\loc}+\sum_{j=2}^k v^{(j),\loc}$ and $u^{(k)}:=u^{(1)}+\sum_{j=2}^k v^{(j)}$,
we have for $k\in \{1,\ldots,q-1\}$,
$\|u^{(k)} - u^{(k),\loc}\| \leq   \epsilon \|g\|_{2}$ and $\|u^{(k+1)}-u^{(k)}-(u^{(k+1),\loc}-u^{(k),\loc})\|\leq \frac{\epsilon}{2 (k+1)^2} \|g\|_{2}$.
\end{Theorem}
\begin{Remark}
In the context of Examples \ref{egkdejkdhkjhdjk}, \ref{egprotoalsobolevl} and \ref{egprotoalsobolevlori},  Condition \ref{Condonrhok} is satisfied with $H=h^s$, $d_s=d/s$, and $d$ is the dimension of the physical space.
\end{Remark}

\subsection{Complexity}

Theorem \ref{tmdiscrete1} allows us to derive the complexity of Algorithm \ref{fastgambletsolvecase1g}. To state this complexity we will assume that the dimension $d$ introduced in Item \ref{itt1} of Condition \ref{condilwhiuhd} is sharp in the sense defined by the following condition.
\begin{Condition}\label{condexpdecaybis}
We have for $k\in \{1,\ldots,q-1\}$,  $\Card(\I^{(k)})\geq C_{\loc,3}^{-1} H^{-kd_s}$ and, for $i\in \I^{(k)}$ and $\rho \leq H^{-k d_s/d}$,
 $\Card\{j: \db^{(k)}(i,j)\leq \rho\}\geq C_{\loc,3}^{-1} \rho^{d}$.
\end{Condition}
Observe that under Conditions \ref{condilwhiuhd} and \ref{condexpdecaybis} the number of degrees of freedom of the discrete system \eqref{def_discretetete} is $N\approx H^{-qd_s}$.
The following table summarizes the complexity of Algorithm \ref{fastgambletsolvecase1g}. Note that the complexity bottleneck in the first solve lies in lines \ref{step11gf} and \ref{step13gf}. Once gamblets and stiffness matrices have been computed the complexity drops and the bottleneck is line \ref{step8gf}. Line 3 of Table \ref{tabcomplexity} gives the (sub-linear) complexity of the Algorithm for subsequent solves in the numerical homogenization regime (i.e. the desired accuracy is  $H^k \gg H^q$
  and $g$ has sufficient regularity to compute the coefficients $g^{(k),\loc}$ in $\mathcal{O}(H^{-kd_s})$ complexity).

\begin{table}[!h]
\begin{center}
\begin{tabular}{ l || c | r }
  \hline			
 Compute and store $\psi_i^{(k),\loc}$, $\chi_i^{(k),\loc}$, $A^{(k),\loc}$, $B^{(k),\loc}$   & $\epsilon\leq H^{q}$ & $\epsilon\geq H^{q}$\\
 and $u^{\loc}$ s.t. $\|u - u^{\loc}\|\leq  \epsilon \|g\|_2$ & &  \\\hline
  First solve & $N \ln^{3d} \frac{1}{\epsilon} $  & $N \ln^{3d} N$ \\\hline
  Subsequence solves & $N \ln^{d+1} \frac{1}{\epsilon}$  & $N \ln^{d} N \ln \frac{1}{\epsilon} $ \\ \hline  \hline
  Subsequent solves to compute the coefficients $c_i^{(k)}$  & & \\ of $u^{(k),\hom}=\sum_{i\in \I^{(k)}} c_i^{(k)} \psi_i^{(k)}$  & &$\epsilon^{-d_s} \ln^{d+1}  \frac{1}{\epsilon}$ \\
  s.t.  $\|u - u^{(k),\hom}\| \leq  C \epsilon \|g\|_{c}$  & &  \\
   \hline
\end{tabular}
 \end{center}
  \caption{Complexity of Algorithm \ref{fastgambletsolvecase1g}.}
    \label{tabcomplexity}
\end{table}

\subsection{Equivalent conditions for localization}

As in Section \ref{secexpdecloc}, Item  \ref{itt2} of Condition \ref{condilwhiuhd} can be expressed as equivalent conditions on the operator $A$ instead of its inverse $A^{-1}$. We will present these conditions here. We also refer to Theorem \ref{thmequivcondbis} for a  discrete version of Lemma \ref{lempropjguyug6}.

\begin{Theorem}\label{thmequivcond}
Item  \ref{itt2} of Condition \ref{condilwhiuhd} is equivalent to the following inequalities being satisfied for
for $x\in \R^{\I^{(q)}}/\Img(\pi^{(q,k)})$,
\begin{equation}\label{eqldodjdoijebisloc}
\frac{1}{C_{\loc,1}}  \leq  \frac{\sum_{i\in \I^{(k)}} \sup_{z \in \Ker(\pi^{(k,q),i})} \frac{(z^T \Pr_i^{(k,q)} x)^2}{z^T  A^i z}}{\sup_{z\in \Ker(\pi^{(k,q)})} \frac{(x^T z)^2}{z^T A z}} \leq C_{\loc,2}
\end{equation}
\end{Theorem}

Write $n_{\max}$ the maximum number of neighbors of a subset $\I^{(k)}_i$ defined as
\begin{equation}
n_{\max}=\max_{k\in \{1,\ldots,q\}, i\in \I^{(k)}} \operatorname{Card}\{ j \in \I^{(k)}\mid \db^{(k)}(j,j')\leq 1 \text{ for some }j'\in \I^{(k)}_i \}\,.
\end{equation}
Let $K_{\min}$ be the largest constant such that for all $k\in \{1,\ldots,q\}$ and $z\in \Ker(\pi^{(k,q)})$, there exists a decomposition $z=\sum_{i\in \I^{(k)}} z^i$ with $z^i\in \Ker(\pi^{(k,q)})$, $z^i_j=0$ for $j\not \in \I^{(k,q)}_i$, and
\begin{equation}
K_{\min}\sum_{i\in \I^{(k)}} (z^i)^T A z^i \leq z^T A z
\end{equation}

\begin{Theorem}\label{thmdlkjdhjdh}
It holds true that Item (2) of Condition \ref{condilwhiuhd} holds with $\C_{\loc,2} \leq n_{\max}$ and
$\C_{\loc,1} = 1/K_{\min}$.
\end{Theorem}

\section{Correspondence between Numerical Analysis, Approximation Theory, and Statistical Inference}
\label{sec_correspondence}

The correspondence between Numerical Analysis and Statistical Inference is not new.  As exposed by Diaconis \cite{Diaconis:1988},
a compelling example of such a correspondence is the rediscovery of classical quadrature  rules  (such as the trapezoidal rule) by reformulating  numerical integration as a Bayesian inference problem with integrals of the Brownian Motion as priors and values of the integrand at quadrature points as data.  Although this correspondence can be traced back to  Poincar{\'e}'s course in Probability Theory  \cite{Poincare:1896} it appears to have remained ignored until the pioneering works of Sul'din \cite{sul1959wiener}, Palasti and Renyi \cite{PalastiRenyi1956}, Sard
 \cite{sard1963linear},  Kimeldorf and Wahba \cite{Kimeldorf70} (on the correspondence between Bayesian estimation and spline smoothing/interpolation, see also Van der Linde \cite{van1995splines}),
and the systematic investigation of  Larkin \cite{larkin1972gaussian} of the connections between conditioning Gaussian measures/processes and numerical approximation. As noted by Larkin \cite{larkin1972gaussian}, the application of probabilistic concepts and techniques to numerical integration/approximation ``attracted little attention among numerical analysts'', perhaps, as observed in \cite{OwhadiMultigrid:2015}, ``due to the counterintuitive nature of the process of randomizing a known function''.

However, as presented in \cite{OwhadiMultigrid:2015}, a natural framework for understanding this  process of randomization emerged
in the pioneering works of Kadane and Wasilkowsi \cite{kadane1983average}, Wo{\'z}niakowski \cite{Woniakowski1986}, Packel \cite{Packel1987}, and Traub, Wasilkowski,
and Wo\'{z}niakowski \cite{Traub1988} on Information Based Complexity (IBC)  \cite{Nemirovsky1992, Woniakowski2009}, where the performance of an algorithm operating on incomplete information can be analyzed in the usual worst case setting or the average case (randomized) setting \cite{Ritter2000, Novak2010} with respect to the missing information. Although the measure of probability (on the solution space) employed in the average case setting  may be arbitrary, as observed by Packel \cite{Packel1987}, the average case setting could be interpreted as a possible mixed strategy in an adversarial game obtained by lifting a (worst case) min max problem  to a min max problem over mixed (randomized) strategies. Such results for certain
 minmax statistical estimators have been found in
Li \cite{li1982minimaxity} and
 Sacks and Ylvisaker
\cite{sacks1978linear}. This observation, as presented in \cite{OwhadiMultigrid:2015, OwhadiScovel:2015w} and \cite{owhadi2017game},  initiates a natural connection between Numerical Analysis (and model reduction) and  Wald's Decision Theory \cite{Wald:1945}, evidently influenced by Von Neumann's Game Theory \cite{VNeumann28, VonNeumann:1944}.

The randomized setting has also been investigated, independently from the IBC purpose of computing with limited resources and partial information, from the perspective of providing statistical descriptions of numerical errors through
a natural correction between Numerical Analysis and Bayesian Inference (where confidence intervals can be derived from posterior distributions by conditioning prior distributions on values at quadrature points). Here, we refer to the pioneering works of  Diaconis \cite{Diaconis:1988}, Shaw \cite{Shaw:1988}, O'Hagan  \cite{Hagan:1991, Hagan:1992} (where new quadratures are discovered by exploiting this connection) and Skilling \cite{Skilling1992}. Moreover, Bayesian interpretation of  inverse problems
  have been developed in Evans and Stark \cite{evans2002inverse},
O'Sullivan \cite{o1986statistical},
 Tenorio \cite{tenorio2001statistical},
 Tarantola \cite{tarantola2005inverse}, Backus \cite{backus1988bayesian}
 Backus \cite{backus1996trimming},
 O'Sullivan \cite{o1986statistical}, Tenorio \cite{tenorio2001statistical}, Stuart \cite{stuart2010inverse},
Hennig \cite{Hennig2015b}, Hennig and Kiefel \cite{hennig2013quasi},
Cockayne, Oates,  Sullivan, and Girolami \cite{cockayne2016probabilistic, Cockayne2017}.

The possibilities offered by combining numerical uncertainties/errors with  model uncertainties/errors in a unified framework  are now stimulating a resurgence of  the statistical inference approach to numerical analysis, we refer in particular to recent results
by Briol,  Chkrebtii,  Campbell,  Calderhead,
Conrad, Duvenaud, Girolami,  Hennig,    Karniadakis, Maziar, Oates, Osborne, Owhadi, Paris, Sejdinovic,  S\"{a}rk\"{a},  Sch\"{a}fer, Schober, Scovel,  Sullivan,  Stuart, Venturi, Zhang and Zygalakis
\cite{ChkrebtiiCampbell2015, schober2014nips, Owhadi:2014,Hennig2015, Hennig2015b, Briol2015, Conrad2015, OwhadiMultigrid:2015,  OwhadiScovel:2015w, OwZh:2016, cockayne2016probabilistic, perdikaris2016multifidelity, raissi2017inferring, Cockayne2017,  SchaeferSullivanOwhadi17}. We also refer to \cite{schober2014nips} where it is shown that by placing a (carefully chosen) probability distribution  on the solution space of an ODE  and conditioning on quadrature points, one obtains a posterior distribution on the solution whose mean may coincide with classical numerical integrators such as Runge-Kutta methods. We also refer to  \cite{ChkrebtiiCampbell2015} where it is shown that the statistical approach is particularly well suited for chaotic dynamical systems for which deterministic worst case error bounds may provide little information.

For PDEs or integro-differential operators, \cite{Owhadi:2014} shows that numerical homogenization has a Bayesian Inference interpretation in which, filtering Gaussian noise through the inverse operator, when combined with conditioning, produces accurate finite-element basis functions for the solution space whose  deterministic  worst case errors can be bounded by standard deviation errors using the reproducing kernel structure of the covariance function of the filtered Gaussian field.

Now let us turn to the correspondence between
Approximation Theory and Statistical Inference.  Evidently, the first such connection is the Gauss-Markov
Theorem, see e.g. Kruskal \cite{kruskal1968gauss}.
 It says that, for a random vector in a finite dimensional euclidean space
 whose first moment is known to live in a subspace $S$ and whose covariance
is an unknown multiple of a known positive semidefinite matrix $V$, that the least squares estimate
of the mean
using the euclidean structure is the same as the minimum variance linear unbiased estimator if and only if
$V$ leaves $S$ invariant.  According to  Rao \cite{rao1976estimation}, ``ever since Gauss introduced the theory of least squares there
has been considerable interest in the estimation of parameters by linear functions of observations."
He also says that ``with the advent of decision theory by Wald, attempts are being made
to find estimators which may be biased but closer to the true values in some sense", and asserts
that the methods developed are all special cases of Bayes linear estimators, and that these in turn
are examples of admissible linear estimators.  Moreover, in \cite[Thm.~5.1]{rao1976estimation} Rao
demonstrates when the set  $X$  of  parameters $x$ is  finite dimensional
ellipse, and for each $x$ the random variable has mean $x$  and the variance  known up to a scalar multiple, and the objective is to estimate a rank one linear function $Sx$,
that minimax estimators  are Bayes linear estimators,   and a Bayes linear estimator obtained using a prior
with covariance matrix this same scalar multiple of  the matrix defining the constraint ellipse $X$, is minimax.
As a consequence he obtains the famous result of Kuks and Olman  \cite{kuks1972minimax,kuks1971minimax}, see also
 Speckman \cite[Lem.~3.1]{speckman1985spline} for a nice
 statement and proof of the Kuks Olman result, for minimax estimation along with the assertion of
Bunke \cite{bunke1975minimax} that the Kuks-Olman estimator is  minimax with respect to a natural matrix risk. Evidently,
  L{\"a}uter \cite{lauter1975minimax} generalizes Kuks and Olman's results
 to arbitrary $S$  in finite dimensions.

Now consider the more general linear setting of
 Donoho  \cite{donoho1994statistical}, where
for a convex subset $X$ of a  separable Hilbert space $H$, we consider
a linear model
\begin{equation}
\label{def_linearmodel}
  y=\Phi x+z
\end{equation}
where $x \in X$ and $z$ is a noise term, and we are interested in estimating
$Sx$ using a  (not necessarily linear) functions of the observation $y$,  where $S$ is a linear operator. When $z$ is random with covariance $\Sigma$, then this is a statistical estimation problem. However,
when $z$ is not random but is only known to lie in some subset $Z$,  then this is a problem in optimal recovery,
see e.g. Micchelli  and Rivlin  \cite{micchelli1977survey} and Golomb and Weinberger \cite{bounds1959michael}.    Donoho states that
``While the two problems are superficially different, there are a number of underlying similarities. Suppose that $S$,
 $\Phi$
 and $\Sigma$ are fixed, but we approach the problem two different ways: one time assuming the noise is random Gaussian, and the other time assuming the noise is chosen by an antagonist, subject to a quadratic constraint.
 In some cases both ways of stating the problem have been solved, and what happens is that while the
two solutions are different in detail, they belong to the same family -i.e.~the same family of splines,
 of kernel estimators, or of regularized least squares estimates -only the "tuning" constants" are chosen differently.
Also, a number of theoretical results in the two different fields bear a resemblance. For example,
 Micchelli \cite{micchelli1975optimal} showed in the optimal recovery model that minimax linear estimates are
 generally minimax even among all nonlinear estimates."
The  results of Donoho \cite{donoho1994statistical} amount to a rather comprehensive analysis of the connection
between these two problems.

However, here
we are interested in the case where the observations are made without noise, so that
this connection between optimal recovery and linear statistical estimation
 appears to have limited utility for us.
For example, consider the case of numerical quadrature. That is suppose that there is a
real function $f$ and observations are made of the values $f(x_{i}), i=1,..,n$ at $n$ points
$x_{i}, i=1,..,n$. Then we desire to estimate a function of $f$, such as
$f(x^{*})$ for a specified $x^{*}$ or $\int{fd\mu}$, the integral of $f$ with respect to  some measure
$\mu$.
That is,   instead of the classical linear model
\eqref{def_linearmodel}, to estimate the value
\[Sx\]
of a linear operator $S$ based on the values of linear observations
\begin{equation}
 \label{def_homlinearmodel1}
 y=\Phi x \, .
\end{equation}
In the Information-Based Complexity (IBC) approach to this problem, see e.g.~
Traub, Wasilkowski and Wo{\'z}niakowski \cite{traubh},   there is the worst-case approach, which amounts to Optimal Recovery, and the
average case approach.  Here, similar to that described by Donoho \cite{donoho1994statistical},
 there is also a fairly complete connection between the worst case approach and the
average case approach.  To describe an important example,  let $X_{1}$ and $Y$ be real linear spaces,
 $Z$s and $ X_{2}$  be
 Hilbert spaces,
and consider  bounded linear operators $S:X_{1}\rightarrow X_{2}$ and $T:X_{1}\rightarrow Z$ such that
$T$ is injective with closed range.
Errors obtained by constraining to a balanced subset $X'\subset X_{1}$ can be analyzed. However,
the application of Gaussian measures, restricted to such sets  according to
\cite[Sec.~6.5.8]{traubh}, appears to
essentially has to assume large radius to obtain approximations. Consequently, to obtain the connection
between worst case and average, it appears more appropriate to consider relative error. In the
case where the optimal solution is known to be linear then this relative error amounts to a constraint.
Indeed, the question of which IBC problems admit linear optimal solutions is an important complexity reduction.
According to Novak \cite{novak1993quadrature} ``Although adaptive methods are widely used, most
theoretical results show that adaption does not help under various conditions, see, for
example, Bakhvalov \cite{bakhvalov1959approximate}, see \cite{bakhvalov2015approximate} for the English translation,
 Gal and Micchelli \cite{gal1980optimal}, Micchelli and  Rivlin \cite{micchelli1977survey,micchelli1985lectures}
Traub and Wo´zniakowski \cite{traub1980general}, Traub, Wasilkowski and Wo´zniakowski \cite{traubh}, and
Wasilkowski and Wo´zniakowski \cite{wasilkowski1984can}." In particular, note that in the discussion below
the choice of Gaussian measure admits linear optimal solutions.

Let the relative error of an estimator $v=\phi\circ \Phi$, defined by a measurable function
$\phi:Y \rightarrow X$ which uses only the information provided by $\Phi$ to estimate $x \in X$,
be defined as
\begin{equation}
\label{def_wc}
 e^{wc}(v):= \sup_{x \in X_{1}: x \neq 0}\frac{\|Sx-v(x)\|_{X_{2}}}{\|Tx\|_{Z}}\, .
\end{equation}
On the other hand consider the average case error
\begin{equation}
\label{def_avg}
 e^{avg}(v):= \Bigl(\int_{X_{1}}{\|Sx-v(x)\|_{X_{2}}^{2}d\mu(x)}\Bigr)^{\frac{1}{2}}\,,
\end{equation}
where $\mu$ is a probability measure on $X_{1}$. In each case the objective is to minimize
this error. Moreover, according to Wasilkowski \cite{wasilkowski1983local}
this minimal error
   can be defined independent of algorithms and are the same as those defined using algorithms, see
\cite[Thm.~3.2.1, Pg.~50]{traubh}.
Let
\[\s(y):=\arg\min_{x \in X_{1}}{\{\|Tx\|_{Z}: \Phi x=y\}}\]
denote the $T$-spline determining the function
 $\s:Y \rightarrow X_{1}$, and
consider the corresponding spline algorithm
$\hat{s}:=S\s\circ \Phi$. Then, according to Traub Wasilkowski and Wo{\'z}niakowski \cite{traub1984average},
the spline algorithm $\hat{s}$ is a worst case optimal minmax solution, that is, it minimizes
\eqref{def_wc}.
Moreover, when $X_{1}=\R^{m}$ is finite dimensional and $\mu$ is a centered  Gaussian measure
on $X_{1}$ with  covariance operator $T^{*}T$,  then
the $T$-spline is also an optimal solution to the average case problem, that is, it minimizes
\eqref{def_avg}.  Remarkably, these results depend very weakly on the structure of the spaces
other than $Z$. Evidently, this is related to the importance
of the hypercircle inequality in optimal recovery, see Golomb and  Weinberger \cite{bounds1959michael},
Larkin \cite{larkin1972gaussian}.
 For example, Wasilkowski and Wo{\'z}niakowski \cite{wasilkowski1986average}
show that if  $r:X_{2} \rightarrow \R$ is convex and symmetric about the origin, then
$\hat{\s}$ also is  minimizer  of the more general
average error function
\begin{equation}
\label{def_avg2}
 e^{avg}(v):= \int_{X_{1}}{r\bigl(Sx-v(x)\bigr)d\mu(x)}\, .
\end{equation}
Wasilkowski and Wo{\'z}niakowski \cite{wasilkowski1986average}, see also
Traub, Wasilkowski and Wozniakowski \cite[Rmk.~6.5.4:1]{traubh}, demonstrate that
these average case results do not apply when $X_{1}$ is infinite dimensional. For example,
if we let $X_{1}$ be a separable Hilbert space and $\mu$ a centered Gaussian measure
with covariance operator $C$, then if we define the $C^{-1}$ spline by
\[\s(y):=\arg\min_{x \in X_{1}}{\{\langle C^{-1}x,x\rangle : x \in C(X_{1}),  \Phi x=y\}}\]
    then Wasilkowski and Wo{\'z}niakowski \cite{wasilkowski1986average} assert that
 this spline is optimal for the average case error \eqref{def_avg2}, see also
Traub, Wasilkowski and Wozniakowski \cite[Rmk.~6.5.4:1]{traubh} and Novak and Wozniakowski \cite[Thm.~4.28]{novak2008tractability}.
 However, since it is known
that the covariance operator $C$ of a Gaussian measure is trace class, and therefore compact,
then any $T$ such that $C=T^{*}T$ cannot have closed range, thus violating the assumptions of the worst case
result.  More examples of this phenomena, along with an
 analysis of both the worst-case and average-case
situation for separable Banach spaces can be found in
Lee and Wasilkowski \cite{lee1986approximation}.

In this paper, we are interested in a variation on this theme.
Let $X_{1}$ and $X_{2}$ be separable Hilbert spaces and
let \[  \L:X_{2} \rightarrow X_{1} \]
be an isomorphism whose inverse  $S:=\L^{-1}$ determines implicitly a solution
operator $S:X_{1} \rightarrow X_{2}$. We are interest in estimating
the solution $Sx$ but, in this situation, we want to do so
using observation data obtained from a linear map $\Phi:X_{2} \rightarrow Y$.
Formally,  there is an equivalence to that above, obtained by changing the roles
of $X_{1}$ and $X_{2}$   and determining a new information operator
$\Phi':=\Phi S$. However, now the transformed observation operator  $\Phi':X_{1} \rightarrow Y$ involves
the implicitly defined solution operator $S$. Moreover, here it is important to represent
a convex balanced subset of $X_{2}$  through the an injection $i:X_{0}\rightarrow X_{2}$,  thus determining
a convex balanced subset $\L^{-1}(i(X_{0}))$ as a constraint subset. Such a constraint set in general
cannot be represented through a constraint operator $T$ as above. Then, instead of wanting
to estimate the value of the solution operator, we are interested in
determining an optimal spline in the following sense.
 Push forward the norm  $\|\cdot\|_{0}$   of $X_{0}$ to an extended norm on
$ X_{2}$, which we also denote by $\|\cdot\|_{0}$,    by
\[\|u\|_{0}=
\begin{cases}
\|u'\|_{0}&  u=iu'\\
\infty &  u \notin R(i)\, .
\end{cases}
\]
and consider the relative error criteria for a function $v:=\phi\circ \Phi$ defined by
 \[e(v)=\sup_{u\in B: u \neq 0}\frac{\|u -v(u)\|}{\|\L u\|_{0}}
\]
 That is, this situation is like the estimation of the identity operator and the operator
$\L$ and the injection $i:X_{0} \rightarrow X_{2}$ determine the denominator.
Although, such constraints sets generally can not be obtained through the application of a restriction operator,
 Packel \cite{packel1986linear} asserts that when the constraint set is a balanced subset, then
it is generally known that the optimal solution may not be linear. However, he provides
simple general criteria so that it has extended-real valued linear optimal solutions.

   Smale's \cite{smale1985efficiency} discusses the computational complexity of the quadrature problem
mentioned above line \eqref{def_homlinearmodel1}
in the context of Traub and Wozniakowski's \cite{traub1984information} theory of
Information-Based Complexity (IBC) and mentions the 1972 paper of Larkin \cite{larkin1972gaussian} as an
 ``important earlier paper in this area".
 Somewhat later
Diaconis \cite{diaconis1988bayesian} introduces
 Bayesian Numerical Analysis, citing  O'Hagan \cite{OHagan1985}
 (see also  O'Hagan \cite{Hagan:1992}),
Smale \cite{smale1985efficiency}, and the IBC results of
Lee and Wasilkowsi \cite{lee1986approximation} as related approaches. We note
that Kadane and Wasilkowski \cite{kadane1983average}
address the Bayesian nature of the IBC average case approach in an unpublished report.
Moreover, in the introduction,  Larkin \cite{larkin1972gaussian}  appears to take a  different position
than Donoho with regard the numerical analysis problem and a problem of statistical estimation as follows: The numerical analyst will assume that the function lies in some special class, such as a polynomial of a certain order, or a reproducing kernel Hilbert space, and then
 as an estimator, select from this space an element which interpolates the observational data.
To make this selection, one would generally minimize some metric associated with this special class. On the other hand,
he mentions, that a statistician might put a probability measure on the space of feasible functions
and then, as an estimate, compute its conditional expectation given the observed values
$f(x_{i}), i=1,..,n$.  This is identical to the approach mentioned in the introduction
to Diaconis \cite{diaconis1988bayesian}.

Larkin, on the other hand,
 mentions that his approach is a sort of hybrid, where
 to solve such a quadrature problem we put a (prior) measure on a  Hilbert  subspace of  {\em interpolation} functions
 and then compute
the conditional expectation of the function to be estimated conditioned on the observations.
 That is, let $X$ denote a linear space corresponding to the feasible functions,  and let
$H \subset X$ denote a Hilbert subspace of interpolation functions. Then,  instead of the classical linear model
\eqref{def_linearmodel}, we consider the homogeneous linear observation model \eqref{def_homlinearmodel1}
where we observe $y$ and wish to determine $x$, and to do so we put a probability measure
on $H\subset X$ making $x$ into a random variable with values in $H\subset X$ and
 then  estimate the solution to \eqref{def_homlinearmodel1} by computing the conditional expectation of $x$
conditioned on the observation $ y=\Phi x$.

The primary motive of considering a Hilbert subspace $H\subset X$ is that probability measures  and cylinder measures
on separable
Hilbert spaces are well understood and that  Hilbert space geometry, in particular that
  associated with optimal approximation, is well understood.   Moreover, it is well known that the value of a discontinuous linear function of a state
provides little information about the state. Consequently, in the quadrature problem, to have  pointwise
valuations $f \mapsto f(x)$  be continuous for all $x$, it follows that reproducing kernel
spaces, in particular reproducing kernel Hilbert spaces,
make their appearance naturally. Moreover, if the injection $H \rightarrow X$  is continuous with a dense image,
 it follows that we have a continuous
 injection  $X' \rightarrow H'$ of the topological duals,
and using the self duality $H'=H$ of Hilbert space, we obtain
the Gelfand triple
\[  X' \subset H \subset X \, ,\]
see Gelfand and Vilenkin \cite{gelfand1964vilenkin}, which is central to both the theory of Abstract Wiener
spaces, see Gross \cite{gross1967abstract}, and the full development of the Dirac formulation of Quantum Mechanics,
see e.g.~De la Madrid \cite{de2005role}.

According to
Larkin  \cite{larkin1972gaussian},  it was Sard \cite{sard1963linear}  who introduced probabilistic concepts
into the theory of optimal linear approximation.  Moreover, he also asserts that
 it was Sul'din \cite{sul1959wiener,sul1960wiener} who began the investigation
of the use of Wiener measure in approximation theory.
In particular, Larkin  mentions that with the exception of Sul'din,
  developments of interpolation and quadrature methods based on Optimal Approximation,
initiated by Sard \cite{sard1949best},  developments in Splines as initiated by
Schoenberg \cite{schoenberg1964spline}, and the developments of Stochastic Processes, in particular
the developments of Time Series Analysis in the context of RKHSs, initiated by
Parzen, see e.g. \cite{parzen1961regression},  ``the concepts and techniques developed in these areas
have attracted little attention among numerical analysts."
Since then, other works include
Ritter \cite{ritter1990approximation}, Lee \cite{lee1986approximation2},
Lee and Wasilkowski \cite{lee1986approximation},  Wasilkowski \cite{wasilkowski1993integration},
Wasilkowski and Wo{\'z}niakowski \cite{wasilkowski1986average}, and the early work along different lines of
  Kuelbs \cite{kuelbs1969abstract},  based, curiously enough,
  on Cameron and Martin \cite{cameron1944expression}.

  Larkin's idea is essentially, to extend the idea of Sul'din from  the classic Wiener measure
to the Abstract Wiener measure as initiated by Gross  \cite{gross1967abstract}. To prepare for our treatment
of using probabilistic methods for optimal approximation on Banach spaces,
    let us briefly describe how the Abstract Wiener space formulation
of Gross is relevant here.
To begin, for a Banach space $X$, the cylinder sets are the sets  of the form
$F^{-1}(B)$ where $F:X \rightarrow E$ is a continuous linear map to a finite dimensional topological vector space
$E$
and $B$ is a Borel subset of $E$. The cylinder set algebra is the $\s$-algebra generated
by all choices of $F,E$, and $B$. According to Bogachev \cite[Thm.~A.3.7]{bogachev1998gaussian} when  $X$
is separable, this $\s$-algebra is the Borel $\s$-algebra. Now recall the Gaussian cylinder measure
$\nu$ defined on a separable Hilbert space as the Gaussian field  on $H$ such that each element
  $h\in H$, considered as a continuous linear function on $H$,
has the distribution  of a Gaussian measure, with mean zero and variance equal to $\|h\|^{2}$.
According to Gross   \cite[Pg.~33]{gross1967abstract}, the notion of a cylinder measure is equivalent
to the alternatively defined notion of weak distribution introduced by Segal \cite{segal1956tensor}.
Then a  seminorm $\|\cdot\|_{1}$ on $H$ is said to be a {\em  measurable  seminorm} if for each $\epsilon >0$ there is a
finite dimensional  projection $P_{0}$ such that, for every finite dimensional projection $P$
 orthogonal to $P_{0}$, we have
$\nu(\|Ph\|_{1} > \epsilon) < \epsilon$. A {\em measurable norm} is a measurable seminorm which is a norm.
   Let  $X:=C$ be the Banach space of continuous functions on $[0,1]$ which vanish
at $0$, equipped with the Wiener measure and consider the subspace
$C' \subset C$ of absolutely continuous functions equipped with the Hilbert norm
$\|u\|^{2}:=\int_{[0,1]}{u'(t)^{2}dt}$.
In particular, $C$ is the completion of $C'$ with respect to the $\sup$ norm on $C'$ which is much weaker that the Hilbert norm $\|u\|^{2}:=\int_{[0,1]}{u'(t)^{2}dt}$ on $C'$ and enjoys the property of being
a  measurable norm on $C'$.

  Gross' contribution is an abstraction of this relationship by
replacing $C'$ by a separable Hilbert space $H$ and the $\sup$ norm by its generalization -a measurable norm on
$H$. His  principal result \cite[Thm.~1]{gross1967abstract}  asserts that on the completion of a separable Hilbert space
$H$ with respect to a measurable norm,
the standard Gaussian cylinder measure on $H$ becomes a bonafide (countably additive) measure on its completion with respect to the measurable norm.  Conversely, Gross   \cite[Rmk.~2]{gross1967abstract} asserts that
for any separable real Banach space $X$, there is a real Hilbert space and a measurable norm
defined on it such that $B$ is the completion of $H$ in this norm.
As important examples, note that  Gross   \cite[Ex.~1]{gross1967abstract} asserts that
$\|x\|_{1}:=\langle Ax,x\rangle_{H}$ is a measurable seminorm when $A$ is symmetric, nonnegative and trace class, and when
 $A$ is injective it is a measurable norm. This indeed produces an abstraction of the classical Wiener measure,
since for  the Hilbert subspace $C' \subset C$ of the Wiener space,
the sup norm is a measurable norm on $C'$, and $C$ is the completion of $C'$ with respect to this norm.
When we are in the situation of Gross' theorem, that is, $H$ is a separable
Hilbert space equipped with the Gaussian cylinder measure, and $B$  is a Banach space
 which is the completion of $H$ with respect to a measurable norm,  producing by Gross' theorem
 a Gaussian measure $\mathcal{W}$ on $B$, then we say
that $(H,B,\mathcal{W})$ is an abstract Wiener space and that $\mathcal{W}$ is an abstract Wiener measure.
Moreover, in this case we say that
the separable Hilbert space $H$ generates $(H,B,\mathcal{W})$. We also say that $H$ generates $B$.

Now let us turn to the utility of Abstract Wiener measure in the development of infinite dimensional minimum
variance estimation problems.
 Paraphrasing Gross \cite[Intro]{gross1967abstract},
``Although $C'$ is a set of Wiener measure zero, the Euclidean structure of this Hilbert space
determines the form of the formulas developed by Cameron and Martin, and, to a large extent, also the nature of the
hypothesis of their theorems. However it only became apparent with the work of Segal
 \cite{segal1956tensor,segal1958distributions}, dealing with the normal distribution on Hilbert space, that the role
of the Hilbert space $C'$ was central, and that in so far as analysis is concerned, the role of $C$ itself
was auxiliary for many of Cameron and Martin's theorems, and some instances even unnecessary. Thus Segal's
theorem
\cite[Thm.~3]{segal1958distributions} on the transformation of the normal distribution under affine transformations,
which is formulated for an arbitrary Hilbert space $H$, extends and clarifies the corresponding theorem of Cameron
and Martin  \cite{cameron1944transformations,cameron1945transformations} when $H$ is specialized to $C'$."

To fully develop the framework of
Larkin \cite{larkin1972gaussian},
 Kuelbs, Larkin and Williamson \cite{kuelbs1972weak} develop
the  Hilbertian integration theory and its relationship to Wiener measure. It is interesting to note
that Kuelbs \cite{kuelbs1969abstract,kuelbs1971expansions} began this development a bit earlier,
where he develops a stochastic inner product
of $\langle x, h\rangle $ for arbitrary $x \in B, h \in H$ where $H$ is the generating Hilbert space
for $B$, and from that a stochastic expansion for each element of $B$. Of particular interest is the fact that for
 a Banach space with a
Schauder basis, that this basis determines, in an elementary way, an orthonormal basis of a
Hilbert space  generating the Banach space $B$, and  that the
stochastic expansion of an arbitrary element of $B$ is the same as its basis expansion.

The Cameron-Martin RKHS is well known in the theory of Wiener measure.
Let us make some remarks about the relationship between Gaussian measures on separable Banach spaces,
a Hilbert space $H$ which generates it and the resulting measure on $B$, and reproducing kernel Hilbert spaces.
Let $B$ be a separable Banach space and let $\mu$ be a Gaussian measure on $B$ with $0$ mean.
Then the covariance
\[  K(s,t):=\E_{ b \sim \mu}{\langle s, b\rangle\langle t,b\rangle},\quad s,t \in B^{*}\]
is easily seen to be a reproducing kernel defining a RKHS $H(K)$ of real valued
functions on $B^{*}$. Since $B$ is separable, it follows that it is Polish and therefore $\mu$ is a Radon measure.
Consequently, see e.g.~Bogachev \cite[Thm.~3.2.3]{bogachev1998gaussian},
one obtains that  $H(K) \subset B$.
Now let $(H,B,\mathcal{W})$ be an abstract Wiener space. Then,
  Kallianpur \cite[Cor.~1]{kallianpur1971abstract} asserts that, for the RKHS $H(\Gamma)$  corresponding
to the covariance
kernel
\[ \Gamma(s,t):=\E_{ b \sim \mathcal{W}}{\langle s, b\rangle\langle t,b\rangle},\quad s,t \in B^{*}\, ,\]
we have $H=H(\Gamma)$. See also  Bogachev \cite[Thm.~3.9.4]{bogachev1998gaussian}. In particular, if $(H_{1},B,\mathcal{W})$ and $(H_{2},B,\mathcal{W})$  are two abstract Wiener
spaces with identical measures, then $H_{1}=H_{2}$. $H(\Gamma)$ is known as the Cameron-Martin
(reproducing kernel Hilbert) space. Conversely, Bogachev\cite[Thm.~3.9.6]{bogachev1998gaussian}
asserts that if $\mu$ is a centered Gaussian measure on a separable Banach space $B$ with
covariance $\Gamma$,  such that $H(\Gamma)$ is dense in $B$,  then
$(H(\Gamma),B,\mu)$ is an abstract Wiener space. More generally, by using the Banach-Mazur
Theorem, see e.g.~Albiac and Kalton \cite[Thm.~1.4.3]{albiac2006topics}, which asserts
that every separable real Banach space is isometrically isomorphic to a closed subspace
of $C[0,1]$, the continuous functions on the unit interval with the sup norm, Kallianpur
provides a more general analysis. In particular,  Kallianpur \cite[Thm.~7]{kallianpur1971abstract}
asserts that for an arbitrary Gaussian measure $\mu$ on a separable Banach space with covariance
kernel $\Gamma$, that
 \[\overline{H(\Gamma)}=\support{(\mu)}\, ,\]
where $\overline{H(\Gamma)}$ is the closure
of $H(\Gamma)$ in the topology of $B$ and $\support$ is the support of the measure $\mu$, that is the unique
closed set $F$ of full measure such that for every open set $G$ such that $F\cap G \neq \emptyset$, we have
$\mu(F\cap G) > 0$. It is interesting to note that his proof uses a special structure of the Banach-Mazur theorem,
 namely that the isometry $\Psi:B \rightarrow C_{0}$ with the closed subspace of $C[0,1]$
is obtained  by
\[ b  \mapsto b(t):=\langle b, f_{t}\rangle,  t \in [0,1],\quad b \in B\, ,\]
where the map $t \mapsto f_{t} \in B^{*}$ is a continuous map to the unit ball of $B^{*}$.

Let us mention here that the well known fact that the Cameron-Martin space has Wiener measure
$0$ has a strong parallel with  $0-1$ laws regarding the membership of stochastic paths in RKHSs.
For a comprehensive treatment, see Lukic and Beder \cite{lukic2001stochastic}, who uses
Kallianpur \cite{kallianpur1970zero} to
 fully develop results of Driscoll \cite{driscoll1975signal}, based on
Driscoll \cite{driscoll1973reproducing}.

\section{Additional  properties of the gamblet transform}\label{sec8}
This section presents supplementary material to that of
 Section \ref{sec1or} and assumes its notations.

\subsection{Dual variational formulation and dual gamblets}\label{subsecdual}
 Let $\bar{\psi}_1,\ldots,\bar{\psi}_m$ be linearly independent elements of $\B$ and  consider
\begin{equation}
\bar{A}_{i,j}:=[Q^{-1}\bar{\psi}_i, \bar{\psi}_j].
\end{equation}
The linear independence of the $\bar{\psi}_i$ implies that $\bar{A}$ is invertible and we write its inverse as
 $\bar{A}^{-1}$.  For $i=1,\ldots,m$, let $\bar{\phi}_1,\ldots,\bar{\phi}_m \in \B^*$ be defined as
\begin{equation}\label{eqhgudgudphi}
\bar{\phi}_i:=\sum_{j=1}^m \bar{A}^{-1}_{i,j}Q^{-1} \bar{\psi}_j
\end{equation}
and observe that $\{\bar{\phi}_i,\bar{\psi}_{j}\}$ form a biorthogonal system, in that
 $[\bar{\phi}_i,\bar{\psi}_j]=\delta_{i,j}$.
The following theorems shows that the elements $\bar{\phi}_i$ have optimal variational properties in $\|\cdot\|_*$-norm $\|\cdot\|_*:=\sqrt{[\cdot, Q\cdot]}$.

\begin{Theorem}\label{thmsudgdygdgyphi}
Let $(\bar{\phi}_i)_{i=1,\ldots,m}$ be defined as in \eqref{eqhgudgudphi}. Then it holds true that
for $c\in \R^m$, $\sum_{i=1}^m c_i \bar{\phi}_i$ is the unique minimizer of
\begin{equation}\label{eqhihdiudiduhphi}
\begin{cases}
\text{Minimize }\|\phi\|_*\\
\text{Subject to }\phi\in \B^*\text{ and }[\phi,\bar{\psi}_j]=c_j\text{ for }j=1,\ldots,m\, .
\end{cases}
\end{equation}
Furthermore, consider  arbitrary  $\phi_i \in \B, i=1,\ldots, m$, and define
 $\Theta$ by $\Theta_{ij}:=[\phi_{i},Q\phi_{j}]$ as  in \eqref{eqdefthet} and
  $\psi_{i}:= \sum_{j=1}^{m}{\Theta^{-1}_{ij}Q\phi_{j}}$ as in  \eqref{eqhgudgud}. Then if
$\bar{\psi}_i=\psi_i, i=1, \ldots, m$, it follows that $\bar{\phi}_i=\phi_i, i=1,\ldots, m$ and $\bar{A}=\Theta^{-1}$.
\end{Theorem}

\begin{Theorem}\label{thmoptgalerkinphi}
 Define $\phi^\two:\B^* \rightarrow \B^*$ by
\begin{equation}\label{eqbdudbuysbdphi}
\phi^\two(\phi)=\sum_{i=1}^m [\phi,\bar{\psi}_i] \bar{\phi}_i ,\, \phi \in \B^*\, .
\end{equation}
The, for $\phi\in \B^*$, it holds true that
\begin{equation}
\|\phi-\phi^\two(\phi)\|_*=\inf_{\phi' \in \operatorname{span}\{Q^{-1} \bar{\psi}_1,\ldots,Q^{-1} \bar{\psi}_m\}} \|\phi-\phi'\|_*\, .
\end{equation}
\end{Theorem}

\subsection{Dual measurement functions of the hierarchy of orthogonalized gamblets}
Let $k\in \{2,\ldots,q\}$ let $N^{(k)}:=A^{(k)} W^{(k),T} B^{(k),-1}$ be defined as in \eqref{eqjdhiudhiue} and
for $i\in \J^{(k)}$ write
\begin{equation}\label{eqhdgudgu}
\phi^{(k),\chi}_i:=\sum_{j\in \I^{(k)}} N^{(k),T}_{i,j}\phi^{(k)}_j\,.
\end{equation}
The following propositions shows that $(\phi_i^{(1)})_{i\in \I^{(1)}}$ and $(\phi^{(k),\chi}_j)_{2\leq k \leq q, j\in \J^{(k)}}$ form the dual gamblets of $(\psi_i^{(1)})_{i\in \I^{(1)}}$ and $(\chi^{(k)}_j)_{2\leq k \leq q, j\in \J^{(k)}}$ respectively.
\begin{Proposition}\label{propdkfjffh}
It holds true that for $k\in \{2,\ldots,q\}$, $(i,j)\in \J^{(k)}\times \J^{(k)}$ and $l\in \I^{(k-1)}$
\begin{equation}\label{eqjgjhgyjh}
[\phi^{(k),\chi}_i,\chi^{(k)}_j]=\delta_{i,j}\text{ and } [\phi^{(k),\chi}_i,\psi^{(k-1)}_l]=0\,.
\end{equation}
For $2\leq k,k'\leq q$, $k\not=k'$ and $(i,j)\in \J^{(k)}\times \J^{(k')}$, $[\phi^{(k),\chi}_i,\chi^{(k')}_j]=0$.
For $i\in \I^{(k)}$ and $j\in \J^{(k)}$, $N^{(k)}_{i,j}=[\phi^{(k),\chi}_j,\psi^{(k)}_i]$.
Furthermore, under the Conditions \ref{cond1OR}, there exists a constant $C\geq 1$ depending only on $C_{\Phi}$ such that for $z\in \R^{\J^{(k)}}$,
\begin{equation}\label{eqidyiuyihs}
\frac{ H^{2}}{ C} |z|^2  \leq \|\sum_{j\in \J^{(k)}} z_j  \phi_j^{(k),\chi}\|_0^2 \leq C |z|^2\,.
\end{equation}
\end{Proposition}

For $k\in \{2,\ldots,q\}$ write
\begin{equation}
\bar{W}^{(k)}= (W^{(k)}W^{(k),T})^{-1} W^{(k)}\, .
\end{equation}
\begin{Proposition}\label{propdkfjffother}
It holds true that for $k\in \{2,\ldots,q\}$ and $c \in \R^{\J^{(k)}}$, $\sum_{i\in \J^{(k)}}c_i \phi^{(k),\chi}_i$ is the unique minimizer of
\begin{equation}\label{eqhihdiudiduhphichior}
\begin{cases}
\text{Minimize }\|\phi\|_*\\
\text{Subject to }\phi\in \B^*\text{ and }[\phi,\chi^{(k)}_j]=c_j\text{ for }j\in \J^{(k)}\,.

\end{cases}
\end{equation}
Furthermore for $i\in \J^{(k)}$,
\begin{equation}\label{eqmhjgjhgy}
\phi_i^{(k),\chi}=\sum_{j\in \I^{(k)}} \bar{W}^{(k)}_{i,j}\phi_j^{(k)} +\sum_{l\in \I^{(k-1)}} (N^{(k),T}\bar{\pi}^{(k,k-1)})_{i,l}\phi_l^{(k-1)}
\end{equation}
which implies
\begin{equation}\label{eqjskdkddj}
N^{(k),T}=\bar{W}^{(k)}+N^{(k),T}\bar{\pi}^{(k,k-1)}\pi^{(k-1,k)}
\end{equation}
Moreover $\tilde{\phi}^{(k-1)}_i:=-\sum_{l\in \I^{(k-1)}} (N^{(k),T}\bar{\pi}^{(k,k-1)})_{i,l}\phi_l^{(k-1)}$ is the unique minimizer of
\begin{equation}\label{eqhihdiudiduhphichiorbis}
\begin{cases}
\text{Minimize }\|\sum_{j\in \I^{(k)}} \bar{W}^{(k)}_{i,j}\phi_j^{(k)}-\phi\|_*\\
\text{Subject to }\phi\in \Phi^{(k-1)}
\end{cases}
\end{equation}
Therefore, $\bar{\pi}^{(k-1,k)} N^{(k)}=-A^{(k-1)} \pi^{(k-1,k)}\Theta^{(k)}\bar{W}^{(k),T}$, i.e.,
\begin{equation}\label{eqkjddhk}
\bar{\pi}^{(k-1,k)}A^{(k)} W^{(k),T} = -A^{(k-1)} \pi^{(k-1,k)}\Theta^{(k)}\bar{W}^{(k),T} B^{(k)}
\end{equation}
\end{Proposition}

The following theorem is a direct consequence of \eqref{eqjskdkddj}.
\begin{Theorem}
It holds true that
\begin{equation}\label{eqjskdkjkfjhddj}
N^{(k),T}(I^{(k)}- \bar{\pi}^{(k,k-1)}\pi^{(k-1,k)})=\bar{W}^{(k)}
\end{equation}
\end{Theorem}
Let $U^{(k)}$ be the $\J^{(k)}\times \J^{(k)}$ matrix defined by
\begin{equation}
U^{(k)}= W^{(k)} \Theta^{(k)} W^{(k),T}
\end{equation}
The following proposition is a direct consequence of \eqref{eqkhiduhdf7d} and Lemma \ref{lemddjoj3ir} in
  Section \ref{sec6}, which  contains the proofs of the results of  Sections  \ref{sec1or} and \ref{secsolvinvpb}.
\begin{Proposition}
Under Conditions \ref{cond1OR} it holds true that
\begin{equation}
\frac{ H^{2k}}{C} \leq \lambda_{\min}(U^{(k)}) \text{ and }
\lambda_{\max}(U^{(k)})\leq C H^{2(k-1)},
\end{equation}
for some constant $C$ depending only on $C_{\Phi}$.
\end{Proposition}

\subsection{Minimum angle between gamblets}
The following Proposition provides a lower bound on the angle between linear spaces spanned by distinct gamblets.
\begin{Proposition}\label{lemddkjhed}
Let $k\in \{2,\ldots,q\}$. Let $\S_1$ and $\S_2$ be disjoint non empty subsets of $\J^{(k)}$. Under Conditions \ref{cond1OR} it holds true that for $\chi_1 \in \Span\{\chi_i^{(k)}|i\in \S_1\}$ and $\chi_2 \in\Span\{\chi_i^{(k)}|i\in \S_2\}$,
\begin{equation}
\big|\<\chi_1,\chi_2\>\big|\leq \delta \|\chi_1\| \|\chi_2\|
\end{equation}
with $\delta=1-\frac{H^2 }{ C } $ for some constant $C\geq 1$ depending only on $C_{\Phi}$.
\end{Proposition}

\subsection{Dirac deltas as measurement functions, averaging vs sub-sampling, higher order PDEs and relaxed conditions}

Consider Examples \ref{egprotoalsobolevl} or \ref{egkljkhdekjd}  with $s>d/2$.
Although, in those examples, measurement functions can be designed as linear combinations of Dirac delta functions, the space $\B_0$ can no longer be defined as $L^2(\Omega)$ since Dirac deltas  are elements of $\B^*$ but not $L^2(\Omega)$.
However, although Theorem \ref{corunbcnOR} and Condition \ref{cond1OR} assume  that measurement functions  belong to  $\B_0$, this requirement is not necessary and Condition \ref{cond1OR} can be replaced by the following conditions.

\begin{Condition}\label{cond1ORd}
There exists some constants $C_{\Phi}\geq 1$ and $H\in (0,1)$  such that
\begin{enumerate}
\item\label{itsubst1} $\frac{1}{C_\Phi} H^k \leq  \frac{\|\phi \|_*}{|x|}$ for $k\in \{1,\ldots,q\}$, $x\in \R^{\I^{(k)}}$ and
$\phi=\sum_{i\in \I^{(k)}} x_i \phi_i^{(k)}$.
and $x\in \R^{\I^{(q)}}$.
\item $\frac{\| \phi\|_*}{|x|}\leq C_{\Phi}$ for $\phi=\sum_{i\in \I^{(1)}} x_i \phi_i^{(1)}$
and $x\in \R^{\I^{(1)}}$.
\item $\|\pi^{(k-1,k)}\|_2 \leq C_{\Phi}$ for $k\in \{2,\ldots,q\}$.
\item $\frac{1}{C_{\Phi}} J^{(k)}\leq W^{(k)}W^{(k),T}\leq C_{\Phi} J^{(k)}$ for $k\in \{2,\ldots,q\}$.
\end{enumerate}
\end{Condition}

\begin{Condition}\label{cond1ORdvar1}
The constants $C_{\Phi}$ and $H$ in Condition \ref{cond1ORd} also satisfy
\begin{enumerate}
\item \label{itdftf} $\inf_{\phi'\in \Phi^{(k-1)}}\frac{\| \phi-\phi'\|_*}{|x|}\leq C_{\Phi} H^{k-1}$ for $\phi=\sum_{i\in \I^{(k)}} x_i \phi_i^{(k)}$, $k\in \{2,\ldots,q\}$,
and $x\in \R^{\I^{(k)}}$.
\item\label{itsjhdhjdgdj}
 $\frac{\| \phi\|_*}{|x|}\leq C_{\Phi} H^{k-1}$ for $\phi=\sum_{i\in \I^{(k)}} x_i \phi_i^{(k)}$, $k\in \{2,\ldots,q\}$,
and $x\in \Ker(\pi^{(k-1,k)})$.
\end{enumerate}
\end{Condition}

The proof of the following theorem is similar to that of Theorem \ref{corunbcnOR}.
\begin{Theorem}\label{corunbcnORd}
Under Conditions \ref{cond1ORd} and \ref{cond1ORdvar1} the results of Theorem \ref{corunbcnOR} (except  \eqref{corunbcnORfe}) hold true.
\end{Theorem}

The core element in the proof of complexity and accuracy of the Fast Gamblet Transform is the exponential decay of gamblets, which has been established for Example \ref{egprotoalsobolevl} (with $s>d/2$) in Theorems \ref{thmegdkj3hrrsuite}, \ref{thmegdkj3hrrsuitebis} and
Example \ref{egkdejkdhdjkbis}.

\begin{Remark}
To attain Theorem \ref{corunbcnORd} including  a bound on the  numerical homogenization approximation error of the type \eqref{corunbcnORfe}, one may need, as in \cite{OwhadiZhangBerlyand:2014}, to quantify the approximation in a weaker norm than the $\|\cdot\|$-norm.
\end{Remark}

 Two natural strategies could be employed in the design of a hierarchy measurement functions with Dirac delta functions.
 One is based on
sub-sampling, as presented in Figure \ref{figsubset} and Illustration \ref{nekdjdhhus2}, and
 the other is based on averaging, as presented in Figure \ref{figaverage}
 and Illustration \ref{nekdjdhhus2aver}.
  Although an averaging strategy may satisfy
 Conditions \ref{cond1ORd} and \ref{cond1ORdvar1}, we will demonstrate that  a sub-sampling strategy
 cannot  satisfy Item \ref{itsjhdhjdgdj} of Condition \ref{cond1ORdvar1}.

 \begin{figure}[h!]
	\begin{center}
			\includegraphics[width=0.9\textwidth]{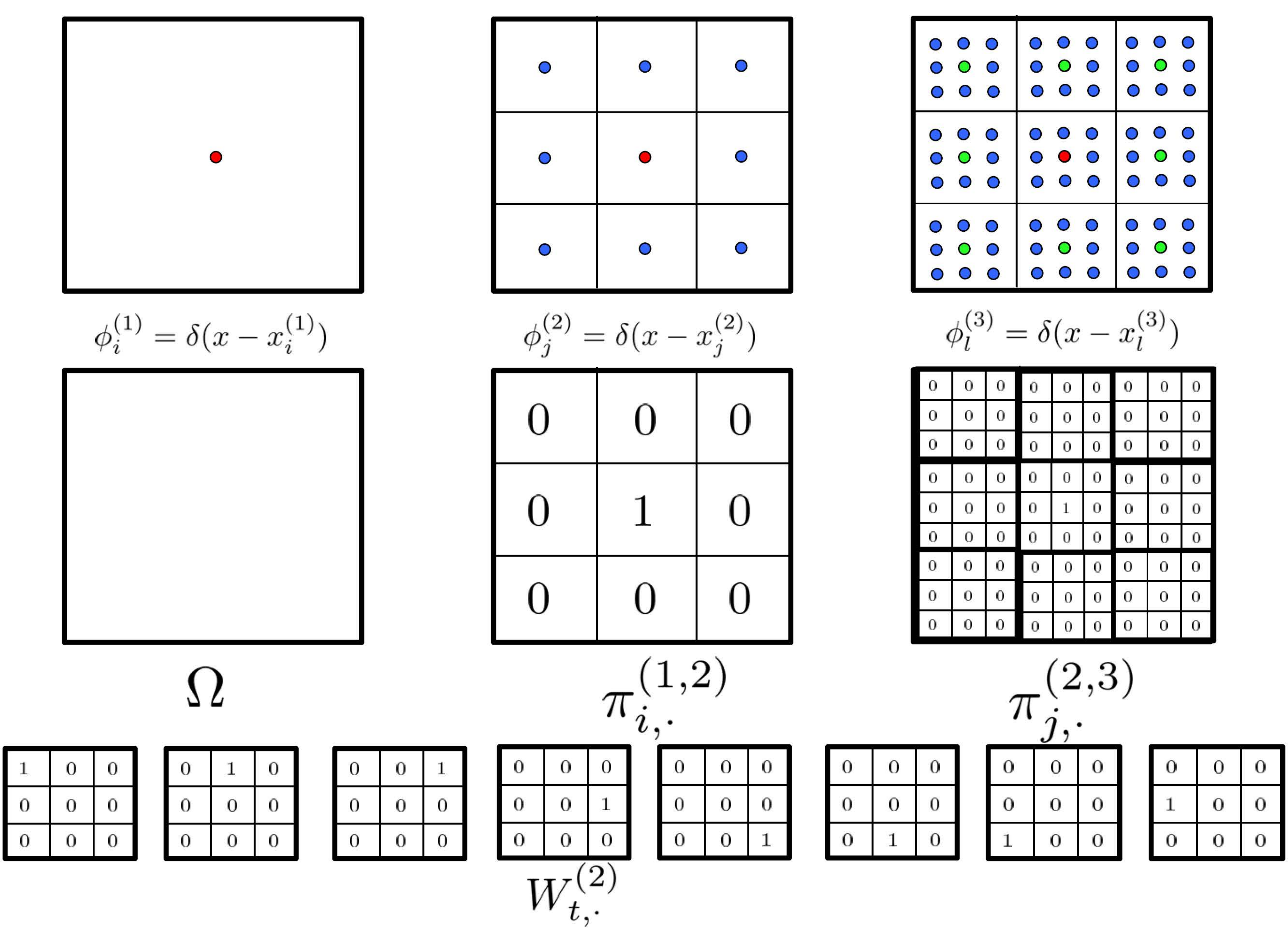}
		\caption{See Illustration \ref{nekdjdhhus2}.  $\Omega=(0,1)^2$. $\phi_i^{(k)}$ are Dirac delta functions at centers of $3^{1-k}\times 3^{1-k}$ squares forming a uniform partition of $\Omega$. The top row shows the support of $\phi_i^{(1)}, \phi_j^{(2)}$ and $\phi_l^{(3)}$.  The middle row shows the  entries of $\pi^{(1,2)}_{i,\cdot}$ and $\pi^{(2,3)}_{j,\cdot}$. The bottom row shows the entries of $W^{(k)}_{t,\cdot}$.
}\label{figsubset}
	\end{center}
\end{figure}

However, as shown in \cite{SchaeferSullivanOwhadi17}, Condition \ref{cond1ORdvar1} in Theorem \ref{corunbcnORd} can be relaxed to the following (weaker) condition (at the cost of a slight increase in the uniform on $\Cond(B^{(k)})$) that can be satisfied by a sub-sampling strategy.
\begin{Condition}\label{cond1ORdmod}
There exists  $d_{\Phi}\geq 0$ such that the constants $C_{\Phi}$ and $H$ in Condition \ref{cond1ORd} also satisfy
 \begin{equation}
 \inf_{y\in \R^{\I^{(k-1)}}, |y|\leq C_{\Phi} |x| H^{-d_{\Phi}/2}}\frac{\| \phi-\sum_{i\in \I^{(k-1)}}y_i \phi_i^{(k-1)}\|_*}{|x|}\leq C_{\Phi} H^{k-1}
 \end{equation}
  for $\phi=\sum_{i\in \I^{(k)}} x_i \phi_i^{(k)}$, $k\in \{2,\ldots,q\}$,
and $x\in \R^{\I^{(k)}}$.
\end{Condition}
The following theorem is proven in \cite{SchaeferSullivanOwhadi17}.
\begin{Theorem}\label{corunbcnORdmod}
Under Conditions \ref{cond1ORd} and \ref{cond1ORdmod} it holds true that there exists a constant $C$ depending only on $C_{\Phi}$ such that
 $C^{-1} I^{(1)} \leq A^{(1)}\leq C H^{-2} I^{(1)}$,  $\operatorname{Cond}(A^{(1)})\leq  C H^{-2}$,
and $C^{-1} H^{-2(k-1)+d_{\Phi}} J^{(k)} \leq B^{(k)}\leq C H^{-2k} J^{(k)}$
 for $k\in \{2,\ldots,q\}$.  In particular,
   \begin{equation}
   \operatorname{Cond}(B^{(k)})\leq  C H^{-2-d_{\Phi}}\,.
   \end{equation}
\end{Theorem}

When $\B=H^s_0(\Omega)$ ($s\geq 1$),  \cite{SchaeferSullivanOwhadi17} shows that Conditions \ref{cond1ORd} and \ref{cond1ORdmod}
are satisfied with $d_{\Phi}=d$ with weighted indicator functions or Dirac delta functions as measurement functions (i.e. for each $k$, the $\phi_i^{(k)}$ are as in Example \ref{egkdejkdhdjk} or Example \ref{egkdejkdhdjkbis}).

\begin{figure}[h!]
	\begin{center}
			\includegraphics[width=0.9\textwidth]{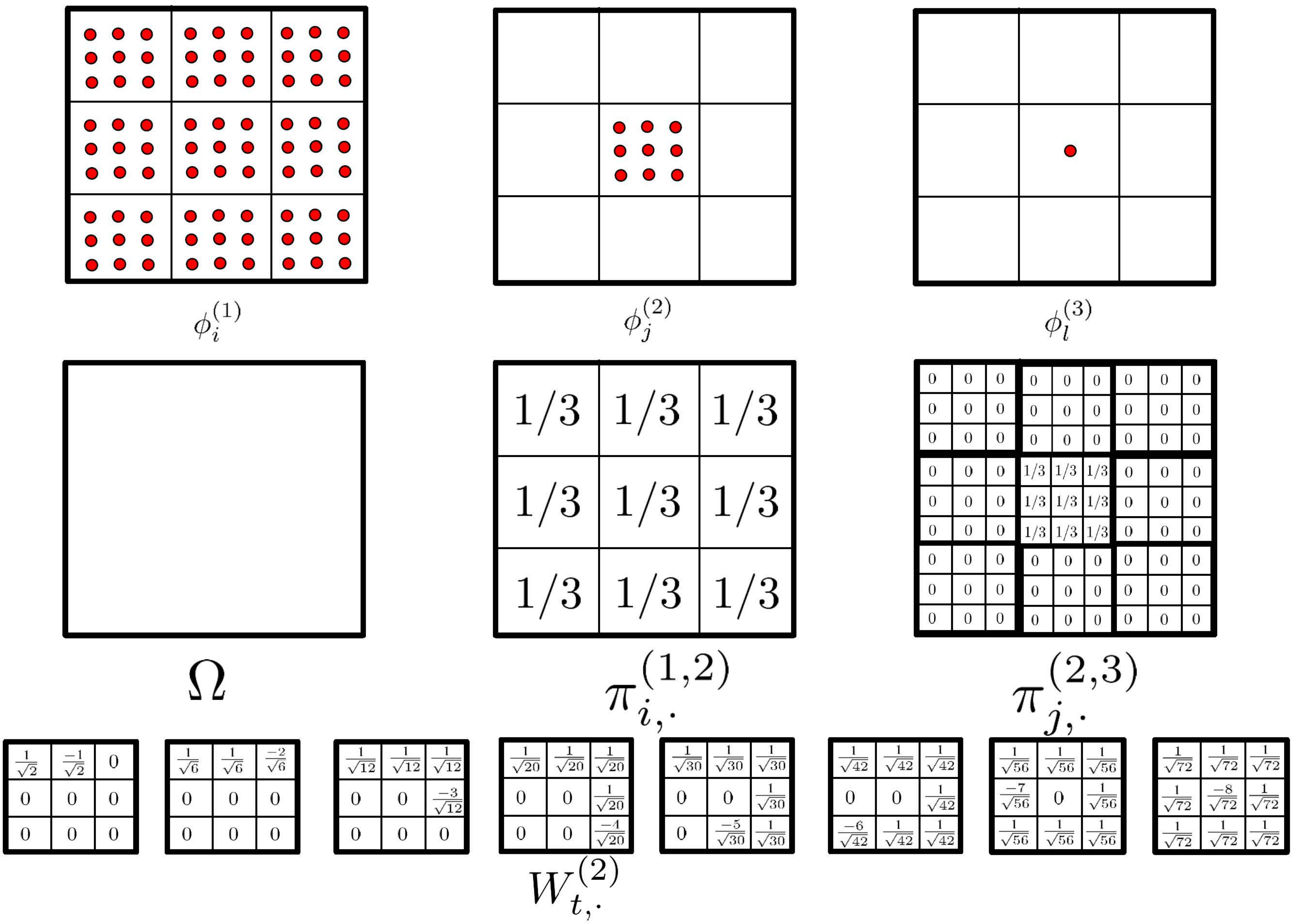}
		\caption{See Illustration \ref{nekdjdhhus2aver}. $\Omega=(0,1)^2$. $\phi_i^{(k)}$ are sums of Dirac delta functions contained in $3^{1-k}\times 3^{1-k}$ squares forming a uniform partition of $\Omega$. The top row shows the support of $\phi_i^{(1)}, \phi_j^{(2)}$ and $\phi_l^{(3)}$.  The middle row shows the  entries of $\pi^{(1,2)}_{i,\cdot}$ and $\pi^{(2,3)}_{j,\cdot}$. The bottom row shows the entries of $W^{(k)}_{t,\cdot}$.
}\label{figaverage}
	\end{center}
\end{figure}

\begin{NE}\label{nekdjdhhus2}
Consider  Figure \ref{figsubset}. For $k\in \N^*$, let $\Omega_{k}$ be a regular grid partition of $\Omega=(0,1)^2$ into $3^{1-k}\times 3^{1-k}$ squares $\tau_i^{(k)}$ with centers $x_i^{(k)}$ and let $\phi_i^{(k)}=\delta(x-x_i^{(k)})$ where  $\delta(x-x_i^{(k)})$ is the Dirac delta function at $x_i^{(k)}$. Level $k$ masses of Dirac are obtained by taking a subset of the masses of Dirac at level $k+1$.  The  matrices $\pi^{(k,k+1)}$ and $W^{(k)}$ are  cellular and satisfy $\pi^{(k,k+1)}\pi^{(k+1,k)}=I^{(k)}$ and $W^{(k)}W^{(k),T}=J^{(k)}$. Although Condition \ref{cond1ORd} is satisfied, Item \ref{itsjhdhjdgdj} of Condition \ref{cond1ORdvar1} cannot be satisfied by sub-sampling. However, as shown in \cite{SchaeferSullivanOwhadi17}, the relaxed Condition \ref{cond1ORdmod} is satisfied by sub-sampling and the proposed choice of measurement functions $\phi_i^{(k)}$.
\end{NE}

\begin{NE}\label{nekdjdhhus2aver}
Consider  Figure \ref{figaverage}. For $k\in \N^*$, let $\Omega_{k}$ be a regular grid partition of $\Omega=(0,1)^2$ into $3^{1-k}\times 3^{1-k}$ squares $\tau_i^{(k)}$ with centers $x_i^{(k)}$ and let $\phi_i^{(k)}=\sum_{x_i^{(q)} \in \tau_i^{(k)} }3^{k-q}\delta(x-x_i^{(q)})$ where  $\delta(x-x_i^{(q)})$ is the Dirac delta function at $x_i^{(q)}$.  The  matrices $\pi^{(k,k+1)}$ and $W^{(k)}$ are  cellular and satisfy $\pi^{(k,k+1)}\pi^{(k+1,k)}=I^{(k)}$ and $W^{(k)}W^{(k),T}=J^{(k)}$. Although Condition \ref{cond1ORd} is satisfied, Item \ref{itsjhdhjdgdj} of Condition \ref{cond1ORdvar1} not satisfied for $s\geq 2$ for the proposed choice of matrices $\pi^{(k,k+1)}$ and resulting measurement functions $\phi_i^{(k)}$. We refer to Remark \ref{rmkreason} for an elucidation of how to choose  the matrices $\pi^{(k,k+1)}$ for the satisfaction of Condition \ref{cond1ORdvar1}
 when $s\geq 2$ and the $\phi_i^{(k)}$ are weighted sums of masses of Diracs or weighted sums of indicator functions of a partition of $\Omega$.
We also refer to \cite{SchaeferSullivanOwhadi17} for a demonstration that the relaxed Condition \ref{cond1ORdmod} is satisfied for the proposed choice of measurement functions $\phi_i^{(k)}$.
\end{NE}

\begin{Remark}\label{rmkreason}
Consider the measurement functions of Illustration \ref{nekdjdhhu}. Note that Item \ref{itsjhdhjdgdj} of
Condition \ref{cond1ORdvar1} is equivalent to item \ref{itcor3} of Condition \ref{cond1OR}
and Item \ref{itdftf} of
Condition \ref{cond1ORdvar1} is equivalent to item \ref{itcor2} of Condition \ref{cond1OR}.
 Although Item \ref{itcor2} of Condition \ref{cond1OR} is satisfied when the $\phi_i^{(k)}$ are weighted indicator functions of the subsets $\tau_i^{(k)}$ (see  Example \ref{egkdejkdhdjk} and Theorem \ref{thmpropegkdejkdhdjkbis}), Item \ref{itcor3} requires
 choosing the matrices $\pi^{(k,k+1)}$ so that $\phi \in \Phi^{(k),\chi}$ implies that $\phi$ is $L^2$ orthogonal to polynomials of degree at most $s-1$ on local patches of size $h^{k-1}$ (it is easy to see, as in the proof of Proposition \ref{propkajhdlkjd}, that this requirement is sufficient). To understand this requirement consider the measurement functions  of Illustration \ref{nekdjdhhu}.
Item \ref{itcor3} of Condition \ref{cond1OR} requires that $\|\phi\|_{H^{-s}(\Omega)}\leq C h^{(k-1)s}\|\phi\|_{L^2(\Omega)}$ for $\phi \in\Phi^{(k),\chi}$. However
$\| \phi\|_{H^{-s}(\Omega)}=\sup_{v\in H^s_0(\Omega)}\big[2 \int_{\Omega} v\phi-\|v\|^2_{H^s_0(\Omega)}\big]$ implies that
if there exists a polynomial $p$ of degree $s'$ (in $H^s_0(\Omega)$) and $\phi \in \Phi^{(k),\chi}$
such that $\int_{\Omega} p \phi \not=0$ then, for $s>s'$
$\| \phi\|_{H^{-s}(\Omega)}$ is bounded from below by  $\int_{\Omega} p \phi$.
Since $C h^{(k-1)s}\|\phi\|_{L^2(\Omega)}$ converges towards $0$ as $s\rightarrow \infty$ and $\int_{\Omega} p \phi$ is independent from $s$, Item \ref{itcor3} of Condition \ref{cond1OR} cannot hold for all $s$.
\end{Remark}

\subsection{Conditions for Algorithm  \ref{algdiscgambletsolvecase1g} on discretized Banach spaces}
In subsection \ref{subseckedjh}, Theorem  \ref{thmconddisbndisbismatdis} regarding
Algorithm \ref{algdiscgambletsolvecase1g} is proven by identifying the discretized space
$\B^{\d}:=\Span\{\varPsi_i\mid i\in \N\}$ with
  $\R^N$, and expressing conditions  on $\R^N$. One may also wonder what
conditions directly  on the subspace $\B^\d\subset \B$ would also lead to the same results as Theorem  \ref{thmconddisbndisbismatdis}. The  subtle distinctions between these two approaches will be discussed here.

\subsubsection{Identification of measurement functions}\label{subsectiondiimspa}
Observe that Algorithm  \ref{algdiscgambletsolvecase1g} produces a hierarchy of gamblets from the specification
\begin{equation}\label{eqjhddhghd}
\psi_i^{(q)}=\varPsi_i \text{ for } i\in \I^{(q)}
\end{equation}
of the level $q$ gamblets as the basis elements spanning a subspace $\B^{\d}$ of $\B$,
and the interpolation/restriction operators obtained in Theorem \ref{thmpsdk}. In the proof of Theorem \ref{thmconddisbndisbismatdis} provided in Subsection \ref{subseckedjh} the level $q$ gamblets and measurement functions are the unit vectors of the canonical basis of $\R^{I^{(q)}}$.
What are the measurement functions associated to \eqref{eqjhddhghd}?

 Consider the situation where the $\|\cdot\|_{0}$ norm is Hilbertian, i.e. writing $[\cdot,\cdot]_{0}$ the duality product between $\B_0^*$ and $\B_0$,
$\|\phi\|^2_{0}=[\M^{-1}\phi,\phi]_{0}$ for $\phi\in \B_0$,  and $\M:\B_0^* \rightarrow \B_0$
is a symmetric positive linear bijection. Observe that since $\B_0$ is a dense subspace of $\B^*$, $\B$ is a dense subspace of $\B_0^*$ and $[\phi,\varPsi]=[\varPsi,\phi]_{0}$ for $\varPsi\in \B$ and $\phi \in \B_0$.
We summarize these embeddings in the following diagram
\begin{equation}\label{eqcase1comp}
\text{\xymatrix{
(\B \subset) \B_0^* \ar[r]^{\M} &  \B_0 (\subset \B^*) }}
\end{equation}
Let $M$ be the $\I^{(q)}\times \I^{(q)}$ symmetric matrix defined by
\begin{equation}
M_{i,j}:=[\M \varPsi_i, \varPsi_j]\,
\end{equation}
and observe that, for $x\in \R^{\I^{(q)}}$, we have $\|\sum_{i\in \I^{(q)}} x_i \M\varPsi_i\|_{0}^2 =x^T M x$.
For $i \in \I^{(q)}$, define
\begin{equation}\label{eqkjdhkdhstpsiphitil}
\phi_i^{(q)}:=\sum_{j\in \I^{(q)}}M_{i,j}^{-1} \M \varPsi_j\, .
\end{equation}
Write $\B^{*,\d}:=\Span\{\M \varPsi_i \mid i \in I^{(q)}\}$, and observe, that since
 the number of constraints is equal to the dimension of $\B^\d$, that
 $\varPsi_i$ is the unique minimizer
of $\|\psi\|$ over $\psi\in \B^\d$ such that $[\phi^{(q)}_j,\psi]=\delta_{i,j}$ for $j\in \I^{(q)}$.
Therefore, by selecting \eqref{eqkjdhkdhstpsiphitil} as  level $q$ measurement functions in the discrete space $\B^{*,\d}$, level $q$ gamblets (in the discrete set $\B^\d$) are then given
by $\psi_i^{(q)}=\varPsi_i$  for $i\in \I^{(q)}$ which corresponds to \eqref{eqjhddhghd} and to Line \ref{step3gdgpbdis} of Algorithm \ref{algdiscgambletsolvecase1g}.
Moreover,  if we
let
\begin{equation}\label{eqnormb2dnorm2}
\|\phi\|_{*,\d}:=\sup_{\varPsi \in \B^\d} \frac{[\phi,\varPsi]}{ \|\varPsi\|}\,
\end{equation}
denote dual norm on $\B^{*,\d}$ associated with the naturally induced  norm on the subspace
$\B^{\d} \subset \B$,
the
 fact that the gamblet transform using $\M$ generates a biorthogonal system  produces the following
 result, surprising in its apparent lack of dependence on $\M$.
\begin{Lemma}\label{lemddjkhkdd}
 Defining $(\phi_i^{(q)})_{i\in \I^{(q)}}$ as in \eqref{eqkjdhkdhstpsiphitil},  we have for
$x\in \R^{\I^{(q)}}$ and $\phi=\sum_{i\in \I^{(q)}}x_i \phi_i^{(q)}$,
\begin{equation}\label{eqjhkug5uy}
\|\phi\|_{*,\d}^2=x^T A^{-1} x=
\begin{cases}\min \|\varphi\|^2_* \\ \varphi \in \B^*\text{ s.t. }[\varphi,\varPsi_i]=x_i, \quad i \in  \I^{(q)}\, .\end{cases}\,
\end{equation}
Furthermore,
\begin{equation}\label{eqkddkhd}
x^T A^{-1}x=\begin{cases}\min \|\psi\|^2 \\\psi \in \B\text{ s.t. }\<\psi,\varPsi_i\>=x_i, \quad i \in
 \I^{(q)}\, .\end{cases}
\end{equation}
\end{Lemma}

\begin{Remark}\label{rmkdkjdhdj}
Write $\delta_\d:=\inf_{\phi \in \B^{*,\d}}\sup_{\varPsi \in \B^\d} \frac{[\phi,\varPsi]}{\|\phi\|_* \|\varPsi\|}$.
Observe that for $\phi\in \B^{*,\d}$, $\delta_\d \|\phi\|_* \leq \|\phi\|_{*,\d} \leq \|\phi\|_*$.
Let $\Theta$ be the $\I^{(q)}\times \I^{(q)}$ matrix defined  $\Theta_{i,j}=[\M \varPsi_i,Q \M\varPsi_j]$.
Therefore the above inequality is equivalent to
$\delta_\d^2 \Theta \leq M A^{-1}M^T \leq \Theta$
and to $\delta_\d^2 A \leq M^T \Theta^{-1}M \leq A$.
\end{Remark}

\subsubsection{Conditions on discretized Banach spaces}
Given the discretized subspace $\B^{\d}:=\Span\{\varPsi_i, i \in  \I^{(q)}\}$,
let   $\phi_i^{(q)}$
\eqref{eqkjdhkdhstpsiphitil} be
 the measurement functions corresponding to the  $q$-level gamblets $\varPsi_i$, and
define the hierarchy of measurement functions $\phi^{(k)}_i$ as in \eqref{eq:eigdeiud3dd} using the
 $\phi_i^{(q)}$ and
the hierarchy of matrices $\pi^{(k,k+1)}$ appearing in Algorithm \ref{algdiscgambletsolvecase1g}.
Define $\Phi^{(k)}$ as in \eqref{eqdefPhik} and $\Phi^{(k),\chi}$ as in \eqref{eqphikchi}.
\begin{Condition}\label{conddiscrip3ordis}
There exists constants
$C_\d\geq 1$ and $H\in (0,1)$  such that the following conditions are satisfied.
\begin{enumerate}
\item\label{connum1d} $C_\d^{-1} J^{(k)}\leq W^{(k)}W^{(k),T} \leq C_\d J^{(k)}$ for $k\in \{2,\ldots,q\}$.
\item\label{connum2d}   $C_\d^{-1}|x|^2 \leq \|\sum_{i\in \I^{(k)}}x_i \phi_i^{(k)}\|_{0}^2 \leq C_\d |x|^2$ for $k\in \{1,\ldots,q\}$ and $x\in \R^{\I^{(k)}}$.
\item $\sup_{\phi' \in \Phi^{(q)}}  \frac{\|\phi'\|_{*,\d}}{\|\phi'\|_{0}}\leq C_\d $
\item\label{connum3d}  $\sup_{\phi' \in \Phi^{(q)}} \inf_{\phi \in \Phi^{(k)}} \frac{\|\phi'-\phi\|_{*,\d}}{\|\phi'\|_{0}}\leq C_\d H^k$ for $k\in \{1,\ldots,q-1\}$.
\item\label{connum4d}  $\frac{1}{C_\d}H^k \leq \inf_{\phi\in \Phi^{(k)}} \frac{\| \phi \|_{*,\d}}{\|\phi \|_{0}}$ for $k\in \{1,\ldots,q\}$.
\item\label{connum5d}
$\sup_{\phi\in \Phi^{(k),\chi} }\frac{\|\phi\|_{*,\d}}{\|\phi\|_{0}}\leq C_\d H^{k-1}$ for  $k\in \{2,\ldots,q\}$.
\end{enumerate}
\end{Condition}
The following theorem is a direct consequence of Theorem \ref{corunbcnOR} by
  considering the Banach space $\B^\d$
gifted with the norm  inherited from $\B$ and the dual space is $\B^{*,\d}$ gifted with
 the natural dual norm \eqref{eqnormb2dnorm2}.
\begin{Theorem}\label{thmconddisbndis}
The results of  Theorem \ref{thmconddisbndisbismatdis} hold true with Condition \ref{conddiscrip3ordis} instead of Condition \ref{conddiscrip3ordismatdis}.
\end{Theorem}
\begin{Remark}
Observe that Item \ref{connum2d} of Condition \ref{conddiscrip3ordis} corresponds, for $k=q$, to a bound on the condition number of the mass matrix $M$.
\end{Remark}

\begin{Example}\label{egprotocase12cond}
Consider Example \ref{egproto00}.
Select $(\varPsi_i)_{i\in \N}$ as linearly independent finite elements of $H^1_0(\Omega)$. By selecting $\B_0=L^2(\Omega)$ we have  $\M \varPsi_i=\varPsi_i$ and $M$ is the mass matrix $M_{i,j}=\int_{\Omega}\varPsi_i \varPsi_j$.
Lines \ref{connum2d} to \ref{connum5d} of Condition \ref{conddiscrip3ordis} are then classical regularity conditions on the elements $\phi_i^{(k)}$: if these elements are derived from a hierarchical triangulation of $\Omega$, then these conditions translate to conditions on the homogeneity and aspect ratios of the triangles at each scale. For instance, Line \ref{connum2d} of
 Condition \ref{conddiscrip3ordis} is the Riesz stability condition $C_\d^{-1}|x|^2 \leq \|\sum_{i\in \I^{(k)}}x_i \phi_i^{(k)}\|_{L^2(\Omega)}^2 \leq C_\d |x|^2$. Writing $\|\phi\|_{H^{-1}(\Omega),\d}=\|\phi\|_{*,\d}=\sup_{\varPsi \in \B^\d}\int_{\Omega}\phi \varPsi/\|\varPsi\|_{H^1_0(\Omega)}$ (using the regular $H^1_0$ norm instead of the $a$-energy norm and integrating $\lambda_{\max}(a)$ and $\lambda_{\min}(a)$ into $C_\d$), Line \ref{connum3d} corresponds to the approximation property\
  $ \inf_{\phi \in \Phi^{(k)}} \|\phi'-\phi\|_{H^{-1}(\Omega),\d}\leq C_\d H^k \|\phi'\|_{L^2(\Omega)}$, Line \ref{connum4d} corresponds to the inverse Sobolev inequality
 $\|\phi \|_{L^2(\Omega)} \leq C_\d H^{-k}\| \phi \|_{H^{-1}(\Omega),\d}$ (for $\phi\in \Phi^{(k)}$) and Line \ref{connum5d} corresponds to the Poincar\'{e}'s inequality
$ \|\phi\|_{H^{-1}(\Omega),\d}\leq C_\d H^{k-1} \|\phi\|_{L^2(\Omega)}$ (for $\phi\in \Phi^{(k),\chi}$).
\end{Example}

\subsubsection{Alternative initialization of Algorithm \ref{algdiscgambletsolvecase1g}}
The choice of measurement functions is not unique and it impacts the initialization of Algorithm \ref{algdiscgambletsolvecase1g} and its analysis. The discretization of $\B^*$ done in  Subsection \ref{subsectiondiimspa} can be generalized to
 $\B^{*,\d}:=\Span\{\varphi_i \mid i \in \N\}$  where $(\varphi_i)_{i\in \N}$ are independent linear elements of $\B^*$.
Given such elements, define  a matrix $M$ by
 \begin{equation}
 M_{i,j}=[\varphi_i,\varPsi_j]\, .
 \end{equation}
Then, if the matrix is invertible,
 define the measurement functions
 \begin{equation}\label{eqaltphiq}
 \phi^{(q)}_i=\sum_{j\in \I^{(q)}} M_{i,j}^{-1} \varphi_j \text{ for $i \in \I^{(q)}$,}
 \end{equation}
and observe that
the unique minimizer of $\|\psi\|$ in $\B^\d$ such that
$[\phi^{(q)}_i,\psi]=\delta_{i,j}$ is  $\psi_i^{(q)}=\varPsi_i$  for $i\in \I^{(q)}$  (which corresponds to  the initialization of Algorithm \ref{algdiscgambletsolvecase1g}).
However under the measurement functions $\phi^{(q)}_i= \varphi_i$, the unique minimizer of $\|\psi\|$ in
$\B^\d$ such that
$[\phi^{(q)}_i,\psi]=\delta_{i,j}$ is
\begin{equation}\label{eqkjjdhdkjh}
\psi_i^{(q)}=\sum_{j\in \I^{(q)}}M^{-1,T}_{i,j}\varPsi_j
\end{equation}
 which leads to   an alternative initialization of Algorithm \ref{algdiscgambletsolvecase1g}  introduced in \cite{OwhadiMultigrid:2015}. Numerical experiments suggest the initialization \eqref{eqjhddhghd} is superior than \eqref{eqkjjdhdkjh}.

\begin{Remark}
Examples of choices for $\varphi_i$ are $\M \varPsi_i$.  For specific choices of primal basis functions $(\varPsi_i)_{i\in \N}$ the analysis of Algorithm \ref{algdiscgambletsolvecase1g}
can also be conducted by considering the choice $\varphi_i=Q^{-1}\varPsi_i$ and the measurement functions
$\bar{\phi}^{(q)}_i:=\sum_{j\in \I^{(q)}} A_{i,j}^{-1} Q^{-1}\varPsi_j$, which are the
 dual gamblets (in $\B^*$) associated with the elements $(\varPsi_i)_{i\in \I^{(q)}}$ (in $\B$) described in Theorem \ref{thmsudgdygdgyphi}.
\end{Remark}

\section{Proofs of the results of Sections  \ref{sec1or} and \ref{secsolvinvpb}}\label{sec6}

\subsection{Proof of Proposition \ref{prop_Gambletprojection}}
\label{sec_gambletorthogonal}
The proof of biorthogonality is straightforward.
It is clear that the range of $P$ is in $Q\Phi$. Now let us fix $k$ and consider  $\psi:=Q\phi_{k}$.
Since \[
P\psi= PQ\phi_{k}
=
\sum_{i=1}^{m}{\psi_{i}[\phi_{i},Q\phi_{k}]}
=
\sum_{i=1}^{m}{\psi_{i}\langle \phi_{i},\phi_{k}\rangle_{*}}
=
\sum_{i=1}^{m}{\psi_{i}\Theta_{ik}}\]
\[
=
\sum_{i=1}^{m}{\sum_{j=1}^{m}{\Theta^{-1}_{ij}S\phi_{j}}\Theta_{ik}}
=
Q\sum_{j=1}^{m}{\sum_{i=1}^{m}{\Theta^{-1}_{ij}\phi_{j}}\Theta_{ik}}
=
Q\sum_{j=1}^{m}{\delta_{jk}\phi_{j}}
=
Q\phi_{k}
= \psi\]
we obtain that $P\psi=\psi$. Since $Q\Phi$ is the span of $Q\phi_{k}, k=1,\ldots m$ it follows that
$P\psi=\psi$ for $\psi \in Q\Phi$. Now suppose that $\psi$ is orthogonal to $Q\Phi$, so that
$0=\langle Q\phi_{k}, \psi \rangle=[\phi_{k},\psi ]$ for all $k$.  Then it follows that  $P\psi=0$ for all $\psi$ orthogonal to $Q\Phi$, establishing the  second assertion. The third follows in a similar way. The assertion
that $P^{*}=Q^{-1}PQ$ and that $P^{*}$ is the adjoint to $P$ are straightforward.

\subsection{Proof of Theorem \ref{thmwhdguyd}}
If we use the identity $[\varphi,\psi]=\langle Q\varphi,\psi\rangle, \varphi \in \B^{*},\psi \in \B$ to  write the definition \ref{defpsi} of the gamblet $ \psi_i^{(k)}$ as
\[
\begin{cases}
\text{Minimize }  &\|\psi\|\\
\text{Subject to } &\psi \in \B\text{ and }\langle Q\phi_j^{(k)}, \psi\rangle=\delta_{i,j}\text{ for } j\in \I^{(k)}\,,
\end{cases}
\]
then the usual orthogonality arguments show that $\psi_i^{(k)}\in Q\Phi,i=1,\ldots m$
if we can solve the constraints there. So let
$\psi_i^{(k)}=\sum_{j'}{\alpha_{j'}Q\phi_{j'}^{(k)}}$ and define
$\Theta^{(k)}_{jj'}:=[\phi_j^{(k)} , Q\phi_{j'}^{(k)}]$.  Then we  obtain
\[ \delta_{i,j} =\langle Q\phi_j^{(k)}, \psi_i^{(k)}\rangle
=\sum_{j'}{\alpha_{j'}\langle Q\phi_j^{(k)} , Q\phi_{j'}^{(k)}\rangle}=
\sum_{j'}{\alpha_{j'}[\phi_j^{(k)} , Q\phi_{j'}^{(k)}]}=
\sum_{j'}{\alpha_{j'}\Theta^{(k)}_{jj'}}
\]
and so conclude that
$\alpha_{j'}=\Theta^{(k),-1}_{ij'}$
and therefore
$\psi_i^{(k)}=\sum_{j'}{\Theta^{(k),-1}_{ij'}Q\phi_{j'}^{(k)}}$.

\subsection{Proof of Theorem \ref{thmgugyug0OR}}
It follows from Theorem \ref{thmwhdguyd} and Proposition  \ref{prop_Gambletprojection} that
 the gamblets
$\psi_i^{(k)}$ are those components which
 make
the operator $P^{(k)}_{Q\Phi}:=\sum_{i}{\psi_i^{(k)}\otimes \phi_{j'}^{(k)}}$ the orthogonal projection onto
$Q\Phi$, establishing that $u^{(k)}(u)=P^{(k)}_{Q\Phi}u$ and the assertion follows.

\subsection{Theorem \ref{thmsudgdygdgy}}
The following theorem, which implies Theorems \ref{thmgugyug0OR} and \ref{thmwhdguyd}, shows that the gamblets $\psi_i$ have optimal variational properties in the  $\|\cdot\|$-norm (which is also a consequence of Proposition \ref{prop_Gambletprojection}).
\begin{Theorem}\label{thmsudgdygdgy}
Let $(\psi_i)_{i=1,\ldots,m}$ be defined as in \eqref{eqhgudgud}. Then it holds true that
for $c\in \R^m$, $\sum_{i=1}^m c_i \psi_i$ is the unique minimizer of
\begin{equation}\label{eqhihdiudiduh}
\begin{cases}
\text{Minimize }\|v\|\\
\text{Subject to }v\in \B\text{ and }[\phi_j,v]=c_j\text{ for }j=1,\ldots,m
\end{cases}
\end{equation}
In particular $\psi_i$ is the minimizer of $\|v\|$  over $v\in \B$ such that $[\phi_j,\psi_i]=\delta_{i,j}$ for $j\in \{1,\ldots,m\}$. Furthermore, $\<\psi_i,\psi_j\>=\Theta^{-1}_{i,j}$, and  for $u\in \B$, $\bar{u}$ is the minimizer of $\|u-v\|$ over $v\in \operatorname{span}\{\psi_1,\ldots,\psi_m\}$.
\end{Theorem}
\begin{proof}
Let $V_m:=\operatorname{span}\{\psi_1,\ldots,\psi_m\}$ and $V_m^c$ be the orthogonal complement of $V_m$ in $\B$ with respect to the inner product $\<\cdot,\cdot\>$. Let $v=v_1+v_2$ be the corresponding orthogonal decomposition of $v\in \B$ into $v_1\in V_m$ and $v_2\in V_m^c$.  The constrains in \eqref{eqhihdiudiduh} imply $v_1=\sum_{i=1}^m c_i \psi_i$ and $\|v\|^2=\|v_1\|^2+\|v_2\|^2$ implies that the minimum of \eqref{eqhihdiudiduh} is achieved at $v_2=0$. Note that for $u\in \B$, $u=\bar{u}+u_2$ with $u_2\in V_m^c$, which implies that  $\bar{u}$ is the minimizer of $\|u-v\|$ over $v\in V_m$.
\end{proof}

Let $\Phi^{(k)}$ and $\V^{(k)}$ be defined as in \eqref{eqdefPhik} and \eqref{eqdefvk}.
\eqref{eq:eigdeiud3dd} implies that $\Phi^{(k)}\subset \Phi^{(k+1)}$. \eqref{eqhgudgud} implies that $\V^{(k)}=Q\Phi^{(k)}$. We deduce that $\V^{(k)}\subset \V^{(k+1)}$.
Write $\V^{(q+1)}:=\B$ and for $k\in \{2,\ldots,q+1\}$ let $\W^{(k)}$ be the orthogonal complement of $\V^{(k-1)}$ in $\V^{(k)}$ with respect to the inner product $\<\cdot,\cdot\>$, i.e.
\begin{equation}\label{eqklggkjg7668}
\W^{(k)}=\{w\in \V^{(k)}\mid \<v,w\>=0\text{ for } v\in \V^{(k-1)}\}
\end{equation}
 Let $u^{(k)}$ be defined as in \eqref{eqbdudbuysbdneq1} and write $u^{(q+1)}:=u$.
\begin{Lemma}\label{lemdgdg3ge}
It holds true that
\begin{itemize}
\item For $k\in \{1,\ldots,q-1\}$, $w \in \W^{(k+1)}$ if and only if $Q^{-1} w\in \Phi^{(k+1)}$ and $\V^{(k)}\subset \operatorname{Ker}(Q^{-1} w)$.
\item For $1\leq k\leq k' \leq q$, $u\in \B$ and $\phi \in \Phi^{(k)}$, one has $[\phi, u^{(k')}(u)]=[\phi,u]$.
\end{itemize}
\end{Lemma}
\begin{proof}
The first property follows by observing that $Q^{-1}\V^{(k+1)}=\Phi^{(k+1)}$ and that for $v \in \V^{(k)}$ and $w\in \V^{(k+1)}$ one has $\<w,v\>=[Q^{-1}w,v]$.
For the second property, observe that $[\phi_j^{(k')},\psi_i^{(k')}]=\delta_{i,j}$ implies that  $[\phi', u^{(k')}(u)]=[\phi',u]$ for $\phi'\in \Phi^{(k')}$ and the nesting  \eqref{eq:eigdeiud3dd} implies that for $\phi \in \Phi^{(k)}$, $[\phi, u^{(k')}(u)]=[\phi,u]$.
\end{proof}
Let $W^{(k)}$ and $\chi^{(k)}_i$ be defined  as in Construction \ref{conswk} and \eqref{eqjkhdkdh}.
\begin{Lemma}\label{lemwk}
It holds true that the elements  $(\chi_i^{(k)})_{i\in \J^{(k)}}\in \V^{(k)}$ defined in \eqref{eqjkhdkdh}
form a basis of $\W^{(k)}$ defined in \eqref{eqklggkjg7668}. Therefore \eqref{eqklggkjg7668} and \eqref{eqdefwk} define the same space.
\end{Lemma}
\begin{proof}
Let $\W^{(k)}$ be defined in \eqref{eqklggkjg7668}. Since $\chi^{(k)}_i \in Q \Phi^{(k)}$,
Lemma \ref{lemdgdg3ge} implies that $\chi^{(k)}_i$ belongs to $\W^{(k)}$ if and only if $[\phi^{(k-1)}_j,\chi^{(k)}_i]=0$ for all $j\in \I^{(k-1)}$ which, using \eqref{eq:eigdeiud3dd}, translates into $(\pi^{(k-1,k)}W^{(k),T})_{j,i}=0$, which is satisfied. Writing $|\J^{(k)}|$ the cardinality  of $\J^{(k)}$, observe that $|\J^{(k)}|=|\I^{(k)}|-|\I^{(k-1)}|$. Therefore $\Img(W^{(k),T})=\Ker(\pi^{(k-1,k)})$ also implies that the $|\J^{(k)}|$ elements $\chi^{(k)}_i$ are linearly independent and, therefore, form a basis of $\W^{(k)}$.
\end{proof}

\subsection{Proof of Theorems \ref{thmgugyug2OR} and \ref{thmreffytfuyf}}
Theorems \ref{thmgugyug2OR} and \ref{thmreffytfuyf} are a direct consequence of Lemma \ref{lemwk} and the following proposition.
\begin{Proposition}\label{propduyge}
It holds true that for $u\in \B$ and $k\in \{1,\ldots,q\}$, $u^{(k+1)}(u)-u^{(k)}(u) \in \W^{(k+1)}$ and
 \begin{equation}\label{eqkhdhjkdh}
 \big\|u-(u^{(k+1)}(u)-u^{(k)}(u))\big\| =\inf_{w\in \W^{(k+1)}} \|u-w\|\,.
 \end{equation}
\end{Proposition}
\begin{proof}
By construction $(u^{(k+1)}(u)-u^{(k)}(u)) \in \V^{(k+1)}$.
Let $v \in \V^{(k)}$. Since $\<v,u^{(k+1)}(u)-u^{(k)}(u)\>=[\phi,u^{(k+1)}(u)-u^{(k)}(u)]$ for  $\phi=(Q^{-1}v)\in \Phi^{(k)}$, the second statement of Lemma \ref{lemdgdg3ge} implies that $\<v,u^{(k+1)}(u)-u^{(k)}(u)\>=0$, i.e. $u^{(k+1)}(u)-u^{(k)}(u)$ is orthogonal to $\V^{(k)}$ and must therefore be an element of $\W^{(k+1)}$.
To obtain \eqref{eqkhdhjkdh} we simply observe that $u-(u^{(k+1)}(u)-u^{(k)}(u))$ is orthogonal to $\W^{(k+1)}$. Indeed by  Lemma \ref{lemdgdg3ge}, for $w\in \W^{(k+1)}$,
$\<w,u-(u^{(k+1)}(u)-u^{(k)}(u))\>=[Q^{-1} w,u-u^{(k+1)}(u)]=0$.
\end{proof}

\subsection{Proof of Theorem \ref{thmpsdk}}

For $s\in \I^{(k)}$ write $\bar{\psi}_s^{(k)}:= \sum_{l\in \I^{(k+1)}}\bar{\pi}^{(k,k+1)}_{s,l}\psi^{(k+1)}_l$ and $\bar{\V}^{(k)}:=\Span\{\bar{\psi}_s^{(k)}\mid s\in \I^{(k)} \}$. Let $x\in \R^{I^{(k)}}$, $y\in \R^{\J^{(k+1)}}$ and
\begin{equation}\label{eqdhdihedu65c}
\psi =\sum_{s\in \I^{(k)}} x_s \bar{\psi}_s^{(k)} + \sum_{j\in \J^{(k+1)}}y_j \chi_j^{(k+1)}\,.
 \end{equation}
 If $\psi=0$ then pairing $\psi$ against $\phi_i^{(k)}$ for $i\in \I^{(k)}$ (and observing that $[\phi_i^{(k)}, \bar{\psi}_s^{(k)}]= \delta_{i,s}$) implies $x=0$ and $y=0$. Therefore the elements $\bar{\psi}_s^{(k)}, \chi_j^{(k+1)}$ form a basis for  $\bar{\V}^{(k)}+\W^{(k+1)}$.  Observing that $\operatorname{dim}(\V^{(k+1)})=\operatorname{dim}(\bar{\V}^{(k)})+\operatorname{dim}(\W^{(k+1)})$ we deduce that $\V^{(k+1)}=\bar{\V}^{(k)}+\W^{(k+1)}$. Therefore, since $\V^{(k)}\subset \V^{(k+1)}$, $\psi_i^{(k)}$ can be decomposed as in \eqref{eqdhdihedu65c}. The constraints $[ \phi_s^{(k)},\psi_i^{(k)}]=[ \phi_s^{(k)},\bar{\psi}_i^{(k)}]=\delta_{i,s}$ lead to $x_s=  \delta_{i,s}$.  The orthogonality between $\psi_i^{(k)}$ and $\W^{(k+1)}$ leads to the equations $\<\psi_i^{(k)},\chi^{(k+1)}_j\>=0$ for $j\in \J^{(k+1)}$, i.e.\\
$\sum_{l\in \I^{(k+1)}}\bar{\pi}^{(k,k+1)}_{i,l}\<\psi^{(k+1)}_l,\chi^{(k+1)}_j\>+\sum_{j'\in \J^{(k+1)}} y_{j'} \<\chi^{(k+1)}_{j'},\chi^{(k+1)}_j\>=0$, which translates into $ W^{(k+1)} A^{(k+1)} \bar{\pi}^{(k+1,k)}_{\cdot,i} + B^{(k+1)} y=0$, that is \eqref{eqdidhduhh}.
 Plugging \eqref{eqjkhdkdh} in \eqref{eqdidhduhh} and comparing with \eqref{eq:ftfytftfx} leads to \eqref{eqhuhiddeuv}.

\subsection{Nested computation}
\label{sec6a}
The section establishes Theorem \ref{thm38dgdn} and Proposition \ref{propfund}  below,
 which is used to provide uniform bounds on condition numbers  in
Section  \ref{sec_conoiinon}, which in turn is used to establish our main result
 Theorem \ref{corunbcnOR} in Section \ref{sec_ibuiuviviviv}.

Define $A^{(k)}$ as in \eqref{eqtheta2} and $\Theta^{(k)}$ as in \eqref{eqtheta1}. \eqref{eqhgudgud} implies that
\begin{equation}\label{eqdefak}
A^{(k)}:=\Theta^{(k),-1},
\end{equation}
Observe in particular that,
\begin{equation}\label{eqhguddeddsgud}
\psi_i^{(k)}=\sum_{j\in \I^{(k)}} A^{(k)}_{i,j}Q  \phi_j^{(k)}\,.
\end{equation}
Define $R^{(k,k+1)}$ as in \eqref{eq:ftfytftfx}.
Pairing \eqref{eq:ftfytftfx} against $\phi_j^{(k+1)}$ implies
\begin{equation}\label{eqhjgjhgjgjg}
R^{(k,k+1)}_{i,j}= [\phi_j^{(k+1)},\psi^{(k)}_i]\,.
\end{equation}
\begin{Theorem}\label{thm38dgdn}
For $b\in \R^{\I^{(k)}}$, $R^{(k+1,k)}b$ is the (unique) minimizer  $x\in \R^{\I^{(k+1)}}$ of
\begin{equation}\label{eq:dfdytffdeytfewdaisq}
\begin{cases}
\text{Minimize }  &x^T A^{(k+1)} x \\
\text{Subject to } &\pi^{(k,k+1)}x=b
\end{cases}
\end{equation}
Furthermore $R^{(k,k+1)} \pi^{(k+1,k)}=\pi^{(k,k+1)}R^{(k+1,k)} =I^{(k)}$, $R^{(k,k+1)}=A^{(k)}\pi^{(k,k+1)}\Theta^{(k+1)}$, $\Theta^{(k)}=\pi^{(k,k+1)}\Theta^{(k+1)}\pi^{(k+1,k)}$ and
\begin{equation}\label{eqhuhiuv}
A^{(k)}= R^{(k,k+1)}A^{(k+1)}R^{(k+1,k)}\,.
\end{equation}
\end{Theorem}
\begin{proof}
Using the decompositions \eqref{eq:ftfytftfx} and \eqref{eq:eigdeiud3dd} in $[\phi_j^{(k)},\psi_i^{(k)}]=\delta_{i,j}$ leads to
$R^{(k,k+1)} \pi^{(k+1,k)}=I^{(k)}$. Using \eqref{eqhguddeddsgud} and \eqref{eq:eigdeiud3dd} to expand $\psi^{(k)}_i$ in \eqref{eqhjgjhgjgjg} leads to $R^{(k,k+1)}=A^{(k)}\pi^{(k,k+1)}\Theta^{(k+1)}$. Using \eqref{eq:eigdeiud3dd} to expand $\phi_i^{(k)}$ and $\phi_j^{(k)}$ in \eqref{eqtheta1} leads to $\Theta^{(k)}=\pi^{(k,k+1)}\Theta^{(k+1)}\pi^{(k+1,k)}$. Using \eqref{eq:ftfytftfx} to expand $\psi^{(k)}_i$ and $\psi^{(k)}_j$ in  \eqref{eqtheta2} leads to \eqref{eqhuhiuv}. Let $b\in \R^{\I^{(k)}}$. Theorem \ref{thmsudgdygdgy} implies that $\sum_{i\in \I^{(k)}}b_i \psi_i^{(k)}$ is the unique minimizer
of $\|v\|^2$ subject to $v\in \B$ and $[\phi_j^{(k)},v]=b_j$ for $j\in \I^{(k)}$. Since $\V^{(k)}\subset \V^{(k+1)}$ and since the minimizer is in $\V^{(k)}$, the minimization over $v\in \B$ can be reduced to $v\in \V^{(k+1)}$ of the form $v=\sum_{i\in \I^{(k+1)}}x_i \psi_i^{(k+1)}$, which after using \eqref{eq:eigdeiud3dd} to expand the constraint $[\phi_j^{(k)},v]=b_j$, corresponds to \eqref{eq:dfdytffdeytfewdaisq}.
\end{proof}

The following proposition is a direct consequence of Theorem \ref{thm38dgdn}.
\begin{Proposition}\label{propfund}
 $N^{(k)}W^{(k)}$ is the projection on $\Img(A^{(k)} W^{(k),T})$ with null space (parallel to) $\Ker(W^{(k),T})$ and
$I^{(k)}-N^{(k)}W^{(k)}$ is the projection on $\Ker(W^{(k),T})$ with null space (parallel to) $\Img(A^{(k)} W^{(k),T})$. Furthermore we have the following identities: (1) $W^{(k)}N^{(k)}=J^{(k)}$ (2) $R^{(k-1,k)}N^{(k)}=0$ and (3) $R^{(k-1,k)}A^{(k)}W^{(k),T}=0$.
\end{Proposition}

\subsection{Uniformly bounded condition numbers}
\label{sec_conoiinon}
 This section  provides uniform bounds on condition numbers that
 are used to establish our main result
 Theorem \ref{corunbcnOR} in Section \ref{sec_ibuiuviviviv}.

For $k\in \{1,\ldots,q\}$ write
\begin{equation}\label{eqbarhk}
\bar{H}_k:= \sup_{g \in \B_0, g\not=0} \inf_{\phi \in \Phi^{(k)}} \frac{\|\varphi- \phi\|_*}{\|\varphi\|_0},
\end{equation}
and
\begin{equation}\label{equbarhk}
\ubar{H}_{k}:=\inf_{\phi\in \Phi^{(k)}} \frac{\| \phi \|_*}{\| \phi \|_{0}}\,.
\end{equation}
Observe that Theorem \ref{thmgugyug0OR} and $\V^{(k)}=Q\Phi^{(k)}$ imply \eqref{corunbcnORfe}, i.e.  that for all $u\in \B$,
\begin{equation}\label{eqdkjkduedn}
\|u-u^{(k)}(u)\|=\inf_{v\in \V^{(k)}}   \|u - v\| \leq  \bar{H}_k \|Q^{-1} u\|_0\,.
\end{equation}
To simplify the presentation we will also write $\W^{(1)}:=\V^{(1)}$ and
\begin{equation}\label{eqh0or}
\bar{H}_0:=\sup_{\phi\in \B^*} \frac{\| \phi \|_*}{\| \phi \|_{0}}.
\end{equation}

\begin{Theorem}\label{thmcondlnum}
Let   $k\in \{1,\ldots,q\}$.  If $v\in \V^{(k)}$ then
\begin{equation}\label{eqguygugu68lhs}
 \ubar{H}_{k} \leq  \frac{\|v\|}{\|Q^{-1} v\|_{0}}\,.
\end{equation}
Furthermore, for $k\in \{1,\ldots,q\}$ and $w\in \W^{k}$,
\begin{equation}\label{eqkjddhuedf}
\ubar{H}_{k}\leq \frac{\|w\|}{\|Q^{-1} w\|_{0}} \leq  \bar{H}_{k-1}\,.
\end{equation}
\end{Theorem}
\begin{proof}
For $k=1$ the r.h.s. of \eqref{eqkjddhuedf} follows from the definition of $\bar{H}_0$ and $\sup_{w\in \V^{(1)}}\frac{\|w\|}{\|Q^{-1} w\|_{0}}=\sup_{\phi\in \Phi^{(1)}} \frac{\| \phi \|_*}{\| \phi \|_{0}}$.
For $k\geq 2$, $\V^{(k)}=\V^{(k-1)}\oplus \W^{(k)}$ and \eqref{eqdkjkduedn} imply
\begin{equation}
\sup_{w\in \W^{(k)}} \frac{\|w\|}{\|Q^{-1} w\|_{0}} \leq  \sup_{w\in \V^{(k)}} \inf_{v \in \V^{(k-1)}} \frac{\|w-v\|}{\|Q^{-1} w\|_{0}} \leq  \bar{H}_{k-1}.
\end{equation}
For the lower bound of \eqref{eqkjddhuedf} observe that
\[
\inf_{w\in \W^{(k)}} \frac{\|w\|}{\|Q^{-1} w\|_{0}}  \geq  \inf_{v \in \V^{(k)}} \frac{\|v\|}{\|Q^{-1} v\|_{0}}= \inf_{\phi\in \Phi^{(k)}}  \frac{\|  \phi \|_*}{\| \phi \|_{0}},
\]
which also proves \eqref{eqguygugu68lhs}.
\end{proof}

For $k\in \{1,\ldots,q\}$ let
\begin{equation}\label{eqgam1}
\ubar{\gamma}_k:=\inf_{x\in \R^{\I^{(k)}}} \frac{\| \sum_{i\in \I^{(k)}} x_i  \,  \phi_i^{(k)} \|_{0}^2}{|x|^2} \text{ and }\bar{\gamma}_k:=\sup_{x\in \R^{\I^{(k)}}} \frac{\| \sum_{i\in \I^{(k)}} x_i \, \phi_i^{(k)} \|_{0}^2}{|x|^2}\,.
\end{equation}
\begin{Lemma}\label{lemdekjdhkdjhd}
It holds true that for $k\in \{2,\ldots,q\}$, (1) $\lambda_{\max}(\pi^{(k-1,k)} \pi^{(k,k-1)})\leq \frac{\bar{\gamma}_{k-1}}{\ubar{\gamma}_k}$ and (2) $\lambda_{\min}(\pi^{(k-1,k)}\pi^{(k,k-1)})\geq \frac{\ubar{\gamma}_{k-1}}{\bar{\gamma}_k}$.
\end{Lemma}
\begin{proof}
 Using \eqref{eqgam1} we obtain that for $z\in \R^{\I^{(k-1)}}$ and $x=\pi^{(k,k-1)} z$ that
$|\pi^{(k,k-1)} z|^2 \leq \frac{1}{\ubar{\gamma}_k} \| \sum_{i\in \I^{(k)}} x_i  \,\phi_i^{(k)} \|_{0}^2 $. Using
$\sum_{i\in \I^{(k)}} x_i  \, \phi_i^{(k)}=\sum_{i\in \I^{(k-1)}} z_i  \, \phi_i^{(k-1)}$ and \eqref{eqgam1} again we obtain that
$|\pi^{(k,k-1)} z|^2 \leq \frac{\bar{\gamma}_{k-1}}{\ubar{\gamma}_k}  |z|^2$ which implies (1). The proof of (2) is similar, i.e. based on $|\pi^{(k,k-1)} z|^2 \geq \frac{1}{\bar{\gamma}_k} \| \sum_{i\in \I^{(k)}} z_i  \,  \phi_i^{(k)} \|_{0}^2 $ and
$  \| \sum_{i\in \I^{(k-1)}} z_i  \,  \phi_i^{(k-1)} \|_{0}^2 \geq \ubar{\gamma}_{k-1}|z|^2 $.
\end{proof}

For $k\in \{1,\ldots,q\}$ let $\Phi^{(k),\chi}$ be defined as in \eqref{eqphikchi}
and
\begin{equation}\label{eqHhat}
\hat{H}_{k-1}:=\sup_{\phi\in \Phi^{(k),\chi} }\frac{\| \phi\|_*}{\|\phi\|_{0}}
\end{equation}

Let $A^{(k)}$ be as in \eqref{eqdefak}, \eqref{eqtheta2} and $B^{(k)}$ as in \eqref{eqjgfytfjhyyyg}. Theorem \ref{corunbcnOR} is a direct consequence of the following Theorem.
\begin{Theorem}\label{thmodhehiudhehd}
It holds true that
 \begin{equation}
\frac{1}{\bar{H}_{0}^2} \frac{1}{\bar{\gamma}_1 }  \leq \lambda_{\min}(A^{(1)}) \text{ and }\lambda_{\max}(A^{(1)})\leq \frac{1}{\ubar{H}_{1}^2}\frac{1}{ \ubar{\gamma}_1}
 \end{equation}
  and
\begin{equation}
\operatorname{Cond}(A^{(1)})\leq  \frac{\bar{\gamma}_1}{\ubar{\gamma}_1} \left(\frac{\bar{H}_{0}}{\ubar{H}_{1}}\right)^2
\end{equation}
Furthermore, for $k\in \{2,\ldots,q\}$ it holds true that
  \begin{equation}
 \lambda_{\min}(W^{(k)}W^{(k),T})\frac{1}{\bar{H}_{k-1}^2} \frac{(\bar{\gamma}_k )^{-1}}{1+   \frac{\hat{H}_{k-1}^2}{\ubar{H}_{k-1}^2 }  \frac{\bar{\gamma}_k \bar{\gamma}_{k-1}}{\ubar{\gamma}_k \ubar{\gamma}_{k-1}}}\leq \lambda_{\min}(B^{(k)}),
\end{equation}
  \begin{equation}
 \lambda_{\max}(B^{(k)})\leq \frac{1}{\ubar{H}_{k}^2 } \frac{1}{\ubar{\gamma}_k} \lambda_{\max}(W^{(k)}W^{(k),T}),
\end{equation}
   and
\begin{equation}
\operatorname{Cond}(B^{(k)})\leq  \frac{\bar{H}_{k-1}^2}{\ubar{H}_{k}^2 }   \frac{\bar{\gamma}_k}{ \ubar{\gamma}_k}\big(1+  \frac{\hat{H}_{k-1}^2}{\ubar{H}_{k-1}^2 }  \frac{\bar{\gamma}_k \bar{\gamma}_{k-1}}{\ubar{\gamma}_k \ubar{\gamma}_{k-1}}\big)\Cond(W^{(k)}W^{(k),T})
\end{equation}

\end{Theorem}
\begin{proof}
Let $k\in \{1,\ldots,q\}$ and $x\in \R^{\I^{(k)}}$. Write $v=\sum_{i\in \I^{(k)}} x_i \psi_i^{(k)}$. Observing that $\|v\|^2=x^T A^{(k)} x$ and
$\|Q^{-1} v\|_{0}^2=\|\sum_{i\in \I^{(k)}} (A^{(k)} x)_i \phi_i^{(k)}\|_{0}^2 \geq \ubar{\gamma}_k |A^{(k)} x|^2$, \eqref{eqguygugu68lhs} implies that $\ubar{\gamma}_k \ubar{H}_{k}^2 \leq x^T A^{(k)}x/ |A^{(k)} x|^2$, which, after taking the minimum in $x$ leads to (for $k\geq 1$)
\begin{equation}\label{eqkhiduhdf7d}
 \lambda_{\max}(A^{(k)})\leq \frac{1}{\ubar{H}_{k}^2 \ubar{\gamma}_k},
\end{equation}
and  (using \eqref{eqjgfytfjhyyyg}) for $k\geq 2$
\begin{equation}\label{eqkhiduhde2df7d}
 \lambda_{\max}(B^{(k)})\leq  \frac{1}{\ubar{H}_{k}^2 \ubar{\gamma}_k} \lambda_{\max}(W^{(k)}W^{(k),T})\,.
\end{equation}
Similarly for $k=1$ the r.h.s. of \eqref{eqkjddhuedf} leads to $ x^T A^{(1)}x/ |A^{(1)} x|^2\leq \bar{\gamma}_1 \bar{H}_{0}^2$ and
\begin{equation}
\lambda_{\min}(A^{(1)})\geq 1/\big(\bar{\gamma}_1  \bar{H}_{0}^2\big)\,.
\end{equation}

Now let us consider $k\in \{2,\ldots,q\}$ and $x\in \R^{\J^{(k)}}$. Write $w=\sum_{i\in \J^{(k)},\,j \in \I^{(k)}} x_i W_{i,j}^{(k)} \psi_j^{(k)}$.
\eqref{eqjkhdkdh} and \eqref{eqjgfytfjhyyyg} imply that $\|w\|^2=x^T B^{(k)} x$ and (using \eqref{eqgam1})\\  $\|Q^{-1} w \|_{0}=\| \sum_{i\in \J^{(k)},\,j \in \I^{(k)}}  (A^{(k)} W^{(k),T} x)_{j}   \phi_j^{(k)}\|_{0}\leq \bar{\gamma}_k |A^{(k)} W^{(k),T} x|^2$. Observing that $w\in \W^{(k)}$, the r.h.s.~of
\eqref{eqkjddhuedf} implies that
$\frac{x^T B^{(k)} x}{x^T W^{(k)} (A^{(k)})^2 W^{(k),T} x} \leq \bar{\gamma}_k\,  \bar{H}_{k-1}^2$.
Taking  $x=B^{(k),-1} y$ for $y\in \R^{\J^{(k)}}$ we deduce that
$ \frac{y^T B^{(k),-1} y}{|A^{(k)} W^{(k),T} B^{(k),-1} y|^2 } \leq \bar{\gamma}_k\, \bar{H}_{k-1}^2$.
Writing $N^{(k)}$ as in \eqref{eqjdhiudhiue}
 we have obtained that
\begin{equation}\label{eqjgjhdjhgdy}
\frac{(\bar{\gamma}_k\, \bar{H}_{k-1}^2)^{-1}}{\lambda_{\max}(N^{(k),T}N^{(k)})} \leq \lambda_{\min}(B^{(k)})\,.
\end{equation}
We will now need the following lemmas to bound $\lambda_{\max}(N^{(k),T}N^{(k)})$.
For $k\in \{2,\ldots,q\}$ let
\begin{equation}
P^{(k)}:=\pi^{(k,k-1)} R^{(k-1,k)}\,.
\end{equation}
 Using $R^{(k-1,k)}\pi^{(k,k-1)}=I^{(k-1)}$ (Theorem \ref{thm38dgdn}) we obtain that $(P^{(k)})^2=P^{(k)}$, i.e. $P^{(k)}$ is a projection. Write $\|P^{(k)}\|_{\Ker(\pi^{(k-1,k)})}:=\sup_{x\in \Ker(\pi^{(k-1,k)})} |P^{(k)} x|/|x|$, where $|x|$ is the Euclidean norm of $x$.
\begin{Lemma}\label{lemfdhgdf}
It holds true that for $k\in \{2,\ldots,q\}$,
\begin{equation}
\lambda_{\max}(N^{(k),T}N^{(k)})\leq \frac{1+\|P^{(k)}\|_{\Ker(\pi^{(k-1,k)})}^2}{\lambda_{\min}(W^{(k)}W^{(k),T})}\,.
\end{equation}
\end{Lemma}
\begin{proof}
Since $\Img(W^{(k),T})$ and $\Img(\pi^{(k,k-1)})$ are orthogonal and $\dim(\R^{\I^{(k)}})=\dim\big(\Img(W^{(k),T})\big)+\dim\big(\Img(\pi^{(k,k-1)})\big)$,
for $x\in \R^{\I^{(k)}}$ there exists a unique $y\in \R^{\J^{(k)}}$ and $z\in \R^{\I^{(k-1)}}$ such that
$x=W^{(k),T}y+\pi^{(k,k-1)} z$
and
$
|x|^2=|W^{(k),T}y|^2+|\pi^{(k,k-1)} z|^2.
$
Observe that $W^{(k)} x=W^{(k)}W^{(k),T}y$ (since $W^{(k)}\pi^{(k,k-1)}=0$) and  $R^{(k-1,k)}x =R^{(k-1,k)} W^{(k),T}y+z$ (since $R^{(k-1,k)}\pi^{(k,k-1)}=I^{(k-1)}$ from Theorem \ref{thm38dgdn}). Therefore,
$
|x|^2=|W^{(k),T} y|^2+|P^{(k)} (x-W^{(k),T}y)|^2
$
with $y=(W^{(k)} W^{(k),T})^{-1}W^{(k)} x$.
Let $v\in \R^{\J^{(k)}}$. Taking $x=A^{(k)}W^{(k),T}v$  and observing that $P^{(k)}x=0$ (since $R^{(k-1,k)}A^{(k)}W^{(k),T}=0$ from the $\<\cdot,\cdot\>$-orthogonality between $\V^{(k-1)}$ and $\W^{(k)}$, see Proposition \ref{propfund})
 leads to
$
|A^{(k)}W^{(k),T}v|^2=|W^{(k),T} y |^2+ |P^{(k)} W^{(k),T}y |^2
$
with $y=(W^{(k)} W^{(k),T})^{-1} B^{(k)} v$. Therefore
$
|A^{(k)}W^{(k),T}v|^2\leq (1+\|P^{(k)}\|_{\Ker(\pi^{(k-1,k)})}^2)\frac{|B^{(k)}v|^2}{\lambda_{\min}(W^{(k)} W^{(k),T})},
$
which concludes the proof after taking $v=B^{(k),-1}v'$ and maximizing the l.h.s. over $|v'|=1$.
\end{proof}

\begin{Lemma}\label{lemdjoidjdi}
Writing $\|M\|_2:=sup_x |M x|/|x|$ the spectral norm, we have
\begin{equation}
\|P^{(k)}\|_{\Ker(\pi^{(k-1,k)})}^2 \leq \|\pi^{(k,k-1)}A^{(k-1)}\pi^{(k-1,k)}\|_2  \sup_{x\in \Ker(\pi^{(k-1,k)})}\frac{x^T \Theta^{(k)} x}{x^T x}
\end{equation}
\end{Lemma}
\begin{proof}
Let $x\in \Ker(\pi^{(k-1,k)})$. Using
$P^{(k)}=\pi^{(k,k-1)}A^{(k-1)}\pi^{(k-1,k)}\Theta^{(k)}$
 we obtain that
$
|P^{(k)}x|=\|\pi^{(k,k-1)}A^{(k-1)}\pi^{(k-1,k)}(\Theta^{(k)})^\frac{1}{2}\|_2 |(\Theta^{(k)})^\frac{1}{2} x|
$.
Observing that   for \\$M=\pi^{(k-1,k)}(\Theta^{(k)})^\frac{1}{2}$ we have $M M^T=\Theta^{(k-1)}$ and for $N=\pi^{(k,k-1)}A^{(k-1)}\pi^{(k-1,k)}(\Theta^{(k)})^\frac{1}{2}$ we have $\lambda_{\max}(N^T N)=\lambda_{\max}(N N^T)$, so that we deduce
\\$\|\pi^{(k,k-1)}A^{(k-1)}\pi^{(k-1,k)}(\Theta^{(k)})^\frac{1}{2}\|_2^2=\|\pi^{(k,k-1)}A^{(k-1)}\pi^{(k-1,k)}\|_2$ and conclude by taking the supremum over
$x\in \Ker(\pi^{(k-1,k)})$.
\end{proof}

\begin{Lemma}\label{lemddjoj3ir}
It holds true that
\begin{equation}
 \sup_{x\in \Ker(\pi^{(k-1,k)})}\frac{x^T \Theta^{(k)} x}{x^T x} \leq    \hat{H}_{k-1}^2   \bar{\gamma}_k\,.
\end{equation}
\end{Lemma}
\begin{proof}
Let  $\phi:=\sum_{i\in \I^{(k)}} x_i \phi_i^{(k)}$ with $x\in \Ker(\pi^{(k-1,k)})$. Observe that
$ x^T \Theta^{(k)} x =[\phi,Q\phi]=\|\phi\|^2_*$. Therefore,
$ \frac{x^T \Theta^{(k)} x}{x^T x} \leq \frac{\| \phi\|_*^2}{\|\phi\|_{0}^2}\frac{\|\phi\|_{0}^2}{x^T x}\leq  \bar{\gamma}_k \frac{\| \phi\|_*^2}{\| \phi\|_{0}^2}$. We conclude using \eqref{eqHhat}.
\end{proof}

We are now ready to finish the proof of  Theorem \ref{thmodhehiudhehd}.
Observing that $ \|\pi^{(k,k-1)}A^{(k-1)}\pi^{(k-1,k)}\|_2 \leq \lambda_{\max}(\pi^{(k,k-1)}\pi^{(k-1,k)}) \lambda_{\max}(A^{(k-1)})$ and using \eqref{eqkhiduhdf7d}, we derive from lemmas \ref{lemdjoidjdi} and \ref{lemddjoj3ir} that
\begin{equation}\label{eqdihduhuq2he}
\|P^{(k)}\|_{\Ker{\pi^{(k-1,k)}}}^2 \leq \lambda_{\max}(\pi^{(k,k-1)}\pi^{(k-1,k)}) \frac{\hat{H}_{k-1}^2}{\ubar{H}_{k-1}^2 }  \frac{\bar{\gamma}_k}{\ubar{\gamma}_{k-1}}\,.
\end{equation}
 Therefore \eqref{eqjgjhdjhgdy} and Lemma \ref{lemfdhgdf} imply, after simplification, that
\begin{equation}\label{eqjgjhdjghgfhgdy}
 \lambda_{\min}(B^{(k)}) \geq \frac{(\bar{\gamma}_k\, \bar{H}_{k-1}^2)^{-1}}{1+\lambda_{\max}(\pi^{(k,k-1)}\pi^{(k-1,k)})   \frac{\hat{H}_{k-1}^2}{\ubar{H}_{k-1}^2 }  \frac{\bar{\gamma}_k}{\ubar{\gamma}_{k-1}}} \lambda_{\min}(W^{(k)}W^{(k),T})\,.
\end{equation}
Combining  \eqref{eqjgjhdjghgfhgdy} with the result (1) of Lemma \ref{lemdekjdhkdjhd} concludes the proof of  Theorem \ref{thmodhehiudhehd} (recall that
$\lambda_{\max}(\pi^{(k,k-1)}\pi^{(k-1,k)}) =\lambda_{\max}(\pi^{(k-1,k)}\pi^{(k,k-1)}) $).
\end{proof}

\begin{Theorem}\label{lemdjkdj}
Let $k\in \{2,\ldots,q\}$ and $N^{(k)}:=A^{(k)} W^{(k),T} B^{(k),-1}$ be defined as in \eqref{eqjdhiudhiue}. Under Condition \ref{cond1OR} it holds true that there exists a constant $C$ depending only on $C_{\Phi}$ such that
\begin{equation}
C^{-1} H^{2} \leq N^{(k),T}N^{(k)} \leq  C,
\end{equation}
and
\begin{equation}
\Cond(N^{(k),T}N^{(k)})\leq C H^{-2}\,.
\end{equation}
\end{Theorem}
\begin{proof}
We have from the proof of Theorem \ref{thmodhehiudhehd}, for $y\in \R^{\J^{(k)}}$,
$ \frac{y^T B^{(k),-1} y}{|N^{(k)} y|^2 } \leq \bar{\gamma}_k\,\bar{H}_{k-1}^2$. Therefore taking the minimum over $y$ we deduce that $ \big(\lambda_{\max}(B^{(k)})\bar{\gamma}_k\, \bar{H}_{k-1}^2\big)^{-1} \leq \lambda_{\min}(N^{(k),T}N^{(k)}) $. Similarly Lemma \ref{lemfdhgdf}, \eqref{eqdihduhuq2he} and the result (1) of Lemma \ref{lemdekjdhkdjhd} lead to
\begin{equation}
\lambda_{\max}(N^{(k),T}N^{(k)})\leq \frac{1+\frac{\bar{\gamma}_{k-1}}{\ubar{\gamma}_k}   \frac{\hat{H}_{k-1}^2}{\ubar{H}_{k-1}^2 }  \frac{\bar{\gamma}_k}{\ubar{\gamma}_{k-1}}}{\lambda_{\min}(W^{(k)}W^{(k),T})}\,.
\end{equation}
We conclude using the assumed conditions \ref{cond1OR}.
\end{proof}

\subsection{Proof of Theorem \ref{corunbcnOR}}
\label{sec_ibuiuviviviv}
Theorem \ref{corunbcnOR} is a direct consequence of Theorem \ref{thmodhehiudhehd}.

\subsection{Proof of Theorem \ref{thmconddisbndisbismatdis}}\label{subseckedjh}
Theorem \ref{thmconddisbndisbismatdis} is implied by Theorem \ref{corunbcnOR} by considering the Banach space $\bar{\B}=\R^N$
gifted with the norm $\|x\|^2=x^T A x/\lambda_{\min}(A)$ for $x\in \bar{\B}$. Note that the induced dual space is $\bar{\B}^*=\R^N$ gifted with the norm
$\|x\|^2_*=(x^T A^{-1} x) \lambda_{\min}(A)$ for $x\in \bar{\B}^*$. We also select $\B_0=\R^N$  gifted with the norm $\|x\|_{0}^2=x^T x$ for $x\in \B_0$.
Defining $\phi_i^{(q)}=e_i$ as the unit vector of $\R^N$ in the direction $i$ we then have $\psi_i^{(q)}=e_i$. Condition \ref{cond1OR} naturally translates into Condition \ref{conddiscrip3ordismatdis}.

\subsection{Inverse problems}
\subsubsection{Proof of  Lemma \ref{lemdkljedhlkfjh}}
 Assume that \eqref{eqjkhkkjhuiiu} holds.
Using  \eqref{eqjkhkkjhuiiu}, observe that $C_{e}^{-2}\|u\|_{H^s_0(\Omega)}^2 \leq [ \L u,u] \leq \|\L u\|_{H^{-s}(\Omega)} \|u\|_{H^s_0(\Omega)}$, which implies
$\|u\|_{H^s_0(\Omega)} \leq C_{e}^2 \|\L u\|_{H^{-s}(\Omega)} $ and $C_{\L^{-1}}\leq C_{e}^2$. Similarly \eqref{eqjkhkkjhuiiu} and \eqref{eqjjhykhkkjdedduiiu} imply
$C_{e,2}^{-1}\|\L u\|_{H^{-s}(\Omega)} \leq \| \L u\|_* \leq C_{e,2} \|u\|_{H^s_0(\Omega)}$ and $\|\L u\|_{H^{-s}(\Omega)} \leq C_{e,2}^2 \|u\|_{H^s_0(\Omega)}$. Therefore $C_{\L}\leq C_{e,2}^2$.
Assume that $\L$ and its inverse are continuous.
$[ \L u,u]\leq \|\L u\|_{H^{-s}(\Omega)} \| u\|_{H^{s}_0(\Omega)} \leq C_\L \| u\|_{H^{s}_0(\Omega)}^2$ implies that $C_{e,2}\leq \sqrt{C_{\L}}$. Furthermore,
$[\L u, u]\leq C_{\L} \|u\|_{H^s_0(\Omega)}^2 \leq C_{\L} C_{\L^{-1}}^2 \|\L u\|^2_{H^{-s}(\Omega)}$, implies (writing $\L u=g$) that for $g\in H^{-s}(\Omega)$, $[g, L^{-1} g]\leq C_{\L} C_{\L^{-1}}^2 \|g\|^2_{H^{-s}(\Omega)}$, i.e. (using quadratic form inequalities in $H^{-s}(\Omega)$)
$\L^{-1}\leq C_{\L} C_{\L^{-1}}^2 (-\Delta)^{-s}$. We deduce that $ (-\Delta)^{s} \leq C_{\L} C_{\L^{-1}}^2 \L $ which implies the left hand side of \eqref{eqjkhkkjhuiiu} with $C_{e}\leq  C_{\L} C_{\L^{-1}}^2$.
Similarly writing $\L u=g$ we have
$[ \L u,u]=[g,\L^{-1} g]\leq \|g\|_{H^{-1}(\Omega)} \|\L^{-1} g\|_{H^s_0(\Omega)}\leq C_{\L^{-1}}\|g\|_{H^{-s}(\Omega)}^2$.

\subsubsection{Proof of  Theorem \ref{thmsjdhdhgd}}

$\|\phi\|_*=\|Q \phi\|$ for $\phi \in \B^*$ implies that $\frac{1}{C_\L}\|\L Q \phi\|_2\leq \|\phi\|_*\leq C_{\L^{-1}}\|\L Q \phi\|_2$. We conclude by observing that $f^{(k)}_i=\L Q \phi^{(k)}_i$.

\subsubsection{Proof of  Proposition \ref{propkajhdlkjd}}

 We will keep writing $C$ for any constant depending only on $d,\delta$ and $s$.
Item (1) of Example \ref{egleddoiddedejd} is well known \cite{brenner2007mathematical}.  (5) is trivial.

Let us now prove (2). We will use the notations of Example \ref{egkajhdlkjdini}  and the identity $\I^{(k)}=\bar{I}^{(k)}\times \aleph$.
For $i\in \bar{\I}^{(k)}$ let $x_i^{(k)}\in \tau_i^{(k)}$ such that $\tau_i^{(k)}$ contains the ball
$B(x_i^{(k)}, \delta h^k)$  of center $x_i^{(k)}$ and radius $\delta h^k$, and is contained in $B(x_i^{(k)},  2 h^k)$.  $\phi \in \Phi^{(k)}$ can be decomposed as
$\phi=\sum_{i\in \bar{\I}^{(k)}} p_i 1_{\tau_i^{(k)}}$ where $p_i \in \cP_{s-1}$ (writing $\cP_{s-1}$ the space of polynomials of degree at most $s-1$). We have
$\|\phi \|_{H^{-s}(\Omega)}^2=\sup_{v \in H^s_0(\Omega)} \big[ 2\int_{\Omega}v \phi-\|v\|_{H^s_0(\Omega)}^2\big]$.
By restricting the sup over $v \in H^s_0(\Omega)$ to $v=\sum_{i\in \bar{\I}^{(k)}} v_i$ with $v_i\in H^s_0(B(x_i^{(k)}, \delta h^k))$ we obtain that $ \|\phi \|_{H^{-s}(\Omega)}^2\geq \sum_{i\in \bar{\I}^{(k)}}\sup_{v_i \in H^s_0(B(x_i^{(k)}, \delta h^k))} \big[ 2\int_{B(x_i^{(k)}, \delta h^k))}v_i p_i-\|v_i\|_{H^s_0(B(x_i^{(k)}, \delta h^k)))}^2\big]$, i.e.
 \begin{equation}\label{eqkjhdlkdjhdued}
 \|\phi \|_{H^{-s}(\Omega)}^2\geq \sum_{i\in \bar{\I}^{(k)}} \|p_i \|_{H^{-s}(B(x_i^{(k)}, \delta h^k))}^2\,.
 \end{equation}
 Similarly observe that $\|\phi\|^2_{L^2(\Omega)}=\sum_{i\in \bar{\I}^{(k)}} \|p_i\|_{L^2(\tau_i^{(k)})}^2$, which implies that
 \begin{equation}\label{eqdrlkdjhdued}
 \|\phi \|_{L^2(\Omega)}^2\leq \sum_{i\in \bar{\I}^{(k)}} \|p_i \|_{L^2(B(x_i^{(k)}, 2 h^k))}^2\,.
 \end{equation}
 Write $\hat{p}_i(x)=p(x_i^{(k)}+h^k x)$. We then have
 $\|p_i \|_{L^2(B(x_i^{(k)}, 2 h^k))}^2 = h^d \|\hat{p}_i \|_{L^2(B(0, 2))}^2$
 and $\|p_i \|_{H^{-s}(B(x_i^{(k)}, \delta h^k))}^2 = h^{kd} h^{-2ks}  \|\hat{p}_i \|_{H^{-s}(B(0, \delta ))}^2$ (the second equality follows by using the change of variables $\hat{x}=(x-x_i^{(k)})h^{-k} $ in the identity $\|\hat{p}_i \|_{H^{-s}(B(0, \delta ))}^2=\sup_{v\in H^s_0(B(0,\delta))}\big[2 \int_{B(0, \delta )}v\hat{p}_i-\|v\|_{H^s_0(B(0, \delta ))}^2\big]$).
 It follows that $\|p_i \|_{H^{-s}(B(x_i^{(k)}, \delta h^k))}^2/ \|p_i \|_{L^2(B(x_i^{(k)}, 2 h^k))}^2 \leq h^{-2ks} C_0$ with
 $C_0= \sup_{\hat{p}\in \cP_{s-1}} \|\hat{p} \|_{H^{-s}(B(0, \delta ))}^2/  \|\hat{p}_i \|_{L^2(B(0, 2))}^2$. Since $\cP_{s-1}$ is a linear space of finite dimension ${s+d-1 \choose d}$ and the norms $\|\hat{p} \|_{H^{-s}(B(0, \delta ))}^2$ and $\|\hat{p}_i \|_{L^2(B(0, 2))}^2$ are quadratic it follows that  (this argument is a generalization of \cite[Lem.~3.12]{OwhadiMultigrid:2015} and similar to \cite[Prop.~A1]{HouZhang2017II}), $C_0$ is finite and depends only on $d,\delta$ and $s$. Using \eqref{eqkjhdlkdjhdued} and
 \eqref{eqdrlkdjhdued} we have obtained that $ \|\phi \|_{H^{-s}(\Omega)}^2\leq C h^{-2ks}  \|\phi \|_{L^2(\Omega)}^2$ where $C$ depends only on $d,s$ and $\delta$, which finishes the proof of (2).

Let us now prove  the approximation property (3).  Let $\varphi \in L^2(\Omega)$ and $\phi\in \Phi^{(k)}$. We have
$\|\varphi-\phi \|_{H^{-s}(\Omega)}^2=\sup_{v \in H^s_0(\Omega)} \big[ 2\int_{\Omega}v  (\varphi-\phi)-\|v\|_{H^s_0(\Omega)}^2\big]$
Since $\|v\|_{H^s(\Omega)} \leq C \|v\|_{H^s_0(\Omega)} $ for $v\in H^s_0(\Omega)$ we have
$\sup_{v \in H^s_0(\Omega)} \big[ 2\int_{\Omega}v  (\varphi-\phi)-\|v\|_{H^s_0(\Omega)}^2\big] \leq \sup_{v \in H^s(\Omega)} \big[ 2\int_{\Omega}v  (\varphi-\phi)- C^{-1}\|v\|_{H^s(\Omega)}^2\big]$. Therefore,
\begin{equation}\label{eqkjhssedkhd}
\|\varphi-\phi \|_{H^{-s}(\Omega)}^2\leq \sum_{i\in \bar{\I}^{(k)}} \sup_{v \in H^s(\tau_i^{(k)})} \big[ 2\int_{\tau_i^{(k)}}v  (\varphi-\phi)- C^{-1}\|v\|_{H^s(\tau_i^{(k)})}^2\big]\,.
\end{equation}
 Using the notations of Example \ref{egkajhdlkjdini},
recall that (see \cite[Chap.~6]{john2013numerical}, \cite[Sec.~7]{dupont1980polynomial}, \cite[Chap.~4]{brenner2007mathematical}),
if $p$ is the $L^2(\tau_i^{(k)})$ projection of $\varphi$ onto $\cP_{s-1}(\tau_i^{(k)})$ (i.e. if  $\int_{\tau_i^{(k)}}q  (\varphi-p)=0$ for $q\in \cP_{s-1}(\tau_i^{(k)})$) then
$\|v-p\|_{L^2(\tau_i^{(k)})}\leq C h^s \|v\|_{H^s(\tau_i^{(k)})}$ for $v\in H^s(\tau_i^{(k)})$.
Let $\phi$ on $\tau_i^{(k)}$ be equal to the $L^2(\tau_i^{(k)})$ projection of $\varphi$ onto $\cP_{s-1}(\tau_i^{(k)})$. We then have
for $v\in H^s(\tau_i^{(k)})$, and $q\in \cP_{s-1}(\tau_i^{(k)})$,
$ 2\int_{\tau_i^{(k)}}v  (\varphi-\phi)- C^{-1}\|v\|_{H^s(\tau_i^{(k)})}^2\leq 2\int_{\tau_i^{(k)}}(v-q)  (\varphi-\phi)
-C^{-1}\|v\|_{H^s(\tau_i^{(k)})}^2 \leq 2 \|v-q\|_{L^2(\tau_i^{(k)})} \|\varphi-\phi\|_{L^2(\tau_i^{(k)})}  -C^{-1}\|v\|_{H^s(\tau_i^{(k)})}^2$.
Taking the inf over $q\in \cP_{s-1}(\tau_i^{(k)})$ we obtain that
$2\int_{\tau_i^{(k)}}v  (\varphi-\phi)- C^{-1}\|v\|_{H^s(\tau_i^{(k)})}^2 \leq C h^{k s} \|\varphi-\phi\|_{L^2(\tau_i^{(k)})} \|v\|_{H^s(\tau_i^{(k)})} -C^{-1}\|v\|_{H^s(\tau_i^{(k)})}^2$ which leads to
$ \sup_{v \in H^s(\tau_i^{(k)})} \big[ 2\int_{\tau_i^{(k)}}v  (\varphi-\phi)- C^{-1}\|v\|_{H^s(\tau_i^{(k)})}^2\big] \leq
C h^{2 k s} \|\varphi-\phi\|_{L^2(\tau_i^{(k)})}^2
$. Using $\|\varphi-\phi\|_{L^2(\tau_i^{(k)})} \leq  \|\varphi\|_{L^2(\tau_i^{(k)})}$ and patching up pieces we deduce from \eqref{eqkjhssedkhd}
$\|\varphi-\phi \|_{H^{-s}(\Omega)}^2\leq C h^{2 k s} \|\varphi\|_{L^2(\Omega)}$, which concludes the proof of (3) with
$H=h^s$.

The proof of (4) is similar to that of $(3)$. Simply observe that if $\varphi\in \Phi^{(k),\chi}$ then $\varphi$ is orthogonal in $L^2(\Omega)$ to $\Phi^{(k-1)}$, i.e. $\varphi=\varphi-\phi$ where $\phi$ is the $L^2(\Omega)$-projection of $\varphi$ onto $\Phi^{(k-1)}$.

\section{Proofs of the results of Section \ref{seccigsec}}\label{sec7}

\subsection{Proof of Theorem \ref{thm_micchelli}}
\label{sec_micchelli}
Consider a  putative solution $\Psi'$ such  that its value
$\nu(\Psi'):=\sup_{u\in \B}\frac{
\|u-\Psi'(\Phi(u))\|^2}{\|u\|^{2}} $ is finite.
 Then choosing a nontrivial
$u^{*}$  such that $\Phi(u^{*})=0$, and considering the pencil
 $u_{\lambda}:=\lambda u^{*}$ for $\lambda >0$,  it follows
 that
\[
\nu(\Psi')=\sup_{u\in \B}\frac{
\|u-\Psi'(\Phi(u))\|^2}{\|u\|^{2}}
\geq
\sup_{\lambda >0}\frac{
\|\lambda u^{*}-\Psi'(\Phi(\lambda u^{*}))\|^2}{\|\lambda u^{*}\|^{2}}
=
\sup_{\lambda >0}\frac{
\|\lambda u^{*}-\Psi'(0)\|^2}{\|\lambda u^{*}\|^{2}}\, ,
\]
so that the finiteness of $\nu(\Psi')$ implies that $\Psi'(0)=0$.
  Consequently,
\[
\nu(\Psi')=\sup_{u\in \B}\frac{
\|u-\Psi'(\Phi(u))\|^2}{\|u\|^{2}}
\geq  \sup_{u\in \B:\Phi(u)=0}\frac{
\|u-\Psi'(\Phi(u))\|^2}{\|u\|^{2}}
= \sup_{u\in \B:\Phi(u)=0}\frac{
\|u-\Psi'(0)\|^2}{\|u\|^{2}}\]
\[
=  \sup_{u\in \B:\Phi(u)=0}\frac{
\|u\|^2}{\|u\|^{2}}
=  1\,,
\]
so that we conclude that
\begin{equation}
\label{eq_yvuyvuynoko}
\nu(\Psi') \geq 1, \quad  \Psi':\R^{m} \rightarrow \B\, .
\end{equation}
On the other hand consider the solution  mapping $\Psi:\R^{m} \rightarrow \B$ with components
$\Psi=(\psi_{i})_{i=1}^{m}$  in the assumptions.  By
Proposition \ref{prop_Gambletprojection}, the operator
\[P_{Q\Phi}=\sum_{i=1}^{m}{\psi_{i}\otimes \phi_{i}}\]
is the orthogonal projection onto $Q\Phi$. Writing its action as
$P_{Q\Phi}u=\Psi(\Phi(u))$, observe that
\[
\sup_{u\in \B}\frac{
\|u-\Psi(\Phi(u))\|^2}{\|u\|^{2}}=
\sup_{u\in \B}\frac{
\|u-P_{Q\Phi}u\|^2}{\|u\|^{2}}
 \leq 1\,.
\]
Consequently, the optimality of $\Psi$  follows from   \ref{eq_yvuyvuynoko}.

\subsection{Proof of Theorem \ref{thmdlkdjh3e}}
The assertion that $\xi \sim \mathcal{N}(0,Q)$ is a worst case measure corresponding to the saddle function \eqref{def_saddle}
follows from Theorems \ref{thm_minmax}, \ref{thm_gaussweak} and  \ref{thm_saddle} Section
\ref{sec_worstcase} after we establish the other assertions. To that end,
for $\varphi \in \B^{*}$, consider
$\E[\xi(\varphi)|\s(\Phi)]$. From the fact that \ref{def_gausscond}  is a Gaussian field, and so is an isometry,
we obtain that
\begin{eqnarray*}
E[\xi(\varphi)|\s(\Phi)]&=&P_{\Phi}(\xi(\varphi))\\
&=&\xi(P_{\Phi}\varphi)
\end{eqnarray*}
Recall that Proposition \ref{prop_Gambletprojection} implies that
$P_{\Phi}=\sum_{i=1}^{m}{\phi'_{i}\otimes Q\phi_{i}}\,, $
with $
\phi'_{i}:= \sum_{j=1}^{m}{\Theta^{-1}_{ij}\phi_{j}},\, i=1,\ldots, m
$
and $ \Theta_{ij}:=[ \phi_{i},Q\phi_{j}],\, i,j=1,\ldots m ,$
is the orthogonal projection onto $\Phi$.

Consequently, for  $u \in \B$
we have
\[
\E[\xi(\varphi)|\s(\Phi)](u)=\xi(P_{\Phi}\varphi)(u)
=\xi(\sum_{i=1}^{m}{\phi'_{i} [\varphi,Q\phi_{i}}])(u)
=\sum_{i=1}^{m}{[\varphi,Q\phi_{i}}]\xi(\phi'_{i})(u)\]
\[
=\sum_{i=1}^{m}{[\varphi,Q\phi_{i}}] \sum_{j=1}^{m}{\Theta^{-1}_{ij}\xi(\phi_{j}})(u)
=\sum_{i=1}^{m}{[\varphi,Q\phi_{i}}] \sum_{j=1}^{m}{\Theta^{-1}_{ij}[\phi_{j}},u]
=\sum_{i=1}^{m}{[\varphi,Q\phi'_{i}][\phi_{i},u]}\, ,
\]
so that we obtain
\[\E[\xi|\s(\Phi)](u)=\sum_{i=1}^{m}{Q\phi'_i[\phi_{i},u]}\, .\]
Since Proposition  \ref{prop_Gambletprojection} implies that
$\sum_{i=1}^{m}{Q\phi'_i \otimes \phi_{i}}=P_{Q\Phi}$ is the orthogonal projection onto $Q\Phi$,
we obtain that
$\E[\xi|\s(\Phi)](u)=P_{Q\Phi}u$, and therefore
$u^{\two}(u)=P_{Q\Phi}u$,
where $P_{Q\Phi}$ is the orthogonal projection onto $Q\Phi$,
 The first assertion, that it is an optimal minmax solution, follows from Corollary \ref{cor_micchelli}.
The second assertion follows from the first and the identity
$u^{\two}(u)=P_{Q\Phi}u$, since the projection $P_{Q\Phi}$ fixes
$u:=\psi_{i} \in Q\Phi$.

\subsection{Proof of Theorem \ref{thmdoiedhiu}}
The first two assertions follow directly from  Theorem \ref{thmdlkdjh3e} applied to $\Phi^{(k)}$.
 To obtain the third \eqref{eqhjgjhgjgjgOR},
first use the
 classical martingale property
\[\E(\xi|\sigma(\Phi^{(k})]=\E\bigl[\E[\xi|\sigma(\Phi^{(k+1)})]|\sigma(\Phi^{(k})\bigr]\]
 corresponding the sub $\s$-algebra
$\sigma(\Phi^{(k)})\subset \sigma(\Phi^{(k+1)})$
 along with the representation
\[\E[\xi|\sigma(\Phi^{(k+1)})]=\sum_{j\in  \I^{(k+1)}} \psi^{(k+1)}_j [\phi^{(k+1)}_j,\xi]\,.\]
of the  inner  conditional expectation obtained from
Theorem \ref{thmgugyug0OR}
and the first assertion \eqref{eqbhbdhbdjhb3e}.
Then the second assertion \eqref{def_gambletexpk} implies that
\begin{eqnarray*}
\psi^{(k)}_i=\E\big[\xi\big| [\phi_l^{(k)},\xi]=\delta_{i,l},\,  l \in \I^{(k)}\big]\,
 & =&
\E\Bigl[\E[\xi|\sigma(\Phi^{(k+1)})]\Big| [\phi_l^{(k)},\xi]=\delta_{i,l},\,  l\in \I^{(k)}\Bigr]\\
&=&\\
& =&
\E\Bigl[\sum_{j\in  \I^{(k+1)}} \psi^{(k+1)}_j [\phi^{(k+1)}_j,\xi]\Big| [\phi_l^{(k)},\xi]=\delta_{i,l},\,  l\in \I^{(k)}\Bigr]\\
&=&
\sum_{j\in  \I^{(k+1)}} \psi^{(k+1)}_j \E\big[[\phi^{(k+1)}_j,\xi]\big| [\phi_l^{(k)},\xi]=\delta_{i,l}\text{ for } l\in \I^{(k)}\big]
\end{eqnarray*}
for all $i \in  \I^{(k}$, establishing the third  assertion. The final assertion
follows directly from the fact that Theorem \ref{thmdlkdjh3e} asserts that
 $\xi \sim \mathcal{N}(0,Q)$ is a {\em universal worst case measure} in the sense that it is worst case independent of the measurement functions.

\subsection{Proof of Theorem   \ref{thmdgdjdgygugyd}}
The nesting \eqref{eq:eigdeiud3dd} of the measurement functions implies  $\sigma(\Phi^{(k)}) \subset \sigma(\Phi^{(k+1)})$ and $\big(\sigma(\Phi^{(k)})\big)_{k\geq 1}$ is therefore filtration. The fact that $\xi^{(k)}$ is a martingale follows from  $\xi^{(k)}=\E\big[\xi\big| \sigma(\Phi^{(k)}) \big]$. Since $\xi^{(1)}$ and the increments $(\xi^{(k+1)}-\xi^{(k)})_{k\geq 1}$ are  Gaussian fields belonging to the same Gaussian space their independence is equivalent to zero covariance, which follows from the martingale property, i.e. for $k\geq 1$
$\E\big[\xi^{(1)}(\xi^{(k+1)}-\xi^{(k)})\big]=\E\Big[\E\big[\xi^{(1)}(\xi^{(k+1)}-\xi^{(k)})\big|\sigma(\Phi^{(k)})\big]\Big]=\E\Big[\xi^{(1)} \E\big[(\xi^{(k+1)}-\xi^{(k)})\big|\sigma(\Phi^{(k)})\big]\Big]=0$
and for $k>j\geq 1$,
$\E\big[(\xi^{(j+1)}-\xi^{(j)})(\xi^{(k+1)}-\xi^{(k)})\big]=\E\Big[(\xi^{(j+1)}-\xi^{(j)}) \E\big[(\xi^{(k+1)}-\xi^{(k)})\big|\sigma(\Phi^{(k)})\big]\Big]=0$.

\subsection{Proof of Proposition \ref{propdfflkj}}
Let $\xi_1$ be image of $\xi_2$ under $\L^{-1}$.
By linearity $\xi_1$ is centered and Gaussian. Its covariance operator $Q_1$ is defined as the symmetric positive operator mapping $\B^*$ onto $\B$ such that for $v^*,w^* \in \B^*$,
\[
[ v^*, Q_1 w^*]=\E\big[ [v^*,\xi_1][w^*,\xi_1]\big]=\E\big[ [v^*,\L^{-1}\xi_2][w^*,\L^{-1}\xi_2]\big]=\E\big[ [\L^{-1,*}v^*,\xi_2][\L^{-1,*}w^*,\xi_2]\big]
\]
 i.e. $[ v^*, Q_1 w^*]=[ v^*, \L^{-1} Q_2 \L^{-1,*}w^*]$ and the result follows.

\subsection{Proof of Theorem \ref{thm_minmax}}
First observe that
\[
\inf_{v \in L(\Phi,\B)}\sup_{\mu \in \mathcal{M}_{2}(\B)} {\Psi(v,\mu)}\leq
\sup_{\mu \in \mathcal{M}_{2}(\B)} {\Psi(0,\mu)}
=\sup_{\mu \in \mathcal{M}_{2}(\B)} {1}
\]
implies that
\begin{equation}
\label{biyvuv3}
\inf_{v \in L(\Phi,\B)}\sup_{\mu \in \mathcal{M}_{2}(\B)} {\Psi(v,\mu)} \leq 1 \,.
\end{equation}

For the $\sup\inf$,  consider a Gaussian measure $\mu$ with zero mean and covariance $S_{\mu}$
 along with
$P^{S_{\mu}}:= \sum_{i=1}^{m}{\psi_{i}\otimes \phi_{i}}\, $
where
$
\psi_{j}:= \sum_{k=1}^{m}{(\Theta^{S_{\mu}})^{-1}_{jk}S_{\mu}\phi_{k}}, \quad j=1,\ldots, m
$
 and  $\Theta^{S_{\mu}}$ is the Grammian
$  \Theta^{S_{\mu}}_{ij}:=\langle S_{\mu} \phi_{i},\phi_{j}\rangle,\quad i,j=1,\ldots, m\, $.
It is well known that $\mu\in \mathcal{M}_{2}(\B)$.
 Then Wasilkowski and Wo{\'z}niakowski  \cite{wasilkowski1986average} show that
$v^{*}(x):=P^{S_{\mu}}x, x \in \B$
minimizes
$\int{\|x -v(x)\|^{2}d\mu(x)}$
and therefore $\Psi(v,\mu)$ over $L(\Phi,\B)$.
The value of this minimum
is computed in  \cite{wasilkowski1986average} to be
\begin{equation}
\label{kjhkjhkjh}
 \inf_{v \in L(\Phi,\B)}{\Psi(v,\mu)}
=\Psi(v^{*},\mu)
= 1-\frac{\sum_{i,j=1}^{m}{\Theta^{S_{\mu}}_{ij}\langle \psi_{i}, \psi_{j}\rangle}}
{\int{\|x\|^{2}d\mu(x)}}\, .
\end{equation}

Let $T:\B^{*}\rightarrow \B$ be any continuous symmetric bijection and let $S_{0}:\B^{*}\rightarrow \B$
be any  nontrivial symmetric continuous linear transformation such that $\Phi \in ker(S_{0})$. Then if we define
$S_{\epsilon}:=S_{0}+\epsilon T$,
 it follows that
$ \Theta^{S_{\epsilon}}_{ij}:= \epsilon  \Theta^{T}_{ij}$ where
$\Theta^{T}_{ij}:= \langle T \phi_{i},\phi_{j}\rangle,\quad i,j=1,\ldots, m\, $. Moreover,  define
$
\psi_{j}:= \sum_{k=1}^{m}{(\Theta^{T})^{-1}_{jk}T\phi_{k}}, \quad j=1,\ldots, m
$ and uses it to define
$P^{S_{\epsilon}}:= \sum_{i=1}^{m}{\psi_{i}\otimes \phi_{i}}\, $, which we note is independent of $\epsilon$.
 Then if we let $\mu_{\epsilon}$ denote the centered Gaussian measure with covariance operator $S_{\epsilon}$,  we find that \eqref{kjhkjhkjh} becomes
\begin{equation*}
\label{bivvyivi2}
 \inf_{v \in L(\Phi,\B)}{\Psi(v,\mu_{\epsilon})}
=\Psi(v^{*},\mu)
= 1- \epsilon \frac{\sum_{i,j=1}^{m}{\Theta^{T}_{ij}\langle \psi_{i}, \psi_{j}\rangle}}
{\int{\|x\|^{2}d\mu_{\epsilon}(x)}}\, .
\end{equation*}
Since  the denominator can be calculated as  $\int{\|x\|^{2}d\mu_{\epsilon}(x)}=tr(S_{\epsilon})=tr(S_{0})+\epsilon tr(T)$,   we obtain
\begin{equation*}
\label{iivvuuvuvuyv2}
 \sup_{\epsilon >0}\inf_{v \in L(\Phi,\B)}{\Psi(v,\mu_{\epsilon})}
\geq 1\, ,
\end{equation*}
and therefore
\begin{equation*}
\label{iivvuuvuvuyv2}
 \sup_{\mu \in \mathcal{M}_{2}(\B)}\inf_{v \in L(\Phi,\B)}{\Psi(v,\mu)}
\geq 1\, .
\end{equation*}

Combining  with \eqref{biyvuv3},  the classical minmax inequality
 implies the   assertion
\[\sup_{\mu \in \mathcal{M}_{2}(\B)}\inf_{v \in L(\Phi,\B)}{\Psi(v,\mu)}=
\inf_{v \in L(\Phi,\B)}\sup_{\mu \in \mathcal{M}_{2}(\B)}{\Psi(v,\mu)}=1\,.\]

\subsection{Proof of Proposition \ref{prop_cmcomplete}}
Let $\mathcal{F}(\B)$ be the set of continuous linear finite-rank projections on $\B$ and  for each
$F \in \mathcal{F}(\B)$, let $\Sigma_{F}:=\{F^{-1}(A), A \in \s(F)\}$ be the $\sigma$-algebra
of cylinder sets in $\B$ based on $F$. Then consider the algebra
\[ \mathcal{A}_{cyl}:=\cup_{F \in  \mathcal{F}(\B)}{\Sigma_{F}}\]
of cylinder sets. A cylinder measure $\mu \in CM$ is a collection
of  probability measures
\[ \mu =\{\mu_{F} \in \mathcal{M}(F\B), F \in \mathcal{F}(\B)\}\, \]
such that
\begin{equation}
\label{nobino}
 \mu_{F_{2}}=G_{*} \mu_{F_{1}},\quad F_{2}=GF_{1},\quad  G:F_{1}\B \rightarrow F_{2}\B\,\, \text{continuous and linear}
\end{equation}
 where $G_{*}$ is the  pushforward operator on Borel measures corresponding to the
continuous map $G$.
It is straightforward to show that using the definition
\[ \mu(A):=\mu_{F}(F(A)), \quad A \in \Sigma_{F}\] is well defined,
in the sense that if $A \in  \Sigma_{F_{1}}$ and $A \in  \Sigma_{F_{2}}$, then the result is the same
using $F_{1}$ or $F_{2}$,
see e.g.~Gelfand and Vilenkin \cite[Assertion, Pg.~309]{gelfand1964vilenkin}.  When the cylinder measure
$\mu$ is a bonafide countably additive measure then it follows that
$\mu_{F}=F_{*}\mu, F \in  \mathcal{F}(\B)$
  where $F_{*}$ is the pushforward operator acting on measures.
 Abusing notation, even when $\mu$ is not a measure we nevertheless denote
the image measures $\mu_{F}$ by $F_{*}\mu$.
Let us define the {\em weak cylinder measure topology} $\omega_{CM}$ as in \eqref{def_omegaCM}.
This is the initial topology defined by
the maps $F_{*}:CM  \rightarrow \mathcal{M}(F\B), F \in \mathcal{F}(\B)$, where
$\mathcal{M}(F\B)$ is endowed with the weak topology.
Since the connecting maps $G:F_{1}\B \rightarrow F_{2}\B$
 in the consistency relations \eqref{nobino} of the cylinder measures in $CM$
are continuous, it follows, see e.g.~Aliprantis and Border \cite[Thm.~15.14]{Aliprantis2006},
that the corresponding pushforward operators
$G_{*}:\mathcal{M}(F_{1}\B) \rightarrow \mathcal{M}(F_{1}\B)$ are continuous. Therefore, it follows
that if a sequence of cylinder measures converges in the weak cylinder measure topology, then its limit
consisting of a family of image measures
satisfies the consistency conditions, and therefore is a cylinder measure. That is, the space
of cylinder measures $(CM, \omega_{CM})$, equipped with
 the weak cylinder measure topology, is sequentially complete.

\subsection{Proof of Theorem \ref{thm_saddle}}
Let $\B$ be a separable Banach space with inner product defined by
$\langle u_{1},u_{2} \rangle:=[Q^{-1}u_{1},u_{2}]$ where $Q:\B^{*}\rightarrow \B$ is a symmetric continuous bijection. The covariance operator
 $S_{\mu}:\B\rightarrow \B$ of a centered Gaussian  measure $\mu$ on $\B$ is defined by
$\langle S_{\mu}u_{1},u_{2}\rangle =\E_{u \sim \mu}[\langle u, u_{1}\rangle\langle u, u_{1}\rangle], \, u_{1},u_{2} \in \B$. It is clearly symmetric, in that
$QS_{\mu}^{*}=S_{\mu}Q$, and non-negative.
 It is also known that it is trace class, see e.g.~Bogachev \cite{bogachev1998gaussian}.
The  covariance operator
$\acute{S}_{\mu}:\B^{*}\rightarrow \B$ corresponding to the dual pairing is instead defined
by
$[\varphi_{1},\acute{S}_{\mu}\varphi_{2}] =
\E_{u \sim \mu}\bigl[[\varphi_{1}, u][\varphi_{2}, u]\bigr], \, \varphi_{1},\varphi_{2} \in \B^{*}$.
It easily follows that $\acute{S}_{\mu}=S_{\mu}Q=QS_{\mu}^{*}$.

We now begin the proof of Theorem \ref{thm_saddle}.
We address the finite dimensional case first, it providing the basic ideas for the infinite dimensional problem.
 Consider the
short $\mu_{\Phi^{\perp}(Q)}$ of the operator $Q$ to the subspace $(Q\Phi)^{\perp}$ and
the corresponding Gaussian measure $\mu_{\Phi^{\perp}(Q)}$. This measure is degenerate in that it has support on $(Q\Phi)^{\perp}$.
 Consequently,  we obtain that
\begin{eqnarray*}
\int_{\B}{
\| x-P_{Q\Phi}x\|^{2}d\mu_{\Phi^{\perp}(Q)}(x)}&=&
\int_{\B}{
\| x\|^{2}d\mu_{\Phi^{\perp}(Q)}(x)}
\end{eqnarray*}
from which we conclude that
\[\Psi(P_{Q\Phi},\mu_{\Phi^{\perp}(Q)})=1\, .\]
However, since $\| x-P_{Q\Phi}x\| \leq \|x\|$,  for all $\mu \in  \mathcal{M}_{2}(\B)$, we  have
$\int_{\B}{
\| x-P_{Q\Phi}x\|^{2}d\mu(x)} \leq  \int_{\B}{
\| x\|^{2}d\mu(x)}$
so that $ \Psi(P_{Q\Phi},\mu)\leq 1$ and therefore we obtain
\begin{equation}
\label{kkjbkjblbj}
\Psi(P_{Q\Phi},\mu) \leq  \Psi(P_{Q\Phi},\mu_{\Phi^{\perp}(Q)}), \quad \mu \in  \mathcal{M}_{2}(\B)\, .
\end{equation}
For the upper bound, observe   Wasilkowski and Wo{\'z}niakowski's
  \cite[(ii),pg.~23]{wasilkowski1986average} assertion that
the Gaussian measure $\mu_{Q}$ is invariant under the Householder transformation
$D:=2P_{Q\Phi}-I$. It follows  by the characterization \eqref{id_short}
of the shorted operator that the Gaussian measure $ \mu_{\Phi^{\perp}(Q)}$ derived from the shorted operator is also invariant under the Householder transformation $D$.  Consequently,
 Wasilkowski and Wo{\'z}niakowski
  \cite[Thm.~2.2]{wasilkowski1986average} implies that
$P_{Q\Phi}$ is average case optimal for the measure $\mu_{\Phi^{\perp}(Q)}$ also. That is, we have
$\int_{\B}{
\| x-P_{Q\Phi}x\|^{2}d\mu_{\Phi^{\perp}(Q)}(x)} \leq  \int_{\B}{
\| x-v(x)\|^{2}d\mu_{\Phi^{\perp}(Q)}(x)}$  for all $ v \in L(\Phi,\B)$, which implies that
\[ \Psi(P_{Q\Phi},\mu_{\Phi^{\perp}(Q)}) \leq \Psi(v,\mu_{\Phi^{\perp}(Q)}), v\in L(\Phi,\B)  .\]
Combining with \eqref{kkjbkjblbj} we obtain the assertion.

For the infinite dimensional case,
recall the representation \ref{prop_Gambletprojection}
$P_{Q\Phi}=\sum_{i=1}^{m}{\psi_{i}\otimes \phi_{i}}$ for the orthogonal projection,
where  $\psi_{i}:= \sum_{j=1}^{m}{\Theta^{-1}_{ij}Q\phi_{j}}, \quad i=1,\ldots, m$
and
$ \Theta_{ij}:= [\phi_{i},Q\phi_{j}], \quad i,j=1,\ldots m\, $.
For a centered Gaussian measure $\mu_{S}$ with  covariance operator $S$, we obtain
\begin{eqnarray*}
\int_{\B}{
\| x-P_{Q\Phi}x\|^{2}d\mu_{S}(x)}&=&
\int_{\B}{
\| x\|^{2}d\mu_{S}(x)}-2\int_{\B}{\langle P_{Q\Phi}x,x\rangle
d\mu_{S}(x)}+\int_{\B}{\langle P_{Q\Phi}x,P_{Q\Phi}x\rangle
d\mu_{S}(x)}\\
&=& \int_{\B}{
\| x\|^{2}d\mu_{S}(x)}-
\int_{\B}{\langle P_{Q\Phi}x,P_{Q\Phi}x\rangle
d\mu_{S}(x)}\\
&=& \int_{\B}{
\| x\|^{2}d\mu_{S}(x)}-
\int_{\B}{\langle P_{Q\Phi}x,x\rangle
d\mu_{S}(x)}\\
&=&\int_{\B}{
\| x\|^{2}d\mu_{S}(x)}-\sum_{i=1}^{m}{\langle  \psi_{i},S \phi_{i}\rangle}
\,,
\end{eqnarray*}
so that
\begin{eqnarray*}
\Psi(P_{Q\Phi},\mu_{S})&=&\frac{\int_{\B}{
\| x-P_{Q\Phi}x\|^{2}d\mu_{S}(x)}}{\int_{\B}{
\| x\|^{2}d\mu_{S}(x)}}\\
&=&1-\frac{\sum_{i=1}^{m}{\langle  \psi_{i},S \phi_{i}\rangle}}{\int_{\B}{
\| x\|^{2}d\mu_{S}(x)}}\, .
\end{eqnarray*}
Now consider a sequence of covariance operators
  $S_{n}=\Phi^{\perp}(Q_{n})$ where $Q_{n}$ is trace class, increasing
and strongly convergent to $Q$. Then, since $\Phi \in ker\bigl(\Phi^{\perp}(Q_{n})\bigr)$ is equivalent to
$S_{n}\phi_{i}=0, i=1,\ldots m$, we obtain
\[\Psi(P_{Q\Phi},\mu_{S_{n}})=
1-\frac{\sum_{i=1}^{m}{\langle  \psi_{i},S_{n} \phi_{i}\rangle}}{\int_{\B}{
\| x\|^{2}d\mu_{S_{n}}(x)}}\, =1\, .
\]
Since $\| x-P_{Q\Phi}x\| \leq \|x\|, x \in \B$ implies that
\[\Psi(P_{Q\Phi},\mu)  \leq 1,\quad  \mu \in \mathcal{M}_{2}(\B)\, ,\]
it follows that
\begin{equation}
\label{oubyvbhhhh}
\Psi(P_{Q\Phi},\mu)  \leq \Psi(P_{Q\Phi},\mu_{S_{n}}), \quad   \mu \in \mathcal{M}_{2}(\B)\, .
\end{equation}

On the other hand, for an upper bound
consider the map $P_{Q_{n}\Phi}:\B \rightarrow \B$ defined by
$P_{Q_{n}\Phi}:=\sum_{i=1}^{m}{\psi^{n}_{i}\otimes \phi_{i}}$,
  $\psi^{n}_{i}:= \sum_{j=1}^{m}{\bigl(\Theta^{n}\bigr)^{-1}_{ij}Q_{n}\phi_{j}}, \quad i=1,\ldots, m$,
and
$ \Theta^{n}_{ij}:=[ \phi_{i},Q_{n}\phi_{j}], \quad i,j=1,\ldots m\, $.
Since $Q_{n}$ strongly converges to $Q$ it follow that
$\Psi(P_{Q_{n}\Phi},\mu_{S_{n}})$ converges to $\Psi(P_{Q\Phi},\mu_{S_{n}})$ as $n \rightarrow \infty$.
Moreover,
Wasilkowski and Wo{\'z}niakowski  \cite{wasilkowski1986average} show that
$P_{Q_{n}\Phi}\in \arg\min_{v\in L(\Phi,\B)}{\int_{\B}{\|x-v(x)\|^{2}d\mu_{S_{n}}(x)}}$, that is
$P_{Q_{n}\Phi}$ is an optimal solutions to the average error problem determined by the measure $\mu_{S_{n}}$. Therefore, given $\epsilon >0$,  we have
\[\Psi(P_{Q\Phi},\mu_{S_{n}}) -\epsilon \leq
 \Psi(P_{Q_{n}\Phi},\mu_{S_{n}})
\leq  \Psi(v,\mu_{S_{n}}),
\quad v \in L(\Phi,\B)\,
\]
for large enough $n$,
and therefore conclude
\begin{equation}
\label{iubiyviv}
\Psi(P_{Q\Phi},\mu_{S_{n}})-\epsilon \leq
 \Psi(v,\mu_{S_{n}}),
\quad v \in L(\Phi,\B)\,
\end{equation}
for large enough $n$.
Consequently, combining with   \eqref{oubyvbhhhh}
 establishes the saddle identity. Since $Q_{n}$ converges strongly to $Q$ it follows from Theorem
 \ref{thm_gaussweak} below that
$\mu_{\Phi^{\perp}(Q_{n})}$ converges to $\mu_{\Phi^{\perp}(Q)}$ in the cylinder measure topology, thus establishing the assertion.

\subsection{Gaussian cylinder measures as weak limits of Gaussian measures}
The following theorem shows that the standard Gaussian cylinder measure is the limit
in the weak cylinder measure topology of a sequence of Gaussian measures. It follows that
all Gaussian cylinder measures are such limits.
\begin{Theorem}
\label{thm_gaussweak}
Consider a centered  Gaussian measure $\mu_{S}$ on a separable Hilbert space $\B$ with covariance
operator $S:\B \rightarrow \B$. It is well known that $S$ is  non-negative, symmetric
and trace class and therefore there exists
an orthonormal eigenbasis $\{e_{i}\in \B\}$ such that $Se_{i}=s_{i}e_{i}, i=1,\ldots$ where
the sequence $s_{i}$ is nonnegative, non-increasing and $\sum_{i=1}^{\infty}{s_{i}}< \infty$.
Consider modifications $S_{n}, n=1,\ldots$ defined  using the same orthonormal basis and
modifying the eigenvalues by  $S_{n}e_{i}=s^{(n)}_{i}e_{i}, i=1,\ldots$
where
\begin{equation}
s^{(n)}_{i}=
\begin{cases}
s_{1}& i \leq n\\
s_{i}& i >  n
\end{cases}
\end{equation}
Since the modifications $S_{n}$ are trace class it follows that they correspond to
Gaussian measures $\mu_{S_{n}}, n=1,\ldots$.
Let $\mu_{s_{1}I}$ denote the Gaussian cylinder measure with covariance operator
$s_{1}I$. Then we have
\[  \mu_{S_{n}}   \xrightarrow{\omega_{CM}} \mu_{s_{1}I}\, .\]
\end{Theorem}

\begin{proof}
Consider $F \in \mathcal{F}(\B)$. Then $F_{*}\mu_{S_{n}} = \mu_{FS_{n}F^{*}}$
and $F_{*}\mu_{s_{1}I}= \mu_{s_{1}FF^{*}} $.
 Since, for $x=\sum_{i=1}^{\infty}{x_{i}e_{i}} \in \B$ we have
$\sum_{i=1}^{\infty}{x^{2}_{i}}  < \infty$ and
\[
\|(s_{1}I-S_{n})x\|^{2}= \|\sum_{i=n+1}^{\infty}{(s_{1}-s_{i})x_{i}e_{i}} \|^{2}
= \sum_{i=n+1}^{\infty}{(s_{1}-s_{i})^{2}x^{2}_{i}}
\leq s_{1}^{2}\sum_{i=n+1}^{\infty}{x^{2}_{i}}
\]
it follows that
$S_{n} \rightarrow s_{1}I$  in the strong operator topology. Consequently, since $F$ is of finite rank,
 it follows that
 $FS_{n}F^{*} \rightarrow s_{1}FF^{*}$  in the strong operator topology, and therefore the weak operator topology.
 Since Mourier's theorem,
see e.g.~Vakhania, Tarieladze and Chobanyan \cite[Thm.~IV.~2.4]{vakhania1987probability}, implies that the characteristic
function $\phi_{n}$ of  $\mu_{FS_{n}F^{*}}$ is, for  $x \in F\B$,
$\phi_{n}(x)=e^{-\frac{1}{2}\langle FS_{n}F^{*}x, x\rangle}$  and
the characteristic function $\phi$
of  $\mu_{s_{1}FF^{*}}$ is
$\phi(x)=e^{-\frac{s_{1}}{2}\langle FF^{*}x, x\rangle}$
it follows that $\phi_{n} \rightarrow \phi$  pointwise.  Therefore,
 by the  Levy theorem, see e.g.~Vakhania, Tarieladze and Chobanyan \cite[Thm.~IV.3.2]{vakhania1987probability} , we conclude that
$F_{*}\mu_{S_{n}}  \xrightarrow{\omega} F_{*}\mu_{s_{1}I}$. Since $F\in \mathcal{F}(\B)$ was arbitrary
the assertion follows from the definition  \eqref{def_omegaCM}  of the weak cylinder measure
topology $\omega_{CM}$.
\end{proof}

\section{Proofs of the results of Section \ref{secexpdecloc}}\label{sec9}

\subsection{Proofs of the results of Subsection \ref{subsecejhdg9877eg8e}}

Let $P$ and $P_i$ be as in Subsection \ref{subsecejhdg9877eg8e}.

\subsubsection{Proof of Lemma \ref{lemequivpv}}

Since $P:=\sum_{i\in \beth}{P_{i}}$ is the sum of the orthogonal projections according to the internal direct sum, $P$ must be symmetric, i.e. the symmetry of $P$ ($\<P\chi,\chi'\>=\<\chi,P\chi'\>$ for $\chi,\chi' \in \V^\perp$) follows from that of each $P_i$ (for $i\in \beth$, $\<\chi,P_i \chi'\>=\<P_i \chi,P_i \chi'\>=\<P_i \chi, \chi'\>$ for $\chi,\chi'\in \V^\perp$).
Observe that for $\chi\in \V^\perp$, $\sum_{i\in \beth} \|P_i \chi\|^2=\sum_{i\in \beth} \<P_i \chi,\chi\>= \<P \chi,\chi\>$. Therefore $P\chi=0$ implies that, for $i\in \beth$, $P_i \chi=0$, i.e. $\<\chi,\chi_i\>=0$ for $\chi_i\in \V^{\perp}_i$.
Using $\V^\perp=\sum_{i\in \beth} \V_i^\perp$ we deduce that $P\chi=0$ implies that $\chi=0$ and obtain the injectivity of $P$.
Since $P$ is self adjoint with respect to the scalar product $\<\cdot,\cdot\>$, the Closed Range Theorem see e.g.~ \cite[Thm.~4.13]{rudin1991functional}
says that $P$ is surjective if and only if it is bounded below, that is if $P \geq \epsilon I$ for some $\epsilon >0$.
According to
Feshchenko \cite[Prop.~3.2]{feshchenko2012closeness} (whose statement and proof are reminded in Lemma \ref{lemodddoijediod} below for the convenience of the reader) finite sum of closed subspaces of a Hilbert space equals the whole space if and only if the sum of the corresponding orthogonal projections is strictly positive definite.
Since strictly positive definite for symmetric implies bijective, one can say that the sum of the closed subspaces equals the whole space if and only if the sum of the corresponding orthogonal projections is a bijection.

\begin{Lemma}\label{lemodddoijediod}(Feshchenko)
Consider the finite internal sum $\Sigma_{i}{V_{i}}$ of closed Hilbert subspaces $V_{i}\subset V$ (of a Hilbert space $V$) and the sum $\bar{P}:=\sum_{i}{\bar{P}_{i}}$ of the corresponding set of orthoprojectors.
Then $\Sigma_{i}{V_{i}}=V$ if and only of $\bar{P}$ is a homeomorphism.
\end{Lemma}
\begin{proof}
Consider the map $\bar{A}:\oplus{V_{i}}\rightarrow V$ from the external direct sum to $V$
   defined
by $\bar{A}(x_{1},\ldots)=\sum_{i}{x_{i}}\,$.
Then $\bar{A}^{*}:V \rightarrow \oplus{V_{i}}$ is $\bar{A}^{*}=\oplus{\bar{P}_{i}}$ and
 $\bar{A}\bar{A}^{*}= \bar{P}.$ The assumption $\Sigma_{i}{\bar{H}_{i}}=\bar{H}$ implies that $\bar{A}$ is surjective. Consequently, by the closed range theorem, see e.g.~\cite{rudin1991functional}, it follows that $\|\bar{A}^{*}x\| \geq \epsilon \|x\|$ for some $\epsilon >0$.
Since
\[\langle \bar{A}\bar{A}^{*}x,x\rangle=\langle \bar{A}^{*}x,\bar{A}^{*}x\rangle \geq \epsilon^{2}\|x\|^{2}\] we obtain that $\bar{A}\bar{A}^{*}\geq \epsilon^{2} I$.  Since $\bar{A}\bar{A}^{*}=\bar{P}$ we conclude that $\bar{P} \geq \epsilon^{2}I$.
Since $\bar{P}$ is symmetric and bounded below it is also surjective and therefore an homeomorphism.

\end{proof}

\subsubsection{Proof of Lemma \ref{lemdkjdhjh3e}}

The proof is similar to that of of Lemma 3.1 of \cite{KornhuberYserentant16}.
$K_{\max}\leq n_{\max}$ follows by observing that $\|\sum_{i\in \beth} \chi_i\|^2=\sum_{\db^{\C}(i,j)\leq 1}\<\chi_i,\chi_j\>\leq
\sum_{\db^{\C}(i,j)\leq 1}\frac{\|\chi_i\|^2+\|\chi_j\|^2}{2}=\sum_{\db^{\C}(i,j)\leq 1}\|\chi_i\|^2\leq n_{\max}\sum_{i\in \beth} \|\chi_i\|^2$.
For $\chi\in \V^\perp$, using  $P_i\chi \in \V_i^\perp$ and $\|\sum_{i\in \beth}  P_i\chi\|^2 \leq K_{\max}\sum_{i\in \beth} \|P_i\chi\|^2$ we have
$\<P\chi,\chi\>=\sum_{i\in \beth}  \<P_i \chi,\chi\> \leq \|\sum_{i\in \beth}  P_i\chi\| \|\chi\|\leq (K_{\max}\sum_{i\in \beth} \|P_i\chi\|^2 )^\frac{1}{2} \|\chi\|$.
Observing that
\begin{equation}\label{eqjddkjhdkj}
\sum_{i\in \beth} \|P_i \chi\|^2=\sum_{i\in \beth} \<P_i \chi,\chi\>= \<P \chi,\chi\>
\end{equation}
we deduce that
$\<P\chi,\chi\> \leq (K_{\max}\<P \chi,\chi\> )^\frac{1}{2}  \|\chi\|^2$ and conclude that for $\chi\in \V^\perp$,
$\<P \chi,\chi\> \leq K_{\max} \|\chi\|^2$. Therefore we have obtained that $\lambda_{\max}(P)\leq K_{\max}$.

Let us now prove $ K_{\min}\leq \lambda_{\min}(P)$. Let $\chi\in \V^\perp$. There exists a decomposition $\chi=\sum_{i\in \beth} \chi_i$ with $\chi_i\in \V_i^\perp$, such that \eqref{eqlkjdhlkdhd}, $K_{\min}\sum_{i\in \beth} \|\chi_i\|^2 \leq \|\chi\|^2$, is satisfied. We have
$\|\chi\|^2=\<\chi,\sum_{i\in \beth} \chi_i\>=\sum_{i\in \beth} \<\chi, \chi_i\>= \sum_{i\in \beth} \<P_i \chi, \chi_i\>$. Therefore Cauchy-Schwarz' inequality, \eqref{eqjddkjhdkj} and \eqref{eqlkjdhlkdhd} imply that
$\|\chi\|^2\leq (\sum_{i\in \beth} \|P_i \chi\|^2 )^\frac{1}{2}(\sum_{i\in \beth} \|\chi_i\|^2 )^\frac{1}{2} \leq
\<P \chi,\chi\>^\frac{1}{2} \|\chi\| K_{\min}^{-\frac{1}{2}}
$ and $K_{\min}\|\chi\|^2 \leq \<P \chi,\chi\>$. Therefore we have obtained that  $ K_{\min}\leq \lambda_{\min}(P)$.

\subsubsection{Proof of Theorem \ref{thmswkskjsh}}

Let $\psi_{i,\alpha}^n$ be the  minimizer of \eqref{eqhihdjkjhiudiduh}.
Let $\psi_{i,\alpha,0}:=\psi_{i,\alpha}^0$ and  $\psi_{i,\alpha,n}:=\psi_{i,\alpha}^0-\chi_{i,\alpha,n}$ where $\chi_{i,\alpha,n}\in \V^\perp$ is defined via induction by
$\chi_{i,\alpha,0}=0$ and
\begin{equation}\label{eqgjhguygy}
\chi_{i,\alpha,n+1}=\chi_{i,\alpha,n}  + \zeta P (\psi_{i,\alpha}^0 -\chi_{i,\alpha,n})
\end{equation}

\begin{Lemma}\label{lemdkjd4dh}
Under Condition \ref{confviperp}, $\chi_{i,\alpha}=\psi_{i,\alpha}^0-\psi_{i,\alpha}$ is the unique solution in $\V^\perp$ of
$P\chi=P \psi_{i,\alpha}^0$.
\end{Lemma}
\begin{proof}
Using the variational formulation of $\psi_{i,\alpha}$ observe that $\psi_{i,\alpha}=\psi_{i,\alpha}^0-\chi_{i,\alpha}$ where $\chi_{i,\alpha}$ is the minimizer of $\|\psi_{i,\alpha}^0-\chi\|$ over $\chi\in \V^\perp$. The minimum is characterized by $\<\psi_{i,\alpha}^0-\chi_{i,\alpha},\chi\>=0$ for $\chi\in \V^\perp$ which (since $\V^\perp=\sum_{i\in \beth} \V^\perp_i$) is equivalent to $P_j (\psi_{i,\alpha}^0-\chi_{i,\alpha})=0$  for $j=\in \beth$. Therefore one must have $P (\psi_{i,\alpha}^0-\chi_{i,\alpha})=0$. The uniqueness of the solution of $P\chi=P \psi_{i,\alpha}^0$ follows from the injectivity of $P$.
\end{proof}

Since $\chi_{i,\alpha}$ is a fixed point of \eqref{eqgjhguygy} we deduce (writing $I$ the identity operator) that
\begin{equation}\label{eqgjhguygybis}
\chi_{i,\alpha,n+1}-\chi_{i,\alpha}=(I- \zeta P)(\chi_{i,\alpha,n}-\chi_{i,\alpha})
\end{equation}
Using  \eqref{eqkdjdjhdj} and taking $\zeta-\frac{2}{\lambda_{\max}(P)+\lambda_{\min}(P)}$  we deduce from
$\|I-\zeta P\| \leq \frac{\Cond(P)-1}{\Cond(P)+1}$ that
\begin{equation}
\|\chi_{i,\alpha,n}-\chi_{i,\alpha}\|\leq \big(\frac{\Cond(P)-1}{\Cond(P)+1}\big)^n \|\chi_{i,\alpha}\|
\end{equation}
Observe that by the definition of the distance $\db$ we have $\chi_{i,\alpha,n}\in \B_{i}^n$ and therefore $\psi_{i,\alpha,n}\in \B_{i}^n$.
We have $\|\psi_{i,\alpha}-\psi_{i,\alpha,n}\|=\|\chi_{i,\alpha,n}-\chi_{i,\alpha}\|$ and by the variational formulation of $\psi_{i,\alpha}$,
$\|\psi_{i,\alpha}\|^2=\|\psi_{i,\alpha,n}\|^2-\|\psi_{i,\alpha}-\psi_{i,\alpha,n}\|^2=\|\psi_{i,\alpha}^n\|^2-\|\psi_{i,\alpha}
-\psi_{i,\alpha}^n\|^2$.
Similarly, the variational formulation \eqref{eqhihdjkjhiudiduh} implies
$\|\psi_{i,\alpha}^n\|^2=\|\psi_{i,\alpha,n}\|^2-\|\psi_{i,\alpha}^n-\psi_{i,\alpha,n}\|^2$. Therefore,
$\|\psi_{i,\alpha}-\psi_{i,\alpha}^n\| \leq \|\psi_{i,\alpha}-\psi_{i,\alpha,n}\| \leq \big(\frac{\Cond(P)-1}{\Cond(P)+1}\big)^n \|\chi_{i,\alpha}\|$.
Observe that by the variational formulation of $\psi_{i,\alpha}$ we also have $\|\psi_{i,\alpha}^0\|^2=\|\psi_{i,\alpha}\|^2+\|\chi_{i,\alpha}\|^2$, and therefore, $\|\chi_{i,\alpha}\|\leq \|\psi_{i,\alpha}^0\|$. We conclude that for $n\geq 0$,
$\|\psi_{i,\alpha}-\psi_{i,\alpha}^n\| \leq \|\psi_{i,\alpha}-\psi_{i,\alpha,n}\| \leq \big(\frac{\Cond(P)-1}{\Cond(P)+1}\big)^n \|\psi_{i,\alpha}^0\|$.

\subsection{Null space of measurement functions and quotient operator}
Let $(\phi_{i,\alpha})_{(i,\alpha)\in \beth \times \aleph}$ be linearly independent elements of $\B^*$ (as in Section \ref{subsecejhdg9877eg8e}). Write
$\Phi:=\Span\{\phi_{i,\alpha}\mid (i,\alpha)\in \beth \times \aleph\}$.  Let $\V^\perp:=\{\psi \in \B\mid [\phi_{i,\alpha},\psi]=0\text{ for }(i,\alpha)\in \beth \times \aleph\}$

\begin{Lemma}\label{lemjhgjyguyguybis}
For $\varphi\in \B^*$, there exists a unique $\phi\in \Phi$ and a unique $\chi\in \V^\perp$ such that
\begin{equation}\label{eqjhkhjkhkk}
\varphi=\phi+Q^{-1}\chi
\end{equation}
Furthermore, $\phi$ is the minimizer of
\begin{equation}\label{eqdjidhieuhd}
\begin{cases}
\text{Minimize } \|\varphi-\phi'\|_*\\
\text{Subject to } \phi' \in \Phi
\end{cases}
\end{equation}
and
\begin{equation}\label{eqdjjdhkdjh}
\inf_{\phi'\in \Phi}\|\varphi-\phi'\|_*=\|\chi\|=\sup_{\chi'\in \V^\perp}\frac{[\varphi,\chi']}{\|\chi'\|}
\end{equation}
It also  holds true that $\chi$ is the minimizer of $\|Q\varphi-\chi'\|$ subject to $\chi'\in \V^\perp$.
\end{Lemma}
\begin{proof}
Let $\varphi\in \B^*$. Since $\Phi$ is finite dimensional it is closed and therefore, by the projection theorem in Hilbert space, see e.g.~Luenberger \cite[Thm.~5.8.1]{luenberger1969optimization}, there is a unique
  minimizer $\phi$ of \eqref{eqdjidhieuhd}. At the minimum one must have $\<\varphi-\phi,\phi'\>_*=[\phi',Q (\varphi-\phi)]=0$ for $\phi'\in \Phi$. Therefore $\chi:=Q (\varphi-\phi)$ must be an element of $\V^\perp$ and we deduce the first  part of \eqref{eqdjjdhkdjh} as a consequence. The identity between the first and the third term follows from the
 classical duality theory between minimum
 norm problems and their duals, see e.g.~Luenberger \cite[Thm.~5.8.1]{luenberger1969optimization}.  Let $\chi$ be the minimizer of $\|Q\varphi-\chi'\|$ subject to $\chi'\in \V^\perp$. Then at the minimum one must have $\<Q\varphi-\chi,\chi'\>=0$ for $\chi'\in \V^\perp$, which is equivalent  to $Q\varphi-\chi\in \Phi$.
\end{proof}

\begin{Remark}
Note that by Lemma \ref{lemjhgjyguyguybis}  the component $\phi$ in the decomposition \eqref{eqjhkhjkhkk} is the $\<\cdot,\cdot\>_*$-orthogonal projection of $\varphi$ onto $\Phi$ and the component $\chi$ is the $\<\cdot,\cdot\>$-orthogonal projection of $Q \varphi$ onto $\V^\perp$.
\end{Remark}

For $\varphi,\varphi'\in \B^*$ we write $\varphi \sim \varphi'$ if and only if $\varphi-\varphi'\in \Phi$.
Let $\B^{*,\sim}:=\B^*/\sim$ be the corresponding set of equivalence classes.
\begin{Lemma}\label{lemyguy56l}
 The decomposition \eqref{eqjhkhjkhkk} induces a linear bijection $\tilde{Q}:\B^{*,\sim}\rightarrow \V^\perp$. More precisely (1) for $\tilde{\varphi}\in \B^{*,\sim}$ represented by $\varphi\in \B^*$
 we write
$\chi=\tilde{Q}\tilde{\varphi}$ the component $\chi$ in the decomposition \eqref{eqjhkhjkhkk} (2) for $\chi\in \V^\perp$ we write $\tilde{Q}^{-1}\chi$ the equivalence class of $Q^{-1}\chi$.
\end{Lemma}
\begin{proof}
Let $\tilde{\varphi}\in \B^{*,\sim}$ represented by $\varphi\in \B^*$ and $\varphi'\in \B^*$.
We need to make sure that $\varphi$ and $\varphi'$ have the same $\chi$ component in the decomposition \eqref{eqjhkhjkhkk}, which is trivial since $\varphi-\varphi'=\phi''$ for some $\phi''\in \Phi$.
\end{proof}
Note that if $\varphi\sim \varphi'$ then $\<Q\varphi, \chi\>=[\varphi,\chi]=[\varphi',\chi]=\<Q\varphi',\chi\>$ for $\chi\in \V^\perp$. To keep the notations simple we will therefore continue using the same symbols for the duality  pairings between $\B^*$ and $\B_1$ as
$\B^{*,\sim}$ and $\V^\perp$.
\begin{Lemma}\label{lemdihi3}
 $\tilde{Q}$ is symmetric, i.e., for $\tilde{\varphi}, \tilde{\varphi}'\in \B^{*.\sim}$, $[\tilde{\varphi}, \tilde{Q}\tilde{\varphi}']=[\tilde{\varphi}', \tilde{Q}\tilde{\varphi}]$. Moreover,
for any representative $\varphi$ of $\tilde{\varphi}$,
we have
\[[\tilde{\varphi}, \tilde{Q}\tilde{\varphi}]=\|\chi\|^{2}\, ,\] where
$\chi$ is determined by the decomposition
$\varphi=\phi+Q^{-1}\chi$  guaranteed by Lemma  \ref{lemjhgjyguyguybis}.
 \end{Lemma}
 \begin{proof}
Let $\varphi,\varphi'$ be representatives of $\tilde{\varphi}$ and $\tilde{\varphi}'$. Let  $\varphi=\phi+Q^{-1}\chi$ and
$\varphi'=\phi'+Q^{-1}\chi'$  be the decompositions corresponding to \eqref{eqjhkhjkhkk}. We have
$[\tilde{\varphi}, \tilde{Q}\tilde{\varphi}']=[\varphi-\phi, \chi']=[Q^{-1}\chi', \chi]=[\varphi'-\phi', \tilde{Q}\tilde{\varphi}]=[\tilde{\varphi}', \tilde{Q}\tilde{\varphi}]$, establishing symmetry.
From the above, we observe that
$[\tilde{\varphi}, \tilde{Q}\tilde{\varphi}]=[Q^{-1}\chi, \chi]=\langle\chi,\chi\rangle$ establishing
the second assertion.
\end{proof}

We will now prove the exponential decay of gamblets based on  localization properties of the operator $Q^{-1}$.
One of our main results will be the derivation of necessary and sufficient conditions for exponential decay and localization that can  be expressed as conditions on the image space (therefore once these conditions are satisfied for a given image space, they are satisfied for all continuous operators mapping into that image space).

\subsection{Proofs of the results of Subsections  \ref{subsecdklejhd78} and \ref{subsecdhgiudg6}}
For $i\in \beth$ write $\tilde{Q}_i := P_i \tilde{Q}$ for the linear operator mapping $\tilde{\varphi} \in \B_1^{*,\sim}$ onto the unique element $\chi\in \V^\perp_i$ such that $[\tilde{\varphi},\chi']=[Q^{-1}\chi,\chi']=\<\chi,\chi'\>$ for $\chi' \in \V^\perp_i$.  $\tilde{Q}_i \varphi$ is also the $\<\cdot,\cdot\>$ orthogonal projection of $\tilde{Q} \varphi$ onto $\V^\perp_i$, i.e. $\tilde{Q}_i \tilde{Q}^{-1}= P_i $. Let $\Rc_i$ be as in Subsection \ref{subsecdklejhd78}.

\begin{Lemma}\label{lemdihi3bis}
 $\tilde{Q}_i$ is symmetric, i.e., for $\tilde{\varphi}, \tilde{\varphi}'\in \B^{*,\sim}$, $[\tilde{\varphi}, \tilde{Q}_i\tilde{\varphi}']=[\tilde{\varphi}', \tilde{Q}_i\tilde{\varphi}]$. Furthermore, if $\varphi$ is a representative of $\tilde{\varphi}\in \B^{*,\sim}$, then
 \begin{equation}\label{eqkehkdjdhkj}
 \sup_{\chi'\in \V_i^\perp} \frac{[\varphi,\chi']}{\|\chi'\|}= [\tilde{\varphi}, \tilde{Q}_i\tilde{\varphi}]^\frac{1}{2}=\inf_{\phi \in \Phi}\|\Rc_i (\varphi-\phi)\|_{*,i}
 \end{equation}
 \end{Lemma}

 \begin{proof}
Let $\varphi,\varphi'$ be the representatives of $\tilde{\varphi}$ and $\tilde{\varphi}'$ so that by
Lemma \ref{lemjhgjyguyguybis}
we have $\varphi=\phi+Q^{-1}\chi$ and $\varphi'=\phi'+Q^{-1}\chi'$ with $\phi,\phi'\in \Phi$ and
$\chi, \chi'\in \mathfrak{V}^{\perp}$. It follows that $\tilde{Q}\chi=\tilde{\varphi}$ and $\tilde{Q}\chi'=
\tilde{\varphi}'$.
 Let $\chi_i':=\tilde{Q}_i\tilde{\varphi}'$ and
$\chi_i:=\tilde{Q}_i\tilde{\varphi}$. Since $\tilde{Q}_{i}=P_{i}\tilde{Q}$, it follows that
$\chi'_i=P_{i}\chi'$ and $\chi_i=P_{i}\chi$.
Since we obtain
$ [\tilde{\varphi}, \tilde{Q}_i\tilde{\varphi}']=[\tilde{\varphi}, \chi_i']
=[\varphi-\phi, \chi_i']
=[Q^{-1}\chi, \chi_i']
=\langle \chi, \chi_i'\rangle
=\langle \chi, P_{i}\chi'\rangle
$
the symmetry assertion is established, along with
 the identity $[\tilde{\varphi}, \tilde{Q}_i\tilde{\varphi}] =\|P_{i}\chi\|^{2}$.  Since
$\sup_{\chi'\in \V_i^\perp} \frac{[\varphi,\chi']}{\|\chi'\|}=
\sup_{\chi'\in \V_i^\perp} \frac{[\varphi-\phi,\chi']}{\|\chi'\|}
=\sup_{\chi'\in \V_i^\perp} \frac{[Q^{-1}\chi,\chi']}{\|\chi'\|}
=\sup_{\chi'\in \V_i^\perp} \frac{\langle\chi,\chi'\rangle}{\|\chi'\|}
=\sup_{\chi'\in \V_i^\perp} \frac{\langle \chi,P_{i}\chi'\rangle}{\|\chi'\|}
=\sup_{\chi'\in \V^{\perp}} \frac{\langle \chi,P_{i}\chi'\rangle}{\|\chi'\|}
=\sup_{\chi'\in \V^{\perp}} \frac{\langle P_{i}\chi,\chi'\rangle}{\|\chi'\|}
\leq\sup_{\chi'\in \B} \frac{\langle P_{i}\chi,\chi'\rangle}{\|\chi'\|}
\leq \|P_{i}\chi\|\, ,
$
it follows that the supremum is attained with the choice $\chi':=P_{i}\chi$ producing the equality
\begin{equation}
\label{iubiviviyvi}
\sup_{\chi'\in \V_i^\perp} \frac{[\varphi,\chi']}{\|\chi'\|}=\|P_{i}\chi\|\, .
\end{equation}
 Consequently we obtain
the first equality
$\sup_{\chi'\in \V_i^\perp} \frac{[\varphi,\chi']}{\|\chi'\|}=
[\tilde{\varphi}, \tilde{Q}_i\tilde{\varphi}]^\frac{1}{2}$
 in the main assertion.
For the equality between the first and the third term, using the duality of norm minimization, observe that
$\inf_{\phi \in \Phi}\|\Rc_i (\varphi-\phi)\|_{*,i}=
\inf_{\phi \in R_{i}\Phi}\|\Rc_i \varphi-\phi\|_{*,i}
=\sup_{\chi\in \B_{i}\cap (\Rc_{i}\Phi)^{\perp}  } \frac{[\Rc_{i}\varphi,\chi']}{\|\chi'\|}\, .$
However,  for $\chi \in \B_{i}$, the condition $\chi \in (\Rc_{i}\Phi)^{\perp} $ amounts to
$[\chi, \Rc_{i}\phi]=[\chi, \phi]=0$ for all $\phi \in \Phi$, so that consequently
$\B_{i}\cap (\Rc_{i}\Phi)^{\perp} =\B_{i}\cap \V^{\perp}=\V_{i}^{\perp}$ and therefore
we  conclude that
$\inf_{\phi \in \Phi}\|\Rc_i (\varphi-\phi)\|_{*,i}=
\sup_{\chi\in \B_{i}\cap (\Rc_{i}\Phi)^{\perp}  } \frac{[\Rc_{i}\varphi,\chi']}{\|\chi'\|}
=\sup_{\chi\in \V_{i}^{\perp}  } \frac{[\Rc_{i}\varphi,\chi']}{\|\chi'\|}
=\sup_{\chi\in \V_{i}^{\perp}  } \frac{[\varphi,\chi']}{\|\chi'\|}
\, , $
establishing the assertion.
\end{proof}

\begin{Lemma}\label{lempropjguyug6}
It holds true that $\lambda_{\min}(P)$  and $\lambda_{\max}(P)$ are also (respectively) the largest and smallest constants such that for all  $\tilde{\varphi}\in \B^{*,\sim}$,  $\lambda_{\min}(P) \, [\tilde{\varphi},\tilde{Q}\tilde{\varphi}]
\leq  [\tilde{\varphi},\sum_{i\in \beth}  \tilde{Q}_i\tilde{\varphi}] \leq \lambda_{\max}(P) \,[\tilde{\varphi},\tilde{Q}\tilde{\varphi}]$, i.e. the following quadratic form inequalities hold.
\begin{equation}\label{eqkjhkjhkjuih}
\lambda_{\min}(P) \,\tilde{Q} \leq \sum_{i\in \beth}  \tilde{Q}_i \leq \lambda_{\max}(P)\, \tilde{Q}\,.
\end{equation}
\end{Lemma}
\begin{proof}
Recall that for $\tilde{\varphi} \in \B^{*,\sim}$, that $\tilde{Q}\tilde{\varphi}=\chi \in \V^{\perp}$
is equivalent to $\varphi=\phi +Q^{-1}\chi$ for some representative $\varphi \in  \B^{*}$ for $\tilde{\varphi}$ and some $\phi
\in \Phi$.
Then, since for $\chi' \in \V^{\perp}$,
$ \langle \tilde{Q}\tilde{\varphi},\chi' \rangle=
\langle \chi,\chi' \rangle
= [Q^{-1}\chi,\chi']
= [\varphi-\phi,\chi']
= [\varphi,\chi']
= [\tilde{\varphi},\chi']\, ,
$
we obtain
\begin{equation}
 \langle \tilde{Q}\tilde{\varphi},\chi' \rangle=[\tilde{\varphi},\chi']\, ,\qquad  \chi' \in \V^{\perp}, \tilde{\varphi} \in \B^{*,\sim}\, .
\end{equation}
Consequently,
   for $\tilde{\varphi}\in \B^{*,\sim}$ defining
  $\chi:=\tilde{Q}\tilde{\varphi}$, we obtain
$\langle \chi,P_{i}\chi\rangle=\langle \tilde{Q}\tilde{\varphi},P_{i}\tilde{Q}\tilde{\varphi}\rangle
=\langle \tilde{Q}\tilde{\varphi},\tilde{Q}_{i}\tilde{\varphi}\rangle
=[\tilde{\varphi},\tilde{Q}_{i}\tilde{\varphi}]
$ and so conclude $\langle \chi,P_{i}\chi\rangle=[\tilde{\varphi},\tilde{Q}_{i}\tilde{\varphi}]$ and
and similarly $\langle \chi,\chi\rangle=[\tilde{\varphi},\tilde{Q}\tilde{\varphi}]$.
Consequently,
\begin{equation}\label{eqkjhhuhhiuhu}
\frac{\<\chi, P \chi \>}{\|\chi\|^2}=\frac{\sum_{i\in \beth} [\tilde{\varphi}, \tilde{Q}_i \tilde{\varphi}]}{[\tilde{\varphi},\tilde{Q}\tilde{\varphi}]}\,,
\end{equation}
establishing the assertion.
\end{proof}

\subsubsection{Proof of Theorem \ref{propjguyug6}}

We will now use Lemma \ref{lempropjguyug6} to prove \eqref{eqkhohihi} and \eqref{eqljdhelkjdhkh3}.
By Lemma  \ref{lemdihi3}, if $\varphi$ is a representative of $\tilde{\varphi}$, then
$[\tilde{\varphi},\tilde{Q}\tilde{\varphi}]=\|\chi\|^{2},$ where
$\chi$ is determined by the decomposition
$\varphi=\phi+Q^{-1}\chi$  of Lemma  \ref{lemjhgjyguyguybis}.
By Lemma \ref{lemjhgjyguyguybis},  then
$[\tilde{\varphi},\tilde{Q}\tilde{\varphi}]=\|\chi\|^{2}=\sup_{\chi'\in \V^\perp} \frac{[\varphi,\chi']^2}{\|\chi'\|^2}=\inf_{\phi\in \Phi}\|\varphi-\phi\|_*^2$.
Furthermore Lemma \ref{lemdihi3bis} implies that
$[\tilde{\varphi}, \tilde{Q}_i \tilde{\varphi}]=\sup_{\chi'\in \V_i^\perp} \frac{[\varphi,\chi']^2}{\|\chi'\|^2}=\inf_{\phi\in \Phi}\|\Rc_i (\varphi-\phi)\|_{*,i}^2$.  Consequently,
 the assertion follows directly from
  Lemma \ref{lempropjguyug6}.

\subsubsection{Proof of Theorem \ref{thmlkhkjhlkh}}

Simply observe that (since $\|\varphi-\phi\|_*=\|Q(\varphi-\phi)\|$) we have $ \frac{1}{\bar{C}_\L}\|\L Q(\varphi-\phi)\|_2 \leq \|\varphi-\phi\|_* \leq \bar{C}_{\L^{-1}} \|\L Q(\varphi-\phi)\|_2$ and that similar localized inequalities hold.

\subsubsection{Proof of Proposition \ref{propkjshkdjhdkjh}}

We have $K_{\min}\leq \lambda_{\min}(P)$ by Lemma \ref{lemdkjdhjh3e}.
Let us now prove the converse inequality $K_{\min}\geq \lambda_{\min}(P)$. Write $\bar{Q}:=\sum_{i\in \beth}  \tilde{Q}_i$.
If $\lambda_{\min}(P)=0$ then Lemma \ref{lemdkjdhjh3e} implies that $K_{\min}=0$ and the result is trivial. Assume that $\lambda_{\min}(P)>0$. Lemma \ref{lemdihi3bis} then implies that $P$ is a bijection from $\V^\perp$ to $\V^\perp$ and since
$\bar{Q}= P \tilde{Q}$ it follows that $\bar{Q}$ is also a bijection (and in particular invertible).
Let $\chi \in \V^\perp$ and $\chi_i:= \tilde{Q}_i \bar{Q}^{-1} \chi$. Observe that $\sum_{i\in \beth} \chi_i=\chi$ and
$\sum_{i\in \beth} \|\chi_i\|^2=\sum_{i\in \beth} [ \bar{Q}^{-1} \tilde{Q}_i \bar{Q}^{-1} \chi, \chi]$.
\eqref{eqkjhkjhkjuih} implies that
$\lambda_{\min}(P) \,\sum_{i\in \beth} \bar{Q}^{-1} \tilde{Q}_i \bar{Q}^{-1}\leq   \tilde{Q}^{-1}$, which leads to
$\lambda_{\min}(P) \sum_{i\in \beth} \|\chi_i\|^2\leq \|\chi\|^2$ and $K_{\min}\geq \lambda_{\min}(P)$.

\subsubsection{Proof of Lemma \ref{lemdeiudygddf}}

$ \sum_{i\in \beth }\V_{i}^\perp \subset \V^\perp$ is trivial, let us prove the converse inclusion.
Let $v\in \V^\perp$. Construction \ref{constlocop} implies that there exists $(v_1,\ldots, v_{|\beth|})\in \B_{1}\times \cdots \times \B_{|\beth|}$ such that $v=v_1+\cdots+v_{|\beth|}$.  $\tilde{P} v=0$ implies that $v=v-\tilde{P} v=\sum_{i\in \beth} (v_i - \tilde{P}v_i)$.
We conclude the proof by observing that $v_i - \tilde{P}v_i \in \V_i^\perp$ (since $\tilde{P}$ maps $\B_{i}$ into itself under Condition \ref{conduihiuh}). We obtain $\lambda_{\min}(P)>0$ from Lemma \ref{lemequivpv}.

\subsubsection{Proof of  Theorem \ref{thmjhg8g87g}}

Let $v\in \V^\perp$. Let $v_1,\ldots,v_{|\beth|}$ be as in Condition \ref{conduihiuhno2}.
Observe that
$v=v-\tilde{P} v=\sum_{i\in \beth} (v_i-\tilde{P} v_i) $ and $v_i -\tilde{P} v_i \in \V^\perp_i$.
Observe that $\sum_{i\in \beth} \|v_i-\tilde{P} v_i\|^2 \leq 2 \sum_{i\in \beth} \|v_i\|^2 +2\sum_{i\in \beth} \|\tilde{P} v_i\|^2
\leq 4 T_{\max} \|v\|^2$.  We conclude using Lemma \ref{lemdkjdhjh3e}
and Proposition  \ref{propkjshkdjhdkjh}.

\subsubsection{Proof of Theorem \ref{thmegdkj3hrrsuite}}

Using Lemma \ref{lemdkjdhjh3e}   we obtain that
$C_{\max}\leq n_{\max}$  where  $\bar{n}_{\max}-1$ is the maximum number of overlapping neighbors of $\Omega_i$. Since $\bar{n}_{\max}\leq  C \delta^{-d}$ and writing $C$ for any constant depending only on $d, \delta$ and $s$ (we will keep using that convention through this example), we therefore have
\begin{equation}\label{eqkjhkjhurr}
C_{\max} \leq C \,.
\end{equation}
We will now use Theorem \ref{thmjhg8g87g} (and work with $(H^s_0(\Omega),\|\cdot\|_{H^s_0(\Omega)})$ and its localized versions) to obtain a lower bound on $C_{\min}$.
For $t\in \{0,\ldots,s\}$ and $\eta\in C_0^\infty(\Omega)$, write
$\|D^t \eta\|_{L^\infty(\Omega)}:=\max_{i_1,\ldots,i_t}\|\partial_{x_{i_1}}\cdots \partial_{x_{i_t}}\eta(x)\|_{L^\infty(\Omega)}$.
Let $(\eta_i)_{i\in \beth}$ be a partition of unity associated with $(\Omega_i)_{i\in \beth}$, i.e. $\eta_i \in C^\infty_0(\Omega_i)$, $\eta_i=0$ on $\Omega\setminus\Omega_i$, $0\leq \eta_i \leq 1$, $\sum_{i\in \beth} \eta_i=1$, $\|D^t \eta_i\|_{L^\infty(\Omega)}\leq C h^{-t}$ for $t\in \{1,\ldots,s\}$.
Let $v\in \V^\perp$. Write $v_i=\eta_i v$. Observe that $v_i \in H^s_0(\Omega_i)$. We have  $\sum_{i\in \beth} \|v_i\|_{H^s_0(\Omega_i)}^2\leq C \sum_{t=0}^s \sum_{i\in \beth} \|D^t \eta_i\|_{L^\infty(\Omega)}\|D^{s-t} v\|_{L^2(\Omega_i)}
\leq C \sum_{t=0}^s h^{-t} \|D^{s-t} v\|_{L^2(\Omega)}$. Therefore \eqref{eqconlocmeas} implies that
\begin{equation}\label{eqkdjdhdjh}
\sum_{i\in \beth} \|v_i\|_{H^s_0(\Omega_i)}^2\leq   C \|v\|_{H^s_0(\Omega)}
\end{equation}
For $(i,\alpha)\in \beth\times \aleph$ let $\tilde{\psi}_{i,\alpha}$ be the minimizer of $\|\psi\|_{H^s_0(\tau)}$ over $\psi \in H^s_0(\tau_i)$ subject to $[\phi_{i,\beta},\psi]=\delta_{\alpha,\beta}$ for $\beta \in \aleph$. Theorem \ref{thmsudgdygdgy} and \eqref{eqkjdhddkuieheu} imply that $\|\tilde{\psi}_{i,\alpha}\|_{H^s_0(\Omega)}\leq C h^{-s}$ where $C$ depends only $d,\delta$ and $s$.
 Observe the elements $(\tilde{\psi}_{i,\alpha})_{(i,\alpha)\in \beth\times \aleph }$ satisfy
  Condition \ref{conduihiuh} and
 define $\tilde{P}$ as in \eqref{eqkllhkjhju9}, i.e.
$\tilde{P} v:=\sum_{(i,\alpha)\in \beth\times \aleph }  \tilde{\psi}_{i,\alpha} [\phi_{i,\alpha},v]$ for $v\in H^s_0(\Omega)$.
For $v\in \V^\perp$ and $v_i=\eta_i v$, \eqref{eqconlocmeasnext} implies that
$\|\tilde{P} v_i\|_{H^s_0(\Omega)}^2
\leq C   \sum_{(j,\beta)\in\beth \times \aleph}  \|\tilde{\psi}_{j,\beta}\|^2_{H^s_0(\tau_j)}  [\phi_{j,\beta},v_i]^2 \leq
C h^{-2s}   \sum_{(j,\beta)\in\beth \times \aleph}    [\phi_{j,\beta},v_i]^2 \leq
 C \big(h^{-2s} \|v_i\|_{L^2(\Omega)}^2+\|v_i\|_{H^s_0(\Omega)}^2\big)$.
Observe that since $\supp(v_i)\subset \Omega_i$ we have $\sum_{i\in \beth} \|v_i\|_{L^2(\Omega)}^2 \leq C \|v\|^2_{L^2(\Omega)}$ and
\eqref{eqconlocmeas} implies that $\|v\|^2_{L^2(\Omega)}\leq C h^{2s} \|v\|_{H^s_0(\Omega)}^2$.  Therefore using \eqref{eqkdjdhdjh} we deduce that
$\sum_{i\in \beth}\|\tilde{P} v_i\|_{H^s_0(\Omega)}^2
\leq C   \|v\|_{H^s_0(\Omega)}^2$, which corresponds to Item (2) of Condition \ref{conduihiuhno2}.
Using Theorem \ref{thmjhg8g87g} we deduce that
\begin{equation}\label{eqlkjhkjhjhjkhiuirr}
C_{\min}\geq C^{-1},
\end{equation}
for some constant $C$ depending only on $d, s$ and $\delta$. We conclude using Theorem \ref{thmlkhkjhlkh} and Lemma \ref{lemdkljedhlkfjh}.

\subsubsection{Proof of Theorem \ref{thmjff76f57}}
We refer to Subsection \ref{proofthmpropegkde} for the proof of localization with Examples \ref{egprotoh10normNNN} and \ref{egprotoalsobolevl}.

\subsection{Numerical homogenization}

\subsubsection{Proof of Theorem  \ref{thmhgguyg65OR}}

For $i\in \beth$ let $\I_i:=\{(j,\beta)\in \beth\times \aleph \mid \Rc_i \phi_{j,\beta}\not=0\}$.
Let
\begin{equation}\label{eqgsam1maxloc}
\ubar{\gamma}:=\inf_{x\in \R^{\beth\times \aleph}} \frac{\| \sum_{(i,\alpha)\in \beth\times \aleph}x_{i,\alpha}\phi_{i,\alpha} \|_{0}^2}{|x|^2}\text{ and }\bar{\gamma}:=\sup_{x\in \R^{\beth\times \aleph}} \frac{\| \sum_{(i,\alpha)\in \beth\times \aleph}x_{i,\alpha}\phi_{i,\alpha} \|_{0}^2}{|x|^2}\,.
\end{equation}
For $i\in \beth$, write
\begin{equation}\label{eqgam1loc}
\ubar{\gamma}_i:=\inf_{x\in \R^{\I_i}} \frac{\| \sum_{j\in \I_i} x_j  \,  \Rc_i \phi_j \|_{0,i}^2}{|x|^2},
\end{equation}
Write
\begin{equation}\label{equbarhkloc}
\ubar{H}:=\inf_{\phi\in \Phi} \frac{\| \phi \|_{*}}{\| \phi \|_{0}} \text{ and }\ubar{H}_i:=\inf_{\phi\in \Phi: \Rc_i\phi\not=0} \frac{\|\Rc_i \phi \|_{*,i}}{\| \Rc_i \phi \|_{0,i}} \,.
\end{equation}
and
\begin{equation}\label{equbarhklocwithstars}
\underset{*}{H}:=\inf_{x\in \R^{\beth\times \aleph}} \frac{\| \sum_{(i,\alpha)\in \beth\times \aleph}x_{i,\alpha}\phi_{i,\alpha} \|_{*}}{|x|} \text{ and for $i\in \beth$, }\overset{*}{H}:=\sup_{x\in \R^{\beth\times \aleph}} \frac{\| \sum_{(i,\alpha)\in \beth\times \aleph}x_{i,\alpha}\phi_{i,\alpha} \|_{*}}{|x|} \,.
\end{equation}

\begin{Lemma}\label{lemljkhlkhhu}
We have $\lambda_{\max}(A)\leq (\underset{*}{H})^{-2} \leq \frac{1}{\ubar{H}^2 \ubar{\gamma}}$
and
$\lambda_{\min}(A) \geq (\overset{*}{H})^{-2} \geq (\bar{H}_0^2 \bar{\gamma})^{-1}$.
\end{Lemma}
\begin{proof}
Theorem \ref{thmsudgdygdgy} implies that $A^{-1}=\Theta$ where $\Theta$ is the $(\beth\times \aleph)\times (\beth\times \aleph)$ matrix defined by $x^T \Theta x=\| \sum_{(i,\alpha)\in \beth\times \aleph}x_{i,\alpha}\phi_{i,\alpha} \|_{*}^2$ for $x\in \R^{\beth\times \aleph}$. Using \eqref{equbarhkloc} we obtain that $1/\lambda_{\max}(A)=\lambda_{\min}(\Theta)\geq \ubar{H}^2$.
Similarly, using \eqref{eqh0or}, we have $1/\lambda_{\min}(A) \leq \sup_{x\in \R^{\beth\times \aleph}} \frac{\| \sum_{(i,\alpha)\in \beth\times \aleph}x_{i,\alpha}\phi_{i,\alpha} \|_{*}^2}{|x|^2}\leq \bar{H}_0^2 \bar{\gamma}$.
\end{proof}

\begin{Lemma}\label{lemkkjh8}
It holds true that $\|\psi_{i,\alpha}^0\|\leq \frac{1}{\ubar{H}_i \sqrt{\ubar{\gamma}_i}} $
\end{Lemma}
\begin{proof}
As in the proof of Theorem \ref{corunbcnOR} we have $\|\psi_{i,\alpha}^0\|^2= A^{i}_{i,i}$ where $A^{i}$ is the inverse of the $\I_i\times \I_i$ matrix defined by
$\Theta^i_{j,j'}:=[\Rc_i \phi_j,Q_i \phi_{j'}]$ and as  in \eqref{eqkhiduhdf7d}, $\lambda_{\max}(A^{i})\leq \frac{1}{\ubar{H}_i^2 \ubar{\gamma}_i}$.
\end{proof}

The proof of Theorem \ref{thmhgguyg65OR} is contained in that of the following theorem.
\begin{Theorem}\label{thmhgguyg65}
Under Condition \ref{confviperp}, it holds true that for $(i,\alpha),(j,\beta)\in \beth\times \aleph$,
\begin{equation}
|A_{(i,\alpha),(j,\beta)}|\leq C_{1} \exp\big(-C_2 \db(i,j)\big)
\end{equation}
with $C_2=\frac{1}{2} \ln \frac{\Cond(P)+1}{\Cond(P)-1}$ and
$C_1= \frac{\Cond(P)+1}{\Cond(P)-1}  \max_{i}\frac{1}{\ubar{H}_i^2}\max_{i}\frac{1}{\ubar{\gamma}_i}$.
\end{Theorem}
\begin{proof}
We have $A_{(i,\alpha),(j,\beta)}=\<\psi_{i,\alpha},\psi_{j,\beta}\>=\<\psi_{i,\alpha}-\psi_{i,\alpha}^n,\psi_{j,\beta}\>+ \<\psi_{i,\alpha}^n,\psi_{j,\beta}^{n}\>+ \<\psi_{i,\alpha}^n,\psi_{j,\beta}-\psi_{j,\beta}^{n}\>$. Therefore for
$2 n< \db(i,j)$ we have $\<\psi_{i,\alpha}^n,\psi_{j,\beta}^{n}\>=0$ and $|A_{(i,\alpha),(j,\beta)}|\leq \|\psi_{i,\alpha}-\psi_{i,\alpha}^n\| \|\psi_{j,\beta}\|+\|\psi_{i,\alpha}^n\|\|\psi_{j,\beta}-\psi_{j,\beta}^{n}\|$. Using the minimization property of gamblets we have $\|\psi_{i,\alpha}\|\leq \|\psi_{i,\alpha}^n\|\leq \|\psi_{i,\alpha}^0\|$ and $|A_{(i,\alpha),(j,\beta)}|\leq \|\psi_{i,\alpha}-\psi_{i,\alpha}^n\| \|\psi_{j,\beta}^0\|+\|\psi_{i,\alpha}^0\|\|\psi_{j,\beta}-\psi_{j,\beta}^{n}\|$.
 Taking $n=\db(i,j)/2-1$ and using Theorem \ref{thmswkskjsh} we deduce that
$|A_{(i,\alpha),(j,\beta)}| \leq \|\psi_{i,\alpha}^0\|\|\psi_{j,\beta}^0\| \big(\frac{\Cond(P)-1}{\Cond(P)+1}\big)^{\frac{\db(i,j)}{2}-1}$ which corresponds to Theorem \ref{thmhgguyg65OR}. We conclude using
Lemma \ref{lemkkjh8} to  bound $\|\psi_{i,\alpha}^0\|\|\psi_{j,\beta}^0\|$.
\end{proof}

\subsubsection{Control of the exponential decay of interpolation matrices}

The following theorem will allow us to control the exponential decay of the interpolation matrices $R^{(k,k+1)}$ defined in \eqref{eq:ftfytftfx}.
\begin{Theorem}\label{thmjgjuygg67}
Let $i, j\in \beth$, $\alpha\in \aleph$ and $\varphi \in \B^*$. If $[\varphi,\psi]=0$ for $\psi \in \B_{i}^{\db(i,j)-2}$  then
\begin{equation}
\big|[\varphi,\psi_{i,\alpha}]\big|\leq C_1 \exp\big(-C_2 \db(i,j)\big)  \|\varphi\|_*
\end{equation}
with $C_2= \ln \frac{\Cond(P)+1}{\Cond(P)-1}$ and
$C_1=\frac{1}{\ubar{H}_i \sqrt{\ubar{\gamma}_i}} \exp\big(2 C_2\big)$.
\end{Theorem}
\begin{proof}
Observe that $[\varphi,\psi_{i,\alpha}]=[\varphi,\psi_{i,\alpha}-\psi_{i,\alpha}^n]$ for $n\leq  \db(i,j)-2$.
We deduce that $|[\varphi,\psi_{i,\alpha}]|\leq \|\varphi\|_* \|\psi_{i,\alpha}-\psi_{i,\alpha}^n\|$ and  conclude using
Theorem \ref{thmswkskjsh} and Lemma \ref{lemkkjh8}.
\end{proof}

\subsubsection{Numerical homogenization with localized basis functions}

 Let $\bar{H}_0$ be defined as in \eqref{eqh0or} and
\begin{equation}\label{eqbarhkloc}
\bar{H}:= \sup_{\varphi \in \B_0} \inf_{\phi \in \Phi} \frac{\|\varphi- \phi\|_*}{\|\varphi\|_{0}}\,.
\end{equation}
The following theorem corresponds a  numerical homogenization result with localized basis functions.
\begin{Theorem}\label{thmdkehjgdkjdhj}
Let $u^n$ be the finite-element solution of \eqref{eqn:scalar} in $\Span\{\psi_{i,\alpha}^n\mid (i,\alpha)\in \beth \times \aleph\}$.
It holds true that for $n\geq C_1+C_2 \ln \frac{\sqrt{m}}{\bar{H} \min_i \ubar{H}_i }$,
\begin{equation}
 \frac{\|u-u^n\|}{\|g\|_{c}} \leq  2 \bar{H},
\end{equation}
with $C_1=\big(\ln \frac{\Cond(P)+1}{\Cond(P)-1}\big)^{-1} \ln \Big( \frac{ C_{\L^{-1}} \sqrt{\bar{\gamma}} }{ \min_i \sqrt{\ubar{\gamma}_i}  }  \bar{H}_0^2\Big)$, $C_2=\big(\ln \frac{\Cond(P)+1}{\Cond(P)-1}\big)^{-1} $.
\end{Theorem}
\begin{proof}
Let $\bar{u}=\sum_{(i,\alpha)\in \beth \times \aleph}^m c_{i,\alpha} \psi_{i,\alpha}$ be the finite-element solution of \eqref{eqn:scalar} in $\Span\{\psi_{i,\alpha}\mid (i,\alpha)\in \beth \times \aleph\}$.
We have $\|u-u^n\|\leq \|u-\bar{u}\|+\sum_{(i,\alpha)\in \beth \times \aleph} |c_{i,\alpha}|\|\psi_{i,\alpha}-\psi_{i,\alpha}^n\|$.
Using Theorem \ref{thmsudgdygdgy} we have
$\frac{\|u-\bar{u}\|}{\|g\|_{c}} \leq  \bar{H}$.
Using the variational property expressed in Theorem \ref{thmsudgdygdgy} we also have
$\lambda_{\min}(A) |c|^2\leq \|\sum_{(i,\alpha)\in \beth \times \aleph} c_{i,\alpha}\psi_{i,\alpha}\|^2 \leq \|u\|^2 \leq C_{\L^{-1}}^2 \|g\|_2^2$. Recall that $\|g\|_2 \leq \bar{H}_0 \|g\|_{c}$.
Using Theorem \ref{thmswkskjsh} and Lemma \ref{lemkkjh8} we have
 $\|\psi_{i,\alpha}-\psi_{i,\alpha}^n\| \leq \big(\frac{\Cond(P)-1}{\Cond(P)+1}\big)^n \frac{1}{\ubar{H}_i \sqrt{\ubar{\gamma}_i}}$.
Using Lemma \ref{lemljkhlkhhu} we deduce that
$\|u-u^n\|\leq \bar{H}\|g\|_{c} + \sqrt{m} C_{\L^{-1}} \bar{H}_0\|g\|_{c}  \bar{H}_0 \sqrt{\bar{\gamma}} \big(\frac{\Cond(P)-1}{\Cond(P)+1}\big)^n \max_i \frac{1}{\ubar{H}_i \sqrt{\ubar{\gamma}_i}}$ and we conclude the proof after simplification.
\end{proof}

\subsubsection{Proof of Corollary \ref{corjff76f57bis}}
Note that $\B_0=L^2(\Omega)$. The proof is  a consequence of Theorems \ref{thmjff76f57} and \ref{thmdkehjgdkjdhj}. As in Examples \ref{egkdejkdhdjk}, \ref{egkdejkdhkjhdjk} and \ref{egkdejkdhdjkbis}, using Lemma \ref{lemduegduey3}, Lemma \ref{lemduegduswsaey3} or Proposition \ref{propkajhdlkjd}, we obtain the  higher order Poincar\'{e}, compact embedding and inverse Sobolev inequalities,
$\sup_{g \in L^2(\Omega)} \inf_{\phi \in \Phi} \frac{\|g- \phi\|_{H^{-s}(\Omega)}}{\|g\|_{L^2(\Omega)}}\leq C h^s$, $\overset{*}{H}\leq C$, $\underset{*}{H}\geq C^{-1} h^{-s}$.
For Examples \ref{egkdejkdhdjk} and \ref{egkdejkdhkjhdjk}, the $L^2(\Omega)$-orthonormaliy of the $\phi_{i,\alpha}$ implies that
$\| \sum_{(i,\alpha)\in \beth\times \aleph}x_{i,\alpha}\phi_{i,\alpha} \|_{L^2(\Omega)}^2=|x|^2$ for $x\in \R^{\beth\times \aleph}$.
For Example \ref{egkdejkdhdjkbis}, the
 proof and statement of \ref{thmdkehjgdkjdhj} can be modified to remove, as in Lemma \ref{lemljkhlkhhu}, the
 dependence of the constant $C$ in Theorem \ref{thmdkehjgdkjdhj} on $\bar{\gamma}, \ubar{\gamma}_i$,
by using Theorem \ref{thmjff76f57}  to control  localization error (and the arguments in the proofs of examples \ref{egkdejkdhdjk}, \ref{egkdejkdhkjhdjk} and \ref{egkdejkdhdjkbis}).

\section{Proofs of the results of Section \ref{secgamtrinrn}}\label{secproofsecgam}
\subsection{Proof of Theorems \ref{thmdpedjoejdo} and \ref{thmconddisbndisbismatdisQQ}}
Theorems \ref{thmdpedjoejdo} and \ref{thmconddisbndisbismatdisQQ} are a direct consequence of Theorem \ref{thmconddisbndisbismatdis}, Remark \ref{rmknt} and Theorem \ref{thmcondlnum}.

\subsection{Proof of Theorem \ref{thmljshjhxj6QQ}}
Theorem \ref{thmljshjhxj6QQ} is a direct consequence of \eqref{corunbcnORfe} in Theorem \ref{corunbcnOR}.

\subsection{Proof of Theorem \ref{thmfix1}}
Theorem \ref{thmfix1} is a direct consequence of Theorem \ref{thmsudgdygdgy}.

\subsection{Proof of Theorem \ref{thmfix2}}
Theorem \ref{thmfix2} is a direct consequence of Theorems \ref{thmgugyug0OR} and \ref{thmgugyug2OR}.

\subsection{Proof of Theorem \ref{thmfix3}}
Theorem \ref{thmfix3} is a direct consequence of Theorem \ref{corunbcnOR}.

\subsection{Proof of Proposition \ref{propkajhdlkjdsob}}
Proposition \ref{propkajhdlkjdsob} is a direct consequence of Proposition \ref{propkajhdlkjd}.

\subsection{Proof of Theorem \ref{thmfix4}}
Theorem \ref{thmfix4} is a direct consequence of the discussion in Example \ref{egdkj3hrr}, Corollary \ref{corthmlkkjhkdfual}, Lemma \ref{lemdkljedhlkfjh}, and Theorem \ref{thmswkskjsh}.

\subsection{Proof of Theorem \ref{thmegdkj3hrrsuitebis}}
Theorem \ref{thmegdkj3hrrsuitebis} is a direct consequence of  Theorem \ref{thmegdkj3hrrsuite}.

\subsection{Proof of Theorem \ref{thmpropegkdejkdhdjkbis}}\label{proofthmpropegkde}

\begin{Proposition}\label{propegkdejkdhkjhdjk}
The measurement functions $(\phi_{i,\alpha})_{(i,\alpha)\in \beth\times \aleph}$ of Example \ref{egkdejkdhdjk} satisfy Condition \ref{condmeasfuncloc}.
\end{Proposition}
\begin{proof}
\eqref{eqconlocmeas} is a classical higher order Poincar\'{e} inequality \cite[Lem.~6.4]{john2013numerical}.
Using the $L^2(\tau_i)$ orthonormality of the $\phi_{i,\alpha}$
we have, for $f\in H^s_0(\Omega)$, $\sum_{\alpha \in \aleph}[\phi_{i,\alpha},f]^2\leq \|f\|^2_{L^2(\tau_i)}$, which proves \eqref{eqconlocmeasnext}. The proof of \eqref{eqkjdhddkuieheu} is identical to that of inequality  (2) in the proof of Proposition \ref{propkajhdlkjd}.
\end{proof}

The following lemma,  is a consequence of \cite[Thm.~2]{madpotter85} (see also \cite[Cor.~1]{madpotter85}).
\begin{Lemma}\label{lemduegduey3}(Madych and Potter)
Let $\Omega$ be a bounded Lipschitz domain of $\R^d$  and $s\in \mathbb{N}^*$ such that $s> d/2$.
Let $\omega$ be a subset of $\Omega$ and write $\bar{h}:=\sup_{x\in \Omega}\inf_{y\in \omega}|x-y|$.
 There exists a constant $C$ (depending only on $s$ and $d$) such that if $f=0$ on $\omega$ and $f\in H^s_0(\Omega)$ then for $t\in \{0,1,\ldots,s\}$,
 \begin{equation}
 \|D^t f\|_{L^2(\Omega)}\leq C \bar{h}^{s-t} \|f\|_{H^s_0(\Omega)}\,.
 \end{equation}
\end{Lemma}

\begin{Proposition}\label{propegkdejkdhdjkbis}
The measurement functions $\phi_1,\ldots,\phi_m$ of Example \ref{egkdejkdhdjkbis} satisfy Condition \ref{condmeasfuncloc}.
\end{Proposition}
\begin{proof}
 \eqref{eqconlocmeas} is a higher order Poincar\'{e} inequality  implied by Lemma \ref{lemduegduey3}.
 \eqref{eqconlocmeasnext} corresponds to the inequality
 $\sum_{i=1}^m h^d  \big(f(x_i)\big)^2\leq  C \big( \|f\|_{L^2(\Omega)}^2+ h^{2s}\|f\|_{H^s_0(\Omega)}^2\big)$ for $f\in H^s_0(\Omega)$
 (whose proof is fairly classical, we refer to to \cite{brownlee2004error}, the proof can also be derived by slightly adapting those of \cite[Lem.~1]{madpotter85} and \cite[Thm.~1]{madpotter85}).
 To prove \eqref{eqkjdhddkuieheu} we use
 \eqref{eqjdhejhdgidddeyu} with
 $v \in H^s_0(B(x_i,\delta h))$ such that $v(x_i)= 1$ and $\| v \|_{H^s_0(\tau_i)}^2 \leq C h^d h^{-2s}$.
\end{proof}

\begin{Lemma}\label{lemduegduswsaey3}
Let $\Omega$ be a bounded Lipschitz domain of $\R^d$ and $s\in \mathbb{N}^*$. Let $\tau_1, \ldots, \tau_m$ be a partition
of $\Omega$ such that each $\tau_i$ is convex, contained in a ball of radius $h>0$ and contains a ball of radius $\delta h$ ($\delta \in (0,1)$).
 There exists a constant $C$ (depending only on $s, d$ and $\delta$) such that if $\int_{\tau_i}f=0$ for $i\in \{1,\ldots,m\}$ and $f\in H^s_0(\Omega)$ then for $t\in \{0,1,\ldots,s\}$,
 \begin{equation}\label{egjjkkhjh2344}
 \|D^t f\|_{L^2(\Omega)}\leq C h^{s-t} \|f\|_{H^s_0(\Omega)}\,.
 \end{equation}
\end{Lemma}
\begin{proof}
Although \eqref{egjjkkhjh2344} looks like a classical Poincar\'{e} inequality we have not found a precise reference for it. However \eqref{egjjkkhjh2344} can be proven by adapting (in a straightforward manner)   the proof of \cite[Thm.~2]{madpotter85}. The main step is in this adaptation is a simple generalization of \cite[Lem.~1]{madpotter85} obtained by integrating \cite[eq.~4]{madpotter85} over $x_0$.
\end{proof}

\begin{Proposition}\label{propegkdejkdhdjk}
The measurement functions $\phi_1,\ldots,\phi_m$ of Example \ref{egkdejkdhdjk} satisfy Condition \ref{condmeasfuncloc}.
\end{Proposition}
\begin{proof}
\eqref{eqconlocmeas} is a higher order Poincar\'{e} inequality implied by Lemma \ref{lemduegduswsaey3}.
Since the $\phi_i$ are orthonormal in $L^2(\Omega)$ we have
 $\sum_{i=1}^m  [\phi_i,f]^2 \leq \|f\|_{L^2(\Omega)}^2 $ which implies \eqref{eqconlocmeasnext}.
For $i\in \{1,\ldots,m\}$,
\begin{equation}\label{eqjdhejhdgidddeyu}
\| \phi_{i}\|_{H^{-s}(\tau_i)}^2=\sup_{v\in H^s_0(\tau_i)} \frac{[\phi_i, v]}{\|v\|_{H^s_0(\tau_i)}} \,.
\end{equation}
 Since $\tau_i$ contains a ball of radius $\delta h$,
one can select  $v\in H^s_0(\tau_i)$ such that, $|\tau_i|^{-1}\int_{\tau_i} v =1$ and
$\|\psi\|_{H^s_0(\tau_i)}\leq C h^{-s} \sqrt{|\tau_i|}$ (where $C$ depends only on $d, \delta$ and $s$). We deduce from
\eqref{eqjdhejhdgidddeyu} that $C h^{2s} \| \phi_{i}\|_{H^{-s}(\tau_i)}^2 \geq  1$, which proves \eqref{eqkjdhddkuieheu}.
\end{proof}

\section{Proofs of the results of Section \ref{secfgtas}}\label{sec10}

\subsection{Error propagation across scales}

We will first, in this subsection, control the error propagation across scales caused by the localization of the computation of gamblets, introduced in Section \ref{sec1or}. The following results are not limited to the discretization of $\B$ and can be applied when  $\B$ is  infinite dimensional.

\begin{Condition}\label{condexpdecay}
Let $H\in (0,1)$ be the constant of Condition  \ref{cond1OR} or Condition \ref{conddiscrip3ordismatdis}. There exists  constants $d,\zeta_1,\zeta_2\in (0,\infty)$ and a constant $C_{\loc}>0$ such that for $k\in \{1,\ldots,q-1\}$ there exists a pseudo-metric $\db^{(k)}$ on $\I^{(k)}$   such that (1) $\Card(\I^{(k)})\leq C_{\loc} H^{-kd}$ and (2) for $i\in \I^{(k)}$,
 $\Card\{j: \db^{(k)}(i,j)\leq \rho\}\leq C_{\loc} \rho^d$ and (3) $|R^{(k,k+1)}_{i,j}|\leq C_{\loc} H^{-k\zeta_1} e^{-C_{\loc}^{-1}\db^{(k)}(i,j^{(k)})}$ for $(i,j)\in \I^{(k)}\times \I^{(k+1)}$ (4) $|A^{(k)}_{i,j}|\leq C_{\loc}  H^{-2k \zeta_2} e^{-C_{\loc}^{-1}\db^{(k)}(i,j)}$ for $(i,j)\in \I^{(k)}\times \I^{(k)}$.
\end{Condition}
\begin{Remark}
Under Condition \ref{cond7fyf}, Item (3) of Condition \ref{condexpdecay} is implied by Item (4) of Condition \ref{condexpdecay}.
Indeed, using \eqref{eqhuhiddeuv} and \eqref{eqjdhiudhiue} we have $R^{(k+1,k)}=(I^{(k+1)}-  W^{(k+1),T} N^{(k+1),T}) \bar{\pi}^{(k+1,k)}$ with
$N^{(k+1),T}= B^{(k+1),-1}W^{(k+1)} A^{(k+1)}$.
Furthermore, the  off-diagonal exponentially decay of $A^{(k+1)}$ implies that of $B^{(k+1)}$. Since the condition number of $B^{(k+1)}$is uniformly bounded one can obtain a uniform bound on the off-diagonal exponentially decay of  $B^{(k+1),-1}$ (the proof is  an adaptation of that of \cite{Jaffard90}, we also refer to \cite{Demko1984, Nabben99, Bebendorf:2008} for related results on the off-diagonal exponential decay of the entries of the inverse of well conditioned sparse matrices).
\end{Remark}

We will, from now on, assume that the matrices $\pi^{(k-1,k)}$ and $W^{(k)}$ are cellular as defined in
 Condition \ref{cond7fyf} and that $|\aleph|$ is bounded independently from $q$.
Let $\psi_i^{(q),\loc}$ be approximations of the gamblets $\psi_i^{(q)}$ introduced in Section \ref{sec1or}.
Let $\V^{(q),\loc}:=\Span\{\psi_i^{\loc}\mid i \in \I^{(q)}\}$.
For $k\in \{1,\ldots,q-1\}$ define (by induction) $\psi_i^{(k),\loc}$ (and  $\V^{(k),\loc}:=\Span\{\psi_i^{\loc}\mid i\in \I^{(k)}\}$)  as the minimizer of
\begin{equation}\label{eqpdsiuikloc}
\begin{cases}
\text{Minimize } \|\psi\| \\ \text{Subject to } \psi\in \V^{(k+1),\loc}_i \text{ and } [ \phi_j^{(k)},\psi]=\delta_{i,j} \text{ for }\db^{(k)}(i,j)\leq \rho_k
\end{cases}
\end{equation}
where for $i\in \I^{(k)}$, $\V^{(k+1),\loc}_i:=\Span\{\psi_j^{(k+1)}\mid j \in \I^{(k+1)}\text{ and } \db^{(k)}(j^{(k)},i)\leq \rho_k\}$.

To simplify the presentation,  we will from now on, write $C$ any constant that depends only on $C_{\loc}, C_\d, d, \zeta_1, \zeta_2,|\aleph|$ and $C_\Phi$  (e.g., $2 C C_\Phi$ will still be written $C$).
We will  need the following lemma.
\begin{Lemma}\label{lemfyfyfyvh}
Let $k\in \{1,\ldots,q-1\}$ and let $R$ be the $\I^{(k)}\times \I^{(k+1)}$ matrix defined by and $R_{i,j}=0$ for $j: \db^{(k)}(i,j^{(k)}) \leq \rho_k$ and $R_{i,j}=R_{i,j}^{(k,k+1)}$ for $j: \db^{(k)}(i,j^{(k)}) > \rho_k$. Under Condition \ref{condexpdecay} it holds true that $\|R\|_2 \leq  C H^{-k (d+\zeta_1)}  e^{- \rho_k/C }$.
\end{Lemma}
\begin{proof}
Observe that $\|R\|_2^2 \leq |\I^{(k)}| \max_{i\in \I^{(k)}}\sum_{j: \db^{(k)}(i,j^{(k)}) > \rho_k} |R_{i,j}^{(k,k+1)}|^2$.
Therefore under Condition \ref{condexpdecay}, $\|R\|_2^2 \leq C H^{-2k(d+\zeta_1)}e^{-C^{-1}\rho_k}$.
\end{proof}

We will need the following lemma summarizing and simplifying some results obtained in Section \ref{sec1or}.
\begin{Lemma}\label{lembase}
Let Condition \ref{cond1OR} or Condition \ref{conddiscrip3ordismatdis} be satisfied. It holds true that (1) $\operatorname{Cond}(A^{(1)})\leq C H^{-2}$ (2) $1/\lambda_{\min}(A^{(1)})\leq C $
and (3) for $k\in \{1,\ldots,q\}$, $\lambda_{\max}(A^{(k)})\leq C H^{-2k}$. Furthermore, for $k\in \{q,\ldots,2\}$, (4) $\Cond(B^{(k)})\leq C H^{-2}$ (5)  $ \lambda_{\max}(B^{(k)})\leq C H^{-2k}$ (6)
$ 1/\lambda_{\min}(B^{(k)})\leq C H^{2k-2}$ (7) $\|\pi^{(k-1,k)}\|_2\leq C$ (8) $\|\bar{\pi}^{(k-1,k)}\|_2\leq C$ (9)  $\|W^{(k)}\|_2\leq C$ (10) $1/\lambda_{\min}(W^{(k)} W^{(k),T}) \leq C $ and
(11) $\|R^{(k-1,k)}\|_2 \leq C$.
\end{Lemma}
\begin{proof}
(1)-(6) follow from Theorem \ref{corunbcnOR}. (7) and (8) follow from Lemma \ref{lemdekjdhkdjhd} (for (8) note that $\|\bar{\pi}^{(k-1,k)}\|_2 \leq \|\pi^{(k-1,k)}\|_2/\lambda_{\min}(\pi^{(k-1,k)}\pi^{(k,k-1)})$). (9) and (10) follow from Condition \ref{cond1OR} or Condition \ref{conddiscrip3ordismatdis}.
 Using
$R^{(k-1,k)}= \bar{\pi}^{(k-1,k)}- \bar{\pi}^{(k-1,k)}N^{(k)}  W^{(k)}$ (obtained in \eqref{eqhuhiddeuv}) we have
$\|R^{(k-1,k)}\|_2 \leq \|\bar{\pi}^{(k-1,k)}\|_2 (1+ \|N^{(k)}\|_2 \|W^{(k)}\|_2)$.
Using Theorem  \ref{lemdjkdj} we have
$\|N^{(k)}\|_2\leq C$. Summarizing we have obtained (11).
\end{proof}

 For $k\in \{1,\ldots,q\}$, let $A^{(k),\loc}$ be the $\I^{(k)}\times \I^{(k)}$ matrix defined by $A^{(k),\loc}_{i,j}:=\<\psi^{(k),\loc}_i,\psi^{(k),\loc}_j\>$ and let $\er(k)$ be the (localization) error
 $\er(k):=\big(\sum_{i\in \I^{(k)}} \|\psi_i^{(k)}-\psi_i^{(k),\loc}\|^2\big)^\frac{1}{2}$.

The following theorem allows us to control the localization error propagation across scales.
\begin{Theorem}\label{thmerrorpropagation}
Let Condition \ref{cond1OR} or Condition \ref{conddiscrip3ordismatdis} be satisfied. Under Condition \ref{condexpdecay}, it holds true for $k\in \{1,\ldots,q-1\}$ that
\begin{equation}
\er(k)\leq C  \er(k+1) (1+H^{-k(d+\zeta_1)} e^{-\rho_{k}/C})+ C H^{-(1+d/2)-k (2+3d+\zeta_1)/2} e^{-\rho_{k}/C}   \,.
\end{equation}
\end{Theorem}
\begin{proof}
Let $k\in \{1,\ldots,q-1\}$ and $i\in \I^{(k)}$.
We obtain by induction (using the constraints in \eqref{eqpdsiuikloc}) that  $\psi_i^{(k),\loc}$ satisfies the constraints of \eqref{eq:dfddeytfewdaisq}. Note that  if $\psi$ satisfies the constraints of \eqref{eq:dfddeytfewdaisq} then
$\|\psi\|^2= \|\psi_i^{(k)}\|^2+\|\psi- \psi_i^{(k)}\|^2$.
Therefore  $\psi_i^{(k),\loc}$ is also the minimizer of $\|\psi- \psi_i^{(k)}\|$ over functions $\psi$ of the form $\psi=\sum_{j: \db^{(k)}(i,j^{(k)}) \leq \rho_k} c_j \psi_j^{(k+1),\loc}$ satisfying the constraints of \eqref{eqpdsiuikloc}. Thus, writing $\psi^*:= \sum_{j: \db^{(k)}(i,j^{(k)}) \leq \rho_k} R_{i,j}^{(k,k+1)} \psi_j^{(k+1),\loc} $,  we have (since $\psi^*$ satisfies the constraints of \eqref{eqpdsiuikloc})
$\| \psi_i^{(k),\loc}-\psi_i^{(k)}\| \leq \| \psi^*-\psi_i^{(k)}\|$. Write
  $\psi_1:= \sum_{j \in \I^{(k+1)}} R_{i,j}^{(k,k+1)} \psi_j^{(k+1),\loc} $ and  $\psi_2:=\sum_{j: \db^{(k)}(i,j^{(k)}) > \rho_k} R_{i,j}^{(k,k+1)} \psi_j^{(k+1),\loc}$.
Observing that $\psi^*=\psi_1-\psi_2$ we deduce that $\| \psi_i^{(k),\loc}-\psi_i^{(k)}\|^2 \leq 2 \| \psi_1-\psi_i^{(k)}\|^2 + 2 \| \psi_2\|^2$. Summing over $i$ we obtain that
$\big(\er(k)\big)^2\leq 2(I_1+I_2)$ with $I_1= \sum_{i \in  \I^{(k)}} \|\sum_{j \in  \I^{(k+1)}} R_{i,j}^{(k,k+1)} (\psi_j^{(k+1)}-\psi_j^{(k+1),\loc})\|^2$
and $I_2=\sum_{i \in  \I^{(k)}}  \|\sum_{j: \db^{(k)}(i,j^{(k)}) > \rho_k} R_{i,j}^{(k,k+1)} \psi_j^{(k+1),\loc}\|^2$.
Writing $S$ the $\I^{(k+1)}\times \I^{(k+1)}$ symmetric positive matrix with entries $S_{i,j}=\<\psi_i^{(k+1)}-\psi_i^{(k+1),\loc},\psi_j^{(k+1)}-\psi_j^{(k+1),\loc}\>$, note that $I_1=\Tr[R^{(k,k+1)}S R^{(k+1,k)}]$. Writing $S^\frac{1}{2}$ the matrix square root of $S$, observe that for a matrix $U$, using the cyclic property of the trace, $\Tr[U S U^T]=\Tr[S^\frac{1}{2} U^T U S^\frac{1}{2}]\leq \lambda_{\max}(U^T U) \Tr[S]$, which (observing that $\Tr[S]=(\er(k+1))^2$ and $\lambda_{\max}(U^T U)=\|U\|_2^2$) implies
$I_1 \leq \|R^{(k,k+1)}\|_2^2 \big(\er(k+1)\big)^2$. Therefore (using Lemma \ref{lembase}) we have  $\sqrt{I_1} \leq C  \er(k+1)$.
Let us now bound $I_2$. Let $R$ be defined as in Lemma \ref{lemfyfyfyvh}. Noting that $\<\psi_i^{(k+1),\loc},\psi_j^{(k+1),\loc}\>=A^{(k+1),\loc}_{i,j}$ we have (as above)
$I_2=\Tr[R A^{(k+1),\loc} R^T]\leq \lambda_{\max}(R^T R) \Tr[A^{(k+1),\loc}]$. Summarizing and using Lemma \ref{lemfyfyfyvh} we deduce that
$\er(k)\leq C  \er(k+1)+ C H^{-k(d+\zeta_1)} e^{-\rho_{k}/C}  \sqrt{\Tr[A^{(k+1),\loc}]}$. Observing that
$\sqrt{\Tr[A^{(k+1),\loc}]}\leq \er(k+1)+\sqrt{\Tr[A^{(k+1)}]}$ and (from Condition \ref{condexpdecay} and Lemma \ref{lembase})  $\Tr[A^{(k+1)}]\leq C |\I^{(k+1)}| \lambda_{\max}(A^{(k+1)})\leq C H^{-(d+2)(k+1)}$,
we conclude the proof of the theorem.
\end{proof}

For $k\in \{1,\ldots,q\}$, let  $u^{(1),\loc}$ be the finite-element solution of \eqref{eqn:scalar} in $\V^{(1),\loc}:=\Span\{\psi_j^{(k),\loc}\mid j\in \I^{(1)}\}$.
For $k\in \{2,\ldots,q\}$ and
For $i\in \J^{(k)}$, let $\chi^{(k),\loc}_i:=\sum_{j \in \I^{(k)}} W_{i,j}^{(k)} \psi_j^{(k),\loc}$.
For $k\in \{2,\ldots,q\}$ let $u^{(k),\loc}-u^{(k-1),\loc}$ be the finite element solution of \eqref{eqn:scalar} in $\W^{(k),\loc}:=\Span\{\chi_j^{(k),\loc}\mid j\in \J^{(k)}\}$.
For $k\in \{2,\ldots,q\}$, write $u^{(k),\loc}:=u^{(1),\loc}+\sum_{j=2}^k (u^{(j),\loc}-u^{(j-1),\loc})$.
Let $B^{(k),\loc}$ be the $\J^{(k)}\times \J^{(k)}$ matrix defined by $B^{(k),\loc}_{i,j}:=\<\chi^{(k),\loc}_i,\chi^{(k),\loc}_j\>$.
Observe that $B^{(k),\loc}= W^{(k)}A^{(k),\loc}W^{(k),T}$.
Write for $k\in \{2,\ldots,q\}$, $\er(k,\chi):=\big(\sum_{j\in \J^{(k)}} \|\chi_j^{(k)}-\chi_j^{(k),\loc}\|^2\big)^\frac{1}{2}$.

We will  need the following lemma.
\begin{Lemma}\label{lemshgjhgdhg3e}
Let $\chi_1,\ldots,\chi_m$ be linearly independent elements of $\B$. Let $\chi_1',\ldots,\chi_m'$ be another set of linearly independent elements of $\B$. Write $\er:=\big(\sum_{i=1}^m \|\chi_i-\chi_i'\|^2\big)^\frac{1}{2}$. Let $B$ (resp. $B'$) be the $m\times m$  matrix defined by
$B_{i,j}=\<\chi_i,\chi_j\>$ (resp. $B_{i,j}'=\<\chi_i',\chi_j'\>$). Let $u_m$ (resp. $u_m'$) be the finite-element solution of \eqref{eqn:scalar} in $\Span\{\chi_i\mid i=1,\ldots,m\}$ (resp. $\Span\{\chi_i'\mid i=1,\ldots,m\})$. It holds true that for  $\er \leq \sqrt{\lambda_{\min}(B)} /2$ (1)  $\Cond(B')\leq 8 \Cond(B)$  (2)
$\|B-B'\|_2 \leq 3 \sqrt{\lambda_{\max}(B)} \er$ (3) $\|B^{-1}-(B')^{-1}\|_2 \leq 12 \sqrt{\lambda_{\max}(B)}  \big(\lambda_{\min}(B)\big)^{-2} \er$  and (4)
$\|u_m-u_m'\|\leq 5 C_{\L^{-1}} \er \|g\|_2\frac{\Cond(B)}{\sqrt{\lambda_{\min}(B)}}$.
\end{Lemma}
\begin{proof}
For (1) observe that $\sqrt{\lambda_{\max}(B')}=\sup_{|x|=1}\|\sum_{i=1}^m x_i \chi_i'\|\leq \sqrt{\lambda_{\max}(B)}+ \er$ and $\sqrt{\lambda_{\min}(B')}=\inf_{|x|=1}\|\sum_{i=1}^m x_i \chi_i'\|\geq \sqrt{\lambda_{\min}(B)}
- \er$. For (2) observe that for $x,y\in \R^m$ with $|x|=|y|=1$ we have $y^T(B-B')x=\<\sum_{i=1}^m y_i (\chi_i-\chi_i'),\sum_{i=1}^m x_i \chi_i\>-\<\sum_{i=1}^m y_i \chi_i',\sum_{i=1}^m x_i (\chi_i'-\chi_i)\>\leq (\sqrt{\lambda_{\max}(B')}+\sqrt{\lambda_{\max}(B)})\er$.
(3) follows from (2) and $\|B^{-1}-(B')^{-1}\|_2 \leq \|B-B'\|_2/\big(\lambda_{\min}(B) \lambda_{\min}(B')\big)$.
For (4) observe that $u_m=\sum_{i=1}^m w_i\chi_i$ (resp. $u_m'=\sum_{i=1}^m w_i'\chi_i'$) where $w=B^{-1} b$ with $b_i=[Q^{-1}\L^{-1} g, \chi_i]$ (resp. $w'=(B')^{-1} b'$ with $b_i'=[Q^{-1}\L^{-1} g, \chi_i']$). Therefore $\|u_m-u_m'\|\leq |w| \er+|w-w'|\sqrt{\lambda_{\max}(B)}$. Observe that $|b_i-b_i'|=\big|[Q^{-1}\L^{-1} g,\chi_i-\chi_i']\big|\leq \|Q^{-1}\L^{-1}g\|_* \|\chi_i-\chi_i'\|$ and $\|Q^{-1}\L^{-1}g\|_*\leq \|\L^{-1}g\|\leq C_{\L^{-1}}\|g\|_2$.
Therefore,
 $w-w'=B^{-1}(b-b')-B^{-1}(B-B')w'$ leads to
$|w-w'|\leq (C_{\L^{-1}} \|g\|_2\er +\|B-B'\|_2 |w'|)/\lambda_{\min}(B)$. Using (2),
  $\lambda_{\min}(B) |w|^2\leq \|\sum_{i=1}^m w_i\chi_i\|^2 \leq \|u\|^2 \leq C_{\L^{-1}}^2 \|g\|_2^2$, and $\lambda_{\min}(B') |w'|^2\leq C_{\L^{-1}}^2 \|g\|_2^2$ we obtain  $\|u_m-u_m'\|\leq C_{\L^{-1}} \|g\|_2 \er/\sqrt{\lambda_{\min}(B)} +\sqrt{\lambda_{\max}(B)} (C_{\L^{-1}} \|g\|_2\er +3 \sqrt{\Cond(B)} C_{\L^{-1}} \|g\|_2)/\lambda_{\min}(B)$ and conclude the proof of (4) after simplification.
\end{proof}

The following lemma allows
us to control the effect of the localization error on the approximation of the solution of \eqref{eqn:scalar}.
\begin{Lemma}\label{thmdhdjh3}
Let Condition \ref{cond1OR} or Condition \ref{conddiscrip3ordismatdis} be satisfied. Let Condition \ref{condexpdecay} be also satisfied.
It holds true that for $k\in \{2,\ldots,q\}$ (1) $\er(k,\chi) \leq C \er(k)$. Furthermore for $k\in \{2,\ldots,q\}$ and $\er(k,\chi)\leq C^{-1} H^{1-k}$  we have (2)  $\Cond(B^{(k),\loc})\leq C H^{-2}$, and (3) $\|u^{(k)}-u^{(k-1)}-(u^{(k),\loc}-u^{(k-1),\loc})\| \leq C \er(k,\chi) \|g\|_{2} H^{k-3}$. Similarly for $\er(1)\leq C^{-1} $,  we have (4)   $\Cond(A^{(1),\loc})\leq C H^{-2}$, and (5)
$\|u^{(1)}-u^{(1),\loc}\| \leq C \er(1) \|g\|_{2} H^{-2}$.
\end{Lemma}
\begin{proof}
Let $S$ be the $\I^{(k)}\times \I^{(k)}$ matrix defined by $S_{i,j}=\<\psi_i^{(k)}-\psi_i^{(k),\loc}, \psi_j^{(k)}-\psi_j^{(k),\loc}\>$.
Using $\chi_j^{(k)}-\chi_j^{(k),\loc}=\sum_{i\in \I^{(k)}} W^{(k)}_{j,i} (\psi_i^{(k)}-\psi_i^{(k),\loc})$ and the cyclic property of the trace we have
$\big(\er(k,\chi)\big)^2=\Tr[ W^{(k),T}W^{(k)} S]=\Tr[  S^\frac{1}{2}W^{(k),T}W^{(k)} S^\frac{1}{2}]\leq \lambda_{\max}(W^{(k),T}W^{(k)})\Tr[S]$, which combined with $\big(\er(k)\big)^2=\Tr[S]$ implies (1).
 (2) and (3) are a direct application of lemmas \ref{lemshgjhgdhg3e} and \ref{lembase}. For (3), observe that
$u^{(k)}-u^{(k-1)}$ (resp. $u^{(k),\loc}-u^{(k-1),\loc}$) is the finite element solution of \eqref{eqn:scalar} in $\W^{(k)}$ (resp. $\W^{(k),\loc}:=\Span\{\chi_j^{(k),\loc}\mid j\in \J^{(k)}\}$). The proof of (4) and (5) is similar to that of (2) and (3).
\end{proof}

\begin{Theorem}\label{thmdjjuud}
Let Condition \ref{cond1OR} or Condition \ref{conddiscrip3ordismatdis} be satisfied. Let Condition \ref{condexpdecay} be also satisfied.
Let $k\in \{1,\ldots,q\}$. If $\rho_j \geq C (j d/2-2-d)\ln(1/H)$ for $j\in \{k,\ldots,q-1\}$
then
\begin{equation}
\er(k)\leq C \big(\sum_{j=k}^{q-1} e^{-\rho_{j}/C} C^{j-k}H^{-(1+d/2)-j (2+3d+\zeta_1)/2}+ C^{q-k}\er(q)\big) .
\end{equation}
\end{Theorem}
\begin{proof}

By Theorem \ref{thmerrorpropagation}, for $k\in \{1,\ldots,q-1\}$ if $H^{-kd} e^{-\rho_{k}/C})\leq 1$ (i.e $\rho_k \geq C k d\ln(1/H)$)
then $\er(k)\leq a_k+b_k \er(k+1)$ with $a_k=C  e^{-\rho_{k}/C}  H^{-(1+d/2)-k (2+3d+\zeta_1)/2}$ and $b_k=C$. Therefore we obtain by induction that
$\er(k) \leq a_k + b_k a_{k+1}+ b_k b_{k+1} a_{k+2}+\cdots + b_k \cdots b_{q-2}a_{q-1}+ b_k \cdots b_{q-1} \er(q)$ and derive the result after simplification.
\end{proof}

\begin{Theorem}\label{tmshjgeydg}
Let Condition \ref{cond1OR} or Condition \ref{conddiscrip3ordismatdis} be satisfied. Let Condition \ref{condexpdecay} be also satisfied.
Let $\epsilon \in (0,1)$. It holds true that if $\er(q)\leq C^{-q} H^2 \epsilon$ and
$\rho_k\geq C \big((1+\frac{1}{\ln(1/H)})\ln \frac{1}{H^k}+\ln \frac{1}{\epsilon}\big)$ for $k\in \{1,\ldots,q-1\}$ then (1) for $k\in \{1,\ldots,q\}$ we have
$\|u^{(k)} - u^{(k),\loc}\| \leq   \epsilon \|g\|_{2}$ and $\|u - u^{(k),\loc}\| \leq   C (H^k+\epsilon) \|g\|_{c}$ (2) $\Cond(A^{(1),\loc})\leq C H^{-2}$, and for $k\in \{2,\ldots,q\}$ we have (3)
 $\Cond(B^{(k),\loc})\leq C H^{-2}$ and (4) $\|u^{(k)}-u^{(k-1)}-(u^{(k),\loc}-u^{(k-1),\loc})\| \leq  \frac{\epsilon}{2 k^2} \|g\|_{2}$.
\end{Theorem}
\begin{proof}
Theorems \ref{thmgugyug0OR} and \ref{thmdhdjh3} imply that the results of Theorem \ref{tmshjgeydg} hold true if for $k\in \{1,\ldots,q\}$,
$\er(k)\leq C^{-1} H^{3-k} \epsilon/k^2$. Using Theorem \ref{thmdjjuud} we deduce that the results of Theorem \ref{tmshjgeydg} hold true if for $k\in \{1,\ldots,q\}$  we have (1) $C  e^{-\rho_{j}/C} C^{j-k}H^{-(1+d/2)-j (2+3d+\zeta_1)/2}\leq  H^{3-k} \epsilon/(k^2 j^2)$ for $k\leq j \leq q-1$ and (2) $C  C^{q-k}\er(q)\leq  H^{3-k} \epsilon/(k^2 q^2)$. We conclude after simplification.
\end{proof}

We will now derive Condition \ref{condexpdecay} from results conditions presented in Section \ref{secexpdecloc}.

\begin{Construction}\label{condlocopbisC}
Let Condition \ref{cond7fyf} be satisfied.
Let  $(\B_{i}^{(k)},\|\cdot\|_{i}^{(k)})$ be (non empty) Banach spaces indexed by $k\in \{1,\ldots,q\}$ and $i\in \bar{\I}^{(k)}$
 such that  (1) $\B=\sum_{i\in \bar{\I}^{(k)}}\B_{i}^{(k)}$ for  $k\in \{1,\ldots,q\}$ (2) $\|\cdot\|_{i}^{(k)}$ is the norm induced by $\|\cdot\|$ on $\B_{i}^{(k)} \subset \B$ (3) For each $k\in \{1,\ldots,q\}$ and $(i,\alpha)\in \bar{\I}^{(k)}\times \aleph$, there exists a $\psi\in \B_{i}^{(k)}$ such that $[\phi_{j,\beta}^{(k)},\psi]=\delta_{i,j}\delta_{\alpha,\beta}$ for $(j,\beta)\in \bar{\I}^{(k)}\times \aleph$.
\end{Construction}

For $k\in \{1,\ldots,q\}$ write $\V^{(k),\perp}:=\{\psi \in \B\mid [\phi_j^{(k)},\psi]=0\text{ for }j\in \I^{(k)}\}$.
for $i\in \bar{\I}^{(k)}$ write
$\V^{(k),\perp}_i:=\B_{i}^{(k)} \cap \V^{(k),\perp}$.
Write $\|\cdot\|_{*,i}^{(k)}$ the norm induced by $\|\cdot\|_{i}^{(k)}$ on $\B_{i}^{*,(k)}$ the dual space of $\B_{i}^{(k)}$ (i.e. $\|\varphi\|_{*,i}^{(k)}:=\sup_{\psi \in \B_{i}^{(k)}}\frac{[\varphi,\psi]}{\|\psi\|_{i}^{(k)}}$ for $\varphi\in \B_{i}^{*,(k)}$).
For $i\in \bar{\I}^{(k)}$ and $\varphi\in \B^*$ let $\Rc_i^{(k)} \varphi$
 be the unique element of $\B_{i}^{*,(k)}$ such that
$[\varphi,\psi]=[\Rc_i^{(k)} \varphi, \psi]$ for $\psi \in \B_{i}^{(k)}$. Note that $\Rc_i^{(k)} \varphi$ is obtained by restricting the action of $\varphi$ to $\B_{i}^{(k)}$.
For $i\in \bar{\I}^{(k)}$ let $\I_i^{(k)}:=\{j\in \I^{(k)}\mid \Rc_i^{(k)} \phi_j^{(k)} \not=0\}$. Assume that  Construction \ref{condlocopbisC} is such that $\{i\}\times \aleph \subset  \I_i^{(k)}$ and the elements $\{\Rc_i^{(k)} \phi_j\mid j\in \I_i^{(k)}\}$ are linearly independent. Let $(\B_{0,i}^{(k)},\|\cdot\|_{0,i}^{(k)})$ be a Banach subspace of $\B_{i}^{*,(k)}$ such that the natural embedding $\B_{0,i}^{(k)} \rightarrow \B_{i}^{*,(k)}$ is compact and dense.
For $i\in\bar{I}^{(k)}$, let
\begin{equation}\label{eqgam1lockk}
\ubar{\gamma}_i^{(k)}:=\inf_{x\in \R^{\I_i^{(k)}}} \frac{\big(\| \sum_{j\in \I_i^{(k)}} x_j  \,  \Rc_i^{(k)} \phi_j^{(k)} \|_{0,i}^{(k)}\big)^2}{|x|^2},
\end{equation}
and
\begin{equation}\label{equbarhklockk}
\ubar{H}_i^{(k)}:=\inf_{\phi\in \Phi^{(k)}: \Rc_i^{(k)} \phi\not=0} \frac{\|\Rc_i^{(k)} \phi \|_{*,i}^{(k)}}{\|\Rc_i^{(k)} \phi \|_{0,i}^{(k)}} \,.
\end{equation}

\begin{Condition}\label{condlocopbis}
(1) $\V^{(k),\perp}=\sum_{i\in \bar{\I}^{(k)}}\V^{(k),\perp}_i$ for $k\in \{1,\ldots,q\}$  (2) There exists a constant $\bar{C}_{\loc}\geq 1$ independent from $k$ and $q$ such that  for  $k\in \{1,\ldots,q\}$ and $\varphi\in \B^*$,
\begin{equation}
\frac{1}{\bar{C}_{\loc}}\, \inf_{\phi\in \Phi^{(k)}}\|\varphi-\phi\|_*^2 \leq \sum_{i\in \bar{\I}^{(k)}} \inf_{\phi\in \Phi^{(k)}}\big(\|\Rc_i^{(k)} (\varphi-\phi)\|_{*,i}^{(k)}\big)^2\leq \bar{C}_{\loc}\, \inf_{\phi\in \Phi^{(k)}}\|\varphi-\phi\|_*^2 \,,
\end{equation}
(3)
$\bar{C}_{\loc}^{-1}\leq \min_{i\in \bar{\I}^{(k)}} \ubar{\gamma}_i^{(k)}$ and
and $ \bar{C}_{\loc}^{-1} H^k\leq \min_{i\in \bar{\I}^{(k)}} \ubar{H}_i^{(k)}$ where $H$ is the parameter of Condition \ref{cond1OR} or Condition \ref{conddiscrip3ordismatdis}.
\end{Condition}

Let $\C^{(k)}$ be the $\bar{\I}^{(k)}\times \bar{\I}^{(k)}$ (connectivity) matrix defined by $\C^{(k)}_{i,j}=1$ if there exists $(\chi_i,\chi_j) \in \B_{i}^{(k)}\times \B_{j}^{(k)}$ such that $\<\chi_j,\chi_i\>\not=0$ and $\C_{i,j}=0$ otherwise.
Let $\db^{(k)}:=\db^{\C^{(k)}}$ be the graph distance associated with $\C^{(k)}$ (see Definition \ref{defgraphmatdistbis}).
We extend $\db^{(k)}$ to a pseudo-metric on $\I^{(k)}\times \I^{(k)}$ by $\db^{(k)}\big((i,\alpha),(j,\beta)\big):=\db^{(k)}(i,j)$
for $(i,\alpha),(j,\beta) \in \bar{\I}^{(k)}\times \aleph$.

\begin{Theorem}\label{thmdkldjhdj}
Let Condition \ref{cond7fyf} be satisfied.
Let Condition \ref{cond1OR} or Condition \ref{conddiscrip3ordismatdis} be satisfied.
Condition \ref{condlocopbis} imply Items (3) and (4) of Condition \ref{condexpdecay} with $\zeta_1=\zeta_2=1$. More precisely,
$|A^{(k)}_{i,j}|\leq C_{1} \exp\big(-C_2 \db^{(k)}(i,j)\big)$ for $(i,j)\in \I^{(k)}\times \I^{(k)}$
with $C_2=\frac{1}{2} \ln \frac{\bar{C}_{\loc}^2+1}{\bar{C}_{\loc}^2-1}$ and
$C_1=\frac{ C_{\Phi}^{3}}{H^{2k}} \frac{\bar{C}_{\loc}^2+1}{\bar{C}_{\loc}^2-1}$. Furthermore,
$R^{(k,k+1)}_{i,j}\leq C_3 \exp\big(-2 C_2 \db^{(k)}(i,j)\big)$ with $C_3=\frac{C_{\Phi}^{3} }{H^k} \exp\big(4 C_2\big)$.
\end{Theorem}
\begin{proof}
Condition \ref{condlocopbis} implies that Condition \ref{constlocop} holds at each level $k\in \{1,\ldots,q\}$. Theorem
\ref{thmhgguyg65} implies the bound on $|A^{(k)}_{i,j}|$. The bound on $R^{(k)}_{i,j}$ is a direct consequence of Theorem
\ref{thmjgjuygg67} by recalling (equation \eqref{eqhjgjhgjgjg}) that $R^{(k)}_{i,j}=[\phi_j^{(k+1)},\psi_i^{(k)}]$ and using
$\|\phi_j^{(k+1)}\|_*^2\leq \lambda_{\max}(\Theta^{(k+1)})\leq 1/\lambda_{\min}(A^{(k+1)})$ with Lemma \ref{lemljkhlkhhu}.
\end{proof}

\subsection{Proof of Theorem \ref{tmdiscrete1}}\label{subsekjkjckjgg76}

To obtain Theorem \ref{tmdiscrete1} we express Condition \ref{condlocopbis} in the discrete setting  of Subsection \ref{subsecgamdis} where $\B$ is replaced by a discrete subspace $\B^\d$ introduced in \eqref{eqkdklrflkff}, and level $q$ gamblets and measurement functions are defined as in \eqref{eqjhddhghd} and \eqref{eqkjdhkdhstpsiphitil}. We  assume that  Condition \ref{cond7fyf} is satisfied.
As in the proof of Theorem \ref{thmconddisbndisbismatdis} the proof can be obtained by considering the Banach space $\bar{\B}=\R^{\I^{(q)}}$
gifted with the norm $\|x\|^2=x^T A x$ for $x\in \bar{\B}$. Recall that the induced dual space is
$\bar{\B}^*=\R^{\I^{(q)}}$ gifted with the norm $\|x\|^2_*=x^T A^{-1} x$ for $x\in \bar{\B}^*$. We also select $\B_0=\R^{\I^{(q)}}$  gifted with the norm $\|x\|_{0}^2=x^T x$ for $x\in \B_0$. Defining $\phi_i^{(q)}=e_i$ as the unit vector of $\R^{\I^{(q)}}$ in the direction $i$ we then have $\psi_i^{(q)}=e_i$. Condition \ref{cond1OR} naturally translates into Condition \ref{conddiscrip3ordismatdis}.
For $k\in \{1,\ldots,q\}$ note that $\V^{(k),\perp}=\Ker(\pi^{(k,q)})$.
For $k\in \{1,\ldots,q\}$ and $i\in \I^{(k)}$ we select
$\B_{i}^{(k)}=\{x\in \R^{\I^{(q)}}\mid x_j=0 \text{ for }j\not\in \I^{(k,q)}_i\}$.
Using Lemma \ref{lemddjkhkdd}, Condition \ref{condlocopbis} and Condition \ref{condexpdecay} are  implied by Condition \ref{condilwhiuhd}.
The proof of Theorem \ref{tmdiscrete1} is then a straightforward application of Theorem \ref{thmdkldjhdj}, Theorem \ref{tmshjgeydg} and
the following lemma.

\begin{Lemma}\label{propshjgeydg}
The results of Theorem \ref{tmshjgeydg} remain true if $A^{(k-1),\loc}$ is truncated as in Line \ref{step13gf} of Algorithm
\ref{fastgambletsolvecase1g} and definition \ref{deftrunc}.
\end{Lemma}
\begin{proof}
Due to the locality of the variational formulation \eqref{eqhihdjkjhiudiduh} the truncation step does not impact the computation of the gamblets. Although this truncation step affects the accuracy of the computation of subband solutions in line \ref{step8gf}, the exponential decay of the entries of $A^{(k)}$, Lemma \ref{lembase} and Theorem \ref{thmdjjuud} imply that this corresponding error remains below the bounds obtained in Theorem \ref{tmshjgeydg} under the assumptions of Theorem \ref{thmdjjuud} on $\rho_k$.
\end{proof}

\subsection{Equivalent conditions for localization}\label{subseckjgg76}

Consider the setting and notations of Subsection \ref{subsekjkjckjgg76}.
Let
\begin{equation}\label{eqjhgjhgjg}
S^{(k)}:=A^{-1}  - A^{-1}\pi^{(q,k)}\Theta^{(k)}\pi^{(k,q)} A^{-1}
\end{equation}
The following lemma is the discrete version of Lemma \ref{lemjhgjyguyguybis}.
\begin{Lemma}\label{lemuyg76ggkhj}
For $x\in \R^{\I^{(q)}}$ there exists a unique $y\in \R^{\I^{(k)}}$ and a unique $z\in \Ker(\pi^{(k,q)})$ such that
\begin{equation}\label{equiudkhlkdhldkjsh}
x=\pi^{(q,k)} y+ A z
\end{equation}
Moreover $z=S^{(k)}x$, where $S^{(k)}$, defined in \eqref{eqjhgjhgjg}, is symmetric, positive, definite and defines a bijection from $\Ker(\pi^{(k,q)})$ onto itself  that  is the inverse of $P_{\Ker(\pi^{(k,q)})} A P_{\Ker(\pi^{(k,q)})}$ in $\Ker(\pi^{(k,q)})$ (writing $P_{\Ker(\pi^{(k,q)})}:=I-\pi^{(q,k)}\bar{\pi}^{(k,q)}$ the orthogonal projection onto $\Ker(\pi^{(k,q)})$). More precisely, for $x\in \Ker(\pi^{(k,q)})$ we have $x=P_{\Ker(\pi^{(k,q)})} A P_{\Ker(\pi^{(k,q)})}z=P_{\Ker(\pi^{(k,q)})} A P_{\Ker(\pi^{(k,q)})} S^{(k)}x$.
Furthermore, for $x\in \R^{\I^{(q)}}$,  $y$ is the minimizer of $(x-\pi^{(q,k)} y)^T A^{-1} (x-\pi^{(q,k)} y)$ subject to $y\in \R^{\I^{(k)}}$ and
$z$ is the minimizer of $(x-A z)^T A^{-1} (x-A z)$ subject to $z\in \Ker(\pi^{(k,q)})$
\end{Lemma}
\begin{proof}
The proof is similar to that of Lemma \ref{lemjhgjyguyguybis}.
Note that $\pi^{(k,q)} A^{-1}x=\pi^{(k,q)} A^{-1}\pi^{(q,k)} y$ and $y=\Theta^{(k)}\pi^{(k,q)} A^{-1}x$. Therefore,
$z=A^{-1}  - A^{-1}\pi^{(q,k)}\Theta^{(k)}\pi^{(k,q)} A^{-1} x $.
Note also that
$\bar{\pi}^{(k,q)}x=y+\bar{\pi}^{(k,q)} A z$ implies
$P_{\Ker}(\pi^{(k,q)})x=P_{\Ker(\pi^{(k,q)})} A P_{\Ker(\pi^{(k,q)})} z$. Note that the positivity of
$S^{(k)}$ is implied by $x^T S^{(k)}x=(P_{\Ker(\pi^{(k,q)})}S^{(k)}x)^T A (P_{\Ker(\pi^{(k,q)})} S^{(k)}x)$.
Note also that  $z\in \Ker(\pi^{(k,q)})$ determines a unique $y\in \R^{\I^{(k)}}$ and a unique $x\in \Ker(\pi^{(k,q)})$
such that \eqref{equiudkhlkdhldkjsh} holds. Indeed in that case we have $y=-\bar{\pi}^{(k,q)} A z$ and
$x=P_{\Ker(\pi^{(k,q)})} A z$. The variational properties of $y$ and $z$ are straightforward by observing that at the minimum one must have $A^{-1} (x-\pi^{(q,k)} y)\in \Ker(\pi^{(k,q)})$ and $x-A z\in \Img(\pi^{(q,k)})$.
\end{proof}

For $i\in \bar{\I}^{(k)}$ let $\pi^{(k,q),i}$ be the $\I^{(k)}_i\times \I^{(k,q)}_i$ matrix defined by the restriction of
$\pi^{(k,q)}$ to these subset of indices. Define $\pi^{(q,k),i}:=(\pi^{(k,q),i})^T$. Let $\bar{\pi}^{(k,q),i}$ be the pseudo-inverse of
$\pi^{(q,k),i}$ (i.e., $\bar{\pi}^{(k,q),i}=(\pi^{(k,q),i}\pi^{(q,k),i})^{-1}\pi^{(k,q),i}$).

Let  $S^{(k),i}$ be the localized version of $S^{(k)}$ defined as follows:
\begin{equation}\label{eqjhgjhgjgloc}
S^{(k),i}:=A^{i,-1}  - A^{i,-1}\pi^{(q,k),i}(\pi^{(k,q),i} A^{i,-1}\pi^{(q,k),i})^{-1}\pi^{(k,q),i} A^{i,-1}
\end{equation}
The proof of the following lemma is similar to that of Lemma \ref{lemuyg76ggkhj}.
\begin{Lemma}\label{lemuyg76ggkhjlocal}
Let $k\in \{1,\ldots,q\}$ and $i\in \bar{\I}^{(k)}$.
For $x\in \R^{\I^{(k,q)}_i}$ there exists a unique $y\in \R^{\I^{(k)}_i}$ and a unique $z\in \Ker(\pi^{(k,q),i})$ such that
\begin{equation}\label{equiudkhlkdhldkjshloc}
x=\pi^{(q,k),i} y+ A^i z
\end{equation}
Moreover $z=S^{(k),i}x$, where $S^{(k),i}$, defined in \eqref{eqjhgjhgjgloc}, is symmetric, positive, definite and defines a bijection from $\Ker(\pi^{(k,q),i})$ onto itself  that  is the inverse of $P_{\Ker(\pi^{(k,q),i})} A^i P_{\Ker(\pi^{(k,q),i})}$ in $\Ker(\pi^{(k,q),i})$ (writing $P_{\Ker(\pi^{(k,q),i})}:=I-\pi^{(q,k),i}\bar{\pi}^{(k,q),i}$ the orthogonal projection onto $\Ker(\pi^{(k,q),i})$). More precisely, for $x\in \Ker(\pi^{(k,q),i})$ we have $x=P_{\Ker(\pi^{(k,q),i})} A^i P_{\Ker(\pi^{(k,q),i})}z=P_{\Ker(\pi^{(k,q),i})} A P_{\Ker(\pi^{(k,q),i})} S^{(k),i}x$.
Furthermore, for $x\in \R^{\I^{(k,q)}_i}$,  $y$ is the minimizer of $(x-\pi^{(q,k),i} y)^T A^{i,-1} (x-\pi^{(q,k),i} y)$ subject to $y\in \R^{\I^{(k)}_i}$ and
$z$ is the minimizer of $(x-A^i z)^T A^{i,-1} (x-A^i z)$ subject to $z\in \Ker(\pi^{(k,q),i})$.
\end{Lemma}

The following Theorem could be seen as a  discrete version of Lemma \ref{lempropjguyug6}.
\begin{Theorem}\label{thmequivcondbis}
Item  \ref{itt2} of Condition \ref{condilwhiuhd}  is equivalent to
\begin{equation}\label{eqkjhgguygguy}
\frac{1}{C_{\loc,1}} S^{(k)} \leq \sum_{i\in \bar{\I}^{(k)}} \Pr_i^{(k,q)} S^{(k),i}\Pr_i^{(k,q)} \leq C_{\loc,2} S^{(k)}\,.
\end{equation}
\end{Theorem}
\begin{proof}
The equivalence with \eqref{eqkjhgguygguy} is a direct consequence of the variational property of $y$ in lemmas \ref{lemuyg76ggkhj}
and \ref{lemuyg76ggkhjlocal}.
\end{proof}

\subsubsection{Proof of Theorem \ref{thmequivcond}}

The equivalence with \eqref{eqldodjdoijebisloc} follows by observing (using Lemma \ref{lemuyg76ggkhj} and also Lemma \ref{lemuyg76ggkhjlocal} for the localized version) that
for $x\in \R^{\I^{(q)}}$ (writing $z$ the component in \eqref{equiudkhlkdhldkjsh}), $\inf_{y\in \R^{\I^{(k)}}}(x-\pi^{(q,k)} y)^T A^{-1}(x-\pi^{(q,k)} y)=z^TA z=\sup_{z'\in \Ker(\pi^{(k,q)})} \frac{(z^T A z')^2}{(z')^T A z'}=\sup_{z'\in \Ker(\pi^{(k,q)})} \frac{(x^T z')^2}{(z')^T A z'}$.

\subsubsection{Proof of Theorem \ref{thmdlkjdhjdh}}
The proof  is similar to that of Lemma \ref{lemdkjdhjh3e} and Proposition \ref{propkjshkdjhdkjh}.

\section{Proofs of the results of Section \ref{sec8}}\label{sec8proof}
\subsection{Proof of Theorem \ref{thmsudgdygdgyphi}}

Let $\bar{\Phi}_m:=\operatorname{span}\{\bar{\phi}_1,\ldots,\bar{\phi}_m\}$ and $\bar{\Phi}_m^c$ be the orthogonal complement of $\bar{\Phi}_m$ in $\B^*$ with respect to the inner product $\<\cdot,\cdot\>_*$. Let $\phi=\phi_a+\phi_b$ be the corresponding orthogonal decomposition of $\phi\in \B^*$ into $\phi_a\in \bar{\Phi}_m$ and $\phi_b\in \bar{\Phi}_m^c$.  The constrains in \eqref{eqhihdiudiduh} imply $\phi_a=\sum_{i=1}^m c_i \phi_i$ and $\|\phi\|^2_*=\|\phi_a\|_*^2+\|\phi_b\|_*^2$ implies that the minimum of \eqref{eqhihdiudiduhphi} is achieved at $\phi_b=0$.

\subsection{Proof of Theorem \ref{thmoptgalerkinphi}}

The proof follows trivially by observing that $\phi-\phi^\two(\phi)$ is orthogonal (in $\B^*$ with respect to the $\<,\>_*$ inner product) to $Q^{-1} \bar{\psi}_1,\ldots,Q^{-1} \bar{\psi}_m$.

\subsection{Proof of Proposition \ref{propdkfjffh}}

Using \eqref{eqjkhdkdh} and $[\phi_i^{(k)},\psi^{(k)}_j]=\delta_{i,j}$ we have (using Proposition \ref{propfund})
$[\phi^{(k),\chi}_i,\chi^{(k)}_j]=  (N^{(k),T} W^{(k),T})_{i,j}=\delta_{i,j}$. Using \eqref{eq:ftfytftfx} and
Theorem \ref{thm38dgdn} we have (using Proposition \ref{propfund})
$[\phi^{(k),\chi}_i,\psi^{(k-1)}_l]=(R^{(k-1,k)}N^{(k)})_{l,i}=0$. $[\phi^{(k)}_i,\chi^{(k')}_j]=0$ for $k'>k$ and $i\in \I^{(k)}$ implies $[\phi^{(k),\chi}_i,\chi^{(k')}_j]=0$ for $k'>k$ and $i\in \J^{(k)}$. For $k'<k$, $\chi^{(k')}_j\in \Span\{\psi^{(k-1)}_l\mid l\in \I^{(k-1)}\}$ and the second equality in \eqref{eqjgjhgyjh} imply  $[\phi^{(k),\chi}_i,\chi^{(k')}_j]=0$ for $(i,j)\in \J^{(k)}\times \J^{(k')}$.
For $z\in R^{\J^{(k)}}$, $\sum_{j\in \J^{(k)}} z_j \phi_j^{(k),\chi}=\sum_{i\in \I^{(k)}} (N^{(k)} z)_i \phi_i^{(k)}$ and
\eqref{eqgam1} implies that
\begin{equation}
\ubar{\gamma}_k |z|^2 \lambda_{\min}(N^{(k),T}N^{(k)}) \leq \|\sum_{j\in \J^{(k)}} z_j  \phi_j^{(k),\chi}\|_0^2 \leq \bar{\gamma}_k |z|^2 \lambda_{\max}(N^{(k),T}N^{(k)}),
\end{equation}
which combined with Theorem \ref{lemdjkdj} implies the assertion \eqref{eqidyiuyihs} under the Conditions \ref{cond1OR}.

\subsection{Proof of Proposition \ref{propdkfjffother}}

According to Theorem \ref{thmsudgdygdgyphi} the unique minimizer of \eqref{eqhihdiudiduhphichior} is $\phi=\sum_{i\in \J^{(k)}} c_i \bar{\phi}_i$ with
$\bar{\phi}_i=\sum_{j\in \J^{(k)}} \bar{A}^{-1}_{i,j}Q^{-1} \chi_j^{(k)}$ with $\bar{A}_{i,j}=[Q^{-1}\chi_i^{(k)}, \chi_j^{(k)}]=B^{(k)}_{i,j}$. Observing that $Q^{-1} \chi_j^{(k)}=\sum_{l\in \I^{(k)}}(W^{(k)}A^{(k)})_{j,l}\phi_l^{(k)}$ we deduce that $\bar{\phi}_i=\phi^{(k),\chi}_i$.
As in the proof of Theorem \ref{thmpsdk} the dimension of the space spanned by $\{\sum_{j\in \I^{(k)}} \bar{W}_{s,j} \phi_j^{(k)}|s\in \J^{(k)}\}$ plus the dimension of $\Phi^{(k-1)}$ is equal to the dimension of $\Phi^{(k)}$. Furthermore
$\sum_{s\in \J^{(k)}} x_s \sum_{j\in \I^{(k)}} \bar{W}_{s,j} \phi_j^{(k)} -\sum_{l\in \I^{(k-1)}} y_l \phi_l^{(k-1)}=0$ implies (by pairing the equation against $\chi^{(k)}_.$ and $\psi^{(k-1)}_.$) $x=0$ and $y=0$. Therefore there exists $x\in \R^{\J^{(k)}}$ and $y\in \R^{\I^{(k-1)}}$ such that
$\phi_i^{(k),\chi}=\sum_{s\in \J^{(k)}} x_s \sum_{j\in \I^{(k)}} \bar{W}_{s,j} \phi_j^{(k)} -\sum_{l\in \I^{(k-1)}} y_l \phi_l^{(k-1)}$. The constrain $[\phi^{(k),\chi}_i,\chi^{(k)}_j]=\delta_{i,j}$ leads to $x_s=\delta_{i,s}$.
The equation $[\phi^{(k),\chi}_i,\psi^{(k-1)}_l]=0$ leads to $y_l=(R^{(k-1,k)}\bar{W}^{(k),T})_{l,i}$. We conclude by observing that
$\bar{W}^{(k)}R^{(k,k-1)}=-N^{(k),T}\bar{\pi}^{(k,k-1)}$.  \eqref{eqjskdkddj} is a direct consequence of \eqref{eqmhjgjhgy} and \eqref{eqhdgudgu}.

Writing  $\phi=\sum_{l\in \I^{(k-1)}} z_l \phi_l^{(k-1)}$ in \eqref{eqhihdiudiduhphichiorbis} one obtains that $\|\sum_{j\in \I^{(k)}} \bar{W}^{(k)}_{i,j}\phi_j^{(k)}-\phi\|_*^2 = z^T \Theta^{(k-1)}z- 2 z^T \pi^{(k-1,k)}\Theta^{(k)}\bar{W}^{(k),T} e_i+
e_i^T \bar{W}^{(k)}\Theta^{(k)}\bar{W}^{(k),T} e_i$. Therefore, the minimum is achieved for
\begin{equation}
z=A^{(k-1)} \pi^{(k-1,k)}\Theta^{(k)}\bar{W}^{(k),T} e_i
\end{equation}
which leads to $\bar{\pi}^{(k-1,k)} N^{(k)}=-A^{(k-1)} \pi^{(k-1,k)}\Theta^{(k)}\bar{W}^{(k),T}$, i.e.
$\bar{\pi}^{(k-1,k)}A^{(k)} W^{(k),T} B^{(k),-1}= -A^{(k-1)} \pi^{(k-1,k)}\Theta^{(k)}\bar{W}^{(k),T}$.

\subsection{Proof of Proposition \ref{lemddkjhed}}

Using
$\frac{\big|\<\chi_1,\chi_2\>\big|}{\|\chi_1\| \|\chi_2\|}=\inf_{\psi \perp \chi_2}\|\chi_1-\psi\|/\|\chi_1\|$ we obtain (using $\inf_{\psi \perp \chi_2}\|\chi_1-\psi\|/\|\chi_1\|=\inf_{\psi \perp \chi_2}\inf_{t\in \R}\|\chi_1-t\psi\|/\|\chi_1\|$) that $\frac{|\<\chi_1,\chi_2\>|}{\|\chi_1\| \|\chi_2\|}=\sqrt{1-I^2}$
with $I:=\sup_{\psi \perp \chi_2} \frac{\<\chi_1,\psi\>}{\|\chi_1\| \|\psi\|}$. For $\chi_1=\sum_{i\in \S_1} y_i \chi_i^{(k)}$
take $\psi=\sum_{i\in \S_1} y_i Q \phi_i^{(k),\chi}+ Q \phi$ with $\phi \in \Phi^{(k-1)}$.
Since $\S_1$ and $\S_2$ are disjoint, Proposition \ref{propdkfjffh} implies that $[\phi_l^{(k),\chi},\chi_2]=0$  for $l\in \S_1$, so  it follows that
$\psi \perp \chi_2$.   Using $\<\psi,\chi_1\>=|y|^2$ (note that Proposition \ref{propdkfjffh} also implies $[\phi,\chi_1]=0$) we deduce that
$I^2 \geq \frac{|y|^4 }{y^T B^{(k)} y \|\psi\|^2 }$. Since
$\|\psi\|=\|\sum_{j\in \S_1} y_{j} Q \phi_j^{(k),\chi}+ Q \phi \|=
\|\sum_{j\in \S_1} y_{j} \phi_j^{(k),\chi}+ \phi \|_* $, taking the infimum over $\phi \in \Phi^{(k-1)}$ we obtain from \eqref{eqbarhk}  that $\|\psi\|\leq  \bar{H}_{k-1} \|\sum_{j\in \S_1} y_{j} \phi_j^{(k),\chi} \|_0$. Therefore \eqref{eqidyiuyihs} and Conditions \ref{cond1OR} imply that $\|\psi\|^2\leq  C H^{2(k-1)} |y|^2$.
Using Corollary \ref{corunbcnOR} to bound $\lambda_{\max}(B^{(k)})$ we deduce that $I^2\geq \frac{H^2 }{ C} $. Summarizing we conclude the proof using the inequality $\sqrt{1-x^2}\leq 1-x^2/2$.

\subsection{Proof of Theorem \ref{corunbcnORd}}
The proof of Theorem \ref{corunbcnORd} is similar to that of Theorem \ref{corunbcnOR}.
$\frac{1}{C_\Phi} H^k \leq  \frac{\|\phi \|_*}{|x|}$ for  $x\in \R^{\I^{(k)}}$ and
$\phi=\sum_{i\in \I^{(k)}} x_i \phi_i^{(k)}$ is equivalent to $C_\Phi^{-2} H^{2k} |x|^2\leq x^T \Theta^{(k)} x $ which implies that
$\lambda_{\max}(A^{(k)})\leq C_\Phi^{2} H^{-2k}$ for $k\in \{1,\ldots,q\}$. Similarly,
$\frac{\| \phi\|_*}{|x|}\leq C_{\Phi}$ for $\phi=\sum_{i\in \I^{(1)}} x_i \phi_i^{(1)}$
and $x\in \R^{\I^{(1)}}$ implies that $\lambda_{\min}(A^{(1)})\geq C_{\Phi}^{-2}$.
For $k\in \{2,\ldots,q\}$ and $z\in \R^{\J^{(k)}}$ we have
$z^T B^{(k)} z=\|\sum_{j\in \J^{(k)}} z_j \chi_j^{(k)}\|$. Therefore using the orthogonality between $\V^{(k-1)}$ and $\W^{(k)}$ we deduce that
\begin{equation}\label{eqkjdhkjhds}
z^T B^{(k)} z=\inf_{y\in \R^{\I^{(k-1)}}}
\|\sum_{j\in \J^{(k)}} z_j \chi_j^{(k)}-\sum_{i\in \I^{(k-1)}}y_i \psi_i^{(k-1)}\|^2.
\end{equation}
Using $
\|\sum_{j\in \J^{(k)}} z_j \chi_j^{(k)}-\sum_{i\in \I^{(k-1)}}y_i \psi_i^{(k-1)}\|
=$\\$\|\sum_{i\in \I^{(k)}} (z^T W^{(k)} A^{(k)})_i \phi_i^{(k)}-\sum_{i\in \I^{(k-1)}} (y^T A^{(k-1)})_i \phi_i^{(k-1)}\|_*$, Item \ref{itdftf} of Condition \ref{cond1ORdvar1} implies that
\begin{equation}\label{eqkdjhkehddhuie}
z^T B^{(k)} z \leq C_{\Phi}^2 H^{-2(k-1)} |A^{(k)}W^{(k),T}z|^2.
\end{equation}
Writing $N^{(k)}$ as in \eqref{eqjdhiudhiue}
 we have, as in \eqref{eqjgjhdjhgdy}, obtained that
\begin{equation}\label{eqjgjhdjhgdyvar1}
\frac{  C_{\Phi}^{-2}H^{-2(k-1)} }{\lambda_{\max}(N^{(k),T}N^{(k)})} \leq \lambda_{\min}(B^{(k)})\,.
\end{equation}
The remaining part of the proof is similar to that of Theorem \ref{thmodhehiudhehd}. In particular,
to control the l.h.s. of \eqref{eqjgjhdjhgdyvar1}, we use Lemma \ref{lemfdhgdf}, Lemma \ref{lemdjoidjdi}
and use  Item \ref{itsjhdhjdgdj} of Condition \ref{cond1ORdvar1} in the proof of
 Lemma \ref{lemddjoj3ir}.

\subsection{Proof of Lemma \ref{lemddjkhkdd}}

The first equality of \eqref{eqjhkug5uy} follows from \eqref{eqkjdhkdhstpsiphitil},
$\|\sum_{i\in \I^{(q)}} z_i \phi_i^{(q)}\|_{*,\d}=\sup_{y\in \R^{\I^{(q)}}}\frac{z^T  y}{\sqrt{y^T A y}}$
 and the Cauchy-Schwarz inequality $x^T y\leq \sqrt{y^T A y}\sqrt{x^T A^{-1} x}$.
Since  Theorem \ref{thmsudgdygdgyphi}  asserts
that the minimum  of the righthand side of  \eqref{eqjhkug5uy}
 is achieved at $\phi=\sum_{i\in \I^{(q)}} (A^{-1}x)_i Q^{-1}\varPsi_i$, the second
equality of  \eqref{eqjhkug5uy} follows by direct calculation.
For  \eqref{eqkddkhd} observe that the minimum is achieved for $\psi=\sum_{i\in \I^{(q)}} (A^{-1}x)_i \varPsi_i$.

\paragraph{Acknowledgments.}
The authors gratefully acknowledges this work supported by  the Air Force Office of Scientific Research and the DARPA EQUiPS Program under
award   number FA9550-16-1-0054 (Computational Information Games).

\bibliographystyle{plain}
\bibliography{merged,RPS,extra}

\end{document}